\documentclass[reqno,a4paper,11pt]{amsart}
\usepackage[top=32mm, bottom=32mm, left=32mm, right=32mm,height = 615pt, width = 360pt]{geometry}
\usepackage{mathrsfs,enumerate}
\usepackage{amssymb}
\usepackage[colorlinks=true]{hyperref}
\allowdisplaybreaks

\usepackage{bm}

\usepackage{tikz}
\usetikzlibrary{positioning}
\usepackage{mathtools}
\usepackage[arrow]{xy}
\usepackage{extarrows}

\usepackage{float}
\usetikzlibrary{graphs,trees,fadings,patterns}
\usetikzlibrary{circuits.ee}
\usetikzlibrary {shapes.gates.logic.IEC}
\usetikzlibrary {circuits.logic.US}

\newcommand{\ld}{\lambda}
\newcommand{\ga}{\gamma}
\newcommand{\Ga}{\Gamma}
\newcommand{\al}{\alpha}
\newcommand{\be}{\beta}
\newcommand{\de}{\delta}
\newcommand{\De}{\Delta}
\newcommand{\si}{\sigma}
\newcommand{\ep}{\epsilon}
\newcommand{\vep}{\varepsilon}
\newcommand{\Om}{\Omega}

\newcommand{\Ld}{\Lambda}

\newcommand{\tht}{\theta}

\newcommand{\cc}{\mathbb{C}}
\newcommand{\cd}{\mathbb{D}}
\newcommand{\zz}{\mathbb{Z}}
\newcommand{\rr}{\mathbb{R}}
\newcommand{\vv}{\mathbb{V}}
\newcommand{\nn}{\mathbb{N}}

\newcommand{\qq}{\mathbb{Q}}
\newcommand{\ol}{\overline}

\newcommand{\wt}{\widetilde}
\newcommand{\wh}{\widehat}

\newcommand{\beqq}{\begin{equation*}}
\newcommand{\eeqq}{\end{equation*}}
\newcommand{\beq}{\begin{equation}}
\newcommand{\eeq}{\end{equation}}

\newcommand{\beqan}{\begin{align*}}
\newcommand{\eeqan}{\end{align*}}
\newcommand{\beqa}{\begin{align}}
\newcommand{\eeqa}{\end{align}}

\newtheorem{theorem}{Theorem}[section]
\newtheorem{lemma}{Lemma}[section]
\newtheorem{sublemma}{Sublemma}[lemma]

\newtheorem{definition}{Definition}[section]
\newtheorem{proposition}{Proposition}[section]

\newtheorem{corollary}{Corollary}[theorem]
\newtheorem{question}{Question}[section]

\newtheorem{remark}{Remark}[section]

\numberwithin{equation}{section}

\begin{document}
\title[Parabolic and near-parabolic renormalization]{The parabolic and near-parabolic renormalization for a class of polynomial maps and its applications}


\author[X. Zhang]{Xu Zhang}
\address[X. Zhang]{Department of Mathematics, Shandong University, Weihai, Shandong, 264209,  China}
\email{xu$\_$zhang$\_$sdu@mail.sdu.edu.cn}



\begin{abstract}
For a class of polynomial maps of one variable with a parabolic fixed points and degrees bigger than $21$, the parabolic renormalization is introduced based on Fatou coordinates and horn maps, and a type of maps which are invariant under the parabolic renormalization is also given. For the small perturbation of these kinds of maps, the near-parabolic renormalization is also introduced based on the first return maps defined on the fundamental regions. As an application, we show the existence of non-renormalizable polynomial maps with degrees bigger than  $21$ such that the Julia sets have positive Lebesgue measure and Cremer fixed points, this provides a positive answer for the classical Fatou conjecture (the existence of Julia set with positive area) with degrees bigger than $21$.
\end{abstract}

\maketitle

\section{Introduction}
Given a holomorphic function $f:\ \cc\to\cc$ defined on the complex plane $\cc$ and an initial point $z_0\in\cc$, a dynamical system is defined by the iteration  $z_{n+1}= f(z_n)$, $n\geq0$. The sequence $(z_n)_{n\in\nn}\subset\cc$ is called the orbit of $z_0$, denoted by $\text{Orb}(z_0,f)$. The filled-in Julia set $K_f$ is the set of points for which the orbit $\text{Orb}(z,f)$
is bounded in $\cc$. The boundary of $K_f$ is called  the Julia set $J_f$, which have empty interior and might have fractal geometric structure. The complementary of the Julia set is called the Fatou set. If there is a point $z_0\in\cc$ with $f(z_0)=z_0$, then $z_0$ is a fixed point, the multiplier is $f^{\prime}(z_0)$, the fixed point $z_0$ is called parabolic if $f^{\prime}(z_0)$ is a root of unity, $z_0$ is called an irrationally indifferent fixed point if $f^{\prime}(z_0)$ is not a root of unity and $|f^{\prime}(z_0)|=1$ (including the Siegel disk, which is conjugate with rotation,  this is called the linearizability problem, and the Cremer fixed point, which is the non-linearizable case), and the fixed point is called non-degenerate if $f^{\prime\prime}(z_0)\neq0$. The simple case $f^{\prime}(z_0)=1$ and $f^{\prime\prime}(z_0)\neq0$ is called the non-degenerate $1$-parabolic fixed point, and have simple dynamics, where the orbits are attracted towards $z_0$ on one side and repelled away on the other side. However, if there is perturbation of the parameter $f^{\prime}(z_0)$, then the complicated dynamics might appear, this is called the parabolic bifurcation.

A useful tool in the study of the parabolic bifurcation is the method of renormalization. Renormalization means that a particular subset of the dynamical plane is chosen, the return map restricted to this subset is considered such that some iterate of the original map sends this set to itself, and the return map (renormalization) bears a resemblance to the original map on the whole dynamical plane. Renormalization theories are very useful in the complex dynamics, for example, the double-period bifurcation in the unimodal interval maps in the work of Feigenbaum, Coullet-Tresser, and so on \cite{CoulletTresser1978, Feigenbaum1978, Feigenbaum1980, Lanford1982}; Yoccoz's renormalization in the proof of Siegel-Brjuno Theorem on the linearization of irrationally indifferent fixed points \cite{Yoccoz1995}, Inou and Shishikura's work on parabolic and near-parabolic renormalization of a cubic polynomial map \cite {Shishikura2000, InouShishikura2016}, and
some generalization of parabolic renormalization on a class of analytic maps by Lanford and Yampolsky \cite{LanfordYampolsky2012}. The renormalization method has been extended to the study of complex H\'{e}non map \cite{CarvalhoLyubichMartens2006, GaidashevYampolsky}.

An important problem is the study of the area of the Julia set, the analysis of which requires a series of tools from combinatorics,
complex analysis, dynamical systems, geometry, topology, and so on. The study of this problem has a fairly long history, since the work of Koenigs, Schr\"{o}der,
B\"{o}ttcher, and Fatou and Julia around the late 19th century and the beginning of 20th century \cite{Fatou1919,Fatou1920}. There are many mathematicians contributed to this problem, Cremer, Siegel, Brolin, Guckenheimer, Jakobson, Douady, Hubbard, Sullivan, Thurston, Lyubich, McMullen, Milnor, Shishikura, Yoccoz, Buff, Ch\'{e}ritat, Avila, Lyubich, and so on, where the names of these mathematicians are arranged by the years of their work.

Inspired by Ahlfors's conjecture on the area of limit sets of Kleinian groups, it is guessed that there might not exist Julia set of a polynomial with positive area. Many results show the non-existence of Julia sets with positive area. Petersen proved that the
Julia set of a quadratic polynomial with Siegel disc for the rotation number of bounded type is locally connected and
of Lebesgue measure zero \cite{Peterse1996}. Petersen and Zakeri proved that the Julia set of
$P(z)=e^{2\pi i\al}z + z^2$ has zero Lebesgue measure by trans-quasiconformal surgery and
David's theorem, where $\al\in[0,1]$ has a continued fraction expansion $\al=[a_1,a_2,a_3,...]$ with $\log a_n=O(a_n)$(see Subsection \ref{equ2021-2-3-3} for more details on this expansion) \cite{PetersenZakeri2004}. A major advancement on this problem is the construction of quadratic polynomials of Buff and Ch\'{e}ritat's examples of Siegel and
Cremer parameters, and of infinitely renormalizable parameters, with positive measure Julia
sets \cite{BuffCheritat2012}, which is based on the combination of Ch\'{e}ritat's work
on parabolic implosion \cite{Cheritat2002, Cheritata, Cheritatb}, Inou and Shishikura's work on parabolic and near-parabolic renormalization of a cubic polynomial map \cite{Shishikura2000, InouShishikura2016},  and some other work by Buff and Ch\'{e}ritat \cite{BuffCheritat2006, BuffCheritat2007}.
Later, Avila and Lyubich produced a positive
Hausdorff measure set of so-called Feigenbaum infinitely renormalizable quadratic polynomials
with positive measure Julia sets, commenting that it is still unknown if this phenomenon can
occur for real quadratic polynomials \cite{AvilaLyubich}. Another significant refinement of parabolic point theory
was developed by Shishikura to show that the Julia sets of some quadratic polynomials,
obtained by taking fast-increasing sequences of parabolic parameters (and therefore including
some Cremer parameters), have Hausdorff dimension two, and, correspondingly, the boundary
of the Mandelbrot set has Hausdorff dimension two \cite{Shishikura1998}.  McMullen subsequently proved that the Julia set has Hausdorff
dimension less than two and proved some results on ``space filling" of the Siegel disc under
``renormalization" (where this refers to the type of renormalizations used for circle maps, rather
than for polynomials), this is called McMullen's renormalization of Siegel disks of quadratic polynomials, where the rotation number is of bounded type \cite{McMullen1998}. In \cite{DudkoLyubich2020}, Dudko et al. provided the pacman renormalization theory, which combines the features of quadratic-like and Siegel renormalizations. In \cite{Cheritat2021}, Ch\'{e}ritat provided a construction of near-parabolic renormalization, extending the range of applications to unicritical polynomials of all degrees.

The concepts of parabolic and near-parabolic renormalization were introduced in \cite{Shishikura1998} to study the Hausdorff dimension of some quadratic Julia sets and the Mandelbrot set. Inou and Shishikura introduced a class of maps by using $P(z)=z(1+z)^2$, such that the parabolic and the near-parabolic renormalization can be defined \cite{InouShishikura2016}. The Inou-Shishikura class of maps plays an important role in the proof of the well-known result of the existence of Julia set with positive area for quadratic polynomials \cite{BuffCheritat2012}. The elegant part in \cite{InouShishikura2016} is the choice of an ellipse which is important in the function spaces.

Inspiring by the idea of the work of Inou and Shishikura, the parabolic and near-parabolic renormalization for the following class of polynomial maps $P(z)=z(1+z)^m$ is investigated, where $m$ is an integer and $m\geq22$. Following the idea of Inou and Shishikura, the choice of the ellipse should be dependent on the value of $m$, and a case by case study of the ellipse for each $m$ seems to work very well. The difficult part is the choices of the ellipse for each $m$. The key point of our construction is that for large enough $m$, $m\geq22$, we introduce a combinatorial construction of the ellipse, one part is a proper ellipse independent on $m$, and the other part is dependent on $m$. Here, the construction of the function spaces is useful for the definition of the parabolic and near-parabolic renormalization. Furthermore, based on this construction, for the cases $3\leq m\leq 21$, the problem of the choices of the ellipses is a case by case study following the work of Inou and Shishikura and our construction.

Based on the parabolic and near-parabolic bifurcation, the existence of high degree (degree bigger than or equal to $22$) polynomial maps with Julia sets having positive Lebesgue measure and Cremer fixed points is verified. This provides a positive answer to the classical Fatou conjecture (the existence of Julia set with positive area) for polynomials with degrees bigger than $21$.

The strategy of the proof follows the idea in the work of Buff and Ch\'{e}ritat for quadratic polynomials \cite{BuffCheritat2012}. The idea of the proof is divided into three parts:
\begin{itemize}
\item The study of the Hausdorff convergence of perturbed Siegel
disks, or the control of the post-critical sets of perturbations near the Siegel disk, this part is based on the near parabolic renormalization obtained in the first part. This generalizes Inou and Shishikura's results on the parabolic and near parabolic renormalization to high degree polynomials, this part is obtained in Section \ref{renormalization-2022-2-12-1}.
\item Ch\'{e}ritat's techniques of parabolic explosion \cite{Cheritat2000} and Yoccoz's renormalization technique \cite{Yoccoz1995} to control the shape of Siegel disks, the parabolic explosion (the control of the cycle) for the high degree polynomials is obtained in Subsection \ref{equ-21-12-7-1} (see Propositions \ref{explosionfun-1} and \ref{explosionfun-2}, and Lemma \ref{equ-2021-12-8-1} in Subsection \ref{equ-21-12-7-1}).
    \item McMullen's results on Siegel disks of bounded type \cite{McMullen1998}, which studied the measurable density of the filled-in Julia set near the boundary of a Siegel disk, this result was generalized to a class of high degree polynomials in \cite{Shen2006}, this part is summarized in Subsection \ref{lebesgue-density-12-12-1}.
\end{itemize}

\begin{theorem}\label{positivearea-12-11-1}
There exist non-renormalizable polynomial maps of any degree ($\geq22$) that have a Cremer fixed point and a Julia set of positive area.
\end{theorem}

The rest is organized as follows. In Section \ref{renormalization-2022-2-12-1}, the main results of the parabolic and near-parabolic bifurcation are provided.
In Section \ref{julia-set-equ-2022-2-12-2}, the proof of Theorem \ref{positivearea-12-11-1} is given.
In Section \ref{julia-set-equ-2022-2-12-3}, some basic concepts and useful tools are introduced.
In Section \ref{julia-set-equ-2022-2-12-4}, several open problems are provided for further work.

\section{Parabolic bifurcation and near-parabolic bifurcation}\label{renormalization-2022-2-12-1}

\subsection{Parabolic fixed point, Fatou coordiante and horn maps}

The theory of Fatou coordinates and horn maps were developed by Douady, Hubbard, Lavaurs and so on.

Consider a holomorphic function as follows:
\beq\label{equ202091}
f(z)=a_1z+a_2z^2+O(z^3),
\eeq
where $a_1$ and $a_2$ are complex parameters. If $a_1=1$ and $a_2\neq0$, then this map has a non-degenerate 1-parabolic fixed point at $z=0$. Consider a change of coordinate $w=-\frac{1}{a_2z}$ with $a_2\neq0$, where the fixed point $0$ is sent to $\infty$, the dynamics in this coordinate near $\infty$ is given by
\beqq
F(w)=-\frac{1}{a_2f(-\frac{1}{a_2w})}=w\bigg(\frac{1}{1-\frac{1}{w}+O(-\frac{1}{a_2^2w^2})}\bigg)=
w+1+\frac{b_1}{w}+O\bigg(\frac{1}{w^2}\bigg),
\eeqq
where
\beqq
b_1=1+\frac{\text{the coefficients in}\ O(z^3)}{a^2_2}.
\eeqq

\begin{theorem}\label{paraequ-85}\cite{Shishikura2000}
For the holomorphic function in \eqref{equ202091} with $a_1=1$ and $a_2\neq0$, there exists a sufficiently large positive number $L$, and two injective holomorphic functions $\Phi_{att}=\Phi_{att,F}:\{w:\ \mbox{Re}\,w>L\}\to\mathbb{C}$ and  $\Phi_{rep}=\Phi_{rep,F}:\{w:\ \mbox{Re}\,w<-L\}\to\mathbb{C}$ such that
\begin{itemize}
\item[(a)]
\beq
\Phi_{s}(F(w))=\Phi_s(w)+1\quad (s=\text{att},\ \text{rep}),
\eeq
where both sides of the equation are defined on the meaningful regions.
\item[(b)] $\Phi_{att}$ and $\Phi_{rep}$ are unique up to addition of constant.
\item[(c)]There is a large positive constant $L'$ such that $\Phi_{att}$ and $\Phi_{rep}$ can be extended to $\{w:\ |\mbox{Im}\,w|+\mbox{Re}\,w>L'\}$ and $\{w:\ |\mbox{Im}\,w|-\mbox{Re}\,w>L'\}$, respectively.
\item[(d)]$\Phi_{att}$ and $\Phi_{rep}$ have asymptotic expansion $w-b_1\log w+const+o(1)$ as $w\to\infty$.
\end{itemize}
\end{theorem}

\begin{definition}
For the functions $\Phi_{att}$ and $\Phi_{rep}$ introduced in Theorem \ref{paraequ-85}, $\Phi_{att}$ is called attracting Fatou coordinate and $\Phi_{rep}$ is said to be repelling Fatou coordinate. They are considered as the conjugacy between the map $F(w)$ and the translation $T:z\to z+1$ defined on the half-neighborhoods (attracting or repelling petals).

Both the attracting and repelling Fatou coordinates are defined on the common regions $V_{\pm}=\{w:\ \pm\text{Im}\,w>|w|+L'\}$,  the horn map $E_F$ defined on $\Phi_{rep,F}(V_{\pm})$ is given by
\beqq
E_F=\Phi_{att}\circ\Phi_{rep}^{-1}.
\eeqq
\end{definition}

\begin{theorem}
\begin{itemize}
\item[(a)] There exists $L^{\prime\prime}>0$ such that $E_F$ is defined on the region $\{z\in\cc:\ -1\leq\text{Re}\,z\leq1,\ |\text{Im}\,z|\geq L^{\prime\prime}\}$, which is contained in $\Phi_{rep}(V_{\pm})$.
\item[(b)] For $z\in\{z\in\cc:\ -1\leq\text{Re}\,z\leq1,\ |\text{Im}\,z|\geq L^{\prime\prime}\}$,
\beqq
E_F(z+1)=E_F(z)+1,
\eeqq
yielding that $E_F(z)-z$ is a periodic function with period $1$. So, $E_F$ can be extended holomorphically to $\{z:\ |\text{Im}\,z|\geq L^{\prime\prime}\}$.
\item[(c)] There are constants $c_{upper}$ and $c_{lower}$ such that
\beqq
E_F(z)-z\to c_{upper}\ \text{as}\ \text{Im}\,z\to+\infty\ \text{and}\ E_F(z)-z\to c_{lower}\ \text{as}\ \text{Im}\,z\to-\infty,
\eeqq
and
$c_{upper}-c_{lower}=2\pi i b_1$.
\end{itemize}
\end{theorem}

\begin{definition} (Fundamental regions and quotient cylinders)
Let $\xi$ be a sufficiently large positive number, and $l$ be a vertical line with $l=\{w\in\cc:\ \text{Re}\,w=\xi\}$. Then $F(l)$ is on the right hand side of $l$, $F(l)$ and $l$ bound an open region, denoted by $S_{att}$, and $F$ is injective in a neighborhood of $\ol{S_{att}}$. The closed strip $\ol{S_{att}}$ is called a (an attracting) fundamental region for $F$. The quotient space $\ol{S_{att}}/\sim$, where $l\ni w\sim F(w)\in F(l)$, is a topological cylinder and is called attracting Ecalle-Voronin cylinder, denoted by $\mathcal{C}_{att}$. Since $F$ is analytic near $l$, the cylinder has a natural structure as a Riemann surface. Similarly, if $\xi$ is a sufficiently small negative number, $l=\{w\in\cc:\ \text{Re}\,w=\xi\}$,  $S_{rep}$,  $\ol{S_{rep}}$, and $\ol{S_{rep}}/\sim$ are defined similarly, the closed strip $\ol{S_{rep}}$ is called a (repelling) fundamental region for $F$, and the quotient space $\ol{S_{rep}}/\sim$ is said to be repelling Ecalle-Voronin cylinder, denoted by $\mathcal{C}_{rep}$. Furthermore, the Fatou coordinates induce isomorphisms from attracting/repelling cylinders onto $\cc/\zz$ via the natural projection
$\cc\to\cc/\zz$, denoted by $\text{mod}\ \zz$.

The horn map $E_F$ via $\text{mod}\ \zz$ induces  a map on $\cc/\zz$ defined only in the neighborhoods of both ends $\pm i\infty$. If $w\in\ol{S_{rep}}$ with $|\text{Im}\,z|$ sufficiently large, then its orbit will eventually land on $\ol{S_{att}}$. This induces a map from a neighborhood of an upper or lower end of $\mathcal{C}_{rep}$ to $\mathcal{C}_{att}$. This map is well-defined, where the map might be discontinuous when $w\in\partial S_{rep}$ or its orbit arrives in $\partial S_{att}$, since the quotient space is defined by the identification $w\sim F(w)$ on the boundary. And, this map is exactly the one induced by $E_F$ via the Fatou coordinates.

Since the Fatou coordinates are only determined up to additive constant, a normalization is needed for convenience. Take a critical point $cp$, a normalization is chosen such that $\Phi_{att}(cp)=0$. For $\Phi_{rep}$, a normalization is chosen so that $c_{upper}=0$, i.e.,
\beqq
E_F(z)=z+o(1)\ \text{as}\ \text{Im}\,z\to+\infty.
\eeqq

\end{definition}

\begin{definition}\label{renorm2020912-3}
For the original map $f$ with the parabolic fixed point at $z=0$, let $\tau(z)=-\frac{1}{a_2z}$, the attracting Fatou coordinate is $\Phi_{att,f}=\Phi_{att}\circ \tau$, the repelling Fatou coordinate is $\Phi_{rep,f}=\Phi_{rep}\circ \tau$, and the horn map $E_f=\Phi_{att,f}\circ \Phi_{rep,f}^{-1}=\Phi_{att}\circ \tau\circ (\Phi_{rep}\circ \tau)^{-1}=\Phi_{att}\circ \Phi_{rep}^{-1}$. And, the fundamental regions are $S_{att,f}=\tau^{-1}(S_{att})$ and $S_{rep,f}=\tau^{-1}(S_{rep})$, whose shapes are like ``croissant-shaped" regions whose horns point at the fixed point $0$. The horn map $E_f$ is induced by the orbits going from the horns of $S_{rep,f}$ to $S_{att,f}$.
\end{definition}

\begin{definition}
Let $\widehat{\cc}=\cc\cup\{\infty\}$ and $\text{Dom}(f)$ be the domain of a function $f$. A neighborhood of $f$ is defined by
\begin{align*}
&\mathcal{N}=\mathcal{N}(f;K,\vep)\\
=&\bigg\{g:\ \text{Dom}(g)\to\widehat{\cc}:\
K\subset\text{Dom}(g)\cap\text{Dom}(f)\ \text{is compact and}\ \sup_{z\in K}d_{\widehat{\cc}}(g(z),f(z))<\vep\bigg\},
\end{align*}
where $\vep$ is a positive constant and $d_{\widehat{\cc}}(\cdot,\cdot)$ is the hyperbolic metric on $\widehat{\cc}$. A sequence $\{f_n\}$ is called convergent to $f$ uniformly on compact sets if for any neighborhood $\mathcal{N}$ of $f$, there is an positive integer $n_0$ such that $f_n\in\mathcal{N}$ for $n>n_0$.
\end{definition}

The construction $f\rightsquigarrow E_f$ is continuous and holomorphic in the following sense.
\begin{theorem} (Continuity and holomorphic dependence)
\begin{itemize}
  \item[(a)] Let $f$ be a holomorphic map with a non-degenerate $1$-parabolic fixed point at $z=0$. Given a neighborhood $\mathcal{N}$ of its horn map $E_f$, there exists a neighborhood $\mathcal{N}'$ of $f$ such that if $g\in\mathcal{N}'$ and $g$ has a $1$-parabolic fixed point at $0$, then its horn map $E_g$ can be defined so that $E_g\in\mathcal{N}$.
\item[(b)] Suppose $f_{\ld}(z)$ is holomorphic in $(\ld,z)\in\Ld\times\mathcal{U}$, where $\Ld$ is a complex manifold and $\mathcal{U}=\text{Dom}(f_{\ld})\subset\widehat{\cc}$. Assume that $f_{\ld}$ always have a non-degenerate $1$-parabolic fixed point at $z=0$. Then for $\ld_{*}\in\Ld$ and an open set $\mathcal{V}\subset\cc$, whose closure is compact and contained in $\text{Dom}(E_{f_{\ld_{*}}})$, there exists a neighborhood $\Ld_1$ of $\ld_{*}$ in $\Ld$ such that $E_{f_{\ld}}(z)$ is defined and holomorphic in $\Ld_1\times\mathcal{V}$.
\end{itemize}

\end{theorem}

\subsection{Bifurcation of parabolic fixed points}

Consider a holomorphic function $f_0$ with a non-degenerate $1$-parabolic fixed point at $z=0$, and consider a perturbation of $f_0$ in a neighborhood of $0$, this perturbation function is denoted by $f$. Since $f_0$ is non-degenerate, $z=0$ has multiplicity $2$ as a solution to $f_0(z)=z$. So, $f$ has two fixed points near $0$. Without loss of generality, we can suppose that $z=0$ is a fixed point of $f$. The multiplier of $f$ at $z=0$ is close to $1$, it can be written as $e^{2\pi i\al}$ with $\al\in\cc$.

In the following discussions, assume $f$ can be written as follows:
\beq\label{equ2020821-1}
f(z)=e^{2\pi i\al}z+a_2z^2+O(z^3), \text{where}\ a_2\neq0,\ \al=\al(f)\ \text{is small and}\ |\arg\,\al|<\frac{\pi}{4}.
\eeq
Let $\si=\si(f)$ be another fixed point of $f$ near $0$. It is evident that if $\al(f)=0$, then $\si(f)=0$. By direct computation, $\si(f)$ has asymptotic expansion $\si(f)=-2\pi i\al/a_2+o(\al)$, when $f$ is convergent to $f_0$ in a fixed neighborhood of $0$, where $a_2=f^{\prime\prime}_0(0)/2$. For more details on the choice of $\al$, please refer to \cite{Shishikura2000}.

\begin{theorem} \cite{Shishikura2000} \cite[Theorem 2.1]{InouShishikura2016}\label{equ2020821-2}
Suppose $f_0$ has a non-degenerate $1$-parabolic fixed point at $z=0$. Then there exists a neighborhood $\mathcal{N}=\mathcal{N}(f_0;K,\vep)$ ($\text{int}\,K$ contains $0$) such that if $f\in\mathcal{N}$ and $f$ satisfies \eqref{equ2020821-1}, then the fundamental regions $S_{att,f}$, $S_{rep,f}$ are defined near those of $f_0$, except that the horns of
$S_{att,f}$ and $S_{rep,f}$ point to distinct fixed points $0$ and $\si(f)$ (if $\al(f)\neq0$). Moreover, the Fatou coordinates $\Phi_{att,f}$ and $\Phi_{rep,f}$ are also defined in a neighborhood of $\ol{S_{att,f}}\setminus\{0,\si(f)\}$ and $\ol{S_{rep,f}}\setminus\{0,\si(f)\}$ so that they induce isomorphisms from the quotient cylinders $\mathcal{C}_{att,f}$ and $\mathcal{C}_{rep,f}$ onto $\cc/\zz$. The horn map $E_f$ is defined similarly. After a suitable normalization, $\Phi_{att,f}$, $\Phi_{rep,f}$, and $E_f$ depend continuously and holomorphically on $f$.
\end{theorem}

\begin{theorem} \cite{Shishikura2000} \cite[Theorem 2.2]{InouShishikura2016} \label{equ2020821-3}
Let $f$ be a holomorphic function defined in the previous theorem with $f^{\prime}(0)\neq1$. Then for any orbit starting from $\ol{S_{att,f}}\setminus\{0,\si(f)\}$ eventually lands on $\ol{S_{rep,f}}\setminus\{0,\si(f)\}$. Such a correspondence induces an isomorphisms $\chi_f$ from $\mathcal{C}_{att,f}$ onto $\mathcal{C}_{rep,f}$. By identifying these cylinders with $\cc/\zz$ by the Fatou coordinates, $\chi_f$ is written as
\beq
\chi_f(z)=z-\frac{1}{\al(f)}\ \text{on}\ \cc/\zz,
\eeq
where the horn map $E_f$ is normalized so that $E_f(z)=z+o(1)$ as $\text{Im}\,z\to+\infty$.

The composition $h=\chi_f\circ E_f$ corresponds to the first return map of $f$ to the region $\ol{S_{rep,f}}\setminus\{0,\si(f)\}$ near the horns, i.e., if $z\in\ol{S_{rep,f}}\setminus(\{0,\si(f)\}\cup\text{``inner boundary"})$ and $w=\Phi_{rep,f}(z)\in\cc/\zz$ has sufficiently large $|\text{Im}\,w|$, then there is a smallest $n\geq1$ such that $f^n(z)\in \ol{S_{rep,f}}\setminus\{0,\si(f)\}$ such that $\Phi_{rep,f}(f^n(z))=h(w)=\chi_f\circ E_f(w)$ in $\cc/\zz$.
\end{theorem}

\subsection{Parabolic and near-parabolic renormalization}

\begin{definition} \cite{InouShishikura2016}
Denote by $\text{Exp}^{\sharp}(z)=e^{2\pi iz}$ and $\text{Exp}^{\flat}(z)=e^{-2\pi iz}$. Both functions induce isomorphisms from $\cc/\zz$ onto $\cc^*=\cc\setminus\{0\}$;
$\text{Exp}^{\sharp}$ sends upper end $+i\infty$ to $0$ and lower end $-i\infty$ to $\infty$, and for $\text{Exp}^{\flat}$, the role of the ends is interchanged.

Suppose $f$ has a non-degenerate parabolic fixed point at $0$, the parabolic renormalization is defined bo be
\beqq
\mathcal{R}_0f=\mathcal{R}^{\sharp}_0f
=\text{Exp}^{\sharp}\circ E_f\circ (\text{Exp}^{\sharp})^{-1},
\eeqq
where $E_f$ is the horn map of $f$, and normalized as $E_f(z)=z+o(1)$ as $\text{Im}\,z\to+\infty$. Then $\mathcal{R}_0f$ extends holomorphically to $0$ and $\mathcal{R}_0f(0)=0$, $(\mathcal{R}_0f)^{\prime}(0)=1$. So, $\mathcal{R}_0f$ has a $1$-parabolic fixed point at $0$. Similarly, the parabolic renormalization for lower end is defined as
\beqq
\mathcal{R}^{\flat}_0f
=c\cdot\text{Exp}^{\flat}\circ E_f\circ (\text{Exp}^{\flat})^{-1},
\eeqq
where $c\in\cc^*$ is chosen so that $(\mathcal{R}^{\flat}_0f)^{\prime}(0)=1$.
\end{definition}

\begin{definition} \cite{InouShishikura2016}
Suppose $f(z)=e^{2\pi i\al}z+O(z^2)$  has fundamental domains, and a return map $h=\chi_f\circ E_f$ is defined as in Theorems \ref{equ2020821-2} and \ref{equ2020821-3}, where $\al\neq0$ and $f^{\prime\prime}(0)\neq0$. The near-parabolic renormalization (or also called cylinder renormalization) is defined by
\beqq
\mathcal{R}f=\mathcal{R}^{\sharp}f
=\text{Exp}^{\sharp}\circ \chi_f\circ E_f\circ (\text{Exp}^{\sharp})^{-1}.
\eeqq
Then $\mathcal{R}f$ extends to $0$ and $\mathcal{R}f(0)=0$, $(\mathcal{R}f)^{\prime}(0)=e^{-2\pi i\tfrac{1}{\al}}$. For the lower end, set $\mathcal{R}^{\flat}f
=\text{Exp}^{\flat}\circ \chi_f\circ E_f\circ (\text{Exp}^{\flat})^{-1}$.
\end{definition}

\begin{definition}\cite{InouShishikura2016}
For $f(z)=e^{2\pi i\al}z+O(z^2)$ with $\al\neq0$, suppose $f(z)=e^{2\pi i\al}h(z)$ with $h(0)=0$ and $h'(0)=1$, $f$ is related to the pair $(\al,h)$. The near-parabolic renormalization can be expressed as a skew product:
\beqq
\mathcal{R}:\ (\al,h)\to\left(-\frac{1}{\al}\text{mod}\,\zz,\ \mathcal{R}_{\al}h\right),
\eeqq
where $\mathcal{R}_{\al}h=E_{(e^{2\pi i\al})h}$ is the renormaliztion in fiber direction.
\end{definition}

\begin{question}
The existence of invariant function spaces under the parabolic or near-parabolic renormalization is a basic problem.
\end{question}

\subsection{A class of functions} \label{r0function2021-9-2-1}

The following class of functions is invariant under parabolic renormalization $\mathcal{R}_0$, but not invariant for the fiber renormalization $\mathcal{R}_{\al}$ for $\al\neq0$, since the simple covering structure of horn map is destroyed, implying the possible existence of infinitely many critical values, or the assumption of the branched covering does not hold \cite{Shishikura1998}.

\begin{definition} (Class $\mathcal{F}_0$)\cite{Shishikura1998}
\begin{equation*}
\mathcal{F}_{0}=\left\{f: \text{Dom}(f) \rightarrow \mathbb{C}\ \bigg | \begin{array}{l}
0\in\text{Dom}(f) \text { open}\subset\cc,\  f\ \text{is holomorphic in Dom}(f), \\
f(0)=0,\ f'(0)=1,\ f:\ U_f\subset\text{Dom}(f)\setminus\{0\}\to\cc^{*}=\cc\setminus\{0\}\\
 \text{is a branched covering map with a unique critical value}\ cv_{f},\\
\text{all critical points in}\ U_f\ \text{are of local degree}\ 2
\end{array}\right\}.
\end{equation*}
\end{definition}

\begin{lemma} \cite[Chap. VIII, Sec. 4 and 5]{Gamelin2001} (Inverse Function Theorem)
Suppose $\rho>0$ is a positve number, $f(z)$ is holomorphic for $z\in\ol{\cd}(z_0,\rho)$ satisfying $f(z_0)=w_0$, $f^{\prime}(z_0)\neq0$, and $f(z)\neq w_0$ for $z\in\ol{\cd}(z_0,\rho)\setminus\{z_0\}$. Let $\de>0$ be chosen such that $|f(z)-w_0|\geq\de$ for $|z-z_0|=\rho$. Then for each $w$ such that $|w-w_0|<\de$, there is a unique $z_w$ satisfying $|z_w-z_0|<\rho$ and $f(z_w)=w$. Such a $ z_{w}$ is given by the explicit expression
\[
z_{w}=f^{-1}(w)=\frac{1}{2 \pi i} \int_{\left|z-z_{0}\right|=\delta} \frac{z f^{\prime}(z)}{f(z)-w} d z,
\]
in particular, $f^{-1}$ is analytic on its domain of definition.
\end{lemma}

\begin{definition} \cite[Chap. VIII, Sec. 4 and 5]{Gamelin2001}
 (i) We call $f$ attains $m$ times the value $w_{0}$ at $z_{0}$ if the function $f(z)-w_{0}$ has a zero of order $m$ at $z=z_{0}$.
(ii) We call $f$ attains $m$ times the value $w_{0}$ on an open subset of $\cc$ if the function $f(z)-w_{0}$ has $m$ zeros on this open subset counting multiplicity.
\end{definition}

\begin{lemma} \cite[Chap. VIII, Sec. 4 and 5]{Gamelin2001}
 If $f$ achieves $m$ times the value $w_{0}$ at $z_{0},$ then there exists $\rho>0$ and $r>0$ such that for every $w \in \cd(w_{0},\rho)$, $f$ achieves $m$ times the value $w$ on $\cd(z_{0},r)$.
\end{lemma}

Next, a Riemann surface, consisting of ``sheets", which are copies of the plane $\cc$, and cutting along several slits and glued together along pairs of slits, is constructed. The action of the map $f\in\mathcal{F}_0$ on each sheet is a projection onto $\cc$.

Consider the function defined on the complex plane:
\beq
P(z)=z(1+z)^m, (m\geq2,\ m\in\mathbb{N}),\ z\in\cc.
\eeq
Now, the map $P(z)=z(z+1)^m$ will be considered as an example to illustrate the construction of the Riemann surface. For convenience, the domain of $P$ is $\cc$.

\begin{proposition}\label{paraequ-2}
The derivative of $P(z)$ is $(1+z)^{m-1}(1+(m+1)z)$. The critical points are $-1$ with multiplicity $m-1$, and $-\tfrac{1}{m+1}$ with multiplicity $1$. The critical value corresponding to $-1$ is $0$, and the critical value corresponding to $-\tfrac{1}{m+1}$ is $-\tfrac{m^m}{(m+1)^{m+1}}$.
\end{proposition}

Set
\beq\label{critip2011-8-30-1}
cp_{\tiny{P}}:=-\frac{1}{m+1}\ \text{and}\  cv_{\tiny{P}}:=-\frac{m^m}{(m+1)^{m+1}}=-\frac{1}{m+1}\bigg(\frac{m}{m+1}\bigg)^m.
\eeq
 It is evident that $0>cv_{\tiny{P}}>cp_{\tiny{P}}$.

\begin{remark}
For $m=2$, the function is $P(z)=z(1+z)^2$ (Inou-Shishikura class), it has a multiple fixed point
$0$, a critical point $-1/3$ with the corresponding critical value $-4/27$, and a second critical point $-1$ with corresponding critical value $0$. This class of functions plays an important role in the proof of the existence of quadratic polynomial maps with positive area Julia sets.
\end{remark}

Suppose the critical value is $cv=cv_P\in\rr_{-}$.
Denote by
\beqq
 \Ga_b=(-\infty,cv],\ \Ga_a=(cv,0),\ \Ga_c=(0,+\infty),
\eeqq
and
\beqq
\cc_{slit}=\cc\setminus(\{0\}\cup\Ga_b\cup\Ga_c),\ \cc_{slit\pm}=\cc\cap\{z:\ \pm\mbox{Im}z>0\}.
\eeqq
An illustration diagram for  $\Ga_b$, $\Ga_a$, and $\Ga_c$ is given in Figure \ref{illustration-three-1}.
\begin{figure}[htbp]
\begin{center}
\scalebox{0.5}{ \includegraphics{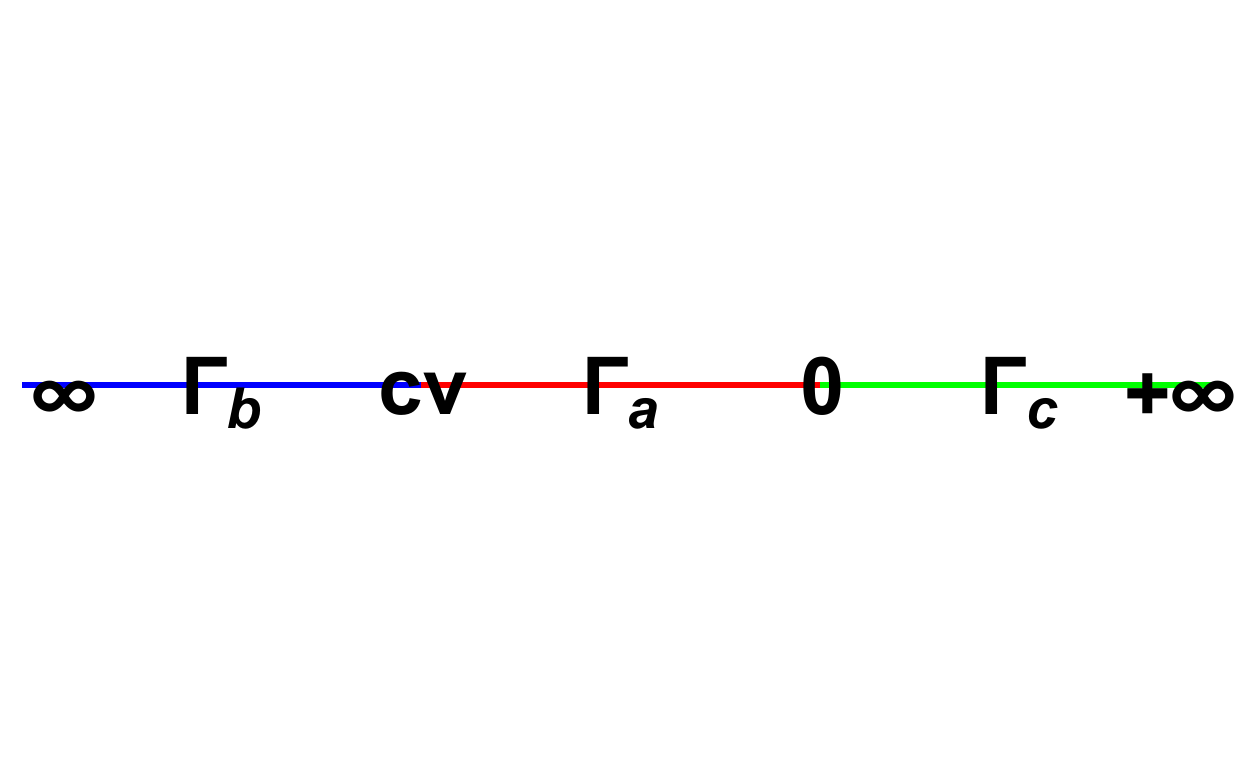}}
\renewcommand{\figure}{Fig.}
\caption{The illustration diagram of  $\Ga_b$, $\Ga_a$, and $\Ga_c$, where blue represents $\Ga_b$, red represents $\Ga_a$, and green represents $\Ga_c$.
}\label{illustration-three-1}
\end{center}
\end{figure}
By the definition of $\cc_{slit}$, it is simply connected and does not contain $0$ and the critical value, $P^{-1}(\cc_{slit})$ consists of connected components, denoted by $(\mathcal{U}_j)_{j\in I}$, where $I=\{1,2,...,m+1\}$ is the index set, each $\mathcal{U}_j$ is mapped by $P$ isomorphically onto $\cc_{slit}$. Denote by
\beq\label{equ2021-9-12-3}
\mathcal{U}_{j\pm}=P^{-1}(\cc_{slit\pm})\cap\mathcal{U}_{j},
\eeq
\beq\label{equ2021-9-12-5}
\ga_{aj}=P^{-1}(\Ga_a)\cap\mathcal{U}_{j},\ \ga_{bj\pm}=P^{-1}(\Ga_b)\cap\overline{\mathcal{U}}_{j\pm},\ \mbox{and}\ \ga_{cj\pm}=P^{-1}(\Ga_c)\cap\overline{\mathcal{U}}_{j\pm}.
\eeq

Since $P$ is homeomorphic near $0$, there is a component containing $0$, denoted by $\mathcal{U}_1$, and $\ga_{c1+}=\ga_{c1-}$. Further, since $P$ has a unique critical value $cv_P$ and the critical point $cp_P$ is of local degree $2$, where $cv_P$ and $cp_P$ are introduced in \eqref{critip2011-8-30-1}, there is another component, denoted by $\mathcal{U}_2$, where $\ga_{b1+}=\ga_{b2-}$ and $\ga_{b2+}=\ga_{b1-}$. And, $P^{-1}(cv)\cap\ga_{b1+}\cap\ga_{b2+}$ is a critical point, denoted by $cp=cp_P$.

Set
\beqq
\mathcal{U}_{12}:=\mathcal{U}_{1}\cup\mathcal{U}_{2}\cup\ga_{b1+}\cup\ga_{b2+}.
\eeqq
So, $f:\ \mathcal{U}_{12}\to\cc_{slit}\cup\Ga_b=\cc\setminus(\{0\}\cup\Ga_c)$ is a branched covering of degree $2$ over the critical value $cv$.

The two components $\mathcal{U}_1$ and $\mathcal{U}_2$ are attached to the critical point $-\tfrac{1}{m+1}$.
Note that the multiplicity of $-1$ is $m-1$, there are $m-1$ connected components of $P^{-1}(\cc_{slit})$ attached to the point $-1$, denoted by $\mathcal{U}_3,...,\mathcal{U}_{m+1}$, where $\mathcal{U}_3,...,\mathcal{U}_{m+1}$ are around the center $-1$ along the counterclock-wise direction with the increasing order.

By the properties of the polynomial $P$, $\ga_{a1}=(cp,0)$ and $\ga_{a2}=(-1,cp)$. Further, the connected components of  $P^{-1}(\Ga_a)\setminus(\{-1\}\cup\ga_{a1}\cup\ga_{a2})$ can be denoted by  $\ga_{a3}$, $\ga_{a4}$,..., $\ga_{a,m+1}$.

By the properties of the inverse of $\Ga_b=(-\infty,cv]$, one has $\ga_{b1+}=\ga_{b2-}$,  $\ga_{b2+}=\ga_{b1-}$,  $\ga_{b1+}$ and $\ga_{b2+}$ are joined together by the critical point $-\tfrac{1}{m+1}$, $\ga_{b1+}$ is above the real axis, $\ga_{b2+}$ is below the real axis. And, $\ga_{c1+}=\ga_{c1-}=(0,+\infty)$.

Consider the two different situations: (i) $m$ is an odd number; (ii) $m$ is an even number.

For Case (i), $m$ is an odd number.

 Since $m$ is odd, and for any $z\in(-\infty,-1)$, $P(z)=z(1+z)^m>0$, one has that $P^{-1}(\Ga_a)\cap(-\infty,-1)=\emptyset$.

The inverse of $\Ga_c=(0,+\infty)$ consists of the following connected components: $\ga_{c1+}=\ga_{c1-}=(0,+\infty)$,  and $(-\infty,-1)$.

Figure \ref{illustration-m-3} is the illustration diagram for $m=3$, and
Figure \ref{illustration-m-5} is the illustration diagram for $m=5$.

\begin{figure}[htbp]
\begin{center}
\scalebox{0.5}{ \includegraphics{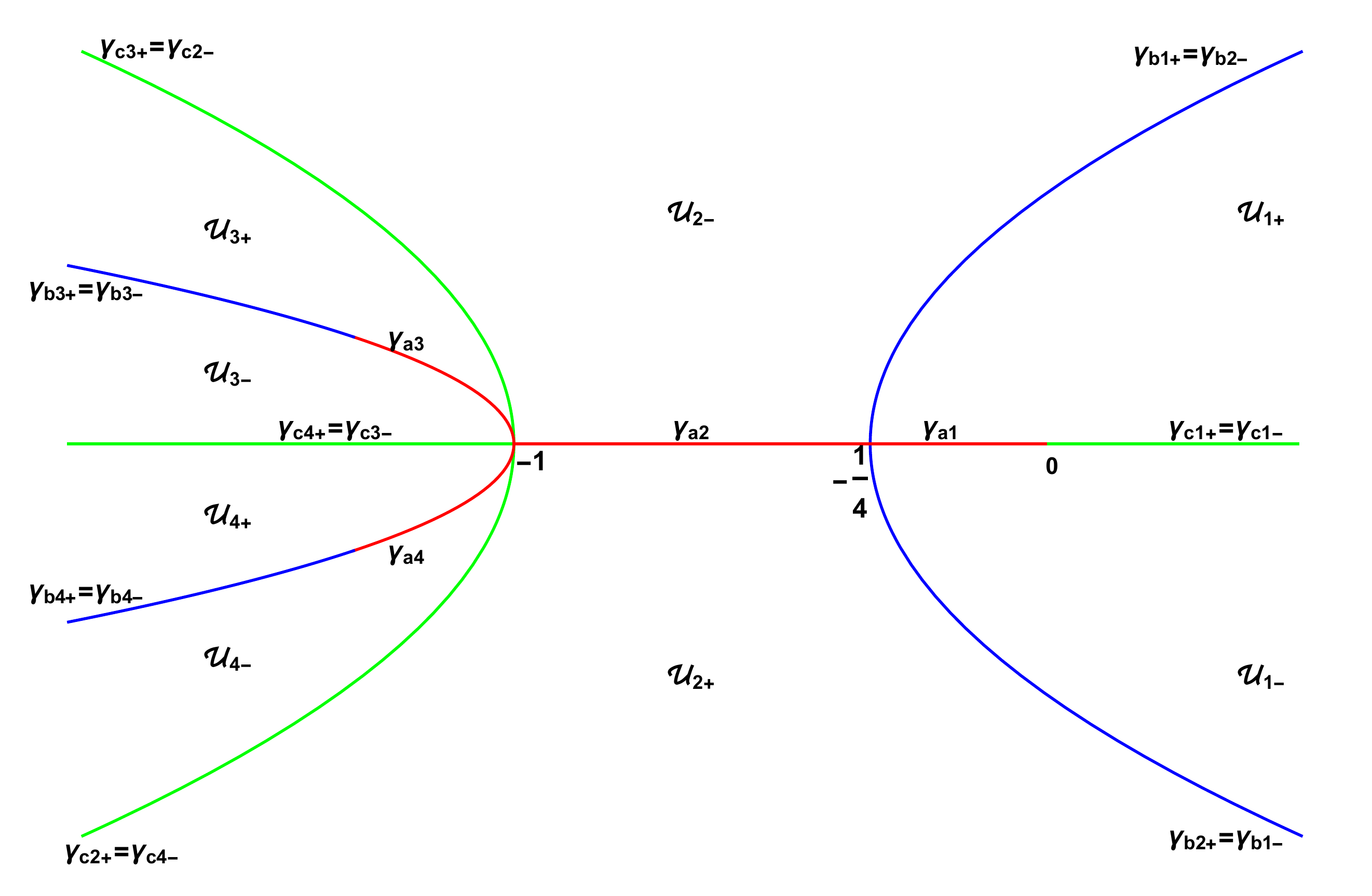}}
\renewcommand{\figure}{Fig.}
\caption{The illustration diagram of  $m=3$, where blue represents $P^{-1}(\Ga_b)$, red represents $P^{-1}(\Ga_a)$, and green represents $P^{-1}(\Ga_c)$.
}\label{illustration-m-3}
\end{center}
\end{figure}

\begin{figure}[htbp]
\begin{center}
\scalebox{0.5}{ \includegraphics{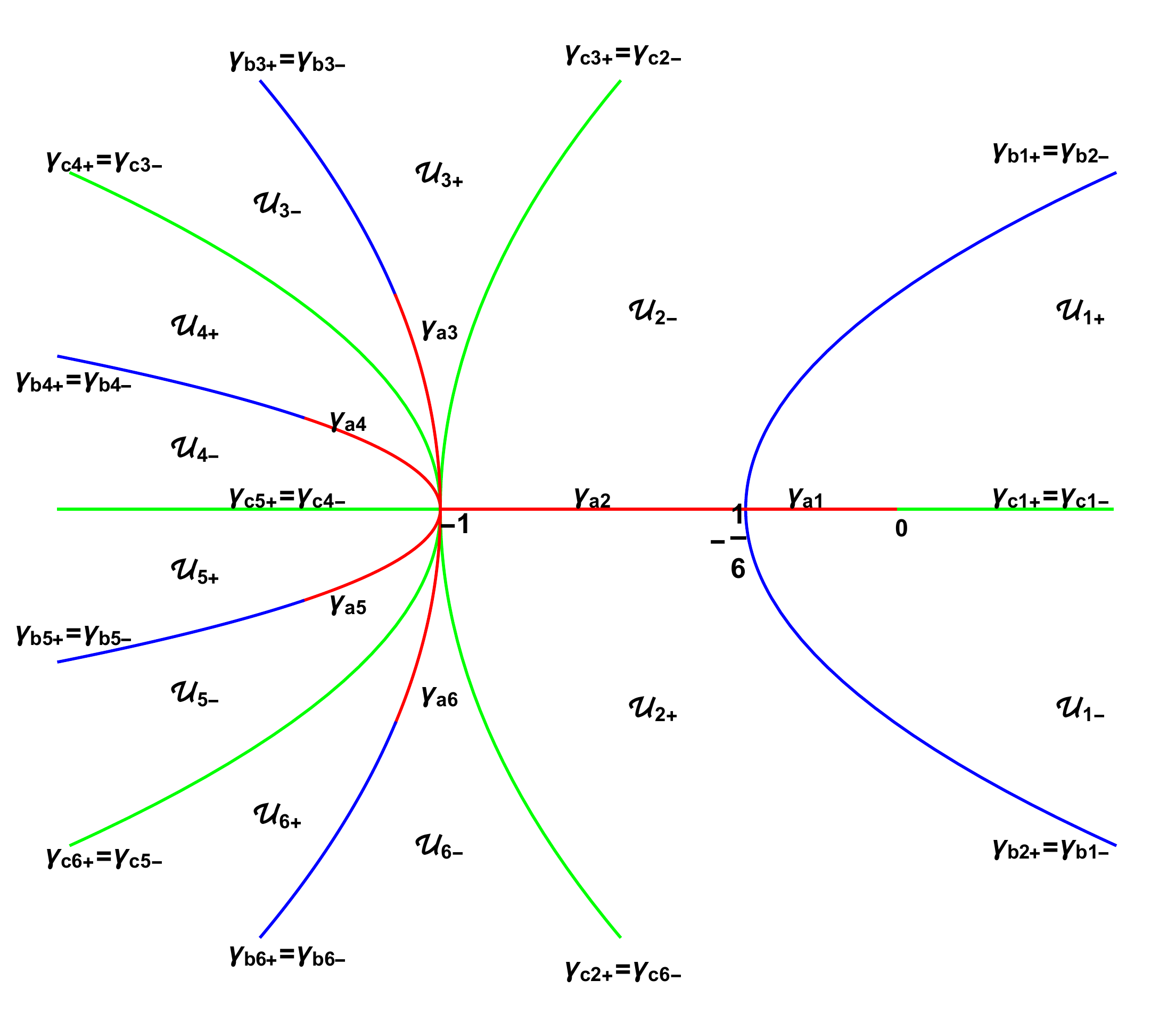}}
\renewcommand{\figure}{Fig.}
\caption{The illustration diagram of  $m=5$, where blue represents $P^{-1}(\Ga_b)$, red represents $P^{-1}(\Ga_a)$, and green represents $P^{-1}(\Ga_c)$.
}\label{illustration-m-5}
\end{center}
\end{figure}

For Case (ii), $m$ is an even number.

Since $m$ is even, and for any $z\in(-\infty,-1)$, $P(z)=z(1+z)^m<0$, one has that there is a unique solution to $P(z)=cv$, denoted by $cv_{-1}$.

The inverse of $\Ga_a=(cv,0)$ consists of the following connected components:
$\ga_{a1}=(cp,0)$, $\ga_{a2}=(-1,cp)$, and $(cv_{-1},-1)$.

The illustration diagram for $m=4$ is in Figure \ref{illustration-m-4}. The illustration diagram for $m=6$ is in Figure \ref{illustration-m-6}.

\begin{figure}[htbp]
\begin{center}
\scalebox{0.5}{ \includegraphics{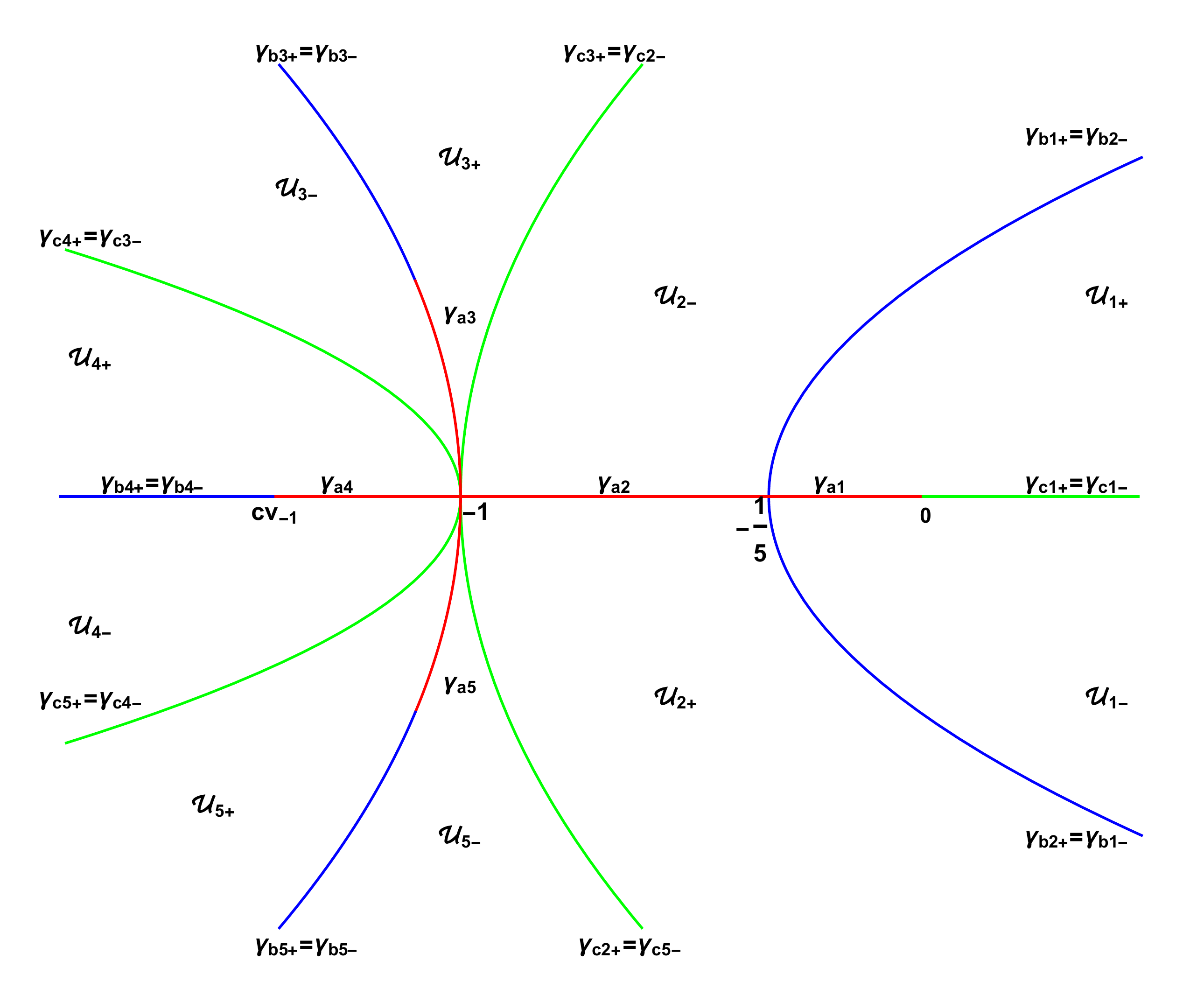}}
\renewcommand{\figure}{Fig.}
\caption{The illustration diagram of  $m=4$, where blue represents $P^{-1}(\Ga_b)$, red represents $P^{-1}(\Ga_a)$, and green represents $P^{-1}(\Ga_c)$.
}\label{illustration-m-4}
\end{center}
\end{figure}

\begin{figure}[htbp]
\begin{center}
\scalebox{0.5}{ \includegraphics{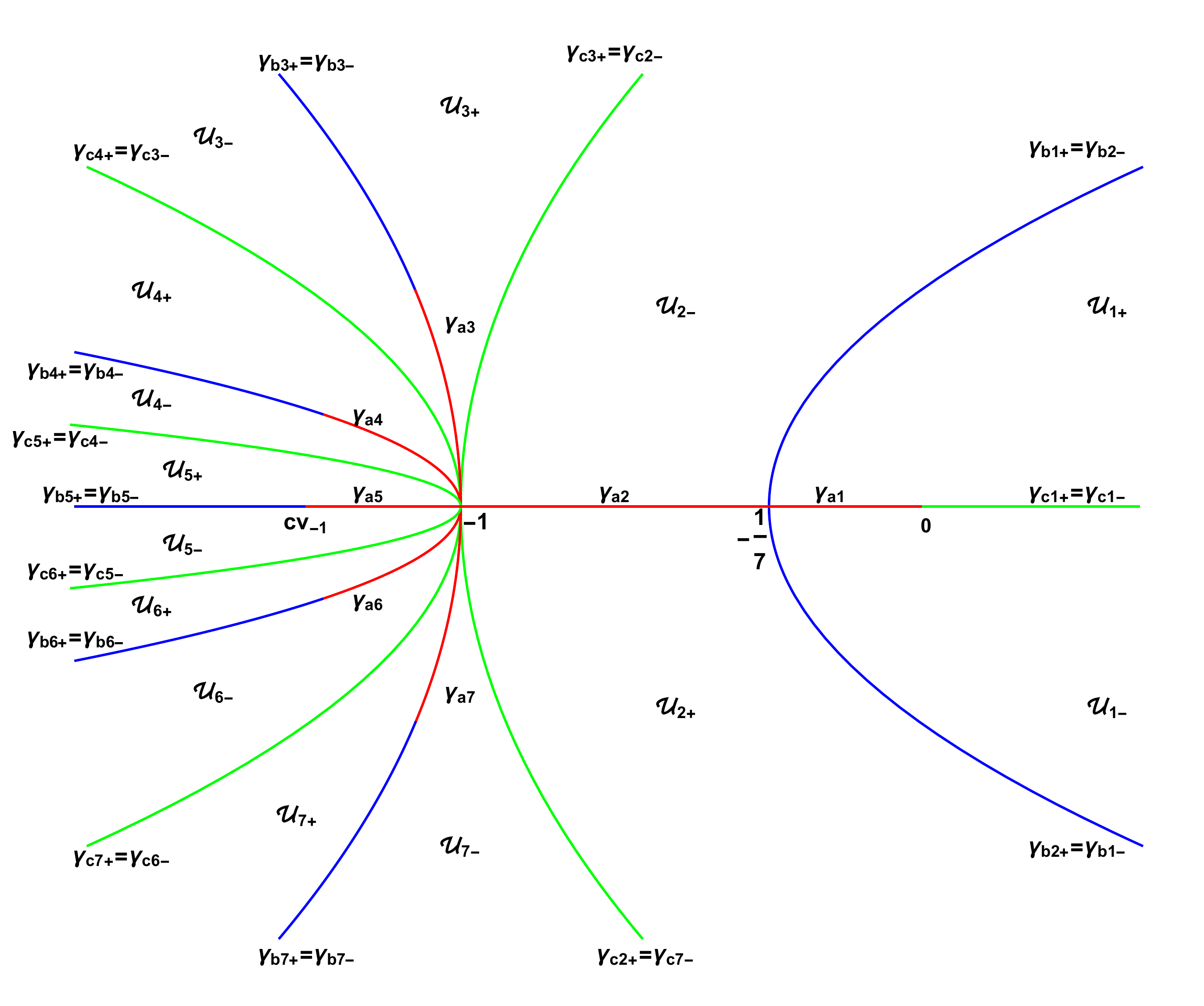}}
\renewcommand{\figure}{Fig.}
\caption{The illustration diagram of  $m=6$, where blue represents $P^{-1}(\Ga_b)$, red represents $P^{-1}(\Ga_a)$, and green represents $P^{-1}(\Ga_c)$.
}\label{illustration-m-6}
\end{center}
\end{figure}

\begin{remark}
The polynomial $P$ can be thought of as an element in $\mathcal{F}_0$, where the point $-1$ should be deleted from the domain of $P$, since $-1$ is a critical point of $P$ with multiplicity $m-1$ and $P(-1)=0$. The domain of $P$ might be different such that $P$ is a branched covering of different degrees, for example, the domain of $P$
might be
\beq\label{defPdiff2021-8-30-2}
\mathcal{U}^{P}_{12}:=\ol{\mathcal{U}}_{1}\cup \mathcal{U}_{2},
\eeq
or
\beq\label{defPdiff2021-8-30-3}
\mathcal{U}^P_{*}:=\ol{\mathcal{U}}_{1}\cup \mathcal{U}_{2}\cup
\mathcal{U}_{3+}\cup\ga_{c2+}\cup
\mathcal{U}_{m+1,-}\cup \ga_{c,2,-},
\eeq
where $\mathcal{U}^{P}_{12}$ is used for the degree $2$, and $\mathcal{U}^P_{*}$ is used for the degree bigger than or equal to $3$.
 \end{remark}

\begin{lemma}\label{paraequ-2021-9-20-2}
Set
\beqq
\be_k:=\frac{k}{m+1}\pi,\ 0\leq k\leq m,
\eeqq
\beqq
S_{k,\pm}:=\{z\in\cc:\ z\neq-1,\ \be_k<\pm\text{arg}\,(z+1)<\be_{k+1}\},\ 0\leq k\leq m,
\eeqq
 and
\beqq
k_0:=\left\{
      \begin{array}{ll}
        \frac{m+1}{2}+1, & \hbox{if}\ m\ \mbox{is odd} \\
        \frac{m}{2}+2, &\hbox{if}\ m\ \mbox{is even}.
      \end{array}
    \right.
\eeqq
Then, one has
\begin{itemize}
\item[(a)]
\beqq
S_{0,+}\subset  \ol{\mathcal{U}}_{1,+}^P\cup\ol{\mathcal{U}}_{2,-}^P\
\mbox{and}\
S_{0,-}\subset  \ol{\mathcal{U}}_{1,-}^P\cup\ol{\mathcal{U}}_{2,+}^P,
\eeqq
\beqq
\ga_{c2-}\subset S_{1,+}\ \mbox{and}\ \ga_{c2+}\subset S_{1,-},
\eeqq
and
\beqq
\ga_{b3+}\cup \ga_{a3}\subset S_{2,+}\ \mbox{and}\ \ga_{b,m+1,+}\cup \ga_{a,m+1}\subset S_{2,-};
\eeqq
\item[(b)] If $m$ is odd, then,
\beqq
\ol{S}_{m,+}\subset\ol{\mathcal{U}}_{k_0,-}^P\ \mbox{and}\ \ol{S}_{m,-}\subset\ol{\mathcal{U}}_{k_0+1,+}^P;
\eeqq
if $m$ is even, then
\beqq
\ol{S}_{m,+}\subset\ol{\mathcal{U}}_{k_0,+}^P\ \mbox{and}\
\ol{S}_{m,-}\subset\ol{\mathcal{U}}_{k_0,-}^P.
\eeqq
\end{itemize}
\end{lemma}

\begin{proof}
For $z\in S_{k,\pm}$, denote by $w=z+1$, one has $z=w-1=w+(-1)$, implying that $\tfrac{k}{m+1}\pi<\pm\text{arg}\,z<\pi$. So, for $z\in S_{k,+}$,
\beqq
\frac{k}{m+1}\pi+m\times\frac{k}{m+1}\pi
<\text{arg}\,P(z)=\text{arg}\,z+m\times\text{arg}\,(z+1)<\pi+m\times\frac{k+1}{m+1}\pi,
\eeqq
that is,
\beq\label{est-2021-9-20-1}
k\pi<\text{arg}\,P(z)<k\pi+\pi+\frac{m-k}{m+1}\pi.
\eeq
Note that $0\leq\frac{m-k}{m+1}<1$ by $0\leq k\leq m$, and $\tfrac{1}{2}<\tfrac{m}{m+1}<1$ by $m>1$.

By \eqref{est-2021-9-20-1}, for $k=0$ and $z\in S_{0,+}$, one has
\beqq
0<\text{arg}\,P(z)<\pi+\frac{m}{m+1}\pi<2\pi.
\eeqq
By \eqref{equ2021-9-12-3},
\beqq
\mathcal{U}^P_{1+}=P^{-1}(\cc_{slit+})\cap\mathcal{U}^P_{1}\
\mbox{and}\ \mathcal{U}^P_{2-}=P^{-1}(\cc_{slit-})\cap\mathcal{U}^P_{2},
\eeqq
one has
\beqq
0<\text{arg}\,P(z)<\pi\ \mbox{for}\ z\in \mathcal{U}^P_{1+}, \mbox{and}\
\pi<\text{arg}\,P(z)<2\pi\ \mbox{for}\ z\in \mathcal{U}^P_{2-}.
\eeqq
By \eqref{est-2021-9-20-1}, for $k=1$ and $z\in S_{1,+}$, one has
\beq\label{equ-2021-9-20-5}
\pi<\text{arg}\,P(z)<2\pi+\frac{m-1}{m+1}\pi<3\pi.
\eeq
By \eqref{est-2021-9-20-1}, for $k=2$ and $z\in S_{2,+}$, one has
\beq\label{equ-2021-9-20-6}
2\pi<\text{arg}\,P(z)<2\pi+\pi+\frac{m-2}{m+1}\pi<4\pi.
\eeq
Note that $\text{arg}\,P(\ga_{c2-})=2\pi$. So,  $\ga_{c2-}\subset S_{1,+}$. By these discussions, one has $S_{0,+}\subset  \ol{\mathcal{U}}_{1,+}^P\cup\ol{\mathcal{U}}_{2,-}^P$.

By \eqref{est-2021-9-20-1}, for $k=3$ and $z\in S_{3,+}$, one has
\beqq
3\pi<\text{arg}\,P(z)<3\pi+\pi+\frac{m-3}{m+1}\pi.
\eeqq
This, together with $\text{arg}\,P(\ga_{b3+}\cup \ga_{a3})=3\pi$, \eqref{equ-2021-9-20-5}, and \eqref{equ-2021-9-20-6}, implies that
\beqq
\ga_{b3+}\cup \ga_{a3}\subset S_{2,+}.
\eeqq

By \eqref{est-2021-9-20-1}, for $k=m$ and $z\in S_{m,+}$, one has
\beqq
m\pi<\text{arg}\,P(z)<(m+1)\pi,
\eeqq
and for $k=m-1$ and $z\in S_{m-1,+}$, one has
\beqq
(m-1)\pi<\text{arg}\,P(z)<(m-1)\pi+\frac{m+2}{m+1}\pi=m\pi+\frac{1}{m+1}\pi.
\eeqq

\end{proof}

For any $f\in\mathcal{F}_0$,  consider the linear conjugacy map $\tilde{f}(\cdot)=\tfrac{cv_{P}}{cv_f}f(\tfrac{cv_f}{cv_P}\cdot)$ of $f$, it is evident that this conjugacy map of $f$ is also in $\mathcal{F}_0$ and has the same critical value $cv_P$, which is contained in $\rr_{-}$. So, we assume the critical value of $f$ is in $\rr_{-}$.

\begin{definition}\label{paraequ-89}
For $f\in\mathcal{F}_0$, suppose the critical value is $cv=cv_f\in\rr_{-}$.
Denote by
\beqq
 \Ga_b=(-\infty,cv],\ \Ga_a=(cv,0),\ \Ga_c=(0,+\infty),
\eeqq
and
\beqq
\cc_{slit}=\cc\setminus(\{0\}\cup\Ga_b\cup\Ga_c),\ \cc_{slit\pm}=\cc\cap\{z:\ \pm\mbox{Im}z>0\}.
\eeqq
\end{definition}

By the definition of $\cc_{slit}$, it is simply connected and does not contain $0$ and the critical value, $f^{-1}(\cc_{slit})$ consists of connected components, denoted by $(\mathcal{U}_j^{f})_{j\in I}$, where $I$ is the index set, each $\mathcal{U}_j^{f}$ is mapped by $f$ isomorphically onto $\cc_{slit}$. Denote by $\mathcal{U}_{j\pm}^{f}=f^{-1}(\cc_{slit\pm})\cap\mathcal{U}_{j}^{f}$, $\ga_{aj}^{f}=f^{-1}(\Ga_a)\cap\mathcal{U}_{j}^{f}$, $\ga_{bj\pm}^{f}=f^{-1}(\Ga_b)\cap\overline{\mathcal{U}}_{j\pm}^{f}$, and $\ga_{cj\pm^{f}}=f^{-1}(\Ga_c)\cap\overline{\mathcal{U}}_{j\pm}^{f}$, where the
closures are taken within the domain of $f$.

Since $f$ is homeomorphic near $0$, there is a component containing $0$, denoted by $\mathcal{U}_1^{f}$, and $\ga_{c1+}^{f}=\ga_{c1-}^{f}$. Further, since $f$ has a unique critical value and all critical points are of local degree $2$, there is another component, denoted by $\mathcal{U}_2^{f}$, where $\ga_{b1+}^{f}=\ga_{b2-}^{f}$ and $\ga_{b2+}^{f}=\ga_{b1-}^{f}$. And, $f^{-1}(cv)\cap\ga_{b1+}^{f}\cap\ga_{b2+}^{f}$ is a critical point, denoted by $cp=cp_f$.

Set
\beqq
\mathcal{U}_{12}^{f}:=\ol{\mathcal{U}_{1}}^{f}\cup\mathcal{U}_{2}^{f}.
\eeqq
It is evident that $f:\ \mathcal{U}_{12}^{f}\to\cc_{slit}\cup\Ga_b=\cc\setminus(\{0\}\cup\Ga_c)$ is a branched covering of degree $2$ over the critical value $cv$.

\subsection{A general class $\mathcal{F}_1^{P}$ and the main results}\label{parabolicbif-main-results}

In this subsection, a class of functions, $\mathcal{F}_1^{P}$, is introduced. Based on this type of functions, the main results are introduced.

\begin{definition}\label{def-2021-9-26-1}
Let $m\in\nn$ be a given constant. For the polynomial $P(z)=z(1+z)^m$, let $V$ be an open subset of $\cc$ containing $0$, where $V$ is dependent on $m$.  Define
\begin{equation*}
\mathcal{F}_{1}^{P}=\left\{f=P\circ \varphi^{-1}: \varphi(V) \rightarrow \mathbb{C}\ \bigg | \begin{array}{l}
\varphi:V\to\cc\ \text { is univalent, } \varphi(0)=0, \varphi^{\prime}(0)=1, \\
\text { and } \varphi \text { has a quasi-conformal extension to } \mathbb{C}
\end{array}\right\},
\end{equation*}
where univalent refers to holomorphic and injective. If $f\in\mathcal{F}_{1}^{P}$, then $0$ is a $1$-parabolic fixed point of $f$; if $cp_P\in V$, then $cp_f=\varphi(cp_P)$ is a critical point and $cv_P$ is a critical value of $f$, where $cp_P$ and $cv_P$ are introduced in \eqref{critip2011-8-30-1}.
\end{definition}

\begin{theorem}\label{paraequ-84} (Invariance of $\mathcal{F}_1^{P}$)
Let $m\in\nn$ be a given constant with $m\geq22$. There exist a Jordan domain $V$ (dependent on $m$) containing $0$ and $cp_P=-\tfrac{1}{m+1}$ with a smooth boundary and an open set $V'$ containing $\ol{V}$ such that
\begin{itemize}
\item [(a)] $f^{\prime\prime}(0)\neq0$, $ |f^{\prime\prime}(0)-(2m+1)|\leq1$, $cp_f\in \text{Basin}(0)$.
\item [(b)] There is a natural injection $((\mathcal{F}_0\setminus\{quadratic\ polynomials\})/\underset{linear}\sim)\hookrightarrow\mathcal{F}_1^{P}$.
\item [(c)] For $f\in\mathcal{F}_1^{P}$, the parabolic renormalization $\mathcal{R}_0f$ is well-defined, $\mathcal{R}_0f=P\circ\psi^{-1}\in\mathcal{F}_1$, where $\psi$ can be extended to a univalent function from $V'$ to $\cc$.
\item [(d)] The parabolic renormalization operator $\mathcal{R}_0$ is holomorphic in the following sense: given a family $f_{\ld}=P\circ\varphi^{-1}_{\ld}$ with a holomorphic function $\varphi_{\ld}(z)$ in two variables $(\ld,z)\in\Ld\times V$, where $\Ld$ is a complex manifold, the renormalization can be written as $\mathcal{R}_0(f_{\ld})=P\circ \psi^{-1}_{\ld}$ with $\psi_{\ld}$ holomorphic in $(\ld,z)\in \Ld\times V'$.
\end{itemize}
\end{theorem}

\begin{theorem} (Contraction in the Teichm\"{u}ller metric)\label{ren2020912-2}
Let $m\in\nn$ be a given constant with $m\geq22$. There is a one-to-one correspondence between each element in $\mathcal{F}_1^P$ and any element in the Teichm\"{u}ller space defined on $\cc\setminus\ol{V}$, on which the parabolic renormalization operator is contraction
\beqq
d_{Teich}(\mathcal{R}_0(f_1),\mathcal{R}_0(f_2))\leq \ld\cdot d_{Teich}(f_1,f_2),
\eeqq
where $d_{Teich}(\cdot,\cdot)$ is the complete distance on $\mathcal{F}_1^P$ induced from the Teichm\"{u}ller distance, and $\ld=e^{-2\pi\cdot\text{mod}(V'\setminus \ol{V})}<1$.
\end{theorem}

\begin{corollary}\label{equ2022-2-8-1}
The parabolic renormalization $\mathcal{R}_0$ on $\mathcal{F}_1^P$ has a unique fixed point, which belongs to $\mathcal{F}_0$. For any $f\in \mathcal{F}_1^P$,
$\{\mathcal{R}_0^nf\}^{\infty}_{n=0}$ converges to the fixed point exponentially fast with respect to the metric introduced in Theorem \ref{ren2020912-2}. Furthermore, for $f\in\mathcal{F}_0$, the renormalization $\mathcal{R}_0^nf$ considered as elements in $\mathcal{F}_0$ converges to the fixed point uniformly in the compact-open topology.
\end{corollary}

\begin{definition}\label{produspace-2021-9-1}
Given a constant $\al_{*}>0$. Denote by $\mathcal{I}(\al_{*})=(-\al_*,\al_*)\setminus\{0\}$. Consider the function space:
\beqq
e^{2\pi i\mathcal{I}(\al_*)}\times\mathcal{F}_1^P=\{e^{2\pi i\al}\cdot h(z):\ \al\in\mathcal{I}(\al_*)\ \text{and}\ h(z)\in\mathcal{F}_1^P\}.
\eeqq
The distance is defined by
$$d(f,g)=d_{Teich}\bigg(\frac{f}{f^{\prime}(0)},\frac{g}{g^{\prime}(0)}\bigg)+|f^{\prime}(0)-g^{\prime}(0)|,$$
where $d_{Teich}$ is the distance induced in Theorem \ref{ren2020912-2}.
\end{definition}

\begin{theorem} (Invariance of $e^{2\pi i\mathcal{I}(\al_*)}\times\mathcal{F}_1^P$ under the near-parabolic renormalization $\mathcal{R}_{\al}$) \label{ren2020912-1}
Let $m\in\nn$ be a given constant with $m\geq22$.
There is a sufficiently small positive number $\al_*$ such that for any $\al\in\cc$ with $0<|\al|\leq\al_*$ and $|\arg\,\al|<\tfrac{\pi}{4}$ (or $|\arg\,(-\al)|<\tfrac{\pi}{4}$), then the near-parabolic renormalization operator $\mathcal{R}_{\al}$ is defined in $e^{2\pi i\mathcal{I}(\al_*)}\times\mathcal{F}_1^P$ has the following properties:
\begin{itemize}
\item [(i)] For $f\in e^{2\pi i\mathcal{I}(\al_*)}\times\mathcal{F}_1^P$, the near-parabolic renormalization $\mathcal{R}_{\al}f$ is well-defined, $\mathcal{R}_{\al}f=e^{2\pi i\al'}\cdot P\circ\psi^{-1}$, $e^{2\pi i\al'}\cdot \psi\in e^{2\pi i\mathcal{I}(\al_*)}\times\mathcal{F}_1^P$, where $\psi$ can be extended to a univalent function from $V'$ to $\cc$.
\item [(ii)] The near-parabolic renormalization operator $\mathcal{R}_{\al}$ is holomorphic in the following sense: given a family $f_{\ld}=e^{2\pi i\al}P\circ\varphi^{-1}_{\ld}$ with a holomorphic function $\varphi_{\ld}(z)$ in two variables $(\ld,z)\in\Ld\times V$, where $\Ld$ is a complex manifold, the renormalization can be written as $\mathcal{R}_{\al}(f_{\ld})=e^{2\pi i\al'}\cdot P\circ \psi^{-1}_{\ld}$ with $\psi_{\ld}$ holomorphic in $(\ld,z)\in \Ld\times V'$.
\item[(iii)]For $f_1,f_2\in e^{2\pi i\mathcal{I}(\al_*)}\times\mathcal{F}_1^P$,
$d(\mathcal{R}_{\al}(f_1),\mathcal{R}_{\al}(f_2))\leq \ld\cdot d(f_1,f_2)$,
where the distance $d$ is introduced in Definition \ref{produspace-2021-9-1} and $\ld=e^{-2\pi\cdot\text{mod}(V'\setminus \ol{V})}<1$.
\end{itemize}

\end{theorem}

\begin{remark}
The statements (i) and (ii) in Theorem \ref{ren2020912-1} are corresponding to statements (c) and (d) in Theorem \ref{paraequ-84}, the statement (iii) is corresponding to Theorem \ref{ren2020912-2}.
\end{remark}

\begin{corollary}\label{cor-2021-9-26-8}
There is a sufficiently large positive integer $N$ such that for $\al\in\text{Irrat}_{\geq N}$, the sequence of functions given by the near-parabolic renormalizaiton, $(f_n)_{n\geq0}$, is given by $f_{n+1}=\mathcal{R}_{\al}(f_n)$ and $(f_n)_{n\geq0}\subset e^{2\pi i\mathcal{I}(\al_*)}\times\mathcal{F}_1^P$. And, the near-parabolic renormalization is contraction under the Teichm\"{u}ller metric.
\end{corollary}

\begin{corollary}\label{cor-2021-9-26-6}
There is a sufficiently large positive integer $N$ such that for $\al\in\text{Irrat}_{\geq N}$
and $f(z)=e^{2\pi i\al}h(z)$ with $h\in\mathcal{F}_1^P$, the post-critical set of $f$ are contained in the domain of the definition of $f$ and can be iterated infinitely many times. Furthermore, there is an infinite sequence of periodic orbits to which the critical orbit does not accumulate.
\end{corollary}

\subsection{Proof of Theorem \ref{paraequ-84}-Invariance of $\mathcal{F}^P_1$ under the parabolic-renormalization $\mathcal{R}_0$}\label{equ2021-12-2}

The main problem is the study of (c) of  Theorem \ref{paraequ-84}, to show the existence of $\psi$ such that
\beqq
\mathcal{R}_0f=\Psi_0\circ E_f\circ\Psi^{-1}_0=P\circ \psi^{-1},
\eeqq
where $\Psi_0=c\cdot\text{Exp}^{\sharp}(z)$ with some constant $c\in\cc^{*}$. So, by Definition \ref{renorm2020912-3}, $\psi$ should be written as
\begin{align}\label{equ-2021-12-1}
\psi=&\Psi_0\circ(E_f)^{-1}\circ \Psi_0^{-1}\circ P=\Psi_0\circ(\Phi_{att,f}\circ\Phi_{rep,f}^{-1})^{-1}\circ \Psi_0^{-1}\circ P\nonumber\\
=&\Psi_0\circ \Phi_{rep,f}\circ\Phi_{att,f}^{-1}\circ\Psi_0^{-1}\circ P
=\Psi_0\circ\Phi_{rep,f}\circ f^{-n}\circ\Phi_{att,f}^{-1}\circ \Psi_0^{-1}\circ P.
\end{align}
The last equality is derived by the fact $\Phi_{rep,f}(f(z))=\Phi_{rep,f}(z)+1$ and $\Psi_0(z+1)=\Psi_0(z)$. And,
$\Phi_{att,f}$ and $\Phi_{rep,f}$ are defined in attracting and repelling half-neighborhoods
of $0$, which are corresponding to $\{z:\ \text{Re}\,z>L\}$ and  $\{z:\ \text{Re}\,z<-L\}$ for $F$ in Theorem \ref{paraequ-85}, then the inverse branch $f^{-n}$ maps part of attracting half-neighborhood to repelling one. Note that the multi-valuedness and branching of $f^{-n}$ should be balanced by the properties of the  map $P$ and the value of $m$ at the beginning of the composition.

The main technique is the movement of the fixed point $0$ to $\infty$, the introduction of a new map $Q$ (see \eqref{renorm2020912-5}) having a parabolic fixed point at $\infty$, several new classes of functions  (see Definition \ref{renorm202091206}), and study the relationship between these functions (see Propositions \ref{paraequ-86} and \ref{equ-20208-23-1}).

On the repelling side of the fixed point, a Riemann surface $X$ with a projection $\pi_X:\ X\to\cc$ is introduced, a function $g:\ X\to X$ corresponding to an inverse branch of $f$ is given,  the repelling Fatou coordinate is defined on the Riemann surface (see Propositions \ref{paraequ-14} and \ref{near-equat-2}).

On the attracting side of the fixed point, Proposition \ref{attracting2020912-1} provides an estimate on $\Phi_{att}$ in the region $\text{Re}\,\Phi_{att}(z)\geq1$, where the normalization $\Phi_{att}(cv_Q)=1$ is used. Two certain regions are given, and the inverse domains are given, which can be lifted to the Riemann surface $X$ (see Proposition \ref{attracting2020912-2}). Based on these results, a domain is partition into smaller pieces, on which the map $\psi$ is defined, this completes the proof of the existence of $\psi$ (see Proposition \ref{attracting202091203}).

Consider the functions
\beqq
\psi_1(z)=4\psi_{Koebe}(-\tfrac{1}{z})=-\frac{4z}{(1+z)^2},\ \psi_0(z)=-\frac{4}{z},
\eeqq
and
\beq\label{renorm2020912-5}
Q=\psi_0^{-1}\circ P\circ \psi_1=\frac{(1+z)^{2+2m}}{z(1-z)^{2m}}=z\frac{\big(1+\frac{1}{z}\big)^{2+2m}}{\big(1-\frac{1}{z}\big)^{2m}}.
\eeq

\begin{lemma}\label{paraequ-27}\cite[Lemma 5.10]{InouShishikura2016}
Consider the function
\beq\label{paraequ-1}
\zeta(w)=e_1 w+e_0+\frac{e_{-1}}{w},\ w\in\cc\setminus\{0\},
\eeq
where $e_1,e_0,e_{-1}\in\rr$ are constants, this function is a conformal map from $\cc\setminus\overline{\cd}\to\cc\setminus E_1$, that is, this map sends $\{w:\ |w|=r,\ r\geq1\}$ onto $\partial E_r$, where
\beq\label{paraequ-3}
E_r=\bigg\{x+iy:\ \bigg(\frac{x-e_0}{a_{E}(r)}\bigg)^2+\bigg(\frac{y}{b_E(r)}\bigg)^2\leq1\bigg\},
\eeq
and
\beq\label{paraequ-90}
a_{E}(r)=e_1 r+\frac{e_{-1}}{r}\ \mbox{and}\ b_{E}(r)=e_1r-\frac{e_{-1}}{r}
\eeq
are non-zero.
\end{lemma}

\begin{proof}
Let $w=re^{i\tht}$ with $r\in\rr_{+}$ and $\tht\in\rr$, one has
\begin{align*}
\zeta(w)&=\zeta(re^{i\tht})=e_1 re^{i\tht}+e_0+\frac{e_{-1}}{re^{i\tht}}=
e_1 r(\cos(\tht)+i\sin(\tht))+e_0+\frac{e_{-1}}{r}e^{-i\tht}\\
=&e_1 r(\cos(\tht)+i\sin(\tht))+e_0+\frac{e_{-1}}{r}(\cos(\tht)-i\sin(\tht))\\
=&e_0+\bigg(e_1 r+\frac{e_{-1}}{r}\bigg)\cos(\tht)+i\bigg(e_1 r-\frac{e_{-1}}{r}\bigg)\sin(\tht).
\end{align*}
\end{proof}

\begin{definition}
In the following discussions, we choose $e_1=0.84$, $e_0=-0.18$, and $e_{-1}=0.6$. Consider the elliptic region with $a_{E}=a_E(1)=0.84+0.6=1.44$ and $b_E=b_E(1)=0.84-0.6=0.24$,
\begin{align*}
E=E_{1}=&\bigg\{z=x+iy\in\cc:\ \bigg(\frac{x-e_0}{a_E}\bigg)^2+\bigg(\frac{y}{b_E}\bigg)^2\leq1\bigg\}\\
=&\bigg\{z=x+iy\in\cc:\ \bigg(\frac{x+0.18}{1.44}\bigg)^2+\bigg(\frac{y}{0.24}\bigg)^2\leq1\bigg\}.
\end{align*}

Given any $m\in\nn$, set
\beq\label{near-equat-6}
\eta:=\frac{1}{2\pi}\left(\log(12(m+1)\times 30^{2+2m})+1\right).
\eeq
Recall in Proposition \ref{paraequ-2}, the critical value of $P$ is $cv_P=-\frac{m^m}{(m+1)^{m+1}}$, set
\begin{align}\label{paraequ-92}
V^{'}:=&\bigg(\mathcal{U}^P_{*}\bigcap P^{-1}\big(\cd(0,|cv_P|e^{2\pi\eta})\big)\bigg)\nonumber\\
&\setminus\big(\mbox{the component of}\ P^{-1}(\cd(0,|cv_P|e^{-2\pi\eta}))\ \mbox{containing}\ -1\big)
\end{align}
and
\beq\label{para2021-9-1-2}
V:=\psi_1\bigg(\widehat{\cc}\setminus \big(E\cup \overline{\cd}\big(\cot(\tfrac{\pi}{m+1})i,\tfrac{1}{\sin(\tfrac{\pi}{m+1})}\big)\cup \overline{\cd}\big(-\cot(\tfrac{\pi}{m+1})i,\tfrac{1}{\sin(\tfrac{\pi}{m+1})}\big)\big)\bigg),
\eeq
where $\mathcal{U}^P_{*}$ is introduced in \eqref{defPdiff2021-8-30-3}.

\end{definition}

\begin{proposition}\label{paraequ-86}
For $\eta$ introduced in \eqref{near-equat-6}, $V$ introduced in \eqref{para2021-9-1-2}, and $V'$ in \eqref{paraequ-92}, one has
\beqq
\ol{V}\subset V'.
\eeqq
\end{proposition}

\begin{definition}
A univalent map is an injective holomorphic map, it is allowed to take value $\infty$ in general; it is called normalized if $\varphi(0)=0$ and $\varphi'(0)=1$ when $0$ is in the domain, or $\varphi(\infty)=\infty$ and $\lim_{z\to\infty}\tfrac{\varphi(z)}{z}=1$ when $\infty$ is in the domain instead of $0$.
\end{definition}

\begin{definition}\label{renorm202091206}
Three classes of functions are introduced:
\begin{equation*}
\mathcal{F}^{P}_{2}=\left\{f=P\circ \varphi^{-1}: \varphi(V') \rightarrow \mathbb{C}\ \bigg | \begin{array}{l}
\varphi:V'\to\cc\ \text { is univalent, } \varphi(0)=0, \varphi^{\prime}(0)=1
\end{array}\right\},
\end{equation*}
\begin{equation*}
\mathcal{F}^{Q}_{1}=\left\{f=Q\circ \varphi^{-1}\ \bigg | \begin{array}{l}
\varphi:\widehat{\cc}\setminus E\to\widehat{\cc}\setminus\{0\}\ \text { is a normalized univalent map } \\
\text { and } \varphi \text { has a quasi-conformal extension to }\widehat{\cc}
\end{array}\right\},
\end{equation*}
and
\begin{equation*}
\mathcal{F}^{Q}_{2}=\left\{f=Q\circ \varphi^{-1}\ \bigg | \begin{array}{l}
\varphi:\widehat{\cc}\setminus \bigg(E\cup\overline{\cd}\big(\cot(\tfrac{\pi}{m+1})i,\tfrac{1}{\sin(\tfrac{\pi}{m+1})}\big)\cup \overline{\cd}\big(-\cot(\tfrac{\pi}{m+1})i,\tfrac{1}{\sin(\tfrac{\pi}{m+1})}\big)\bigg)\\
\to\widehat{\cc}\setminus\{0\}\ \text {is a normalized univalent map and has a conformal} \\
\text {extension to}\ \widehat{\cc}\setminus E,\ \text{has a quasi-conformal extension to }\widehat{\cc}
\end{array}\right\}.
\end{equation*}
\end{definition}

\begin{definition} (Riemann surface $X$)
Let $R=2.66\times 10^m$ (see Remark \ref{nearequ-22}),\ $\rho=0.07$ (see Remark \ref{nearequ-21}), and define the sheets by
\beqq
X_{1\pm}=\{z:\in\cc:\ \pm\text{Im}\,z\geq0,\ |z|>\rho\ \text{and}\ \tfrac{\pi}{5}<\pm\text{arg}\,(z-cv_Q)\leq\pi\},
\eeqq
\beqq
X_{2\pm}=\{z:\in\cc:\ \pm\text{Im}\,z\geq0,\ \rho<|z|<R\ \text{and}\ \tfrac{\pi}{5}<\pm\text{arg}\,(z-cv_Q)\leq\pi\}.
\eeqq
The four sheets are thought of as the joining of distinct copies of $\cc$, the projection maps $\pi_{i\pm}:X_{i\pm}\to\cc$, $i=1,2$, are the natural projection. The Riemann surface $X$ is constructed as follows: $X_{1+}$ and $X_{1-}$ are glued along negative real axis, i.e., $\pi^{-1}_{1+}(x)\in X_{1+}$ and $\pi^{-1}_{1-}(x)\in X_{1-}$ are identified for $x\in(-\infty, -\rho)$; $X_{1+}$ and $X_{2-}$ are glued along $(\rho,cv_Q)$; $X_{1-}$ and $X_{2+}$ are glued along $(\rho,cv_Q)$. The projection $\pi_X:X\to\cc$ is defined as $\pi_X=\pi_{i\pm}$ on $X_{i\pm}$. The complex structure is induced by the projection.   Figure \ref{illustration-riemannsurface} is the illustration diagram of this Riemann surface, where the curves with the same color are glued together.
\end{definition}

\begin{figure}[htbp]
\begin{center}
\scalebox{0.5}{ \includegraphics{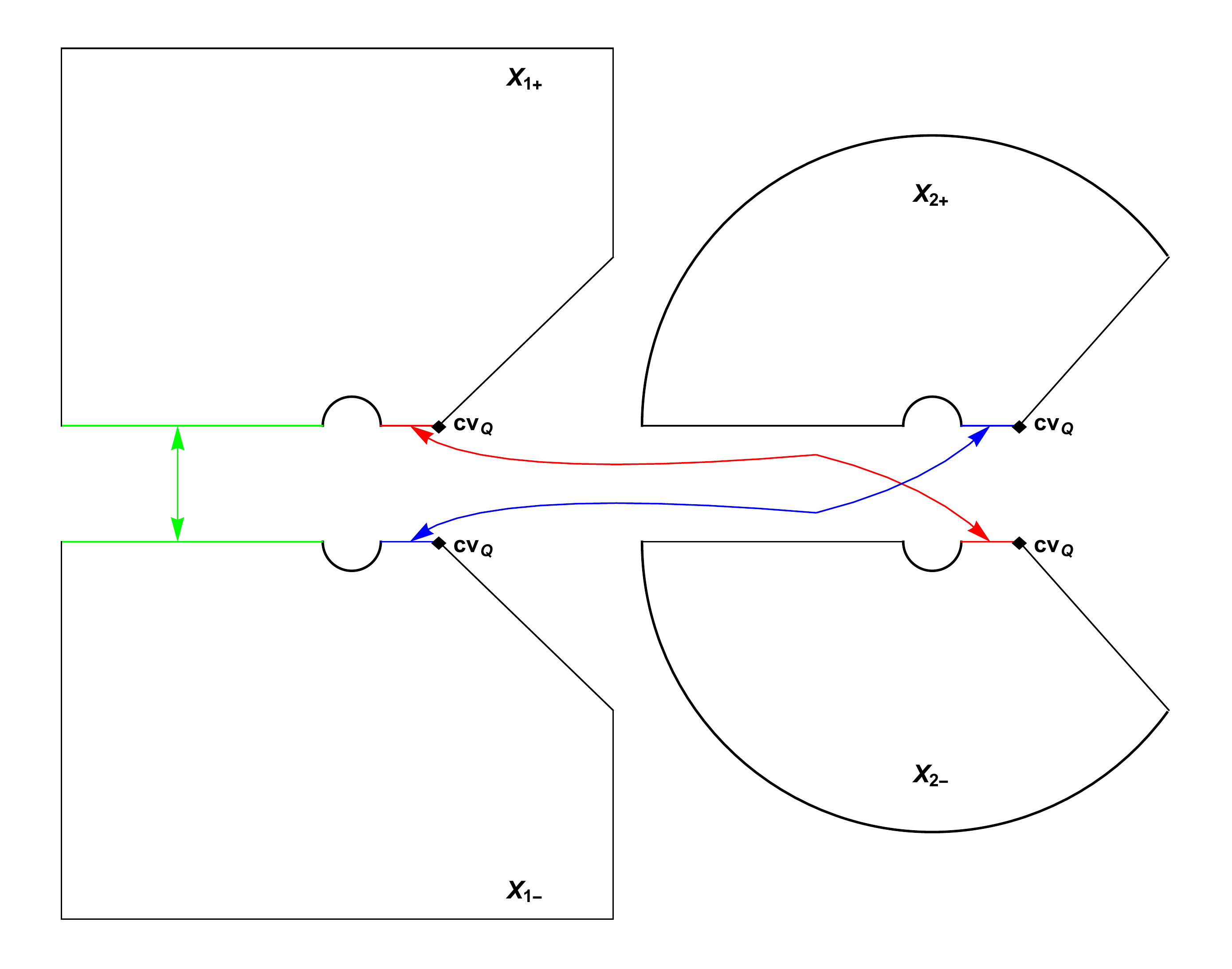}}
\renewcommand{\figure}{Fig.}
\caption{The illustration diagram of the Riemann surface, where the curves with the same color are glued together.
}\label{illustration-riemannsurface}
\end{center}
\end{figure}

\begin{proposition}\label{paraequ-14} (Lifting $Q$ and $\varphi$ to $X$)
Assume $m\geq5$. There exists an open subset $Y\subset\cc\setminus(E_{r_1}\cup\rr_+)$ with the following properties:
\begin{itemize}
\item[(a)] There exists an isomorphism $\widetilde{Q}:Y\to X$ such that $\pi_X\circ\widetilde{Q}=Q$ on $Y$ and $\widetilde{Q}^{-1}(z)=\pi_X(z)-(4m+2)+o(1)$ as $z\in X$ and $\pi_X(z)\to\infty$.
\item[(b)]  The map $\varphi$ restricted to $Y$ can be lifted to a univalent holomorphic map $\tilde{\varphi}:Y\to X$ satisfying $\pi_X\circ\tilde{\varphi}=\varphi$ on $Y$.
\end{itemize}
\end{proposition}

This will be verified in Subsection \ref{riemann-surface-2021-12-29-1}. The introduction of the Riemann surface $X$ makes the lift $F^{-1}=\varphi\circ Q^{-1}$ to a single-valued branch, such that the iteration of the lift of $F^{-1}$ will stay in $X$ by proper definition of $\varphi$.

Set
\beq
g:=\widetilde{\varphi}\circ\widetilde{Q}^{-1}.
\eeq
By the definition, $g$ is a map from $X$ to $X$.

\begin{proposition}\label{near-equat-2} (Repelling Fatou coordinate on $X$)
The map $g$ satisfies the equation $F\circ\pi_X\circ g=\pi_X$. There is an injective holomorphic map $\widetilde{\Phi}_{rep}:X\to\cc$ such that $\widetilde{\Phi}_{rep}(g(z))=\widetilde{\Phi}_{rep}(z)-1$. Furthermore, $\widetilde{\Phi}_{rep}\circ\pi^{-1}_X$ is a repelling Fatou coordinate for $F=Q\circ\varphi^{-1}$ on $\{z:\ \text{Re}\,z<-R\}$, where $R$ is specified in Lemma \ref{paraequ-15}.
\end{proposition}

This will be studied in Subsection \ref{repell-fatou-2021-12-29-2}.

\begin{definition}
For any $z_0\in\cc$ and $\tht>0$, set
\beqq
\vv(z_0,\tht):=\{z:\ z\neq z_0,\ |\text{arg}\,(z-z_0)|<\tht\},
\eeqq
and $\overline{\vv}(z_0,\tht)$ is the closure of $\vv(z_0,\tht)$.

Set
\begin{align}\label{paraequ-81}
W_1:=&\vv\left(cv_Q,\frac{7\pi}{10}\right){\Huge\setminus}  \overline{\vv}\left(F(cv_Q),\frac{3\pi}{10}\right)\nonumber\\
=&\bigg\{z:\ |\arg\,(z-cv_Q)|<\frac{7\pi}{10}\ \text{and}\ |\arg\,(z-F(cv_Q))-\pi|<\frac{7\pi}{10}\bigg\}.
\end{align}
\end{definition}

Lemma \ref{connect-2021-12-29-3} will show that $W_1$ is connected for $m\geq22$.

\begin{proposition}\label{attracting2020912-1} (Attracting Fatou coordinate and shape of $D_1$)
Assume $m\geq22$, $u_0=\tfrac{(\sqrt{m+1}+\sqrt{m})^2+2.1}{\cos(\tfrac{\pi}{5})}$, and $R_1=2.39\times 10^m$.
\begin{itemize}
\item[(a)] The $F$ maps $\vv(u_0,\tfrac{7\pi}{10})$ into itself and $\vv(u_0,\tfrac{7\pi}{10})$ is contained in the attraction basin of $\infty$ under the map $F$. There exists an attracting Fatou coordinate $\Phi_{att}:\vv(u_0,\tfrac{7\pi}{10})\to\cc$ such that
$$\Phi_{att}(F(z))=\Phi_{att}(z)+1\ \text{and}\ \Phi_{att}(cv_Q)=1.$$
\item[(b)] There are domains $D_1$, $D^{\sharp}_1$,
$D^{\flat}_1\subset W_1\subset\vv(u_0,\tfrac{7\pi}{10})$ such that
\begin{align*}
\Phi_{att}(D_1)&=\{z:\ 1<\text{Re}\,z<2,\ -\eta<\text{Im}\,z<\eta\}\ \text{and}\ D_1\subset\cd(cv_Q,R_1),\\
\Phi_{att}(D^{\sharp}_1)&=\{z:\ 1<\text{Re}\,z<2,\ \text{Im}\,z>\eta\}\ \text{and}\ D^{\sharp}_1\subset\{z:\ \tfrac{\pi}{5}<\arg\,(z-cv_Q)<\tfrac{7\pi}{10}\},\\
\Phi_{att}(D^{\flat}_1)&=\{z:\ 1<\text{Re}\,z<2,\ \text{Im}\,z<-\eta\}\ \text{and}\
D^{\flat}_1\subset\{z:\ -\tfrac{7\pi}{10}<\arg\,(z-cv_Q)<-\tfrac{\pi}{5}\},\
\end{align*}
where $\eta$ is specified in \eqref{near-equat-6}.
\end{itemize}

\end{proposition}

The proof of Proposition \ref{attracting2020912-1} is provided in Subsection \ref{attractingfatoucoor-2021-10-25-2}.

\begin{proposition}\label{attracting2020912-2} (Domains around critical point) Assume $m\geq3$.
There exist disjoint Jordan domains $D_0$, $D^{\prime}_{0}$, $D_{-1}$, $D^{\prime\prime}_{-1}$, and a domain $D^{\sharp}_{0}$  such that
\begin{itemize}
\item[(a)] the closures $\ol{D}_0$, $\ol{D}^{\prime}_{0}$, $\ol{D}_{-1}$, $\ol{D}^{\prime\prime}_{-1}$, and $\ol{D}^{\sharp}_{0}$ are contained in $\text{Image}(\varphi)=\text{Dom}(F)$;
 \item[(b)] $F(D_0)=F(D^{\prime}_{0})=D_1$, $F(D_{-1})=F(D^{\prime\prime}_{-1})=D_0$, $F(D^{\sharp}_{0})=D^{\sharp}_{1}$;
 \item[(c)] $F$ is injective on each of these domains;
 \item[(d)] $cp_F=\varphi(cp_Q)\in \ol{D}_0\cap\ol{D}^{\prime}_{0}\cap\ol{D}_{-1}\cap\ol{D}^{\prime\prime}_{-1}$,
$\ol{D}_0\cap\ol{D}_1\neq\emptyset$, $\ol{D}^{\sharp}_{0}\cap\ol{D}^{\sharp}_{1}\neq\emptyset$, $\ol{D}_{-1}\cap\ol{D}^{\sharp}_{0}\neq\emptyset$;
 \item[(e)] $ \ol{D}_0\cup\ol{D}^{\prime}_{0}\cup\ol{D}_{-1}\cup\ol{D}^{\prime\prime}_{-1}\setminus\{cv_Q\}\subset\pi_X (X_{2+})\cup\pi_{X}(X_{2-})=\cd(0,R)\setminus(\ol{\cd}(0,\rho)\cup\rr_{-}\cup\ol{\vv}(cv_Q,\tfrac{3\pi}{10}))$ and $\ol{D}^{\sharp}_0\subset\pi_X(X_{1+})$.
\end{itemize}

\end{proposition}

Proposition \ref{attracting2020912-2} will be shown in Subsection \ref{location2021-10-25-1}.

\begin{proposition}\label{attracting202091203} (Relating $E_F$ to $P$)
The parabolic renormalization $\mathcal{R}_0F$
belongs to the class $\mathcal{F}_2^P$ (possibly after
a linear conjugacy).

Regard $D_0$, $D_0'$, $D^{\prime\prime}_{-1}$, $D^{\sharp}_0$
as subsets of $X_{1+}\cup X_{2-}\subset X$, and let
\beqq
U=\text{the interior of }\ \bigcup^{\infty}_{n=0}
g^n(\overline{D}_0\cup\overline{D}_0^{\sharp}
\cup\overline{D}_0^{\prime}\cup\overline{D}_{-1}^{\prime\prime}).
\eeqq
Then there exists a surjective holomorphic map $\Psi_1:U\to U^{P}_{\eta}\setminus\{0\}=V^{\prime}\setminus\{0\}$ such that
\begin{itemize}
\item[(a)] $P\circ\Psi_1=\Psi_0\circ\widetilde{\Phi}_{att}$ on $U$, where $\Psi_0:\cc\to\cc^*$, $\Psi_0(z)=cv_P\cdot e^{2\pi i z}=cv_P\cdot Exp^{\sharp}(z)$, and $\widetilde{\Phi}_{att}: U\to\cc$ is the natural extension of the attracting Fatou coordinate to $U$;
\item[(b)] $\Psi_1(z)=\Psi_1(z^{\prime})$ if and only if $z^{\prime}=g^n(z)$ or $z=g^n(z^{\prime})$ for some integer $n\geq0$;
\item[(c)] $\psi=\Psi_0\circ \widetilde{\Phi}_{att}\circ \Psi^{-1}_1: V^{\prime}\setminus\{0\}\to\cc^{*}$ is well-defined and extends to a normalized univalent function on $V^{\prime}$;
\item[(d)]  on $\psi(V^{\prime}\setminus\{0\})$, the following holds
\beqq
P\circ \psi^{-1}=P\circ\Psi_1\circ\wt{\Phi}^{-1}_{rep}\circ\Psi^{-1}_0
=\Psi_0\circ\wt{\Phi}_{att}\circ\wt{\Phi}^{-1}_{rep}\circ\Psi^{-1}_0
=\Psi_0\circ E_F\circ\Psi^{-1}_0;
\eeqq
\item[(e)] the holomorphic dependence as in (d) of Theorem \ref{paraequ-84} holds.
\end{itemize}

\end{proposition}

Proposition \ref{attracting202091203} will be proved in Subsection \ref{proofpro2021-10-3-1}. Figure \ref{fig2021-10-3-3} is an illustration diagram of Proposition \ref{attracting202091203}.

\begin{figure}
\begin{center}
\begin{tikzpicture}[scale=1,line width=1pt]

\draw[color=black,line width=1.5pt] (6, 5)to [out=0,in=160] (8, 4);

\draw[color=black,line width=1.5pt] (6,0)--(6,7)--cycle;

\draw[color=black,line width=1.5pt] (6,0)--(6,-7)--cycle;

\draw[color=black,line width=1.5pt] (8,0)--(8,7)--cycle;


\draw[color=black,line width=1.5pt] (8,0)--(8,-7)--cycle;

\draw[dashed,line width=1.5pt] (6,0)--(8,0);

\node at (0,0) {O};

\draw[green,line width=1] (0,2) arc (0:180:1);
\draw[green,line width=1] (-0.25,2) arc (0:180:0.25);
\draw[green,line width=1] (-1.25,2) arc (0:180:0.25);
\draw[green,line width=1] (-1.25,2)--(-0.75,2);
\draw[green,line width=1] (-0.25,2)--(0,2);
\draw[green,line width=1] (-2,2)--(-1.75,2);
\draw[dashed,line width=2] (-0.5,2.25)--(-1,3);
\draw[dashed,line width=2] (-1.5,2.25)--(-1,3);
\draw[line width=1] (-1,2)--(-1,7);

\draw[green,line width=1] (0-2.5,2) arc (0:180:1);
\draw[green,line width=1] (-0.25-2.5,2) arc (0:180:0.25);
\draw[green,line width=1] (-1.25-2.5,2) arc (0:180:0.25);
\draw[green,line width=1] (-1.25-2.5,2)--(-0.75-2.5,2);
\draw[green,line width=1] (-0.25-2.5,2)--(0-2.5,2);
\draw[green,line width=1] (-2-2.5,2)--(-1.75-2.5,2);
\draw[dashed,line width=2] (-0.5-2.5,2.25)--(-1-2.5,3);
\draw[dashed,line width=2] (-1.5-2.5,2.25)--(-1-2.5,3);
\draw[line width=1] (-1-2.5,2)--(-1-2.5,7);

\draw[green,line width=1] (0-5,2) arc (0:180:1);
\draw[green,line width=1] (-0.25-5,2) arc (0:180:0.25);
\draw[green,line width=1] (-1.25-5,2) arc (0:180:0.25);
\draw[green,line width=1] (-1.25-5,2)--(-0.75-5,2);
\draw[green,line width=1] (-0.25-5,2)--(0-5,2);
\draw[green,line width=1] (-2-5,2)--(-1.75-5,2);
\draw[dashed,line width=2] (-0.5-5,2.25)--(-1-5,3);
\draw[dashed,line width=2] (-1.5-5,2.25)--(-1-5,3);
\draw[line width=1] (-1-5,2)--(-1-5,7);


\draw[green,line width=1] (0-7.5,2) arc (0:180:1);
\draw[green,line width=1] (-0.25-7.5,2) arc (0:180:0.25);
\draw[green,line width=1] (-1.25-7.5,2) arc (0:180:0.25);
\draw[green,line width=1] (-1.25-7.5,2)--(-0.75-7.5,2);
\draw[green,line width=1] (-0.25-7.5,2)--(0-7.5,2);
\draw[green,line width=1] (-2-7.5,2)--(-1.75-7.5,2);
\draw[dashed,line width=2] (-0.5-7.5,2.25)--(-1-7.5,3);
\draw[dashed,line width=2] (-1.5-7.5,2.25)--(-1-7.5,3);
\draw[line width=1] (-1-7.5,2)--(-1-7.5,7);

\draw[green,line width=1] (1,2) arc (90:-90:2);
\draw[green,line width=1] (1,1) arc (90:-90:1);
\draw[green,line width=1] (1,1)--(1,2);
\draw[green,line width=1] (1,-1)--(1,-2);
\draw[dashed,line width=2] (2,0)--(3,0);
\draw[line width=2, black,line cap=round] (3,0) to [out=60,in=260](4,7);
\draw[line width=2, black,line cap=round] (3,0) to [out=-60,in=-260](4,-7);

\draw[dashed,line width=2] (3,0)--(6,0);

\fill[black] (8, 0) circle (0.2);

\fill[white] (6,0) circle (0.2);
\draw (6,0) circle (0.21);

\fill[white] (3,0) circle (0.2);
\draw (3,0) circle (0.21);

\draw[line width=1] (-8.5,5)--(6,5);

\draw[line width=1] (-7.5,2)--(-7,2);
\draw[line width=1] (-5,2)--(-4.5,2);
\draw[line width=1] (-2.5,2)--(-2,2);
\draw[line width=1] (0,2)--(1,2);

\draw[line width=1] (0,-2)--(1,-2);
\draw[green,line width=1] (-1,-3) arc (270:360:1);
\draw[line width=1] (-1,-3)--(-1,-7);
\draw[line width=1] (-1,-5)--(8,-5);


\node at  (7,6) {$D^{\sharp}_1$};
\node at  (5,6) {$D^{\sharp}_0$};
\node at  (1.5,6) {$D^{\sharp}_{-1}$};
\node at  (-2.2,6) {$D^{\sharp}_{-2}$};
\node at  (-4.5,6) {$D^{\sharp}_{-3}$};
\node at  (-7,6) {$D^{\sharp}_{-4}$};


\node at  (7,3) {$D_1$};
\node at  (5,3) {$D_0$};
\node at  (1.5,4) {$D_{-1}$};
\node at  (-2.2,4) {$D_{-2}$};
\node at  (-4.5,4) {$D_{-3}$};
\node at  (-7,4) {$D_{-4}$};

\node at (1.5,-4) {$D^{\prime\prime}_{-1}$};
\node at (-0.5,1) {$D^{\prime\prime}_{-2}$};
\node at (-3,1) {$D^{\prime\prime}_{-3}$};
\node at (-5.5,1) {$D^{\prime\prime}_{-4}$};
\node at (-8,1) {$D^{\prime\prime}_{-5}$};

\draw[red,line width=1] (-0.5,1+0.25)--(-0.5,2.1);
\draw[red,line width=1] (-3,1+0.25)--(-3,2.1);
\draw[red,line width=1] (-5.5,1+0.25)--(-5.5,2.1);
\draw[red,line width=1] (-8,1+0.25)--(-8,2.1);

\node at (-0.5-1,1-1) {$D^{\prime}_{-1}$};
\node at (-3-1,1-1) {$D^{\prime}_{-2}$};
\node at (-5.5-1,1-1) {$D^{\prime}_{-3}$};
\node at (-8-1,1-1) {$D^{\prime}_{-4}$};

\draw[blue,line width=1] (-0.5-1,1-1+0.3)--(-0.5-1,2.1);
\draw[blue,line width=1] (-3-1,1-1+0.3)--(-3-1,2.1);
\draw[blue,line width=1] (-5.5-1,1-1+0.3)--(-5.5-1,2.1);
\draw[blue,line width=1] (-8-1,1-1+0.3)--(-8-1,2.1);


\draw[red,line width=1] (0,2) arc (0:90:1);
\draw[red,line width=1] (-2.5,2) arc (0:90:1);
\draw[red,line width=1] (-5,2) arc (0:90:1);
\draw[red,line width=1] (-7.5,2) arc (0:90:1);

\draw[red,line width=1] (-0.25,2) arc (0:180:0.25);
\draw[red,line width=1] (-2.75,2) arc (0:180:0.25);
\draw[red,line width=1] (-5.25,2) arc (0:180:0.25);
\draw[red,line width=1] (-7.75,2) arc (0:180:0.25);

\draw[red,line width=1] (-0.25,2)--(0,2);
\draw[red,line width=1] (-1,2)--(-0.75,2);
\draw[red,line width=1] (-2.75,2)--(-2.5,2);
\draw[red,line width=1] (-3.5,2)--(-3.25,2);
\draw[red,line width=1] (-5.25,2)--(-5,2);
\draw[red,line width=1] (-6,2)--(-5.75,2);
\draw[red,line width=1] (-7.75,2)--(-7.5,2);
\draw[red,line width=1] (-8.5,2)--(-8.25,2);
\end{tikzpicture}
\end{center}
\caption{The illustration diagram of Proposition \ref{attracting202091203}}\label{fig2021-10-3-3}
\end{figure}
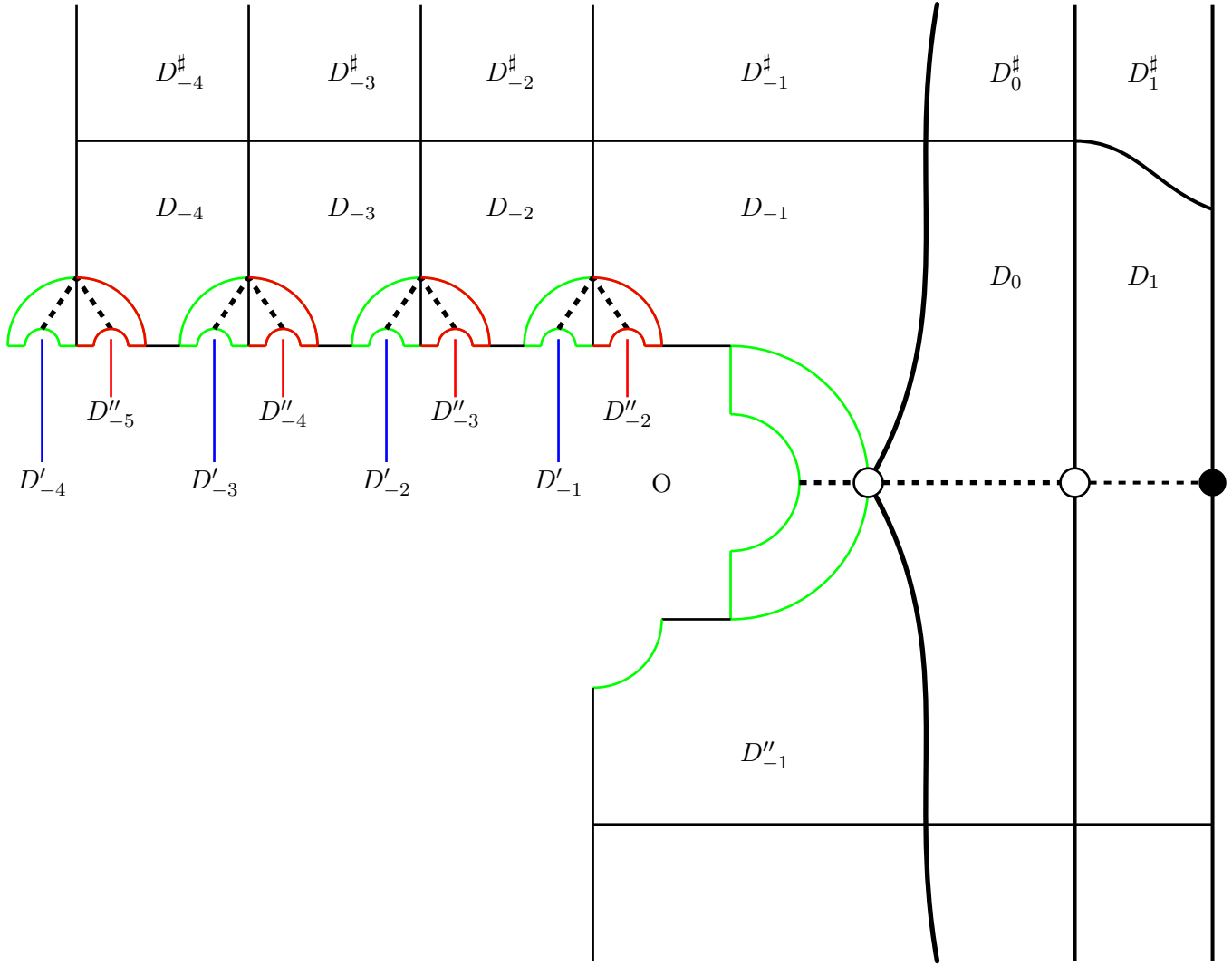

\subsection{Covering property of $f\in\mathcal{F}_0$ and $P$ as ``subcover"}\label{nearequ-20}

\begin{proposition}\label{paraequ-88}
Fix a constant $m\in\nn$. For any $f\in\mathcal{F}_0$, suppose the unique critical value
of $f$ is $cv_Q=-\frac{m^m}{(m+1)^{m+1}}$, since this can be achieved by a linear conjugacy. For the  $\mathcal{U}_j^{f}$ constructed in Subsection \ref{r0function2021-9-2-1}
 index by $I$,  where each $\mathcal{U}_j^{f}$ is mapped by $f$ isomorphically onto $\cc_{slit}$. Then, there exists an open subset contained in $\cc\setminus(-\infty,-1]$ on which there is a conformal map $\varphi$ from  this set onto an open
subset $U\subset\text{Dom}(f)$ such that $\varphi(0)=0$, $\varphi'(0)=1$, and
\beqq
f=P\circ \varphi^{-1}\ \text{on}\ U.
\eeqq
\end{proposition}

\begin{proof}
Suppose the critical value of $f$ is $cv_f=-\frac{m^m}{(m+1)^{m+1}}=cp_P$ by a linear conjugacy.  Since $f:\ \text{Dom}(f)\setminus\{0\}\to\cc^{*}=\cc\setminus\{0\}$, the domain of $f$ consists of the union of the closure of $\mathcal{U}_j^{f}$.

If $f$ is a quadratic polynomial, consider $\mathcal{U}_1^{f}$, $\mathcal{U}_2^{f}$, $\ga_{c2+}^{f}$, and $\ga_{c2-}^{f}$, we glue these two curves $\ga_{c2+}$ and $\ga_{c2-}$ together and add an inverse image of $0$ to $\mathcal{U}_2^{f}$, we would have a branched covering of degree $2$ branched over $cv_f$, and define $\varphi=(f|_{\mathcal{U}_{12}^{f}})^{-1}\circ P$, where $\mathcal{U}_{12}^{f}=\mathcal{U}_1^{f}\cup \mathcal{U}_2^{f}\cup \ga_{b1+}^{f}\cup \ga_{b2+}^{f}$.

If $f$ is not a quadratic polynomial, consider the components $\mathcal{U}_1^{f}$, $\mathcal{U}_2^{f}$, $\mathcal{U}_{3}^{f}$, $\mathcal{U}_{4}^{f}$, where $\mathcal{U}_{3}^{f}$ and $\mathcal{U}_{4}^{f}$ may be the same. Suppose that
$\ga_{c3+}=\ga_{c2-}$ and $\ga_{c2+}^{f}=\ga_{c4-}^{f}$ (see Figures \ref{illustration-m-3}--\ref{illustration-m-6} as an illustration).  As in Subsection \ref{r0function2021-9-2-1}, the above components and curves for $P$ are denoted by $\mathcal{U}^P_j$, $\ga^P_{aj}$, $\ga^P_{bj\pm}$, and $\ga^P_{cj\pm}$.

The map $\varphi:(\mathcal{U}^P_{1\pm}\cup \mathcal{U}^P_{2\pm}\cup \mathcal{U}^P_{3+}\cup \mathcal{U}^P_{m+1,-})\subset \cc\setminus(-\infty,-1]\to
(\mathcal{U}_{1\pm}^{f}\cup\mathcal{U}_{2\pm}^{f}\cup
\mathcal{U}_{3+}^{f}\cup \mathcal{U}_{4-}^{f})\subset\text{Dom}(f)$ is given by
\beqq
\varphi(z)=\left\{
  \begin{array}{ll}
    (f|_{\mathcal{U}_{1\pm}^{f}})^{-1}\circ P, & \hbox{if}\ z\in \mathcal{U}^P_{1\pm}, \\
     (f|_{\mathcal{U}_{2\pm}^{f}})^{-1}\circ P, & \hbox{if}\ z\in \mathcal{U}^P_{2\pm}, \\
     (f|_{\mathcal{U}_{3+}^{f}})^{-1}\circ P, & \hbox{if}\ z\in \mathcal{U}^P_{3+}, \\
     (f|_{\mathcal{U}_{4-}^{f}})^{-1}\circ P, & \hbox{if}\ z\in \mathcal{U}^P_{m+1,-}.
  \end{array}
\right.
\eeqq
It is evident that this definition can be extended continuously to the boundaries of the regions $\mathcal{U}^P_1$, $\mathcal{U}^P_2$, $\mathcal{U}^P_{3+}$, and $\mathcal{U}^P_{m+1,-}$, where the boundaries of these regions are $\ga_{b1+}^P=\ga_{b2-}^P$, $\ga_{c1+}^P=\ga_{c1-}^P$, $\ga_{b2+}^P=\ga_{b1-}^P$, $\ga_{c3+}^P=\ga_{c2-}^P$,
$\ga_{b3+}^P=\ga_{b3-}^P$, $\ga_{a3}^P$,
$\ga_{a,m+1}^P$, $\ga_{a,m+1,+}^P=\ga_{a,m+1,-}^P$, $\ga_{c2+}^P=\ga_{c,m+1,-}^P$. And ,we glue $\ga_{b3+}^P\cup\ga_{a3}^P$ and $\ga_{b,m+1,-}^P\cup\ga_{a,m+1}^P$ together (see Figures \ref{illustration-m-3}--\ref{illustration-m-6} as an illustration).

The origin is mapped to the origin, and $-\tfrac{1}{m+1}$ is mapped to the critical point of $f$. It is evident that $\varphi$ is a homeomorphism. Except for the origin and the critical point, the map $f$ is locally conformal, so $\varphi$ is holomorphic there. For the origin and the critical point, it follows from the removable singularity theorem that $\varphi$ is conformal. It follows from the definition of $\varphi$ that $f=P\circ\varphi^{-1}$ and $\varphi(0)=0$. This, together with the differentiation, implies that $\varphi^{\prime}(0)=1$.
\end{proof}

\begin{corollary}
Fix a constant $m\in\nn$. For $f\in\mathcal{F}_0$, suppose the critical value of $f$ is $cv_f=-\frac{m^m}{(m+1)^{m+1}}=cp_P$. Then,
\beqq
|f^{\prime\prime}(0)-(2m+1)|\leq1.\
\eeqq
\end{corollary}

\begin{proof}
It follows from $f=P\circ \varphi^{-1}$ that
$$f'=P'(\varphi^{-1})\cdot(\varphi^{-1})'\ \mbox{and}\ f^{\prime\prime}=P^{\prime\prime}(\varphi^{-1})\cdot((\varphi^{-1})')^2+P'(\varphi^{-1})\cdot(\varphi^{-1})^{\prime\prime}.$$ So, $f'(0)=P'(0)$ and
$$f^{\prime\prime}(0)=P^{\prime\prime}(0)+P'(0)\cdot(\varphi^{-1})^{\prime\prime}(0).$$

It follows from Proposition \ref{paraequ-2} that $P'(0)=1$ and
$$P^{\prime\prime}(z)=
(m-1)(1+z)^{m-2}(1+(m+1)z)+(1+z)^{m-1}(m+1),$$ implying that
$P^{\prime\prime}(0)=(m-1)+(m+1)=2m$. So, $f^{\prime\prime}(0)=2m+(\varphi^{-1})^{\prime\prime}(0)$.

The Koebe function $\psi_{Koebe}(z)=\tfrac{z}{(1-z)^2}$ is conformal from the unit disk onto $\cc\setminus(-\infty,-1/4]$, $\widehat{\varphi}(z)=\tfrac{1}{4}\varphi(4\psi_{Koebe}(z))$ is a univalent function in the class $\mathcal{S}$, where  $\mathcal{S}$ is the collection of univalent functions in the open unit disk with the fixed point $0$ and derivative $1$ at the fixed point $0$.
By direct calculation, $$\psi'_{Koebe}(z)=\tfrac{1}{(1-z)^2}+\tfrac{2z}{(1-z)^3}=\tfrac{1+z}{(1-z)^3}$$ and
$$\psi^{\prime\prime}_{Koebe}(z)=\tfrac{(1-z)^3-(1+z)(-3(1-z)^2)}{(1-z)^6}=\tfrac{2(z+2)}{(1-z)^4}.$$

By Theorem \ref{paraequ-87} (a), $|\widehat{\varphi}^{\prime\prime}(0)|\leq4$. By direct computation,
one has $$\widehat{\varphi}^{\prime}(z)=\tfrac{1}{4}\varphi^{\prime}(4\psi_{Koebe}(z))\cdot(4\psi^{\prime}_{Koebe}(z))=
\varphi^{\prime}(4\psi_{Koebe}(z))\cdot(\psi^{\prime}_{Koebe}(z))$$
and
$$\widehat{\varphi}^{\prime\prime}(z)
=4\varphi^{\prime\prime}(4\psi_{Koebe}(z))\cdot(\psi^{\prime}_{Koebe}(z))^2+\varphi^{\prime}(4\psi_{Koebe}(z))\cdot(\psi^{\prime\prime}_{Koebe}(z)).$$
So,
$$\widehat{\varphi}^{\prime\prime}(0)
=4\varphi^{\prime\prime}(4\psi_{Koebe}(0))\cdot(\psi^{\prime}_{Koebe}(0))^2+\varphi^{\prime}(4\psi_{Koebe}(0))\cdot(\psi^{\prime\prime}_{Koebe}(0))
=4\varphi^{\prime\prime}(0)+4.$$ Hence, $|\varphi^{\prime\prime}(0)+1|\leq1$.

Since $f=P\circ \varphi^{-1}$, $f\circ \varphi=P$, and $f'(\varphi)\cdot\varphi'=P'$, one has $$f^{\prime\prime}(\varphi)\cdot(\varphi')^2+f'(\varphi)\cdot(\varphi^{\prime\prime})=P^{\prime\prime}.$$ So, $$f^{\prime\prime}(0)+f'(0)\cdot(\varphi^{\prime\prime}(0))=f^{\prime\prime}(0)+P'(0)\cdot(\varphi^{\prime\prime}(0))=f^{\prime\prime}(0)+\varphi^{\prime\prime}(0)=P^{\prime\prime}(0).$$
Thus, $\varphi^{\prime\prime}(0)=P^{\prime\prime}(0)-f^{\prime\prime}(0)=2m-f^{\prime\prime}(0)$.

Therefore, $|2m+1-f^{\prime\prime}(0)|\leq1$.
\end{proof}

\subsection{From $P$ to $Q$}

\begin{proposition} (Relationship among $\mathcal{F}_0$, $\mathcal{F}^{P}_1$, $\mathcal{F}^{P}_2$, $\mathcal{F}^Q_1$, and $\mathcal{F}^Q_2$)\label{equ-20208-23-1} Fix a constant $m\in\nn$.
By the construction above, one has the following statements:
\begin{itemize}
\item[(a)] There is a natural injection $((\mathcal{F}_0\setminus\{quadratic\ polynomials\})/\underset{linear}\sim)\hookrightarrow\mathcal{F}^P_2$.
\item[(b)] There is a natural injection $\mathcal{F}^{P}_2\hookrightarrow\mathcal{F}^{P}_1$, defined by the restriction of $\varphi$ to $V$ for $f=P\circ\varphi^{-1}\in\mathcal{F}^{P}_2$.
\item[(c)] There is a natural injection $\mathcal{F}^{Q}_1\hookrightarrow\mathcal{F}^{Q}_2$, defined by the restriction of $\varphi$ to $\widehat{\cc}\setminus \bigg(E\cup\overline{\cd}\big(\cot(\tfrac{\pi}{m+1})i,\tfrac{1}{\sin(\tfrac{\pi}{m+1})}\big)\cup \overline{\cd}\big(-\cot(\tfrac{\pi}{m+1})i,\tfrac{1}{\sin(\tfrac{\pi}{m+1})}\big)\bigg)$ for $f=P\circ\varphi^{-1}\in\mathcal{F}^{Q}_1$.
\item[(d)] There exists a one-to-one correspondence between $\mathcal{F}^{P}_1$ and $\mathcal{F}^{Q}_2$, defined by $\varphi\in\mathcal{F}^{P}_1\to\widehat{\varphi}= \psi^{-1}_0\circ\varphi\circ\psi_1\in\mathcal{F}^{Q}_2$, since
\begin{align*}
\mathcal{F}^{P}_1\ni f=P\circ\varphi^{-1}\to F=\psi^{-1}_0\circ f\circ \psi_0=\psi^{-1}_0\circ P\circ\psi_1\circ\psi^{-1}_1\circ\varphi^{-1}\circ \psi_0=Q\circ\widehat{\varphi}^{-1}\in\mathcal{F}^{Q}_2,
\end{align*}
where $\widehat{\varphi}^{-1}(z)=\psi^{-1}_1\circ\varphi^{-1}\circ \psi_0(z)$ or $\widehat{\varphi}(z)= \psi^{-1}_0\circ\varphi\circ\psi_1(z)$. In this case, if $\widehat{\varphi}(z)=z+c_0+O(\tfrac{1}{z})$ near $\infty$, then $f^{\prime\prime}(0)=\tfrac{4m+2-c_0}{2}$, where $c_0\in\cc$ is a constant.
\end{itemize}
\end{proposition}

\begin{proof}
Suppose $f\in\mathcal{F}_0$, it follows from Proposition \ref{paraequ-88} that $f$ can be written as $f=P\circ \varphi^{-1}$ on $U\subset\text{Dom}(f)$, where $\varphi:\cc\setminus(-\infty,-1]\to U$ is conformal with $\varphi(0)=0$ and $\varphi'(0)=1$. This, together with $V'\subset\cc\setminus(-\infty,-1]$, yields that the restriction of $f=P\circ\varphi^{-1}$ to $\varphi(V')$ is an element of $\mathcal{F}^P_2$. Since we only need to consider the holomorphic functions in these classes (without the quasi-conformal extension), this kind of embedding is injective.

Note that $\psi_1:\widehat{\cc}\setminus \bigg(E\cup\overline{\cd}\big(\cot(\tfrac{\pi}{m+1})i,\tfrac{1}{\sin(\tfrac{\pi}{m+1})}\big)\cup \overline{\cd}\big(-\cot(\tfrac{\pi}{m+1})i,\tfrac{1}{\sin(\tfrac{\pi}{m+1})}\big)\bigg)\to V$ is conformal and $\varphi$ is normalized at $0$  if and only if $\widehat{\varphi}=\psi^{-1}_0\circ\varphi\circ\psi_1$ is normalized at $\infty$. Since $F=Q\circ\varphi^{-1}$, $\varphi(0)=0$, $\varphi'(0)=1$, by \eqref{paraequ-66}, one has $z+(4m+2-c_0)+O(\tfrac{1}{z})$, implying that $f(z)=\psi^{-1}_0\circ F\circ \psi_0=\tfrac{-4}{-4/z+(4m+2-c_0)+O(z)}=z+\tfrac{4m+2-c_0}{4}z^2+O(z^3)$. So, $f^{\prime\prime}(0)=\tfrac{4m+2-c_0}{2}$.
\end{proof}

\subsection{Preparation}

\begin{lemma}\label{paraequ-13}\cite[Lemma 5.9]{InouShishikura2016}
 \begin{itemize}
 \item[(a)] If $a,b\in\cc$ and $|a|>|b|$, then $|\text{arg}\,(a+b)-\text{arg}\,a|\leq\arcsin\bigg(\frac{|b|}{|a|}\bigg)$.
\item[(b)] If $0\leq x\leq\tfrac{1}{2}$, then $\arcsin(x)\leq\tfrac{\pi}{3}x$.
\end{itemize}
\end{lemma}

\begin{lemma}\label{paraequ-10}\cite[Lemma 5.11]{InouShishikura2016}
 \begin{itemize}
\item[(a)] If $\text{Re}\,(ze^{-i\tht})>t>0$ with $\tht\in\mathbb{R}$, then
\beqq
\frac{1}{z}\in\cd\bigg(\frac{e^{-i\tht}}{2t},\frac{1}{2t}\bigg).
\eeqq
\item[(b)] If $H=\{z:\ \text{Re}\,(z\,e^{-i\tht})>t\}$
and $z_0\in H$ with $u=\text{Re}\,(z_0e^{-i\tht})-t$, then
\beqq
\cd_H(z_0,s(r))=\cd\bigg(z_0+\frac{2ur^2e^{i\tht}}{1-r^2},\frac{2ur}{1-r^2}\bigg),\ 0<r<1,
\eeqq
where the right hand side is an Euclidean disk
and $s(r)=d_{\cd}(0,r)=\log\frac{1+r}{1-r}$.
\end{itemize}
\end{lemma}

\subsection{Construction from $P$ to $Q$}\label{near629-equ-1}

\begin{lemma}\label{paraequ-25}
\begin{itemize}
\item[(a)] The two rational maps $P$ and $Q$ are related by
\beq\label{nearequ-1}
Q=\psi_0^{-1}\circ P\circ \psi_1=\frac{(1+z)^{2+2m}}{z(1-z)^{2m}}=z\frac{\big(1+\frac{1}{z}\big)^{2+2m}}{\big(1-\frac{1}{z}\big)^{2m}}.
\eeq
\item[(b)] The derivative of $Q$ is
\beq\label{paraequ-26}
Q'(z)=
\bigg(1-\frac{(\sqrt{m+1}+\sqrt{m})^2 }{z}\bigg)\cdot\bigg(1-\frac{(\sqrt{m+1}-\sqrt{m})^2}{z}\bigg)\cdot\bigg(\frac{ 1+\frac{1}{z}}{1-\frac{1}{z}}\bigg)^{2 m+1}.
\eeq
\item[(c)] The critical points of  $Q$ are $\pm1$,
\beqq
cp_{Q1}:=(2m+1)+2\sqrt{m+m^2}=(\sqrt{m+1}+\sqrt{m})^2\approx4m+2,
\eeqq
and
\beqq
cp_{Q2}:=(2m+1)-2\sqrt{m+m^2}=(\sqrt{m+1}-\sqrt{m})^2\approx0;
\eeqq
the critical values are
\beqq
Q(-1)=0,\ Q(1)=\infty,\
\eeqq
\beqq
cv_Q:=Q(cp_{Q1})=Q(cp_{Q2})=4(m+1)\frac{(m+1)^{m}}{m^{m}}\to 4e(m+1)\ \text{as}\ m\to+\infty,
\eeqq
where $cp_{Q1}$ and $cp_{Q2}$ are simple critical points, the local degree of $z=1$ is $2m$, and the local degree of $z=-1$ is $2m+2$.
\end{itemize}
\end{lemma}

\begin{remark}\label{nearpara-5}
Consider the map $\psi_{1,1}(z)=\tfrac{z-1}{z+1}$ and the image of the unit circle $\{e^{i\tht}:\ \tht\in\rr\}$, by direct computation,
\begin{align}\label{region2020913-1}
\psi_{1,1}(e^{i\tht})=\frac{e^{i\tht}-1}{e^{i\tht}+1}=\frac{\cos\tht+i\sin\tht-1}{\cos\tht+i\sin\tht+1}=\frac{\sin\tht}{1+\cos\tht}i=\tan(\tfrac{\tht}{2})i.
\end{align}

It is evident that $\psi_1=\psi_{1,2}\circ\psi_{1,1}$, where $\psi_{1,1}=\frac{z-1}{z+1}$ and $\psi_{1,2}=w^2-1$,
$\psi_{1,1}$ is a Mobious transformation, mapping $\cc\setminus\overline{\cd} $ to the right half plane and mapping $\cd$ to the left half plane, and $\psi_{1,2}$ maps the right half plane or the left half plane onto $\cc\setminus(-\infty,-1]$.

This, together with the fact the inverse of $\psi_1$ can be written as $\psi^{-1}_{1,\pm}=\psi^{-1}_{1,1}\circ\psi^{-1}_{1,2,\pm}$, where $\psi^{-1}_{1,1} (z)=\frac{z+1}{1-z}$, $\psi^{-1}_{1,2,\pm}(z)=\pm\sqrt{z+1}$, yields that $\psi^{-1}_{1,+}=\psi^{-1}_{1,1}\circ\psi^{-1}_{1,2,+}$ maps $\cc\setminus(-\infty,-1]$ to $\cc\setminus\overline{\cd}$, and $\psi^{-1}_{1,-}=\psi^{-1}_{1,1}\circ\psi^{-1}_{1,2,-}$ maps $\cc\setminus(-\infty,-1]$ to $\cd$.
\end{remark}

\begin{proof}

By direct calculation, the derivative of $Q$ is
\begin{align*}
Q'(z)&=-\frac{(1-z)^{-2 m-1} (z+1)^{2 m+1} \left(z^2-2 (2 m+1) z+1\right)}{z^2}\\
=&\frac{ \left(z^2-2 (2 m+1) z+1\right)}{z^2}\frac{ (z+1)^{2 m+1}}{(z-1)^{2 m+1}}\\
=&\bigg(1-\frac{2 (2 m+1) }{z}+\frac{1}{z^2}\bigg)\cdot\bigg(\frac{ 1+\frac{1}{z}}{1-\frac{1}{z}}\bigg)^{2 m+1}\\
=&\bigg(1-\frac{(\sqrt{m+1}+\sqrt{m})^2 }{z}\bigg)\cdot\bigg(1-\frac{(\sqrt{m+1}-\sqrt{m})^2}{z}\bigg)\cdot\bigg(\frac{ 1+\frac{1}{z}}{1-\frac{1}{z}}\bigg)^{2 m+1},
\end{align*}
so, the critical points are $(2m+1)\pm2\sqrt{m+m^2}=(\sqrt{m+1}\pm\sqrt{m})^2$, if $m$ is large enough, then $(2m+1)\pm2\sqrt{m+m^2}\approx(2m+1)\pm2\sqrt{m+m^2+1/4}\approx 4m+2,\ or\ 0$.

The critical values are
\begin{align*}
Q(cp)=&\frac{(1+(2m+1)\pm2\sqrt{m+m^2})^{2+2m}}{((2m+1)\pm2\sqrt{m+m^2})(1-((2m+1)\pm2\sqrt{m+m^2}))^{2m}}\\
=&\frac{(2\sqrt{m+1})^{2+2m}(\sqrt{m+1}\pm\sqrt{m})^{2+2m}}{((\sqrt{m+1}\pm\sqrt{m})^2)(-2m\mp2\sqrt{m+m^2}))^{2m}}\\
=&\frac{(2\sqrt{m+1})^{2+2m}(\sqrt{m+1}\pm\sqrt{m})^{2+2m}}{((\sqrt{m+1}\pm\sqrt{m})^2)(2\sqrt{m})^{2m}(\sqrt{m}\pm\sqrt{m+1}))^{2m}}\\
=&\frac{(2\sqrt{m+1})^{2+2m}}{(2\sqrt{m})^{2m}}=4(m+1)\frac{(m+1)^{m}}{m^{m}}\to 4e(m+1).\\
\end{align*}

\end{proof}

\begin{definition}\label{near-para-equ-2}
Note that $Q=\psi_0^{-1}\circ P\circ \psi_1$ and $Q^{-1}=\psi_1^{-1}\circ P^{-1}\circ\psi_0$ by \eqref{nearequ-1}. For $\psi_0(z)=-\tfrac{4}{z}$, let $z=re^{i\tht}$, one has $\psi_0(z)=\psi_0(re^{i\tht})=-\tfrac{4}{z}=-\tfrac{4}{re^{i\tht}}=\tfrac{4}{r}e^{i(\pi-\tht)}$. So, $\psi_0$ maps the upper (lower) half plane to the upper (lower) half plane.

By definition \ref{paraequ-89} and  $P\in\mathcal{F}_0$, denote by
\beqq
 \Ga_b^{P}=(-\infty,cv_P],\ \Ga_a^{P}=(cv_P,0),\ \Ga_c^{P}=(0,+\infty),
\eeqq
and
\beqq
\Ga^Q_b=\psi_0^{-1}( \Ga_b^{P})=(0,cv_Q],\ \Ga^Q_a=\psi_0^{-1}( \Ga_a^{P})=(cv_Q,+\infty),\ \Ga^Q_c=\psi_0^{-1}( \Ga_c^{P})=(-\infty,0).
\eeqq

By Remark \ref{nearpara-5}, for $\ga_{a1}^{P}=(cp_P,0)$, $\ga_{a2}^{P}=(-1,cp_P)$, $\ga_{c1\pm}^{P}=(0,+\infty)$,
one has
\beqq
\ga^{Q+}_{a1}=\psi^{-1}_{1,+}(\ga_{a1}^{P})=(cp_{Q1},+\infty),\ \ga^{Q+}_{a2}=\psi^{-1}_{1,+}(\ga_{a2}^{P})=(1,cp_{Q1}),\ \ga^{Q+}_{c1\pm}=\psi^{-1}_{1,+}(\ga_{c1\pm}^{P})=(-\infty,-1),
\eeqq
and
\beqq
\ga^{Q-}_{a1}=\psi^{-1}_{1,-}(\ga_{a1}^{P})=(0,cp_{Q2}),\ \ga^{Q-}_{a2}=\psi^{-1}_{1,-}(\ga_{a2}^{P})=(cp_{Q2},1),\ \ga^{Q-}_{c1\pm}=\psi^{-1}_{1,-}(\ga_{c1\pm}^{P})=(-1,0).
\eeqq
Similarly, it follows from Remark \ref{nearpara-5} that we can introduce the following subsets
\beqq
\mathcal{U}^{Q+}_{i\pm}=\psi^{-1}_{1,+}(\mathcal{U}^P_{i\pm})\subset\cc\setminus\overline{\cd},\
\mathcal{U}^{Q-}_{i\pm}=\psi^{-1}_{1,-}(\mathcal{U}^P_{i\pm})\subset\cd,
\eeqq
 $\ga^{Q+}_{ai}=\psi^{-1}_{1,+}(\ga^P_{ai})$, $\ga^{Q-}_{ai}=\psi^{-1}_{1,-}(\ga^P_{ai})$, and so on.
\end{definition}

\begin{figure}[htbp]
\begin{center}
\scalebox{0.5}{ \includegraphics{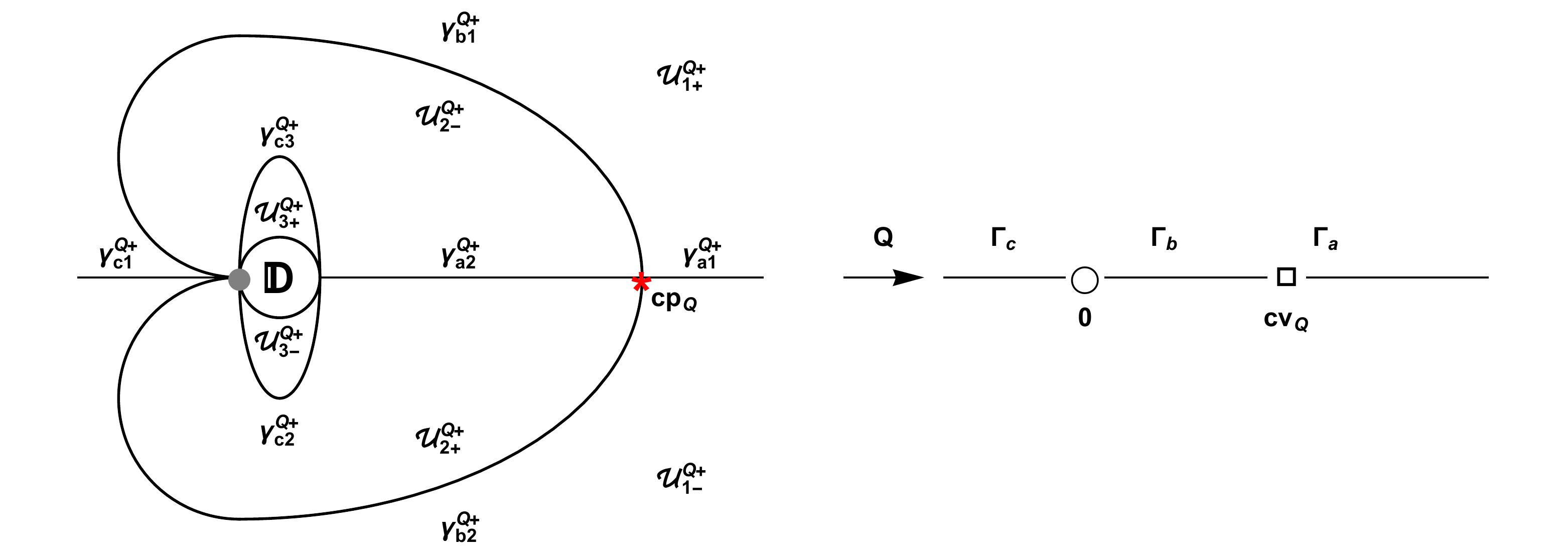}}
\renewcommand{\figure}{Fig.}
\caption{The illustration diagram of the inverse image of $Q$.
}\label{illustration-inverse-Q}
\end{center}
\end{figure}

\begin{lemma}\label{paraequ-16}
For any $\be\in(0,\tfrac{\pi}{2})$ and $\psi_1=\psi_{1,2}\circ\psi_{1,1}$ specified in Remark \ref{nearpara-5}, one has
\beqq
\psi_{1,1}\bigg(\overline{\big(\partial\cd\big(\cot(\be)i,\tfrac{1}{\sin(\be)}\big)\big)\setminus\overline{\cd}}\bigg)
=\{w:\ w=0,\ w=\infty,\ \text{arg}\,w=\be\},
\eeqq
\beqq
\psi_{1,1}\bigg(\overline{\big(\partial\cd\big(-\cot(\be)i,\tfrac{1}{\sin(\be)}\big)\big)\setminus\overline{\cd}}\bigg)
=\{w:\ w=0,\ w=\infty,\ \text{arg}\,w=-\be\},
\eeqq
\beqq
\psi_{1,1}\bigg(\cd\big(\cot(\be)i,\tfrac{1}{\sin(\be)}\big)\setminus\overline{\cd}\bigg)
=\bigg\{w:\ w\neq0,\ \be< \text{arg}\,w<\frac{\pi}{2}\bigg\},
\eeqq
and
\beqq
\psi_{1,1}\bigg(\cd\big(-\cot(\be)i,\tfrac{1}{\sin(\be)}\big)\setminus\overline{\cd}\bigg)
=\bigg\{w:\ w\neq0,\ -\frac{\pi}{2}< \text{arg}\,w<-\be\bigg\}.
\eeqq
Furthermore, one has
\beqq
\psi_{1,2}\circ\psi_{1,1}\bigg(\cd\big(\cot(\be)i,\tfrac{1}{\sin(\be)}\big)\setminus\ol{\cd}\bigg)
=\bigg\{z:\ z\neq-1,\ 2\be< \text{arg}\,(z+1)<\pi\bigg\}
\eeqq
and
\beqq
\psi_{1,2}\circ\psi_{1,1}\bigg(\cd\big(-\cot(\be)i,\tfrac{1}{\sin(\be)}\big)\setminus\ol{\cd}\bigg)
=\bigg\{z:\ z\neq-1,\ -\pi< \text{arg}\,(z+1)<-2\be\bigg\}.
\eeqq
\end{lemma}

\begin{proof}
Now, the set $\psi_1\bigg(\cd\big(\cot(\be)i,\tfrac{1}{\sin(\be)}\big)\setminus\overline{\cd}\bigg)$ is analyzed.

By Remark \ref{nearpara-5}, $\psi_1=\psi_{1,2}\circ\psi_{1,1}$, where $\psi_{1,1}=\frac{z-1}{z+1}$ and $\psi_{1,2}=w^2-1$.  So, $\partial\cd\big(\cot(\be)i,\tfrac{1}{\sin(\be)}\big)$ is a circle that interests the unit circle at the points $\pm1$ where the angle between the tangent line at the $\pm1$ of the unit circle and the tangent line of the circle $\partial\cd\big(\cot(\be)i,\tfrac{1}{\sin(\be)}\big)$ is $\tfrac{\pi}{2}-\be$. The Mobious transformation $\psi_{1,1}=\tfrac{z-1}{z+1}$ maps the unit circle to the imaginary axis, the image of $1$ is $0$, and the image of $-1$ is $\infty$. So,  the image of $\partial\cd\big(\cot(\be)i,\tfrac{1}{\sin(\be)}\big)\setminus\overline{\cd}$ under the map $\psi_{1,1}$ is a line from $0$ to $\infty$, where the image intersects the imaginary axis at $0$ and $\infty$ with angle $\tfrac{\pi}{2}-\be$, and contains the point $\psi_{1,1}((\cot(\be)+\tfrac{1}{\sin(\be)})i)$. Hence,
\beqq
\psi_{1,1}\bigg(\overline{\big(\partial\cd\big(\cot(\be)i,\tfrac{1}{\sin(\be)}\big)\big)\setminus\overline{\cd}}\bigg)
=\{w:\ w=0,\ w=\infty,\ \text{arg}\,w=\be\},
\eeqq
implying that
\beqq
\psi_{1,1}\bigg(\cd\big(\cot(\be)i,\tfrac{1}{\sin(\be)}\big)\setminus\ol{\cd}\bigg)
=\bigg\{w:\ w\neq0,\ \be< \text{arg}\,w<\frac{\pi}{2}\bigg\}.
\eeqq
Furthermore,
\beqq
\psi_{1,2}\circ\psi_{1,1}\bigg(\cd\big(\cot(\be)i,\tfrac{1}{\sin(\be)}\big)\setminus\ol{\cd}\bigg)
=\bigg\{z:\ z\neq-1,\ 2\be< \text{arg}\,(z+1)<\pi\bigg\},
\eeqq
where $\psi_{1,2}(w)=w^2-1$. Hence, one has
\beqq
\psi_1\bigg(\cd\big(\cot(\be)i,\tfrac{1}{\sin(\be)}\big)\setminus\overline{\cd}\bigg)
=\bigg\{z:\ z\neq-1,\ 2\be< \text{arg}\,(z+1)<\pi\bigg\}.
\eeqq

\end{proof}

\begin{corollary}\label{cor-2021-10-2-1}
 Let $\al=\tfrac{\pi}{(m+1)}$,  one has
$$\overline{\cd}\big(\pm\cot(\al)i,\tfrac{1}{\sin(\al)}\big)\setminus\overline{\cd}\subset\cc\setminus(\mathcal{U}^{Q+}_1\cup\mathcal{U}^{Q+}_2),$$
\beqq
\mathcal{U}^{Q+}_1\cup\mathcal{U}^{Q+}_2\subset\cc\setminus\bigg(\overline{\cd}\big(\cot(\al)i,\tfrac{1}{\sin(\al)}\big)\cup \overline{\cd}\big(-\cot(\al)i,\tfrac{1}{\sin(\al)}\big)\bigg),
\eeqq
and
\beqq
\cc\setminus\bigg(\overline{\cd}\big(\cot(\tfrac{\al}{2})i,\tfrac{1}{\sin(\tfrac{\al}{2})}\big)\cup \overline{\cd}\big(-\cot(\tfrac{\al}{2})i,\tfrac{1}{\sin(\tfrac{\al}{2})}\big)\bigg)\subset \mathcal{U}^{Q+}_1\cup\mathcal{U}^{Q+}_2.
\eeqq
\end{corollary}

\begin{proof}
This is a direct conclusion by Lemmas \ref{paraequ-2021-9-20-2} and \ref{paraequ-16}.
\end{proof}

\subsection{Estimates on $Q$: Part I}

\begin{lemma}\label{paraequ-83}
 {Fix the constants $e_1 = 0.84$, $e_{0}=-0.18$, and $e_{-1} = 0.6$.} For $\vep_1=0.1$, $\vep_2=0.1$, and $\eta$ specified in \eqref{near-equat-6}, one has the following properties:
\begin{itemize}
\item[(a)] The disks $\overline{\cd}(1,\vep_1)$ and $\overline{\cd}(-1,\vep_2)$, and $\overline{\cd}$ are contained in the interior of the ellipse $E_{r}$ with $r\geq1.7$, denoted by $r_0=1.7$. 
\item[(b)] The disks $\overline{\cd}(1,\vep_1)$, $\overline{\cd}(-1,\vep_2)$, and $\overline{\cd}(0,0.2)$ are contained in the interior of $E=E_{1}$.
    \item[(c)]  Let
$$U_{\eta}^{Q}=\psi^{-1}_1(V')\setminus\ol{\cd}.$$ Then the set $$\cc\setminus(U_{\eta}^{Q}\cup \cd\big(   \cot(\tfrac{\pi}{2(m+1)})i,\tfrac{1}{\sin(\tfrac{\pi}{2(m+1)})}\big)\cup\cd\big(-\cot(\tfrac{\pi}{2(m+1)})i,\tfrac{1}{\sin(\tfrac{\pi}{2(m+1)})}\big) )$$
is covered by the two disks $\cd(1,\vep_1)$ and $\cd(-1,\vep_2)$.
\end{itemize}
\end{lemma}

\begin{proof}
(a) We show that $\overline{\cd}$, $\overline{\cd}(1,\vep_1)$ $\overline{\cd}(-1,\vep_2)$ are contained in the interior of $E_r$ with $0<\vep_1\leq0.1$, $0<\vep_2\leq0.1$, and $r\geq1.7$.

For fixed $r$, a parametrization of the boundary of $E_r$ is given by a parameter $t\in[-1,1]$:
\beq\label{paraequ-91}
x=a_E(r)t+e_0\ \text{and}\ y=b_E(r)\sqrt{1-t^2},\ t\in[-1,1],
\eeq
where $E_r$, $a_E(r)$, and $b_E(r)$ are specified in \eqref{paraequ-3} and \eqref{paraequ-90}.

Plugging this equation into the expressions of the boundaries of $\overline{\cd}$, $\overline{\cd}(1,\vep_1)$, and $\overline{\cd}(-1,\vep_2)$, one has
\beqq
h_1(t)=x^2+y^2-1=(a_E(r)t+e_0)^2+(b_E(r)\sqrt{1-t^2})^2-1
\eeqq
\beqq
h_2(t)=(x_1-1)^2+y^2-\vep_1^2=(a_E(r)t+e_0-1)^2+(b_E(r)\sqrt{1-t^2})^2-\vep_1^2
\eeqq
\beqq
h_3(t)=(x+1)^2+y^2-\vep_2^2=(a_E(r)t+e_0+1)^2+(b_E(r)\sqrt{1-t^2})^2-\vep_2^2,
\eeqq
where $h_1(t)$ is the boundary of $\overline{\cd}$, $h_2(t)$ is the boundary of $\overline{\cd}(1,\vep_1)$, and $h_3(t)$ is the boundary of $\overline{\cd}(-1,\vep_2)$.

For the minimum value of these three functions, if they are bigger than $0$, then these regions are contained in the interior of $E_r$. Consider the following function
\beqq
(at+b)^2+(c\sqrt{1-t^2})^2+d=a^2t^2+2abt+b^2+c^2(1-t^2)+d=(a^2-c^2)t^2+2abt+b^2+c^2+d
\eeqq
for $t\in[-1,1]$, the minimum value is given by the formula (this can be derived by using standard arguments in analysis):
\beqq
\left\{
  \begin{array}{ll}
    -2ab+b^2+c^2+d, & \hbox{if}\  a^2-c^2=0\ \text{and}\ ab>0\\
    2ab+b^2+c^2+d, & \hbox{if}\  a^2-c^2=0\ \text{and}\ ab<0\\
    b^2+c^2+d-\frac{a^2b^2}{a^2-c^2}, & \hbox{if}\  a^2-c^2>0\ \text{and}\ -1\leq \tfrac{ab}{a^2-c^2}\leq1\\
    (a^2-c^2)-2ab+b^2+c^2+d, & \hbox{if}\  a^2-c^2>0\ \text{and}\ \tfrac{ab}{a^2-c^2}<-1\\
    (a^2-c^2)+2ab+b^2+c^2+d, &  \hbox{if}\  a^2-c^2>0\ \text{and}\ \tfrac{ab}{a^2-c^2}>1\\
    \min\{(a^2-c^2)-2ab,(a^2-c^2)+2ab\}+b^2+c^2+d, & \hbox{if}\  a^2-c^2<0\ \text{and}\ -1\leq \tfrac{ab}{a^2-c^2}\leq1\\
    (a^2-c^2)-2ab+b^2+c^2+d , & \hbox{if}\  a^2-c^2<0\ \text{and}\ \tfrac{ab}{a^2-c^2}<-1\\
     (a^2-c^2)+2ab+b^2+c^2+d, &\hbox{if}\  a^2-c^2<0\ \text{and}\ \tfrac{ab}{a^2-c^2}>1.
  \end{array}
\right.
\eeqq

Take the following constants $e_1 = 0.84$, $e_{0}=-0.18$, $e_{-1} = 0.6$, and $r_0=1.7$.

By direct calculation, one has
\beqq
a_E(r_0)=e_1r_0+\frac{e_{-1}}{r_0}=0.84\times 1.7+0.6/1.7\approx1.78094,
\eeqq
\beqq
b_E(r_0)=e_1r_0-\frac{e_{-1}}{r_0}=0.84\times 1.7-0.6/1.7\approx1.07506.
\eeqq

For the function $h_1(t)$ with $r=r_0$, one has
\beqq
\frac{a_E(r_0)e_0}{(a_E(r_0))^2-(b_E(r_0))^2}\approx-0.159013,
\eeqq
so the minimum value is
\beqq
e_0^2+(b_E(r_0))^2 + (-1) - \frac{(a_E(r_0)\cdot e_0)^2}{((a_E(r_0))^2 - (b_E(r_0))^2)}\approx0.137177>0.
\eeqq

For the $h_2(t)$ with $\vep_1=0.1$ and $r=r_0$, one has
 \beqq
\frac{a_E(r_0)(e_0-1)}{(a_E(r_0))^2-(b_E(r_0))^2}\approx-1.04242,
\eeqq
so the minimum value is
\beqq
(a_E(r_0))^2 - (b_E(r_0))^2  - 2(a_E(r_0)\cdot e_0) + (e_0-1)^2 (b_E(r_0))^2-\vep_1^2\approx8.75717>0.
\eeqq

For the $h_3(t)$ with $\vep_2=0.1$ and $r=r_0$, one has
 \beqq
\frac{a_E(r_0)(e_0+1)}{(a_E(r_0))^2-(b_E(r_0))^2}\approx0.724391,
\eeqq
the minimum value is
\beqq
(e_0+1)^2+(b_E(r_0))^2 + (-\vep_2^2) - \frac{(a_E(r_0)\cdot (e_0+1))^2}{((a_E(r_0))^2 - (b_E(r_0))^2)}\approx0.760272>0.
\eeqq

For these three functions, since the minimum value is bigger than $0$, these regions are contained in the interior of the ellipse.

(b) Plugging \eqref{paraequ-91} into the expression of the boundary of $\overline{\cd}(0,0.2)$, one has
\beqq
h_4(t)=x^2+y^2-0.2^2=(a_E(r)t+e_0)^2+(b_E(r)\sqrt{1-t^2})^2-0.2^2.
\eeqq

Recall the constants $e_1 = 0.84$, $e_{0}=-0.18$, and $e_{-1} = 0.6$.

By direct calculation, one has
\beqq
a_E(1)=e_1+e_{-1}=0.84+0.6=1.44,
\eeqq
\beqq
b_E(1)=e_1-e_{-1}=0.84-0.6=0.24.
\eeqq

For the function $h_4(t)$ with $r=1$, one has
\beqq
\frac{a_E(1)e_0}{(a_E(1))^2-(b_E(1))^2}=\frac{1.44\times (-0.18)}{1.44^2-0.24^2}\approx-0.128571,
\eeqq
so the minimum value is
\begin{align*}
&e_0^2+(b_E(1))^2 + (-0.2^2) - \frac{(a_E(1)\cdot e_0)^2}{((a_E(1))^2 - (b_E(1))^2)}\\
=&(-0.18)^2+0.24^2 -0.2^2 - \frac{(1.44\times (-0.18))^2}{1.44^2 - 0.24^2}\approx0.0166743>0.
\end{align*}

For the $h_2(t)$ with $\vep_1=0.1$ and $r=1$, one has
 \beqq
\frac{a_E(1)(e_0-1)}{(a_E(1))^2-(b_E(1))^2}=\frac{1.44\times(-0.18-1)}{1.44^2-0.24^2}\approx0.842857,
\eeqq
so the minimum value is
\begin{align*}
&(e_0-1)^2+(b_E(1))^2 + (-0.1^2) - \frac{(a_E(1)\cdot (e_0-1))^2}{((a_E(1))^2 - (b_E(1))^2)}\\
=&(-0.18-1)^2+0.24^2 -0.1^2 - \frac{(1.44\times (-0.18-1))^2}{1.44^2 - 0.24^2}\approx0.00781714>0.
\end{align*}

For the $h_3(t)$ with $\vep_2=0.1$ and $r=1$, one has
 \beqq
\frac{a_E(1)(e_0+1)}{(a_E(1))^2-(b_E(1))^2}=\frac{1.44\times(-0.18+1)}{1.44^2-0.24^2}\approx-0.585714,
\eeqq
the minimum value is
\begin{align*}
&(e_0+1)^2+(b_E(1))^2 + (-\vep_2^2) - \frac{(a_E(1)\cdot (e_0+1))^2}{((a_E(1))^2 - (b_E(1))^2)}\\
=&(-0.18+1)^2+0.24^2 -0.1^2 - \frac{(1.44\times (-0.18+1))^2}{1.44^2 - 0.24^2}\approx0.0283886>0.
\end{align*}

(c)
By direct computation, if $|\xi-1|=\vep_1<2/3$, then
\begin{align}\label{nearpara-6}
|Q(\xi)|\leq&\frac{(1+|\xi|)^{2+2m}}{|\xi|\cdot|1-\xi|^{2m}}\leq
\frac{(2+\vep_1)^{2+2m}}{(1-\vep_1)\vep_1^{2m}}\leq \frac{3^{2+2m}}{(1/3)\vep_1^{2m}}\nonumber\\
=&3^{3+2m}\vep_1^{-2m}< 3^{3+2m}\vep_1^{-2-2m}\to+\infty\ \text{as}\ \vep_1\to0,
\end{align}
and
\beqq
|Q(\xi)|=\frac{|2+(\xi-1)|^{2+2m}}{|1+(\xi-1)|\cdot|1-\xi|^{2m}}\geq
\frac{(2-\vep_1)^{2+2m}}{(1+\vep_1)\vep_1^{2m}}\geq \frac{\big(\tfrac{4}{3}\big)^{2+2m}}{(5/3)\cdot\vep_1^{2m}}\to+\infty\ \text{as}\ \vep_1\to0;
\eeqq
if $|\xi+1|=\vep_2<1$, then
\beqq
|Q(\xi)|\geq\frac{\vep_2^{2+2m}}{(1+\vep_2)(2+\vep_2)^{2m}}>\frac{\vep_2^{2+2m}}{2\times3^{2m}}>\frac{\vep_2^{2+2m}}{3^{3+2m}}\to0\ \text{as}\ \vep_2\to0,
\eeqq
and
\beqq
|Q(\xi)|=\frac{|\xi+1|^{2+2m}}{|1-(\xi+1)|\cdot|2-(\xi+1)|^{2m}}\leq\frac{\vep_2^{2+2m}}{(1-\vep_2)(2-\vep_2)^{2m}}\to0\ \text{as}\ \vep_2\to0.
\eeqq

So, for $\vep_1=\vep_2=0.1$, by the choice of $\eta$ in \eqref{near-equat-6}, one has
\beq\label{nearpara-2}
3^{3+2m}\vep_1^{-2-2m}=3\cdot\left(\frac{3}{\vep_1}\right)^{2+2m}=3\cdot30^{2+2m}
\leq \frac{4(m+1)^{m+1}}{m^{m}}e^{2\pi\eta}\leq4(m+1)e^{2\pi\eta+1}
\eeq
and
\beq\label{nearpara-3}
3^{-3-2m}\vep_2^{2+2m}=\frac{1}{3}\cdot\left(\frac{\vep_2}{3}\right)^{2+2m}=\frac{1}{3}\cdot\frac{1}{30^{2+2m}}
\geq 4(m+1)e^{-2\pi\eta+1}\geq \frac{4(m+1)^{m+1}}{m^{m}}e^{-2\pi\eta},
\eeq
where the last inequality can be derived by the inequality
\beqq
3\cdot30^{2+2m}\leq\frac{1}{4(m+1)}e^{2\pi\eta-1}.
\eeqq

Recall $cv_Q=4(m+1)\frac{(m+1)^{m}}{m^{m}}$ in Lemma \ref{paraequ-25} and $cv_{\tiny{P}}=-\frac{m^m}{(m+1)^{m+1}}$ in \eqref{critip2011-8-30-1}. By Lemmas \ref{paraequ-2021-9-20-2} and \ref{paraequ-16},  the set $\cc\setminus(U_{\eta}^{Q}\cup \cd\big(   \cot(\tfrac{\pi}{2(m+1)})i,\tfrac{1}{\sin(\tfrac{\pi}{2(m+1)})}\big)\cup\cd\big(-\cot(\tfrac{\pi}{2(m+1)})i,\tfrac{1}{\sin(\tfrac{\pi}{2(m+1)})}\big) )$ consists of two connected components $W$ and $W'$ such that
$W$ contains $1$ in its boundary and $|Q(\zeta)|\geq cv_Q\cdot e^{2\pi\eta}$ in $W$, whereas $W'$ contains $-1$ in its boundary and $|Q(\zeta)|\leq cv_Q\cdot e^{-2\pi\eta}$ in $W'$.

By \eqref{nearpara-2}, $|Q(\zeta)|<cv_Q\cdot e^{2\pi\eta}$ on $\partial \cd(-1,\vep_1)$; by \eqref{nearpara-3}, $|Q(\zeta)|>cv_Q\cdot e^{-2\pi\eta}$ on $\partial \cd(1,\vep_2)$. Hence, $W\subset \cd(1,\vep_1)$ and $W'\subset \cd(-1,\vep_2)$.

This completes the proof.
\end{proof}

\begin{lemma}\label{paraequ-15}

Let $\al=\tfrac{m\pi}{2(m+1)}$.

Choose real constants $e_{1},e_{0}, e_{-1}$ with $e_{1}>e_{-1}>0$,
\beqq
\tfrac{2}{3}\leq\vep_3<1,\ 1<\vep_4<\frac{1}{\sin(\al)},\ R<\frac{(4+\vep_3^2)^{1+m}}{(1+\vep_3)\vep_3^{2m}},\ \frac{\vep_4^{2+2m}}{\sqrt{1+\vep_4^2}(4+\vep_4^2)^{m}}<\rho=0.07,\ \ r_1=1.4.
\eeqq
\begin{itemize}
\item[(a)] If $\zeta\in\cc\setminus\big(\overline{\cd}\big(\cot(\al)i,\tfrac{1}{\sin(\al)}\big)\cup\overline{\cd}\big(-\cot(\al)i,\tfrac{1}{\sin(\al)}\big)\big)$ and $|\zeta-1|\leq\vep_3$, then $|Q(\zeta)|>R$.
\item[(b)] If $\zeta\in\cc\setminus\big(\overline{\cd}\big(\cot(\al)i,\tfrac{1}{\sin(\al)}\big)\cup\overline{\cd}\big(-\cot(\al)i,\tfrac{1}{\sin(\al)}\big)\big)$ and $|\zeta+1|\leq\vep_4$, then $|Q(\zeta)|<\rho$.
\item[(c)] The set $E_{r_1}$ can be covered by $\overline{\cd}\big(\cot(\al)i,\tfrac{1}{\sin(\al)}\big)$, $\overline{\cd}\big(-\cot(\al)i,\tfrac{1}{\sin(\al)}\big)$, $\overline{\cd}(1,\vep_3)$, and $\overline{\cd}(-1,\vep_4)$, where $E_{r_1}$ is specified in \eqref{paraequ-3}. As a consequence,
\beqq
\cc\setminus\big(\overline{\cd}\big(\cot(\al)i,\tfrac{1}{\sin(\al)}\big)\cup\overline{\cd}\big(-\cot(\al)i,\tfrac{1}{\sin(\al)}\big)\cup
\overline{\cd}(1,\vep_3)\cup\overline{\cd}(-1,\vep_4)\big)\subset\cc\setminus E_{r_1}.
\eeqq
\item[(d)] If $\zeta\in(\mathcal{U}^{Q+}_1\cup\mathcal{U}^{Q+}_2)$ and $\rho\leq|Q(\zeta)|\leq R$, then
$$\zeta\in\cc\setminus \bigg(E_{r_1}\cup\overline{\cd}\big(\cot(\tfrac{\pi}{m+1})i,\tfrac{1}{\sin(\tfrac{\pi}{m+1})}\big)\cup \overline{\cd}\big(-\cot(\tfrac{\pi}{m+1})i,\tfrac{1}{\sin(\tfrac{\pi}{m+1})}\big)\bigg).$$
Moreover, if $\zeta\in\overline{\mathcal{U}}^{Q+}_1$ and $|Q(\zeta)|>R$, then
$$\zeta\in\cc\setminus \bigg(E_{r_1}\cup\overline{\cd}\big(\cot(\tfrac{\pi}{m+1})i,\tfrac{1}{\sin(\tfrac{\pi}{m+1})}\big)\cup \overline{\cd}\big(-\cot(\tfrac{\pi}{m+1})i,\tfrac{1}{\sin(\tfrac{\pi}{m+1})}\big)\bigg).$$
\end{itemize}
\end{lemma}

\begin{remark}\label{nearequ-21}
Since $1<\vep_4<\frac{1}{\sin(\al)}$ and $\al=\tfrac{m}{2(m+1)}\pi$, one has
\begin{align*}
&\frac{\vep_4^{2+2m}}{\sqrt{1+\vep_4^2}(4+\vep_4^2)^{m}}\leq \frac{(1/\sin(\tfrac{\pi}{3}))^{2+2m}}{\sqrt{1+1^2}(4+1^2)^{m}}
=\frac{(\tfrac{2}{\sqrt{3}})^{2+2m}}{\sqrt{2}\cdot5^{m}}\\
=&\frac{4}{3\sqrt{2}}\bigg(\frac{4}{15}\bigg)^{m}\leq\frac{4}{3\sqrt{2}}\bigg(\frac{4}{15}\bigg)^{2}
\approx0.0670442.
\end{align*}
So, we can take $\rho=0.07$.
\end{remark}

\begin{remark}\label{nearequ-22}
For $\tfrac{2}{3}\leq\vep_3<1$, by the arguments in the following proof, one has
\begin{align*}
\frac{(4+\vep_3^2)^{1+m}}{(1+\vep_3)\vep_3^{2m}}\geq\frac{(4+1^2)^{1+m}}{(1+1)\cdot 1^{2m}}=\frac{5^{1+m}}{2}.
\end{align*}
\begin{align*}
\frac{(4+\vep_3^2)^{1+m}}{(1+\vep_3)\vep_3^{2m}}\leq\frac{(4+(\tfrac{2}{3})^2)^{1+m}}{(1+\tfrac{2}{3})\cdot (\tfrac{2}{3})^{2m}}=\frac{(\tfrac{40}{9})^{1+m}}{\tfrac{5}{3}\cdot (\tfrac{4}{9})^{m}}=\frac{(\tfrac{4}{9})^{1+m}\cdot10^{1+m}}{\tfrac{5}{3}\cdot (\tfrac{4}{9})^{m}}=\frac{8}{3}\cdot 10^m.
\end{align*}

So, we can take $R=2.66\times10^{m}$.
\end{remark}

\begin{remark}\label{near-equat-11}
Consider the function $g(x)=x+\frac{a}{x}$ with real parameter $a\neq0$, the domain of this function is $(-\infty,0)\cup(0,\infty)$, and the derivative of this function is
$g'(x)=1-\frac{a}{x^2}$. Consider two different cases: (i) $a>0$ and (ii) $a<0$. (i) For $a>0$, the solution to $g'(x)=0$ is $x=\pm\sqrt{a}$, the derivative $g'(x)>0$ for $x\in(-\infty,-\sqrt{a})\cup(\sqrt{a},+\infty)$, and $g'(x)<0$ for $x\in(-\sqrt{a},0)\cup(0,\sqrt{a})$.  For $x<0$, the maximum value of $g(x)$ is at $-\sqrt{a}$, $g(-\sqrt{a})=-2\sqrt{a}$, and $\lim_{x\to-\infty}g(x)=-\infty$ and $\lim_{x\to0-}g(x)=-\infty$. For $x>0$, the minimum value of $g(x)$ is at $\sqrt{a}$, $g(\sqrt{a})=2\sqrt{a}$, and $\lim_{x\to+\infty}g(x)=+\infty$ and $\lim_{x\to0+}g(x)=+\infty$. (ii) For $a<0$, $g'(x)>0$ for all $x\in(-\infty,0)\cup(0,+\infty)$.  For $x<0$, $\lim_{x\to-\infty}g(x)=-\infty$ and $\lim_{x\to0-}g(x)=+\infty$. For $x>0$, $\lim_{x\to0+}g(x)=-\infty$ and $\lim_{x\to+\infty}g(x)=+\infty$. Hence, $E_{r}\subset E_{r'}$ for $1\leq r<r'$.
\end{remark}

\begin{proof}
Case (a). By direct computation, one has that
\beqq
\overline{\cd}\big(1,\tfrac{1}{\sin(\al)}\big)\cap\{\zeta:\ \text{Re}\zeta\leq1\}\subset \overline{\cd}\big(\cot(\al)i,\tfrac{1}{\sin(\al)}\big)\cup\overline{\cd}\big(-\cot(\al)i,\tfrac{1}{\sin(\al)}\big).
\eeqq
So, for $\zeta$ satisfying the assumptions of (a), one has $\text{Re}\,\zeta>1$ and $|\zeta+1|\geq\sqrt{4+r^2}$ with $r=|\zeta-1|\leq\vep_3$. Hence,
\beqq
|Q(\zeta)|=\bigg|\frac{(1+\zeta)^{2+2m}}{\zeta(1-\zeta)^{2m}}\bigg|
\geq\frac{(\sqrt{4+r^2})^{2+2m}}{(1+r)r^{2m}}=\frac{(4+r^2)^{1+m}}{(1+r)r^{2m}}
:=h_4(r).
\eeqq
Since
\begin{align*}
&(\log(h_4(r))'=((1+m)\log(4+r^2)-\log(1+r)-2m\log r)'=
\frac{2r(1+m)}{4+r^2}-\frac{1}{1+r}-\frac{2m}{r}\\
\leq&\frac{2(1+m)}{4}-0-2m=\frac{1-3m}{2}<0\ \text{for}\ 0<r<1,
\end{align*}
one has
\beqq
|Q(\zeta)|\geq h_4(r)\geq h_4(\vep_3)=\frac{(4+\vep_3^2)^{1+m}}{(1+\vep_3)\vep_3^{2m}}.
\eeqq

Case (b). For $\zeta$ satisfying the assumptions of (b) and $\vep_4<\frac{1}{\sin(\al)}$, one has $\text{Re}\,\zeta<-1$, $|\zeta|\geq\sqrt{1+r^2}$, and $|\zeta-1|\geq\sqrt{4+r^2}$, where $r=|\zeta+1|\leq\vep_4$. So,
\beqq
|Q(\zeta)|=\bigg|\frac{(1+\zeta)^{2+2m}}{\zeta(1-\zeta)^{2m}}\bigg|
\leq\frac{r^{2+2m}}{\sqrt{1+r^2}(\sqrt{4+r^2})^{2m}}=
\frac{r^{2+2m}}{\sqrt{1+r^2}(4+r^2)^{m}}
:=h_5(r^2)=h_5(s).
\eeqq
Since
\begin{align*}
&(\log(h_5(s))'=((1+m)\log s-\frac{1}{2}\log(1+s)-m\log (4+s))'=
\frac{1+m}{s}-\frac{1}{2(1+s)}-\frac{m}{4+s}\\
\geq&\frac{1+m}{s}-\frac{1}{2s}-\frac{m}{s}=\frac{1}{2s}>0\ \text{for}\ s>0,
\end{align*}
one has
\beqq
|Q(\zeta)|\leq h_5(r^2)\leq h_5(\vep^2_4)=\frac{\vep_4^{2+2m}}{\sqrt{1+\vep_4^2}(4+\vep_4^2)^{m}}.
\eeqq

Case (c).  By the symmetric properties of these regions with respect to the real axis, it is sufficient to show that the upper part of $E_{r_1}$ is covered by $\overline{\cd}\big(\cot(\al)i,\tfrac{1}{\sin(\al)}\big)$, $\overline{\cd}(1,\vep_3)$, and  $\overline{\cd}(-1,\vep_4)$.

It is evident that $\overline{\cd}\big(\cot(\al)i,\tfrac{1}{\sin(\al)}\big)\cup \overline{\cd}\big(-\cot(\al)i,\tfrac{1}{\sin(\al)}\big)$ covers $\overline{\cd}(0,1)$.
For $\al=\tfrac{m}{2(m+1)}\pi$ dependent on $m$, it is evident that $\al$ is increasing with respect to $m$, that is, $\tfrac{m}{2(m+1)}\pi<\tfrac{m+1}{2(m+2)}\pi$, and $\lim_{m\to\infty}\al=\tfrac{\pi}{2}$. This, together with the fact that the functions $\cot(x)$ and $\tfrac{1}{\sin(x)}$ are decreasing for $x\in(0,\tfrac{\pi}{2}]$, implies that
the Hausdorff distance between $\overline{\cd}\big(\cot(\al)i,\tfrac{1}{\sin(\al)}\big)$ and $\overline{\cd}(0,1)$ goes to zero as $m\to+\infty$. So,  it only needs to show that the upper part of $E_{r_1}$ is covered by $\overline{\cd}(0,1)$, $\overline{\cd}(1,\vep_3)$, and  $\overline{\cd}(-1,\vep_4)$. Further, by the assumptions of $\vep_3$ and $\vep_4$, this is simplified as that the upper part of $E_{r_1}$ is covered by $\overline{\cd}(0,1)$, $\overline{\cd}(1,2/3)$, and  $\overline{\cd}(-1,1)$.

Consider the curve $\Ga=\partial E_{r_1}\cap\{\zeta:\ \text{Im}\,\zeta\geq0\}$ given by the parametric equation $y(x)=b_E(r_1)\sqrt{1-\big(\tfrac{x-e_0}{a_E(r_1)}\big)^2}$ for $x\in[e_0-a_E(r_1),e_0+a_E(r_1)]$, where  $a_{E}(r)=e_1 r+\tfrac{e_{-1}}{r}$ and $b_{E}(r)=e_1r-\tfrac{e_{-1}}{r}$. Take two points $z_1=x_1+iy(x_1)$ and $z_2=x_2+iy(x_2)$ with  $e_0-a_E(r_1)<x_1<x_2<e_0+a_E(r_1)$, which divides $\Ga$ into three subarcs $\Ga_1$, $\Ga_2$, and $\Ga_3$, from left to right, such that the left endpoints of $\Ga$ and $z_1$ are contained in $\cd(-1,1)$, the points $z_1$ and $z_2$ are contained in $\cd(0,1)$, and the points $z_2$ and the right endpoint of $\Ga$ are contained in $\cd(1,2/3)$, where the following sublemma guarantees that these conditions are sufficient.

\begin{sublemma}\label{arc-2021-10-26-1}\cite[Lemma 5.18]{InouShishikura2016}
Let $\Ga=\{x+iy:\ (\tfrac{x}{a})^2+(\tfrac{y}{b})^2=1,\ y\geq0\}$ with $a>b>0$. If two points
$z_1,z_2\in\Ga$ are contained in a disk $\cd(\zeta_0,r)$ with $\text{Im}\,\zeta_0\geq0$, then the subarc of $\Ga$ between $z_1$ and $z_2$ is contained in the same disk.
\end{sublemma}

Take the following constants $e_1 = 0.84$, $e_{0}=-0.18$, $e_{-1} = 0.6$, $r_1=1.4$, $x_1 = -0.5$, and $x_2 = 0.8$.

So, the arc $\Ga_1\subset\overline{\cd}(-1,1)$ is guaranteed by the following inequalities:
\beqq
a_E(r_1)-e_0<2\ \text{and}\ (x_1+1)^2+(y(x_1))^2<1;
\eeqq
the arc $\Ga_2\subset\overline{\cd}(0,1)$ is guaranteed by the following inequalities:
\beqq
x_1^2+(y(x_1))^2<1\ \text{and}\, x_2^2+(y(x_2))^2<1;
\eeqq
the arc $\Ga_3\subset\overline{\cd}(1,2/3)$ is guaranteed by the following inequalities:
\beqq
a_E(r_1)+e_0<1+\frac{2}{3}=\frac{5}{3}\ \text{and}\ (x_2-1)^2+(y(x_2))^2<\frac{2^2}{3^2}=\frac{4}{9}.
\eeqq

By direct computation, one has
\beqq
a_E(r_1)=0.84\times1.4+0.6/1.4=1.60457, \ b_E(r_1)=0.84\times1.4-0.6/1.4=0.747429,
\eeqq
\beqq
y(x_1)=b_E(r_1)\sqrt{1-\big(\tfrac{x_1-e_0}{a_E(r_1)}\big)^2}=
0.747429\sqrt{1-\big(\tfrac{-0.5-(-0.18)}{1.60457}\big)^2}\approx0.732415,
\eeqq
\beqq
y(x_2)=b_E(r_1)\sqrt{1-\big(\tfrac{x_2-e_0}{a_E(r_1)}\big)^2}=
0.747429\sqrt{1-\big(\tfrac{0.8-(-0.18)}{1.60457}\big)^2}\approx0.591829.
\eeqq
So,
\beqq
a_E(r_1)-e_0=1.60457+0.18=1.78457<2\
 \eeqq
and
\beqq
(x_1+1)^2+(y(x_1))^2=0.5^2 + 0.732415^2\approx0.786432<1;
\eeqq
\beqq
x_1^2+(y(x_1))^2=(-0.5)^2 + 0.732415^2\approx0.786432<1\
\eeqq
and
\beqq
 x_2^2+(y(x_2))^2=0.8^2+0.591829^2\approx0.990262<1;
\eeqq
\beqq
a_E(r_1)+e_0=1.60457-0.18=1.42457<\frac{5}{3}\approx1.66667\
\eeqq
and
\beqq (x_2-1)^2+(y(x_2))^2=0.2^2+0.591829^2\approx0.390262<\frac{4}{9}\approx0.44444.
\eeqq

(d)
For $\zeta\in\mathcal{U}^{Q+}_1\cup\mathcal{U}^{Q+}_2$, by Corollary \ref{cor-2021-10-2-1}, one has
\beq\label{equ2021-10-27-1}
\zeta\in\cc\setminus\bigg(\overline{\cd}\big(\cot(\tfrac{\pi}{m+1})i,\tfrac{1}{\sin(\tfrac{\pi}{m+1})}\big)\cup \overline{\cd}\big(-\cot(\tfrac{\pi}{m+1})i,\tfrac{1}{\sin(\tfrac{\pi}{m+1})}\big)\bigg).
\eeq
By Lemma \ref{paraequ-16}, one has
\begin{align}\label{equ2021-10-27-2}
&\bigg(\overline{\cd}\big(\cot(\tfrac{m\pi}{2(m+1)})i,\tfrac{1}{\sin(\tfrac{m\pi}{2(m+1)})}\big)\cup \overline{\cd}\big(-\cot(\tfrac{m\pi}{2(m+1)})i,\tfrac{1}{\sin(\tfrac{m\pi}{2(m+1)})}\big)\bigg)\nonumber\\
\subset&
\bigg(\overline{\cd}\big(\cot(\tfrac{\pi}{m+1})i,\tfrac{1}{\sin(\tfrac{\pi}{m+1})}\big)\cup \overline{\cd}\big(-\cot(\tfrac{\pi}{m+1})i,\tfrac{1}{\sin(\tfrac{\pi}{m+1})}\big)\bigg).
\end{align}
For $\zeta$ with $\rho\leq|Q(\zeta)|\leq R$, by (a) and (b) of this lemma, $\zeta$ does not lie in $\ol{\cd}(1,\ep_3)\cup\ol{\cd}(-1,\ep_4)$. This, together with (c) of this lemma, yields that
$$\zeta\in\cc\setminus \bigg(E_{r_1}\cup\overline{\cd}\big(\cot(\tfrac{\pi}{m+1})i,\tfrac{1}{\sin(\tfrac{\pi}{m+1})}\big)\cup \overline{\cd}\big(-\cot(\tfrac{\pi}{m+1})i,\tfrac{1}{\sin(\tfrac{\pi}{m+1})}\big)\bigg).$$

Consider the inverse image of $\cc\setminus\ol{\cd}(0,R)$ by the map $Q|_{\cc\setminus\ol{\cd}}$. By \eqref{nearequ-1} and Remark \ref{nearpara-5},
\begin{align*}
&Q^{-1}(\cc\setminus\ol{\cd}(0,R))\bigcap(\cc\setminus\ol{\cd}(0,R))
=\psi^{-1}_{1,+}\circ P^{-1}\circ \psi_0(\cc\setminus\ol{\cd}(0,R))\bigcap(\cc\setminus\ol{\cd}(0,R))\\
=&\psi^{-1}_{1,+}\circ P^{-1}(\cd(0,\tfrac{4}{R}))\bigcap(\cc\setminus\ol{\cd}(0,R))
=\psi^{-1}_{1,+}\bigg(P^{-1}(\cd(0,\tfrac{4}{R}))\bigcap(\psi_1(\cc\setminus\ol{\cd}(0,R)))\bigg).
\end{align*}
By the above discussions in Subsection \ref{nearequ-20} and Remark \ref{nearpara-5} $P^{-1}(\cd(0,\tfrac{4}{R}))\bigcap(\psi_1(\cc\setminus\ol{\cd}(0,R)))$ have two connected components, one component contains the point $0$, the closure of another component contains the point $-1$. So, $Q^{-1}(\cc\setminus\ol{\cd}(0,R))\bigcap(\cc\setminus\ol{\cd}(0,R))$ contains two connected components, denoted by $U$ and $U'$, $U$ is contained in $\mathcal{U}^{Q+}_1\cup \ga^{Q+}_{c1+}$, and $U'$ is contained in $(\mathcal{U}^{Q+}_2\cup\cdots\cup\mathcal{U}^{Q+}_{m+1})\cup(\ga^{Q+}_{c2}\cup\cdots\cup\ga^{Q+}_{c,m+1}) $, respectively. It is evident that $\infty\in\ol{U}$ and $1\in\ol{U'}$. By (a), $\widetilde{W}=\ol{\cd(1,\vep_3)}\setminus\big(\overline{\cd}\big(\cot(\al)i,\tfrac{1}{\sin(\al)}\big)\cup\overline{\cd}\big(-\cot(\al)i,\tfrac{1}{\sin(\al)}\big)\big)$ is contained in the component $U'$. So,
$\widetilde{W}\cap\ol{\mathcal{U}}^{Q+}_1=\emptyset$. This, together with the above discussions in \eqref{equ2021-10-27-1} and \eqref{equ2021-10-27-2}, implies that
$$\zeta\in\cc\setminus \bigg(E_{r_1}\cup\overline{\cd}\big(\cot(\tfrac{\pi}{m+1})i,\tfrac{1}{\sin(\tfrac{\pi}{m+1})}\big)\cup \overline{\cd}\big(-\cot(\tfrac{\pi}{m+1})i,\tfrac{1}{\sin(\tfrac{\pi}{m+1})}\big)\bigg).$$

\end{proof}

\begin{lemma} \label{paraequ-56}
\begin{itemize}
\item[(a)] For $|z|\geq r>cp_{Q1}$, one has
\begin{align*}
 &\log Q'(z)\\
 =&-\frac{1+8m+8m^2 }{z^2}
-\frac{ 32 m + 96 m^2 + 64 m^3}{3z^3}\\
&-\sum^{\infty}_{n=4}\frac{1}{n}\bigg(\frac{(\sqrt{m+1}+\sqrt{m})^2 }{z}\bigg)^n
-\sum^{\infty}_{n=4}\frac{1}{n}\bigg(\frac{(\sqrt{m+1}-\sqrt{m})^2 }{z}\bigg)^n+
\sum^{\infty}_{k=3}\frac{2(2m+1)}{(2k-1)z^{2k-1}}.
\end{align*}
Moreover, if $|z|\geq r>cp_{Q1}$, then
\begin{align*}
 &|\log Q'(z)|\\
 \leq&\text{Log}DQ_{max}(r):=
\frac{1+8m+8m^2 }{r^2}+
\frac{ 32 m + 96 m^2 + 64 m^3}{3r^3}\\
&+\frac{1}{4}\bigg(\frac{\big(\tfrac{(\sqrt{m+1}+\sqrt{m})^2 }{r}\big)^4}{1-\tfrac{(\sqrt{m+1}+\sqrt{m})^2 }{r}}+\frac{\big(\tfrac{(\sqrt{m+1}-\sqrt{m})^2 }{r}\big)^4}{1-\tfrac{(\sqrt{m+1}-\sqrt{m})^2 }{r}}\bigg)+
\frac{\frac{2(2m+1)}{5r^{5}}}{1-\tfrac{1}{r^2}}.\\
\end{align*}
\item[(b)] If $|z|>cp_{Q1}$, then $\text{Re}\,Q'(z)>0$. For any $\tht\in\rr$, $Q$ is injective in $\{z:\ \text{Re}(\zeta\,e^{-i\tht})>cp_{Q1}\}$.
\end{itemize}
\end{lemma}

\begin{proof}
(a) Since $-\log(1-x)=\sum^{\infty}_{n=1}\frac{x^n}{n}=x+\frac{x^2}{2}+\frac{x^3}{3}+\sum^{\infty}_{n=4}\frac{x^n}{n}$ and \eqref{paraequ-26}, one has
\begin{align*}
&\log Q'(z)=\log \bigg(1-\frac{(\sqrt{m+1}+\sqrt{m})^2 }{z}\bigg)+\log\bigg(1-\frac{(\sqrt{m+1}-\sqrt{m})^2}{z}\bigg)+\log\bigg(\frac{ 1+\frac{1}{z}}{1-\frac{1}{z}}\bigg)^{2 m+1}\\
=& \log \bigg(1-\frac{(\sqrt{m+1}+\sqrt{m})^2 }{z}\bigg)+\log\bigg(1-\frac{(\sqrt{m+1}-\sqrt{m})^2}{z}\bigg)\\
&+\log\bigg( 1+\frac{1}{z}\bigg)^{2 m+1}-\log\bigg(1-\frac{1}{z}\bigg)^{2 m+1}\\
=&-\sum^{\infty}_{n=1}\frac{1}{n}\bigg(\frac{(\sqrt{m+1}+\sqrt{m})^2 }{z}\bigg)^n
-\sum^{\infty}_{n=1}\frac{1}{n}\bigg(\frac{(\sqrt{m+1}-\sqrt{m})^2 }{z}\bigg)^n\\
&-(2m+1)\sum^{\infty}_{n=1}\frac{1}{n}\bigg(-\frac{1}{z}\bigg)^n+(2m+1)\sum^{\infty}_{n=1}\frac{1}{n}\bigg(\frac{1}{z}\bigg)^n\\
=&-\sum^{\infty}_{n=1}\frac{1}{n}\bigg(\frac{(\sqrt{m+1}+\sqrt{m})^2 }{z}\bigg)^n
-\sum^{\infty}_{n=1}\frac{1}{n}\bigg(\frac{(\sqrt{m+1}-\sqrt{m})^2 }{z}\bigg)^n+
\sum^{\infty}_{k=1}\frac{2(2m+1)}{(2k-1)z^{2k-1}}\\
=&-\sum^{\infty}_{n=4}\frac{1}{n}\bigg(\frac{(\sqrt{m+1}+\sqrt{m})^2 }{z}\bigg)^n
-\sum^{\infty}_{n=4}\frac{1}{n}\bigg(\frac{(\sqrt{m+1}-\sqrt{m})^2 }{z}\bigg)^n+
\sum^{\infty}_{k=3}\frac{2(2m+1)}{(2k-1)z^{2k-1}}\\
&-\sum^{3}_{n=1}\frac{1}{n}\bigg(\frac{(\sqrt{m+1}+\sqrt{m})^2 }{z}\bigg)^n
-\sum^{3}_{n=1}\frac{1}{n}\bigg(\frac{(\sqrt{m+1}-\sqrt{m})^2 }{z}\bigg)^n+
\sum^{2}_{k=1}\frac{2(2m+1)}{(2k-1)z^{2k-1}}\\
=&-\sum^{\infty}_{n=4}\frac{1}{n}\bigg(\frac{(\sqrt{m+1}+\sqrt{m})^2 }{z}\bigg)^n
-\sum^{\infty}_{n=4}\frac{1}{n}\bigg(\frac{(\sqrt{m+1}-\sqrt{m})^2 }{z}\bigg)^n+
\sum^{\infty}_{k=3}\frac{2(2m+1)}{(2k-1)z^{2k-1}}\\
&-\sum^{3}_{n=2}\frac{1}{n}\bigg(\frac{(\sqrt{m+1}+\sqrt{m})^2 }{z}\bigg)^n
-\sum^{3}_{n=2}\frac{1}{n}\bigg(\frac{(\sqrt{m+1}-\sqrt{m})^2 }{z}\bigg)^n+
\frac{2(2m+1)}{3z^{3}}\\
=&-\sum^{\infty}_{n=4}\frac{1}{n}\bigg(\frac{(\sqrt{m+1}+\sqrt{m})^2 }{z}\bigg)^n
-\sum^{\infty}_{n=4}\frac{1}{n}\bigg(\frac{(\sqrt{m+1}-\sqrt{m})^2 }{z}\bigg)^n+
\sum^{\infty}_{k=3}\frac{2(2m+1)}{(2k-1)z^{2k-1}}\\
&-\frac{1+8m+8m^2 }{z^2}
-\frac{2 + 36 m + 96 m^2 + 64 m^3}{3z^3}+
\frac{2(2m+1)}{3z^{3}}\\
=&-\sum^{\infty}_{n=4}\frac{1}{n}\bigg(\frac{(\sqrt{m+1}+\sqrt{m})^2 }{z}\bigg)^n
-\sum^{\infty}_{n=4}\frac{1}{n}\bigg(\frac{(\sqrt{m+1}-\sqrt{m})^2 }{z}\bigg)^n+
\sum^{\infty}_{k=3}\frac{2(2m+1)}{(2k-1)z^{2k-1}}\\
&-\frac{1+8m+8m^2 }{z^2}
-\frac{ 32 m + 96 m^2 + 64 m^3}{3z^3}.
\end{align*}

(b) For $|z|>(2m+1)+2\sqrt{m+m^2}=(\sqrt{m+1}+\sqrt{m})^2$, suppose that
$\text{Im}\, z\geq0$, it is evident that for $|z|>(2m+1)+2\sqrt{m+m^2}=(\sqrt{m+1}+\sqrt{m})^2$, one has $\text{Im}\,\frac{1}{z}\leq0$ and $\big|\frac{(\sqrt{m+1}+\sqrt{m})^2}{z}\big|<1$, implying that
\begin{align*}
&\arg\, \bigg(1+\frac{1}{z}\bigg)\leq0\leq  \arg\, \bigg(1-\frac{(\sqrt{m+1}-\sqrt{m})^2}{z}\bigg)\\
&\leq  \arg\, \bigg(1-\frac{1}{z}\bigg)\leq \arg\, \bigg(1-\frac{(\sqrt{m+1}+\sqrt{m})^2}{z}\bigg)<\frac{\pi}{2},
\end{align*}
where $(2m+1)+2\sqrt{m+m^2}\approx(2m+1)+2\sqrt{m+m^2+1/4}\approx 4m+2$ and
$(2m+1)-2\sqrt{m+m^2}\approx(2m+1)-2\sqrt{m+m^2+1/4}\approx0$  for large enough $m$.

By Lemma \ref{paraequ-13}, $\arg Q'(\zeta)=\text{Im}\,\log Q'(\zeta)$, and the discussions in (b), one has
\begin{align*}
&\arg\, Q'(z)\geq (2m+1)\text{arg}\, \bigg(1+\frac{1}{z}\bigg)-(2m+1)\text{arg}\, \bigg(1-\frac{1}{z}\bigg)
\geq -(2m+1)\arcsin\bigg|\frac{\tfrac{2}{z}}{1-\tfrac{1}{z}}\bigg|\\
\geq& -2(2m+1)\arcsin\bigg|\frac{1}{z-1}\bigg|\geq -2(2m+1)\arcsin\frac{1}{4m-1}\geq-\frac{\pi}{3}\cdot\frac{4m+2}{4m-1}>-\frac{\pi}{2},
\end{align*}
where $a=1-\tfrac{1}{z}$, $b=\tfrac{2}{z}$,  $4m<(\sqrt{m+1}+\sqrt{m})^2<4(m+1)$, and $m\geq2$.

For two distinct points $z_0$ and $z_1$ joined by a line segment within $\{z:\ |z|>(\sqrt{m+1}+\sqrt{m})^2\}$, one has
\beq\label{paraequ-75}
\frac{Q(z_1)-Q(z_0)}{z_1-z_0}=\frac{1}{z_1-z_0}\int^1_0\frac{d}{dt}Q(z_0+t(z_1-z_0))dt
=\int^1_0Q'(z_0+t(z_1-z_0))dt.
\eeq
By the above discussion, $\text{Re}\,Q'(z_0+t(z_1-z_0))>0$, implying $\text{Re}\,\frac{Q(z_1)-Q(z_0)}{z_1-z_0}>0$. So, $Q(z_0)\neq Q(z_1)$, $Q$ is injective in
$\{z:\ \text{Re}\, (ze^{-i\tht})>(\sqrt{m+1}+\sqrt{m})^2\}$.

\end{proof}

\begin{lemma}\label{paraequ-12}
There is an estimate on $Q(\zeta)$:
\beq\label{paraequ-66}
Q(z)=z+(4m+2)+
\frac{8m(m+1)+1}{z}+Q_2(z),
\eeq
where
\beqq
Q_2(z)=\frac{8C^3_{2m}+16m}{z(z-1)}+\sum^{2m}_{j=4}\frac{C^j_{2m}2^{j}}{z(z-1)^{j-2}}+
\sum^{2m}_{j=2}\frac{4C^j_{2m}2^j}{(z-1)^{j}}
\eeqq
and for $|z|\geq r>1$, one has
\begin{align}\label{near-equat-14}
&|Q_2(z)|\leq Q_{2,max}(r)\nonumber\\
:=& \frac{8C^3_{2m}+16m}{r(r-1)}+\frac{2(2m)(2m-1)(2m-2)(2m-3)}{3r(r-1)^2}\bigg(\frac{r+1}{r-1}\bigg)^{2m-4}\nonumber\\
&+\frac{8(2m)(2m-1)}{(r-1)^2}\bigg(\frac{r+1}{r-1}\bigg)^{2m-2}.
\end{align}

For any positive number $a$, one has
\begin{itemize}
\item if $r>(\tfrac{32}{3a}m^3)^{1/2}+1$, then
\begin{align}\label{paraequ-20}
\frac{8C^3_{2m}}{r(r-1)}<a;
\end{align}
\item if $r>(\tfrac{16m}{a})^{1/2}+1$, then
\begin{align}\label{paraequ-21}
\frac{16m}{r(r-1)}<a;
\end{align}
\item if $r>\max\{(\tfrac{2^5m^4}{a})^{1/3}+1,4m\}$, then
\begin{align}\label{paraequ-22}
\frac{2(2m)(2m-1)(2m-2)(2m-3)}{3r(r-1)^2}\bigg(\frac{r+1}{r-1}\bigg)^{2m-4}
<a;
\end{align}
\item if $r>\max\{\tfrac{4\sqrt{6}m}{\sqrt{a}}+1,4m\}$, then
\begin{align}\label{paraequ-23}
\frac{8(2m)(2m-1)}{(r-1)^2}\bigg(\frac{r+1}{r-1}\bigg)^{2m-2}
<a.
\end{align}
\end{itemize}

\end{lemma}

\begin{proof}
By direct calculation, one has
\begin{align*}
Q(z)&=\frac{(z+1)^2}{z}\cdot\bigg(\frac{z+1}{z-1}\bigg)^{2m}=\frac{(z-1)^2+4z}{z}\cdot\bigg(\frac{(z-1)+2}{z-1}\bigg)^{2m}\\
=&\frac{(z-1)^2}{z}\cdot\bigg(\frac{(z-1)+2}{z-1}\bigg)^{2m}+4\cdot\bigg(\frac{(z-1)+2}{z-1}\bigg)^{2m}\\
=&\frac{(z-1)^2}{z}\cdot\bigg(\frac{\sum^{2m}_{j=0}C^j_{2m}(z-1)^{2m-j}2^{j}}{(z-1)^{2m}}\bigg)+4\cdot\bigg(\frac{\sum^{2m}_{j=0}C^j_{2m}(z-1)^{2m-j}2^j}{(z-1)^{2m}}\bigg)\\
=&\frac{(z-1)^2}{z}\cdot\bigg(\frac{\sum^{3}_{j=0}C^j_{2m}(z-1)^{2m-j}2^{j}}{(z-1)^{2m}}+\frac{\sum^{2m}_{j=4}C^j_{2m}(z-1)^{2m-j}2^{j}}{(z-1)^{2m}}\bigg)\\
&+4\cdot\bigg(\frac{\sum^{1}_{j=0}C^j_{2m}(z-1)^{2m-j}2^j}{(z-1)^{2m}}+\frac{\sum^{2m}_{j=2}C^j_{2m}(z-1)^{2m-j}2^j}{(z-1)^{2m}}\bigg)\\
=&\frac{(z-1)^2}{z}\cdot\bigg(1+\frac{4m}{(z-1)}+
+\frac{4C^2_{2m}}{(z-1)^{2}}
+\frac{8C^3_{2m}}{(z-1)^{3}}+\frac{\sum^{2m}_{j=4}C^j_{2m}(z-1)^{2m-j}2^{j}}{(z-1)^{2m}}\bigg)\\
&+4\cdot\bigg(1+
\frac{4m}{(z-1)}
+\frac{\sum^{2m}_{j=2}C^j_{2m}(z-1)^{2m-j}2^j}{(z-1)^{2m}}\bigg)\\
=&\frac{(z-1)^2}{z}+\frac{4m(z-1)}{z}+
\frac{4C^2_{2m}}{z}
+\frac{8C^3_{2m}}{z(z-1)}+\sum^{2m}_{j=4}\frac{C^j_{2m}2^{j}}{z(z-1)^{j-2}}+4+
\frac{16m}{(z-1)}
+\sum^{2m}_{j=2}\frac{4C^j_{2m}2^j}{(z-1)^{j}}\\
=&z+(4m+2)+
\frac{8m(m-1)+1}{z}+\frac{16m}{(z-1)}
+\frac{8C^3_{2m}}{z(z-1)}+\sum^{2m}_{j=4}\frac{C^j_{2m}2^{j}}{z(z-1)^{j-2}}+
\sum^{2m}_{j=2}\frac{4C^j_{2m}2^j}{(z-1)^{j}}\\
=&z+(4m+2)+
\frac{8m(m+1)+1}{z}
+\frac{8C^3_{2m}+16m}{z(z-1)}+\sum^{2m}_{j=4}\frac{C^j_{2m}2^{j}}{z(z-1)^{j-2}}+
\sum^{2m}_{j=2}\frac{4C^j_{2m}2^j}{(z-1)^{j}}.\\
\end{align*}

Consider the following function with $r\in(1,+\infty)$:
\begin{align*}
&\sum^{2m}_{j=4}\frac{C^j_{2m}2^{j}}{r(r-1)^{j-2}}+
\sum^{2m}_{j=2}\frac{4C^j_{2m}2^j}{(r-1)^{j}}=\sum^{2m-4}_{j=0}\frac{C^{j+4}_{2m}2^{j+4}}{r(r-1)^{j+2}}+
\sum^{2m-2}_{j=0}\frac{4C^{j+2}_{2m}2^{j+2}}{(r-1)^{j+2}}\\
=&\sum^{2m-4}_{j=0}\frac{(2m)!}{(j+4)!(2m-j-4)!}\frac{2^{j+4}}{r(r-1)^{j+2}}+
\sum^{2m-2}_{j=0}\frac{(2m)!}{(j+2)!(2m-j-2)!}\frac{2^{j+4}}{(r-1)^{j+2}}\\
=&\sum^{2m-4}_{j=0}\frac{(2m)(2m-1)(2m-2)(2m-3)}{(j+4)(j+3)(j+2)(j+1)}\frac{(2m-4)!}{j!(2m-4-j)!}\frac{2^{j+4}}{r(r-1)^{j+2}}\\
&+\sum^{2m-2}_{j=0}\frac{(2m)(2m-1)}{(j+2)(j+1)}\frac{(2m-2)!}{j!(2m-2-j)!}\frac{2^{j+4}}{(r-1)^{j+2}}\\
\leq&\frac{(2m)(2m-1)(2m-2)(2m-3)}{4!}\frac{2^4}{r(r-1)^2}\sum^{2m-4}_{j=0}\frac{(2m-4)!}{j!(2m-4-j)!}\cdot\frac{2^{j}}{(r-1)^{j}}\cdot1^{2m-4-j}\\
&+\frac{(2m)(2m-1)}{2!}\frac{2^4}{(r-1)^2}\sum^{2m-2}_{j=0}\frac{(2m-2)!}{j!(2m-2-j)!}\cdot\frac{2^{j}}{(r-1)^{j}}\cdot1^{2m-2-j}\\
=&\frac{2(2m)(2m-1)(2m-2)(2m-3)}{3r(r-1)^2}\bigg(\frac{r+1}{r-1}\bigg)^{2m-4}+\frac{8(2m)(2m-1)}{(r-1)^2}\bigg(\frac{r+1}{r-1}\bigg)^{2m-2}.\\
\end{align*}

Consider the function $g(x)=(1+\tfrac{1}{x})^x$ with $x\in[2,+\infty)$. The derivative is
\begin{align}\label{paraequ-36}
&(\log g(x))'=(x\log(1+\tfrac{1}{x}))'=\log(1+\tfrac{1}{x})+x\frac{-\tfrac{1}{x^2}}{1+\tfrac{1}{x}}=\log(1+\tfrac{1}{x})-\frac{\tfrac{1}{x}}{1+\tfrac{1}{x}}\nonumber\\
=&\sum^{+\infty}_{n=1}\frac{(-1)^{n+1}}{n}\frac{1}{x^n}-\sum^{+\infty}_{n=1}(-1)^{n+1}\frac{1}{x^n}
=\sum^{+\infty}_{n=1}(-1)^{n+1}\bigg(\frac{1}{n}-1\bigg)\frac{1}{x^n}
=\sum^{+\infty}_{n=2}(-1)^{n}\bigg(1-\frac{1}{n}\bigg)\frac{1}{x^n}>0,
\end{align}
where $(1-\tfrac{1}{n})\tfrac{1}{x^n}>(1-\tfrac{1}{n+1})\tfrac{1}{x^{n+1}}$ by $x\geq2$. This, together with $\lim_{x\to+\infty}g(x)=e$, implies that $g(x)\leq e$ for $x\in[2,+\infty)$.  Similarly, one has that  $g(x)=(1+\tfrac{1}{x})^x$ with $x\in(-\infty,-2]$ is increasing and $\lim_{x\to-\infty}g(x)=e$, implying that $e\leq g(x)\leq g(-2)=4$ for $x\in(-\infty,-2]$.

The derivative is
\begin{align*}
\bigg(\log\bigg(\frac{x+1}{x-1}\bigg)\bigg)'=(\log(x+1)-\log(x-1))'=\frac{1}{x+1}-\frac{1}{x-1}<0\ \mbox{for}\ x\in(1,+\infty).
\end{align*}

For $r\geq 4m$, one has
\begin{align}\label{near-equat-12}
 &\bigg(\frac{r+1}{r-1}\bigg)^{2m-4}= \bigg(1+\frac{2}{r-1}\bigg)^{2m-4}\leq \bigg(1+\frac{2}{4m-1}\bigg)^{2m-4}\nonumber\\
=&\bigg(1+\frac{2}{4m-1}\bigg)^{2m-1/2}\cdot\bigg(1+\frac{2}{4m-1}\bigg)^{-7/2}
\leq e\cdot1=e
\end{align}
and
\begin{align}\label{near-equat-13}
 &\bigg(\frac{r+1}{r-1}\bigg)^{2m-2}= \bigg(1+\frac{2}{r-1}\bigg)^{2m-2}\leq \bigg(1+\frac{2}{4m-1}\bigg)^{2m-2}\nonumber\\
=&\bigg(1+\frac{2}{4m-1}\bigg)^{2m-1/2}\cdot\bigg(1+\frac{2}{4m-1}\bigg)^{-3/2}
\leq e\cdot1=e.
\end{align}

Hence, for any positive number $a$, one has
\begin{itemize}
\item if $r>(\tfrac{32}{3a}m^3)^{1/2}+1$, then
\begin{align*}
\frac{8C^3_{2m}}{r(r-1)}<\frac{\tfrac{4}{3}(2m)^3}{(r-1)^2}<a;
\end{align*}
\item if $r>(\tfrac{16m}{a})^{1/2}+1$, then
\begin{align*}
\frac{16m}{r(r-1)}<\frac{16m}{(r-1)^2}<a;
\end{align*}
\item if $r>\max\{(\tfrac{2^5m^4}{a})^{1/3}+1,4m\}$, then
\begin{align*}
\frac{2(2m)(2m-1)(2m-2)(2m-3)}{3r(r-1)^2}\bigg(\frac{r+1}{r-1}\bigg)^{2m-4}
<\frac{2(2m)^4}{3(r-1)^3}e<\frac{2(2m)^4}{(r-1)^3}<a;
\end{align*}
\item if $r>\max\{\tfrac{4\sqrt{6}m}{\sqrt{a}}+1,4m\}$, then
\begin{align*}
\frac{8(2m)(2m-1)}{(r-1)^2}\bigg(\frac{r+1}{r-1}\bigg)^{2m-2}
<\frac{8(2m)^2}{(r-1)^2}e<\frac{24(2m)^2}{(r-1)^2}<a.
\end{align*}
\end{itemize}

\end{proof}

\begin{lemma}\label{near-equati-22}
For the map $Q$ and $m\geq5$, one has
\beqq
Q\big(\overline{\vv}\big((4e-1)m,\tfrac{\pi}{5}\big)\big)\subset\vv\big(4e(m+1),\tfrac{\pi}{5}\big)\subset\vv\big(cv_{Q},\tfrac{\pi}{5}\big),
\eeqq
\end{lemma}

\begin{remark}\label{nearequ-23}
By direct calculation, one has
\[
\sin \left(\frac{\pi}{5}\right)=\sqrt{\frac{5-\sqrt{5}}{8}}=\frac{1}{4} \sqrt{10-2 \sqrt{5}}\approx0.587785\ \text{and}\
\cos \left(\frac{\pi}{5}\right)=\frac{1}{4}(1+\sqrt{5})\approx0.809017.
\]
This can be derived by using the multiple-angle formula
\[
\sin (5 \theta)=5 \sin \theta-20 \sin ^{3} \theta+16 \sin ^{5} \theta.
\]
\end{remark}

\begin{remark}\label{near-equat-21}
 By Lemma \ref{paraequ-25}, for $m\geq5$,
\begin{align*}
&4(m+1)\cdot e\geq cv_Q=4(m+1)\frac{(m+1)^{m}}{m^{m}}\geq4(m+1) \frac{(5+1)^{5}}{5^{5}}\\
=&\frac{7776}{3125}\cdot4(m+1)\approx9.95328(m+1)>(4e-1)m\approx9.87313m.
\end{align*}
\end{remark}

\begin{proof}
Suppose $\zeta\in\overline{\vv}\big((4e-1)m,\tfrac{\pi}{5}\big)$ and $\zeta'=\zeta+(m+4e)$, so $\zeta'\in \vv\big(4e(m+1),\tfrac{\pi}{5}\big)$. It is sufficient to show that $|\arg(Q(\zeta)-\zeta')|=\bigg|\arg\bigg(\tfrac{8m(m+1)+1}{\zeta}+(3m+2-4e+Q_2(\zeta))\bigg)\bigg|<\tfrac{\pi}{5}$ by \eqref{paraequ-66}.

It follows from $\zeta\in\overline{\vv}\big((4e-1)m,\tfrac{\pi}{5}\big)$ that $|\arg\,\zeta|<\tfrac{\pi}{5}$ and $|\arg\,\tfrac{8m(m+1)+1}{\zeta}|<\tfrac{\pi}{5}$.

Now, we give an estimate on $Q_2(\zeta)$.
Since $\zeta\in\overline{\vv}\big((4e-1)m,\tfrac{\pi}{5}\big)$, $|\zeta|\geq(4e-1)m\approx9.87313m$. This, together with $m\geq3$, yields that $|\zeta|>9m+1$.

\begin{itemize}
\item By \eqref{paraequ-20},  if $a=\tfrac{2^5}{3^5}m$ and $r>9m+1$, then
\begin{align*}
\frac{8C^3_{2m}}{r(r-1)}<\frac{2^5}{3^5}m;
\end{align*}
\item by \eqref{paraequ-21},  if $a=\tfrac{2^4}{3^4m}$ and $r>9m+1$, then
\begin{align*}
\frac{16m}{r(r-1)}<\frac{2^4}{3^4m};
\end{align*}
\item by \eqref{paraequ-22}, if $a=\tfrac{2^5}{3^6}m$ and $r>9m+1$, then
\begin{align*}
\frac{2(2m)(2m-1)(2m-2)(2m-3)}{3r(r-1)^2}\bigg(\frac{r+1}{r-1}\bigg)^{2m-4}
<\frac{2^5}{3^6}m;
\end{align*}
\item by \eqref{paraequ-23}, if $a=\tfrac{2^5}{3^3}$ and $r>9m+1$, then
\begin{align*}
\frac{8(2m)(2m-1)}{(r-1)^2}\bigg(\frac{r+1}{r-1}\bigg)^{2m-2}
<\frac{2^5}{3^3}.
\end{align*}
\end{itemize}
So,
\beq\label{near-equat-20}
|Q_2(\zeta)|\leq \frac{2^7}{3^6}m+\frac{2^4}{3^4}\cdot\frac{1}{m}+\frac{2^5}{3^3}\ \forall |\zeta|\geq9m+1.
\eeq

By direct numerical computation, the inequality
\begin{align*}
\frac{\frac{2^7}{3^6}m+\frac{2^4}{3^4}\cdot\frac{1}{m}+\frac{2^5}{3^3}}{3m+2-4e}
<\frac{\frac{2^7}{3^6}m+\frac{2^4}{3^4}\cdot\frac{1}{3}+\frac{2^5}{3^3}}{3m+2-4e}<0.5877
\end{align*}
is equivalent to
\beqq
-6.46577 + 1.58752 m>0,
\eeqq
which can be derived by $m\geq5$.

So,
\begin{align*}
&|\arg(3m+2-4e+Q_2(\zeta))|=\bigg|(3m+2-4e)\arg\bigg(1+\frac{Q_2(\zeta)}{3m+2-4e}\bigg)\bigg|\\
=&\bigg|\arg\bigg(1+\frac{Q_2(\zeta)}{3m+2-4e}\bigg)\bigg|
\leq\arcsin\bigg(\bigg|\frac{Q_2(\zeta)}{3m+2-4e}\bigg|\bigg)\leq\arcsin(0.5877)\\
<&\arcsin\bigg(\frac{1}{4}(1+\sqrt{5})\bigg)=\frac{\pi}{5}.
\end{align*}

This, together with $|\arg\,\tfrac{8m(m+1)+1}{\zeta}|<\tfrac{\pi}{5}$, implies that
$$|\arg(Q(\zeta)-\zeta')|=\bigg|\arg\bigg(\tfrac{8m(m+1)+1}{\zeta}+(3m+2-4e+Q_2(\zeta))\bigg)\bigg|<\tfrac{\pi}{5}.$$
\end{proof}

\subsection{Univalent function}
\cite{Duren1983,Pommerenke1975}

A map $f$ is called conformal/univalent if it is analytic and one-to-one, (holomorphic and injective).

Let $\mathcal{S}$ be the collection of univalent functions in the open unit disk $\De=\{z\in\mathbb{C}:\ |z|<1 \}$ such that
$f(0)=0$ and $f'(0)=1$. The compactness properties of this class of functions are crucial.

\beqq
\Sigma:=\{g:\cc\setminus\cd\to\cc:\ g\ \mbox{is univalent and}\ \lim_{z\to\infty}\tfrac{g(z)}{z}=1\},
\eeqq
\beqq
\Sigma_0:=\{g\in\Sigma:\ c_0= \lim_{z\to\infty}(g(z)-z)=0\},
\eeqq
\beq\label{paraequ-7}
\Sigma_*:=\{g\in\Sigma:\ 0\not\in \mbox{Image}(g)\}.
\eeq

\begin{theorem}\label{paraequ-87} If $f\in\mathcal{S}$, then
\begin{itemize}
\item $|f^{\prime\prime}(0)|\leq4$,
\item $\bigg|\log(z\frac{f'(z)}{f(z)})\bigg|\leq\log\frac{1+|z|}{1-|z|}$\ for $|z|<1$;
\item $\big|\log(\frac{f(z)}{z})+\log(1-|z|^2)\big|\leq\log(\frac{1+|z|}{1-|z|})$ for $|z|<1$.
\end{itemize}

\end{theorem}

\begin{theorem} (Koebe one-quater theorem)\cite[Theorem 1.3]{Carleson1993}
If $f\in\mathcal{S}$, then the image of $f$ covers the open disk centered at $0$ of radius one-quater, that is, $f(\De)\supset \De(0,1/4)$. The Koebe function $\psi_{Koebe}(z)=z/(1-z)^2=z(\tfrac{1}{(1-z)})'=z(1+z+z^2+z^3+\cdots)'=z(1+2z+3z^2+\cdots)=\sum^{\infty}_{n=1}nz^n$
maps the disk to the slit plane $\mathbb{C}\setminus(-\infty,-1/4]$. This is the optimal bound.
\end{theorem}

\begin{theorem}\cite[Theorem 1.5]{Carleson1993}
If $f\in\mathcal{S}$, then
\beqq
\bigg|\frac{zf^{\prime\prime}(z)}{f'(z)}-\frac{2|z|^2}{1-|z|^2}\bigg|\leq\frac{4|z|}{1-|z|^2}.
\eeqq
\end{theorem}

\begin{theorem}(Distortion theorem)\cite[Theorem 1.6]{Carleson1993}
If $f\in\mathcal{S}$, then
\beqq
\frac{|z|}{(1+|z|)^2}\leq|f(z)|\leq\frac{|z|}{(1-|z|)^2}
\eeqq
and
\beqq
\frac{1-|z|}{(1+|z|)^3}\leq|f'(z)|\leq\frac{1+|z|}{(1-|z|)^3}.
\eeqq
\end{theorem}

\begin{theorem} \label{paraequ-8}\cite[Theorem A.2]{InouShishikura2016}
Let $g(z)=z+c_0+g_1(z)\in\Sigma$. Then the following estimates hold:
\begin{itemize}
\item[(a)] $\{z:\in\cc:\ |z-c_0|>2\}\subset\mbox{Image}(g)$. In particular, if $g\in\Sigma_*$, then $|c_0|\leq2$.
\item[(b)] $|g_1(z)|\leq\sqrt{\log\frac{1}{1-|z|^{-2}}}$.
\item[(c)] $|\log g'(z)|\leq\log\frac{1}{1-|z|^{-2}}$.
\item[(d)] If $g\in\Sigma_*$, then
\beqq
\bigg|\log\frac{g(z)}{z}-\log\bigg(1-\frac{1}{|z|^2}\bigg)\bigg|\leq\log\frac{|z|+1}{|z|-1}.
\eeqq
In particular,
\beqq
|z|\bigg(1-\frac{1}{|z|}\bigg)^2\leq|g(z)|\leq|z|\bigg(1+\frac{1}{|z|}\bigg)^2\ \mbox{and}\ \bigg|\mbox{arg}\,\frac{g(z)}{z}\bigg|\leq\log\frac{|z|+1}{|z|-1}.
\eeqq
\end{itemize}
\end{theorem}

\begin{theorem} (A consequence of Golusin inequalities) \cite[Theorem A.3]{InouShishikura2016}
Let $\Om$ be a disk or a half plane in $\widehat{\cc}$ (including the case of the complement of a closed disk). If $g:\Om\to\widehat{\cc}$ is a uninvalent holomorphic map, then for $z,\xi\in\Om$ with $z,\xi,g(z),g(\xi)\neq\infty$ and $z\neq\xi$,
\beqq
\bigg|\log\bigg(\frac{g'(z)g'(\xi)(z-\xi)^2}{(g(z)-g(\xi))^2}\bigg)\bigg|\leq 2\log\cosh\frac{d_{\Om}(z,\xi)}{2}
\eeqq

\end{theorem}

\begin{theorem}\label{paraequ-64} (A general estimate on Fatou coordinate) \cite[Theorem 5.12]{InouShishikura2016}
Let $\Om$ be a disk or a half plane and $f:\Om\to\cc$ a holomorphic function with $f(z)\neq z$. Suppose that $f$ has a univalent Fatou coordinate $\Phi:\Om\to\cc$, i.e., $\Phi(f(z))=\Phi(z)+1$ when $z,f(z)\in\Om$. If $z\in\Om$ and $f(z)\in\Om$, then
\beqq
\bigg|\log\Phi'(z)+\log(f(z)-z)-\frac{1}{2}\log f'(z)\bigg|\leq\log\cosh\frac{d_{\Om}(z,f(z))}{2}=\frac{1}{2}\log\frac{1}{1-r^2},
\eeqq
where $r$ is a real number such that $0\leq r<1$ and $d_{\cd}(0,r)=d_{\Om}(z,f(z))$.
\end{theorem}

\begin{lemma}\label{paraequ-11}\cite[Lemma 5.22]{InouShishikura2016}
 Let $\varphi:\widehat{\cc}\setminus E_1\to\widehat{\cc}\setminus\{0\}$ be a normalized univalent map, which is written as
\beq\label{paraequ-9}
\varphi(\zeta)=\zeta+c_0+\varphi_1(\zeta)
\eeq
where $\zeta(w)$ is specified in \eqref{paraequ-1}, $E_1$ is introduced in \eqref{paraequ-3} with $r=1$, $e_1\neq0$, $c_0\in\cc$,  $a_E(1)=e_1+e_{-1}\neq0$, $b_E(1)=e_1-e_{-1}\neq0$, $\lim_{\zeta\to\infty}\varphi_1(\zeta)=0$. Then, the following estimates hold:
\begin{itemize}
\item[(a)] $|c_0+e_0|\leq 2|e_1|$;
\item[(b)] $\text{Image}\,(\varphi)\supset\{z:\ |z-(c_0+e_0)|>2|e_1|\}\supset\{z:\ |z|>4|e_1|\}$;
\item[(c)] $|e_1||w|(1-\tfrac{1}{|w|})^2\leq|\varphi(\zeta(w))|\leq |e_1||w|(1+\tfrac{1}{|w|})^2$ for $|w|>1$;
\item[(d)] $|\text{arg}\,\tfrac{\varphi(\zeta(w))}{w}|\leq\log\tfrac{|w|+1}{|w|-1}$ for $|w|>1$;
\item[(e)] $|\varphi_1(\zeta)|\leq\varphi_{1,max}(r):=|a_E(1)|\sqrt{-\log\bigg(1-\bigg(\frac{|a_E(1)|}{r-|e_0|}\bigg)^2\bigg)}$ for $|\zeta|\geq r>|a_E(1)|+|e_0|$;
\item[(f)] $|\log \varphi'(\zeta)|\leq \log D\varphi_{max}(r):=-\log\bigg(1-\bigg(\frac{|a_E(1)|}{r-|e_0|}\bigg)^2\bigg)$ for $|\zeta|\geq r>|a_E(1)|+|e_0|$.
\end{itemize}
\end{lemma}

\begin{remark}
In \cite[Lemma 5.22]{InouShishikura2016}, the constants are given. Actually, the arguments there work for general conditions. So, the proof is given here.
\end{remark}

\begin{remark}\label{paraequ-29}
In  (e), the function $\varphi_{1,max}(r)=|a_E(1)|\sqrt{-\log\bigg(1-\bigg(\frac{|a_E(1)|}{r-|e_0|}\bigg)^2\bigg)}$ can be thought of as $\widetilde{\varphi}_{1,max}(r,|a_E(1)|,|e_0|)$. It is evident that this function is increasing with respect to $|a_E(1)|$ or $|e_0|$, and is decreasing with respect to $r$. So, suppose that $|a_E(1)|\leq3$ and $|e_0|\leq1$, consider the following function for convenience $\widehat{\varphi}_{1,max}(r)=3\sqrt{-\log\bigg(1-\bigg(\frac{3}{r-1}\bigg)^2\bigg)}$. By numerical calculation, one has for $r\geq4m+2\sqrt{6}-3\geq5+2\sqrt{6}$, $\widehat{\varphi}_{1,max}(r)\leq\varphi_{1,max}(5+2\sqrt{6})\approx1.03$.

In (f), the function $\log D\phi_{max}(r):=-\log\bigg(1-\bigg(\frac{|a_E(1)|}{r-|e_0|}\bigg)^2\bigg)$, $r>|a_E(1)|+|e_0|$,  can be thought of as $\widetilde{ \log D\varphi}_{max}(r,|a_E(1)|,|e_0|)$, it is evident that this function is increasing with respect to $|a_E(1)|$ or $|e_0|$, and is decreasing with respect to $r$.  So, suppose that $|a_E(1)|\leq3$ and $|e_0|\leq1$, consider the following function for convenience $\widehat{\log D\varphi}_{max}(r):=-\log\bigg(1-\bigg(\frac{3}{r-1}\bigg)^2\bigg)$. By numerical calculation, one has for $r\geq4m+2\sqrt{6}-3\geq5+2\sqrt{6}$, $\widehat{\log D\varphi}_{max}(r)\leq\log D\varphi_{max}(5+2\sqrt{6})\approx0.120641$.

 Take the following constants $e_1 = 0.84$, $e_{0}=-0.18$, $e_{-1} = 0.6$, one has, for $r\geq4m+2\sqrt{6}-3\geq5+2\sqrt{6}\approx9.89898$,
\beq\label{paraequ-30-1-21-1}
\varphi_{1,max}(r)\leq(0.84+0.6)\sqrt{-\log\bigg(1-\bigg(\frac{0.86+0.6}{5+2\sqrt{6}-0.18}\bigg)^2\bigg)}
\approx0.214541
\eeq
and
\beq\label{paraequ-31}
\log D\phi_{max}(r)\leq-\log\bigg(1-\bigg(\frac{0.84+0.6}{5+2\sqrt{6}-0.18}\bigg)^2\bigg)
\approx0.022197.
\eeq

\end{remark}

\begin{proof}
Let $\hat{\varphi}(w)=\tfrac{1}{e_1}\varphi(\zeta(w))$. By the assumption $\varphi:\widehat{\cc}\setminus E\to\widehat{\cc}\setminus\{0\}$, $\hat{\varphi}\in\Sigma_*$, where $\Sigma_*$ is introduced in \eqref{paraequ-7}. By direct computation,
\beqq
\hat{\varphi}(w)=w+\frac{c_0+e_0}{e_1}+\frac{1}{e_1}\bigg(\varphi_1(\zeta(w))+\frac{e_{-1}}{w}\bigg)=
w+\frac{c_0+e_0}{e_1}+O\bigg(\frac{1}{w}\bigg),
\eeqq
this, together with Theorem \ref{paraequ-8} (a), implies that $|\tfrac{c_0+e_0}{e_1}|\leq2$ and $\{z:\ |z|>4\}\subset\{z:\ |z-\tfrac{c_0+e_0}{e_1}|>2\}\subset\text{Image}\,\hat{\varphi}$. So, (a) and (b) holds. And, (c) and (d) can be derived by applying Theorem \ref{paraequ-8} (d) to $\hat{\varphi}$.

Let $\tilde{\zeta}=\frac{\zeta-e_0}{a_E(1)}$. If $|\tilde{\zeta}|>1$, then $\zeta=e_0+a_E(1)\tilde{\zeta}\in\cc\setminus E_1$ and
\begin{align*}
&\tilde{\varphi}(\tilde{\zeta})=\frac{1}{a_E(1)}\varphi(e_0+a_E(1)\tilde{\zeta})
=\frac{1}{a_E(1)}(e_0+a_E(1)\tilde{\zeta}+c_0+\varphi_1(e_0+a_E(1)\tilde{\zeta}))\\
=&\tilde{\zeta}+\frac{e_0+c_0}{a_E(1)}+\frac{\varphi_1(e_0+a_E(1)\tilde{\zeta})}{a_E(1)},
\end{align*}
\beqq
\lim_{\tilde{\zeta}\to\infty}\bigg(\frac{e_0+c_0}{\tilde{\zeta}a_E(1)}+\frac{\varphi_1(e_0+a_E(1)\tilde{\zeta})}{a_E(1)\tilde{\zeta}}\bigg)
=\lim_{\tilde{\zeta}\to\infty}\bigg(\frac{e_0+a_E(1)\tilde{\zeta}}{a_E(1)\tilde{\zeta}}\cdot\frac{\varphi_1(e_0+a_E(1)\tilde{\zeta})}{e_0+a_E(1)\tilde{\zeta}}\bigg)=0,
\eeqq
and it follows from the assumption of $\varphi$ that $0\not\in\text{Image}\,\tilde{\varphi}$. So, $\tilde{\varphi}(\tilde{\zeta})=\tfrac{1}{a_E(1)}\varphi(e_0+a_E(1)\tilde{\zeta})\in\Sigma_*$
is well-defined. Hence, (e) and (f) can be derived by applying Theorem \ref{paraequ-8} (b) and (c) to $\tilde{\varphi}\in\Sigma_*$.
\end{proof}

\begin{lemma}\label{nearequ-24}
Fix the given constants $e_1=0.84$, $e_0=-0.18$, $e_{-1}=0.6$, $r_1=1.4$.
If $\zeta\in\cc\setminus\text{int}\,E_{r_1}$, then $|\varphi(\zeta)|>\rho=0.07$ and $|\text{arg}\,\tfrac{\varphi(\zeta)}{\zeta}|<\pi$.
\end{lemma}

\begin{proof}
For $\zeta\in\cc\setminus\text{int}\,E_{r_1}$ with $r_1=1.4$, $\zeta=\zeta(w)$ with $|w|>1.4$. By Lemma \ref{paraequ-11} (c) and the fact that $r(1-\tfrac{1}{r})^2$ is increasing for $r>1$, one has
\beqq
|\varphi(\zeta)|=|\varphi(\zeta(w))|\geq e_1|w|(1-\tfrac{1}{|w|})^2\geq0.84\times1.4(1-\tfrac{1}{1.4})^2\approx0.096>\rho=0.07.
\eeqq

It follows from Lemma \ref{paraequ-11} (d) that
\beqq
\bigg|\arg\frac{\varphi(\zeta(w))}{w}\bigg|\leq\log\frac{1.4+1}{1.4-1}\approx1.79176\approx 0.570335\pi.
\eeqq
On the other hand,
\beqq
\bigg|\frac{e_0}{e_1w}+\frac{e_{-1}}{e_1w^2}\bigg|\leq
\bigg|\frac{e_0}{e_1w}\bigg|+\bigg|\frac{e_{-1}}{e_1w^2}\bigg|\leq
\frac{0.18}{0.84\times 1.4}+\frac{0.6}{0.84\times 1.4^2}\approx0.517493,
\eeqq
together with Lemma \ref{paraequ-13}, implies that
\begin{align*}
&\bigg|\arg\frac{\zeta(w)}{w}\bigg|=\bigg|\arg\bigg(1+\frac{e_0}{e_1w}+\frac{e_{-1}}{e_1w^2}\bigg)\bigg|=
\bigg|\arg\bigg(1+\frac{e_0}{e_1w}+\frac{e_{-1}}{e_1w^2}\bigg)-\arg(1)\bigg|\\
\leq&
\arcsin\bigg(\bigg|\frac{e_0}{e_1w}+\frac{e_{-1}}{e_1w^2}\bigg|\bigg)
\leq\frac{\pi}{3}\cdot0.517493=0.172498\pi.
\end{align*}
Hence,
\beqq
\bigg|\arg\frac{\varphi(\zeta)}{\zeta}\bigg|\leq\bigg|\arg\frac{\varphi(\zeta(w))}{w}\bigg|+\bigg|\arg\frac{\zeta(w)}{w}\bigg|
<0.570335\pi+0.172498\pi<\pi.
\eeqq
\end{proof}

\begin{lemma}\label{equ2021-11-25-5}
Let $\al\in(0,\tfrac{\pi}{2})$. If
$$\zeta\in\cc\setminus\bigg(\cd\big(   \cot(\al)i,\tfrac{1}{\sin(\al)}\big)
\cup
\cd\big(-\cot(\al)i,\tfrac{1}{\sin(\al)}\big) \cup\mbox{int}\,E_{r_1}\bigg)$$
and $\mbox{Re}\ \zeta\geq e_0=-0.18$,
where $r_1=1.4$, then $\varphi(\zeta)\not\in\rr_{-}$.
\end{lemma}

\begin{proof}
It follows from (d) of Lemma \ref{paraequ-11} that
\beqq
|\arg(\varphi(\zeta(w)))|\leq|\arg\,w|+\bigg|\arg\frac{\varphi(\zeta(w))}{w}\bigg|
\leq|\arg\,w|+\log\frac{|w|+1}{|w|-1}\ \forall\ |w|>1.
\eeqq
By the proof of Lemma \ref{paraequ-27}, for $w=re^{i\tht}$ with $r=|w|$ and $\arg w=\tht$, one has
\begin{align*}
\zeta(w)=e_0+\bigg(e_1 r+\frac{e_{-1}}{r}\bigg)\cos(\tht)+i\bigg(e_1 r-\frac{e_{-1}}{r}\bigg)\sin(\tht).
\end{align*}
For $\zeta\in\cc\setminus E_{r_1}$ and $\text{Re}\,\zeta\geq e_0$, it suffices to consider $\zeta=\zeta(w)$, where $w=re^{i\tht}$, $r=|w|\geq r_1=1.4$, and $\tht=\arg w\in[-\tfrac{\pi}{2},\tfrac{\pi}{2}]$. It is sufficient to study the following statement: if $r\geq r_1$ and $0\leq\tht\leq\tfrac{\pi}{2}$, then either $\tht+\log\tfrac{r+1}{r-1}<\pi$ or $\zeta(re^{i\tht})\in \cd\big(   \cot(\al)i,\tfrac{1}{\sin(\al)}\big)$.

Consider the following different situations:
\begin{itemize}
\item[(a)] $r>r_1=1.4$ and $0\leq\tht\leq0.4\pi$;
\item[(b)] $r\geq 1.54$ and $0\leq\tht\leq\tfrac{\pi}{2}$;
\item[(c)] $1.4\leq r\leq1.54$ and $0.4\pi\leq\tht\leq\tfrac{\pi}{2}$.
\end{itemize}

For Case (a), by direct calculation, one has
\beqq
\log\frac{1.4+1}{1.4-1}\approx1.79176\approx 0.570335\pi.
\eeqq
For Case (b), consider the root of the equation
\beqq
\log\bigg(\frac{x+1}{x-1}\bigg)=\frac{\pi}{2},
\eeqq
the solution is
\beqq
x=\frac{e^{\pi/2}+1}{e^{\pi/2}-1}\approx1.52487<1.54.
\eeqq

For Case (c), one needs to study the following statement.

{\bf Claim} Let $\al\in(0,\tfrac{\pi}{2})$. Let $1\leq s_1<s_2$ and $\tht_1=0.4\pi$. If $\zeta(s_2 i)$ and $\zeta(s_2e^{i\tht_1})$ are contained in  $\cd\big(\cot(\al)i,\tfrac{1}{\sin(\al)}\big)$,
then the set
\beqq
Z(s_1,s_2,\tht_1)=\bigg\{\zeta(w):\ s_1\leq|w|\leq s_2\ \mbox{and}\ \tht_1\leq\tht\leq\frac{\pi}{2}\bigg\}
\eeqq
is also contained in $\cd\big(\cot(\al)i,\tfrac{1}{\sin(\al)}\big)$.

Now, the proof of the statement is provided.

By the proof of Lemma \ref{paraequ-27}, one has
\begin{align*}
\zeta(w)=e_0+\bigg(e_1 r+\frac{e_{-1}}{r}\bigg)\cos(\tht)+i\bigg(e_1 r-\frac{e_{-1}}{r}\bigg)\sin(\tht).
\end{align*}

A hyperbolic curve is represented by the parameter $r\in[s_1,s_2]$:
\beqq
e_0+\bigg(e_1 r+\frac{e_{-1}}{r}\bigg)\cos(\tht_1)+i\bigg(e_1 r-\frac{e_{-1}}{r}\bigg)\sin(\tht_1).
\eeqq

Suppose $x=e_0+\bigg(e_1 r+\frac{e_{-1}}{r}\bigg)\cos(\tht_1)$ and $y=\bigg(e_1 r-\frac{e_{-1}}{r}\bigg)\sin(\tht_1)$, by applying $x$ and $y$, one would eliminate the variable $r$. That is,
\beq\label{equ2021-10-25-5}
\bigg(\frac{x-e_0}{\cos(\tht_1)}\bigg)^2=\bigg(e_1 r+\frac{e_{-1}}{r}\bigg)^2
=e_1^2 r^2+\frac{e_{-1}^2}{r^2}+2e_1e_{-1}
\eeq
and
\beq\label{equ2021-10-25-6}
\bigg(\frac{y}{\sin(\tht_1)}\bigg)^2=\bigg(e_1 r-\frac{e_{-1}}{r}\bigg)^2
=e_1^2 r^2+\frac{e_{-1}^2}{r^2}-2e_1e_{-1}.
\eeq
So, substracting \eqref{equ2021-10-25-6} from \eqref{equ2021-10-25-5}, one has
\beqq
\bigg(\frac{x-e_0}{\cos(\tht_1)}\bigg)^2-\bigg(\frac{y}{\sin(\tht_1)}\bigg)^2=4e_1e_{-1}.
\eeqq
Recall the constants $e_1=0.84$, $e_0=-0.18$, $e_{-1}=0.6$, $\tht_1=0.4\pi$. One has
\beq\label{equ2021-10-26-2}
e_0+2\sqrt{e_1e_{-1}}\cos\tht_1=-0.18+2\sqrt{0.84\times0.6}\cos(0.4\pi)
\approx0.258761<1
\eeq

It follows from the assumption of this claim and Sublemma \ref{arc-2021-10-26-1}, the subarc $\partial E_{s_2}\cap Z(s_1,s_2,\tht_1)$ is contained in
$\cd\big(   \cot(\al)i,\tfrac{1}{\sin(\al)}\big)$.

By the geometric description, $Z(s_1,s_2,\tht_1)$ is the region bounded by $\{\zeta:\ \text{Re}\,\zeta=e_0\}$, $\partial E_{s_1}$, $\partial E_{s_2}$, and the upper right part of a hyperbola
\beqq
\bigg(\frac{x-e_0}{\cos(\tht_1)}\bigg)^2-\bigg(\frac{y}{\sin(\tht_1)}\bigg)^2=4e_1e_{-1},\ x\geq e_0+2\sqrt{e_1e_{-1}}\cos\tht_1\ \mbox{and}\ y\geq0,
\eeqq
which is concave. So, the region $Z(s_1,s_2,\tht_1)$ is contained in the convex hull of $(\partial E_{s_2}\cap Z(s_1,s_2,\tht_1))\cup[e_{0},e_{0}+2\sqrt{e_1e_{-1}}\cos\tht_1]$. This, together with \eqref{equ2021-10-26-2}, implies that this convex hull is contained in $\cd\big(   \cot(\al)i,\tfrac{1}{\sin(\al)}\big)$. This completes the proof of the claim.

Applying this claim to $s_1=1.4$ and $s_2=1.54$, one needs to check the following inequalities:
\begin{align*}
&|\zeta(1.54e^{i0.4\pi})-\cot(\al)i|^2\\
=&\bigg(e_0+\bigg(e_1 r+\frac{e_{-1}}{r}\bigg)\cos(\tht_1)\bigg)^2+\bigg(\bigg(e_1 r-\frac{e_{-1}}{r}\bigg)\sin(\tht_1)-\cot(\al)\bigg)^2\\
=&\bigg(-0.18+\bigg(0.84\times 1.54+\frac{0.6}{1.54}\bigg)\cos(0.4\pi)\bigg)^2+\bigg(\bigg(0.84\times 1.54-\frac{0.6}{1.54}\bigg)\sin(0.4\pi)-\cot(\al)\bigg)^2\\
=&\bigg(-0.18+\bigg(0.84\times 1.54+\frac{0.6}{1.54}\bigg)\cos(0.4\pi)\bigg)^2
+\bigg(\bigg(0.84\times 1.54-\frac{0.6}{1.54}\bigg)\sin(0.4\pi)\bigg)^2\\
&-2\bigg(\bigg(0.84\times 1.54-\frac{0.6}{1.54}\bigg)\sin(0.4\pi)\bigg)\cot(\al)+
\cot^2(\al)\\
\approx&0.854857-1.71949\cot(\al)+
\cot^2(\al)\\
<&\bigg(\frac{1}{\sin(\al)}\bigg)^2
\end{align*}
and
\begin{align*}
&|\zeta(1.54e^{i\pi/2})-\cot(\al)i|^2\\
=&\bigg(e_0+\bigg(e_1 r+\frac{e_{-1}}{r}\bigg)\cos(\pi/2)\bigg)^2+\bigg(\bigg(e_1 r-\frac{e_{-1}}{r}\bigg)\sin(\pi/2)-\cot(\al)\bigg)^2\\
=&\bigg(-0.18+\bigg(0.84\times 1.54+\frac{0.6}{1.54}\bigg)\cos(\pi/2)\bigg)^2+\bigg(\bigg(0.84\times 1.54-\frac{0.6}{1.54}\bigg)\sin(\pi/2)-\cot(\al)\bigg)^2\\
=&\bigg(-0.18+\bigg(0.84\times 1.54+\frac{0.6}{1.54}\bigg)\cos(\pi/2)\bigg)^2
+\bigg(\bigg(0.84\times 1.54-\frac{0.6}{1.54}\bigg)\sin(\pi/2)\bigg)^2\\
&-2\bigg(\bigg(0.84\times 1.54-\frac{0.6}{1.54}\bigg)\sin(\pi/2)\bigg)\cot(\al)+
\cot^2(\al)\\
\approx&0.849597-1.80798\cot(\al)+
\cot^2(\al)\\
<&\bigg(\frac{1}{\sin(\al)}\bigg)^2.
\end{align*}
Hence, in Case (c), $Z(1.4,1.54,0.4\pi)$ is contained in $\cd\big(\cot(\al)i,\tfrac{1}{\sin(\al)}\big)$. This completes the proof.

\end{proof}

\subsection{Riemann surface}\label{riemann-surface-2021-12-29-1}

\begin{definition} Assume $m\geq5$.
Denote by
$$Y_{j\pm}:=(Q|_{\overline{\mathcal{U}^{Q+}_{j\pm}}})^{-1}(\pi_X(X_{j\pm}))=\overline{\mathcal{U}^{Q+}_{j\pm}}\bigcap Q^{-1}(\pi_X(X_{j\pm})),\ j=1,2.$$
Let
\beqq
Y=Y_{1+}\cup Y_{1-}\cup Y_{2+}\cup Y_{2-},
\eeqq
which is a subset of $\mathcal{U}^{Q+}_1\cup\mathcal{U}^{Q+}_2\cup\rr_{-}\subset\cc$. Define $\widetilde{Q}:Y\to X$ (whose well-definiedness is to be verified) by
\beq\label{paraequ-18}
\widetilde{Q}(\zeta)=(\pi|_{X_{j\pm}})^{-1}(Q(\zeta))\in X_{j\pm}\ \text{for}\ \zeta\in Y_{j\pm}.
\eeq
Also, we define
\beqq
\widetilde{Y}=\cc\setminus(E_{r_1}\cup\rr_{+}\cup\overline{\vv}(u_*,\tfrac{\pi}{5})),
\eeqq
where $r_1=1.4$ is introduced in Lemma \ref{paraequ-83}, and $u_*=(4e-1)m$ is specified in Lemma \ref{near-equati-22}.
\end{definition}

\begin{lemma} Assume $m\geq5$.
By the above construction, one has $Y\subset\widetilde{Y}$.
\end{lemma}

\begin{proof}
For $\zeta\in Y$, $\zeta\in Y_{1\pm}$ or $\zeta\in Y_{2\pm}$. If $\zeta\in Y_{1\pm}\subset \overline{\mathcal{U}^{Q+}_{1\pm}}$, by the definition, $|Q(\zeta)|>\rho$; if $\zeta\in Y_{2\pm}\subset \overline{\mathcal{U}^{Q+}_{1\pm}}\cup \overline{\mathcal{U}^{Q+}_{2\pm}}$, then $\rho<|Q(\zeta)|<R$. So, it follows from Lemma \ref{paraequ-15} (d) that $\zeta\in\cc\setminus E_{r_1}$.

Since $\pi_X(X)\cap \overline{\vv}\big(cv_Q,\tfrac{\pi}{5}\big)=\emptyset$, one has $Q(\zeta)\not\in \overline{\vv}\big(cv_Q,\tfrac{\pi}{5}\big)$. It follows from Lemma \ref{near-equati-22} that $Q(\zeta)\not\in \overline{\vv}\big(u_*,\tfrac{\pi}{5}\big)$, where $u_*$ is specified in Lemma \ref{near-equati-22}.

Finally, from $Q((1,+\infty))=[cv_Q,+\infty)\subset \overline{\vv}\big(cv_Q,\tfrac{\pi}{5}\big)$ by Lemma \ref{paraequ-25}, and $Y\subset\cc\setminus\ol{\cd}$, it follows that $\zeta\not\in\rr_{+}$. Hence, $\zeta\in\widetilde{Y}$.
\end{proof}

\begin{proof} (Proof of Proposition \ref{paraequ-14})

(a) By the discussions in Subsections \ref{nearequ-20} and \ref{near629-equ-1}, $Q$
is an isomorphism between $\mathcal{U}^{Q+}_{j\pm}$ and $\{z:\ \text{Im}\,z>0\}$, implying that $\widetilde{Q}$ is an isomorphism between $Y_{j\pm}$ and $X_{j\pm}$, $j=1,2$. So, if the definition of $\widetilde{Q}$ restricted to the boundaries is meaningful, then the map $\widetilde{Q}$ is well-defined. So, we only need to verify the consistence along the boundaries of $X_{j\pm}$.

By the discussions in Subsection \ref{near629-equ-1},
\beqq
\overline{\mathcal{U}^{Q+}_{1+}}\cap \overline{\mathcal{U}^{Q+}_{2+}}=\{cp_{Q1},-1\},\ \overline{\mathcal{U}^{Q+}_{1-}}\cap \overline{\mathcal{U}^{Q+}_{2-}}=\{cp_{Q1},-1\},
\eeqq
\beqq
\overline{\mathcal{U}^{Q+}_{1+}}\cap \overline{\mathcal{U}^{Q+}_{1-}}=\ol{\ga^{Q+}_{a1}}\cup \ol{\ga^{Q+}_{c1}},\ \overline{\mathcal{U}^{Q+}_{1+}}\cap \overline{\mathcal{U}^{Q+}_{2-}}=\ol{\ga^{Q+}_{b1}},\ \overline{\mathcal{U}^{Q+}_{1-}}\cap \overline{\mathcal{U}^{Q+}_{2+}}=\ol{\ga^{Q+}_{b2}},\ \overline{\mathcal{U}^{Q+}_{2+}}\cap \overline{\mathcal{U}^{Q+}_{2-}}=\ol{\ga^{Q+}_{a2}}.
\eeqq

First, $X_{1+}$ and $X_{1-}$ are glued along the negative real axis $\Ga_c$. By the inverse image of $Q$ in Figure \ref{illustration-inverse-Q},
along the negative side of the real axis, $Y_{1+}$ and $Y_{1-}$ might intersect, this is consistent for $Y_{1+}$ and $Y_{1-}$, implying that $\widetilde{Q}$ is continuous.

Second, $X_{1+}$ and $X_{2-}$ are glued along $(\rho,cv_Q)$. And, $Y_{1+}$ and $Y_{2-}$ might intersect along $\ga_{b1}$. So, the gluing there is also consistent. Similarly, the same arguments work for that $X_{2+}$ and $X_{1-}$ are glued along $(\rho,cv_Q)$.

Hence, $\widetilde{Q}: Y\to X$ is an isomorphisim.

By the construction above, $\pi_X\circ\widetilde{Q}=Q$ on $Y$. If $z\in X$ with $|\pi_X(z)|>R$, then $z\in X_{1+}\cup X_{1-}$, so $\widetilde{Q}^{-1}(z)\in Y_{1+}\cup Y_{1-}\subset\mathcal{U}^{Q+}_{1+}\cup\mathcal{U}^{Q+}_{1-}$. If $\pi_X(z)\to\infty$, then $\widetilde{Q}^{-1}(z)$ is in the inverse branch of $Q$ near $\infty$, this, together with Lemma \ref{paraequ-12}, implies that $\widetilde{Q}^{-1}(z)$ has the asymptotic expansion $\pi_X(z)-(4m+2)+o(1)$.

(b) We will show that $\varphi:\widetilde{Y}\subset\cc\to \cc$ can be lifted to  $\widetilde{\varphi}:\widetilde{Y}\to X$, which is well-defined and holomorphic.

Take the following constants $e_1 = 0.84$, $e_{0}=-0.18$, $e_{-1} = 0.6$, one has,
\beq\label{paraequ-30-1-21-2}
\varphi_{1,max}(7)\leq(0.84+0.6)\sqrt{-\log\bigg(1-\bigg(\frac{0.86+0.6}{7-0.18}\bigg)^2\bigg)}
\approx0.30752.
\eeq

If $|\zeta|\geq7$, then, by Lemma \ref{paraequ-11},
\beq\label{near-equat-7}
|\varphi(\zeta)-\zeta|\leq |e_{0}|+2|e_1|+\varphi_{1,max}(7)=0.18+2\times 0.84+0.30752\approx2.16752<3;
\eeq
if $\zeta\in\cc\setminus E$ and $|\zeta|\leq7$, then
\beqq
|\varphi(\zeta)|\leq|\zeta|+|e_{0}|+2|e_1|+\varphi_{1,max}(7)<7+0.18+2\times 0.84+0.30752\approx7+2.16752<10,
\eeqq
where the last term is given by the maximum modulus theorem, the fact that the set $\{\zeta\in\cc\setminus E_1:\ |\zeta|<7\}$ is surrounded by the curve $\{\zeta:\ |\zeta|=7\}$, and $\varphi$ is univalent. Therefore, if $\zeta\in\cc\setminus \overline{\vv}(u_*,\tfrac{\pi}{5})$, then $\varphi(\zeta)$ cannot be in $\overline{\vv}(cv_Q,\tfrac{\pi}{5})$, since the distance between $\partial \vv(u_*,\tfrac{\pi}{5})$ and $\overline{\vv}(cv_Q,\tfrac{\pi}{5})$ is bigger than $
(\tfrac{7776}{3125}\cdot4(m+1)-(4e-1)m)\sin(\tfrac{\pi}{5})\approx(9.95328(m+1)-9.87313m)\sin(\tfrac{\pi}{5})
>9.9\sin(\tfrac{\pi}{5})\approx5.81907>3$ by Remark \ref{near-equat-21}.

Take $\zeta\in\widetilde{Y}$, it follows from Lemma \ref{nearequ-24} that $|\arg\tfrac{\varphi(\zeta)}{\zeta}|<\pi$ and $|\varphi(\zeta)|>\rho$. Define $\widetilde{\varphi}(\zeta)\in X$ as follows
\begin{itemize}
\item $\pi_X(\widetilde{\varphi}(\zeta))=\varphi(\zeta)$;
\item $\widetilde{\varphi}(\zeta)\in X_{1+}\cup X_{2-}$\quad\quad if $\text{Im}\,\zeta\geq0$ and $-\pi<\arg\tfrac{\varphi(\zeta)}{\zeta}\leq0$;
\item $\widetilde{\varphi}(\zeta)\in X_{1-}\cup X_{2+}$\quad\quad if $\text{Im}\,\zeta\leq0$ and $0\leq\arg\tfrac{\varphi(\zeta)}{\zeta}<\pi$;
\item $\widetilde{\varphi}(\zeta)\in X_{1+}\cup X_{1-}$\quad\quad otherwise.
\end{itemize}
Now, we verify this extension is reasonable. First, we show that for $\widetilde{\varphi}(\zeta)\in X_{2\pm}$, $|\varphi(\zeta)|<R$. For the first case of
$\widetilde{\varphi}(\zeta)\in X_{1+}\cup X_{2-}$, it is evident that $\varphi(\zeta)$ is in the half plane $H=\{w:\ \arg\,\zeta-\pi<\arg\,w<\arg\,\zeta\}$ and not in $\ol{\vv}(u_*,\tfrac{\pi}{5})$, and the distance between $\zeta$ and $H\setminus\ol{\cd}(0,R)\cup\ol(u_*,\tfrac{\pi}{5})$ is bounded blow by the distance between $\partial\cd(0,R)\cup\partial\vv(u_*,\tfrac{\pi}{5})$ and the real axis, which is greater than $(R-u_*)\sin(\tfrac{\pi}{5})=(2.66\times 10^m- (4e-1)m)\sin(\tfrac{\pi}{5})>3,$ where the details of the computation can be found in \eqref{near-equat-8}.

Next,  we prove the continuity of $\widetilde{\varphi}$. It is sufficient to study the cases that $\text{Im}\,\zeta=0$ or $\arg\,\tfrac{\varphi(\zeta)}{\zeta}=0$. If $\zeta\in\widetilde{Y}$ and $\text{Im}\,\zeta=0$, then $\zeta\in\rr_{-}$, then $\widetilde{\varphi}(\zeta)\in X_{1+}\cup X_{1-}$. If $\text{Im}\,\zeta\neq0$ and $\arg\,\tfrac{\varphi(\zeta)}{\zeta}=0$, then $\widetilde{\varphi}(\zeta)\in X_{1+}\cup X_{1-}$. Hence, for these cases, $\widetilde{\varphi}(\zeta)$ should be in $X_{1+}\cup X_{1-}$, and this does not cause any discontinuity. Therefore, this extension of $\varphi$ is holomorphic.
\end{proof}

\subsection{Estimates on $F$}

\begin{lemma}\label{paraequ-24}
For the critical value $cp_{Q1}:=(2m+1)+2\sqrt{m+m^2}=(\sqrt{m+1}+\sqrt{m})^2$, one has
\beqq
cp_{Q1}\geq 4m+2\sqrt{6}-3\approx 4m+1.89898\ \forall m\geq2.
\eeqq
\end{lemma}

\begin{proof}
Suppose $(2m+1)+2\sqrt{m+m^2}\geq4m+a$ for some $a>0$. This is equivalent to
$\sqrt{m+m^2}\geq m+\tfrac{a-1}{2}$. It is evident that $a\geq1$. So, assume that $\sqrt{m+m^2}\geq m+b$. That is, $m+m^2\geq(m+b)^2$. So, $(1-2b)m\geq b^2$ and $b\leq1/2$. Since $m\geq2$, one has $2(1-2b)\geq b^2$ and $b\leq 1/2$. Hence, the maximum value for $b$ is $\sqrt{6}-2\approx0.44949$.
\end{proof}

\begin{lemma}\label{paraequ-55}
Take the constants $e_1 = 0.84$, $e_{0}=-0.18$, and $e_{-1} = 0.6$.
Suppose that $r>cp_{Q1}:=(2m+1)+2\sqrt{m+m^2}=(\sqrt{m+1}+\sqrt{m})^2>4m+1$, $m\geq3$, $\tht\in[-\tfrac{\pi}{2},\frac{\pi}{2}]$, $\text{Re}\,(\zeta e^{-i\tht})>r$, and $\varphi$ is introduced in Lemma \ref{paraequ-11}. Then the following estimates hold for $z=\varphi(\zeta)$:
\begin{itemize}
\item[(a)] $F(z)-z\in\cd(\al(r,\tht),\be_{max}(r))$, where
\beqq
\al(r,\tht)=(4m+2)+e_0+
\frac{(8m(m+1)+1)e^{-i\tht}}{2r}
\eeqq
and
\beqq
\be_{max}(r)=2|e_1|+\tfrac{8m(m+1)+1}{2r}+Q_{2,max}(r)+\varphi_{1,max}(r),
\eeqq
$Q_{2,max}(r)$ is given in Lemma \ref{paraequ-12}, and $\varphi_{1,max}(r)$ is defined in Lemma \ref{paraequ-11} (e);
\item[(b)]
\beqq
\text{Arg}\,\De F_{min}(r,\tht)\leq\arg\,(F(z)-z)\leq\text{Arg}\,\De F_{max}(r,\tht),
\eeqq
where
\begin{align*}
&\text{Arg}\,\De F_{\big\{\substack{max\\ min}\big\}}(r,\tht)=-\arctan\bigg(\frac{\tfrac{(8m(m+1)+1)\sin(\tht)}{2r}}{(4m+2)+e_0+
\tfrac{(8m(m+1)+1)\cos(\tht)}{2r}}\bigg)\\
&\pm\arcsin\bigg(\frac{\be_{max}(r)}{\sqrt{((4m+2)+e_0)^2+
\big(\tfrac{8m(m+1)+1}{2r}\big)^2+2((4m+2)+e_0)\tfrac{(8m(m+1)+1)\cos\tht}{2r}}}\bigg)\\
=&\arctan\bigg(\frac{\text{Im}\,\al(r,\tht)}{\text{Re}\,\al(r,\tht)}\bigg)\pm\arcsin\bigg(\frac{\be_{max}(r)}{|\al(r,\tht)|}\bigg),
\end{align*}
$\al(r,\tht)$ is a complex number and $|\al(r,\tht)|$ is the norm of $\al(r,\tht)$;
\item[(c)] one has
\beqq
\text{Abs}\,\De F_{min}(r,\tht)\leq|F(z)-z|\leq\text{Abs}\,\De F_{max}(r,\tht),
\eeqq
where
\begin{align*}
&\text{Abs}\,\De F_{\big\{\substack{max\\ min}\big\}}(r,\tht)\\
=&\sqrt{((4m+2)+e_0)^2+
\big(\tfrac{8m(m+1)+1}{2r}\big)^2+2((4m+2)+e_0)\tfrac{(8m(m+1)+1)\cos\tht}{2r}}\pm\be_{max}(r)\\
=&|\al(r,\tht)|\pm\be_{max}(r);
\end{align*}
\item[(d)] one has
\beqq
|\log F'(z)|\leq\text{Log}DF_{max}(r):=\text{Log}DQ_{max}(r)+\text{Log}D\varphi_{max}(r),
\eeqq
where $\text{Log}\,DQ_{max}(r)$ is introduced in Lemma \ref{paraequ-56}, and $\text{Log}\,D\varphi_{max}(r)$ is specified in Lemma \ref{paraequ-11} (f).
\end{itemize}
\end{lemma}

\begin{proof}
It is evident that
\begin{align*}
&\al(r,\tht)=(4m+2)+
\frac{(8m(m+1)+1)e^{-i\tht}}{2r}+e_0\\
=&(4m+2)+e_0+\frac{(8m(m+1)+1)\cos(\tht)}{2r}-i\frac{(8m(m+1)+1)\sin(\tht)}{2r},
\end{align*}
\begin{align*}
\text{Re}\,\al(r,\tht)=(4m+2)+e_0+\frac{(8m(m+1)+1)\cos(\tht)}{2r},
\end{align*}
\begin{align*}
\text{Im}\,\al(r,\tht)=-\frac{(8m(m+1)+1)\sin(\tht)}{2r},
\end{align*}
and
\begin{align*}
&|\al(r,\tht)|=\sqrt{(\text{Re}\,\al(r,\tht))^2+(\text{Im}\,\al(r,\tht))^2}\\
=&\sqrt{\big((4m+2)+e_0+\tfrac{(8m(m+1)+1)\cos(\tht)}{2r}\big)^2+\big(-\tfrac{(8m(m+1)+1)\sin(\tht)}{2r}\big)^2}\\
=&\sqrt{((4m+2)+e_0)^2+
\big(\tfrac{8m(m+1)+1}{2r}\big)^2+2((4m+2)+e_0)\tfrac{(8m(m+1)+1)\cos\tht}{2r}}.
\end{align*}

(a) For $z=\varphi(\zeta)$, by \eqref{paraequ-9}, $\varphi(\zeta)=\zeta+c_0+\varphi_1(\zeta)$, by \eqref{paraequ-66}, one has
\begin{align}\label{paraequ-57}
&F(z)-z=Q(\zeta)-\varphi(\zeta)\\
=&\bigg(\zeta+(4m+2)+
\frac{8m(m+1)+1}{\zeta}+Q_2(\zeta)\bigg)-(\zeta+c_0+\varphi_1(\zeta))\nonumber\\
=&(4m+2)+
\frac{8m(m+1)+1}{\zeta}
+Q_2(\zeta)-c_0-\varphi_1(\zeta)\nonumber\\
=&\bigg((4m+2)+e_0+\frac{(8m(m+1)+1)e^{-i\tht}}{2r}\bigg)\nonumber\\
&+\bigg(\frac{8m(m+1)+1}{\zeta}-\frac{(8m(m+1)+1)e^{-i\tht}}{2r}+Q_2(\zeta)-(c_0+e_0)-\varphi_1(\zeta)\bigg)\nonumber\\
=&\al(r,\tht)+\be(\zeta,r,\tht).
\end{align}

By Lemma \ref{paraequ-10} (a), $|\tfrac{1}{\zeta}-\tfrac{e^{-i\tht}}{2r}|\leq\tfrac{1}{2r}$. So, by Lemma \ref{paraequ-11} (a) and Lemma \ref{paraequ-12},
\begin{align*}
&|\be(\zeta,r,\tht)|\\
=&\bigg|\frac{8m(m+1)+1}{\zeta}-\frac{(8m(m+1)+1)e^{-i\tht}}{2r}+Q_2(\zeta)-(c_0+e_0)-\varphi_1(\zeta)\bigg|\\
\leq&\bigg|\frac{8m(m+1)+1}{\zeta}-\frac{(8m(m+1)+1)e^{-i\tht}}{2r}\bigg|+|Q_2(\zeta)|+2|e_1|+|\varphi_1(\zeta)|\\
\leq&2|e_1|+\tfrac{8m(m+1)+1}{2r}+Q_{2,max}(r)+\varphi_{1,max}(r)=\be_{max}(r).
\end{align*}
So, this number is contained in a disk with center $\al(r,\tht)$ and radius $\be_{\max}(r)$.
This shows (a).

Next, we show that for $r\geq cp_{Q1}$, $m\geq3$, and $\tht\in[-\tfrac{\pi}{2},\frac{\pi}{2}]$, one has
\beq\label{paraequ-58}
\be_{\max}(r)<|\al(r,\tht)|.
\eeq

It follows from $r> 4m+1$ and $m\geq3$ that
\beqq
\frac{8m(m+1)+1}{r}\leq\frac{8m(m+1)+1}{4m+1}<2.6m.
\eeqq

Furthermore, one has
\begin{itemize}
\item by \eqref{paraequ-20}, if $a=\tfrac{2}{3}m$ and $r>4m+1$, then
\begin{align*}
\frac{8C^3_{2m}}{r(r-1)}<\frac{2}{3}m;
\end{align*}
\item by \eqref{paraequ-21}, if $a=\tfrac{1}{m}$ and $r>4m+1$, then
\begin{align*}
\frac{16m}{r(r-1)}<\frac{1}{m};
\end{align*}
\item by \eqref{paraequ-22}, if $a=\tfrac{m}{2}$ and $r>4m+1$, then
\begin{align*}
\frac{2(2m)(2m-1)(2m-2)(2m-3)}{3r(r-1)^2}\bigg(\frac{r+1}{r-1}\bigg)^{2m-4}
<\frac{m}{2};
\end{align*}
\item by \eqref{paraequ-23}, if $a=6$ and $r>4m+1$, then
\begin{align*}
\frac{8(2m)(2m-1)}{(r-1)^2}\bigg(\frac{r+1}{r-1}\bigg)^{2m-2}
<6.
\end{align*}
\end{itemize}

Hence, combining the above estimates, one has
\beq\label{paraequ-28}
|Q_2(\zeta)|\leq Q_{2,max}(r)\leq \frac{2}{3}m+\frac{1}{m}+\frac{m}{2}+6=\frac{7}{6}m+\frac{1}{m}+6,\ |\zeta|=r\geq 4m+1.
\eeq

 Take the following constants $e_1=0.84$, $e_{0}=-0.18$, $e_{-1}=0.6$.

So, by \eqref{paraequ-30-1-21-1}, the radius is
\begin{align}\label{paraequ-54}
&2|e_1|+\tfrac{8m(m+1)+1}{2r}+Q_{2,max}(r)+\varphi_{1,max}(r)\nonumber\\
\leq& 2\times0.84+1.3m+\bigg(\frac{2}{3}m+\frac{1}{m}+\frac{m}{2}+6\bigg)+0.214541\nonumber\\
<&\frac{37}{15}m+8+\frac{1}{m}\leq(4m+2)+e_0<|\al(r,\tht)|,
\end{align}
where $m\geq5$ and $\tht\in[-\tfrac{\pi}{2},\tfrac{\pi}{2}]$ are used.

Now, we study the case $m=4$.
For $m=4$ and $r=(2m+1)+2\sqrt{m+m^2}
=(2\times4+1)+2\sqrt{4+4^2}=9+4\sqrt{5}$, direct computation gives
\beqq
\frac{8m(m+1)+1}{r}=\frac{8\times4\times5+1}{9+4\sqrt{5}}\approx 8.97222,
\eeqq
\begin{align*}
\frac{8C^3_{2m}}{r(r-1)}
=\frac{8\times\tfrac{8\times7\times6}{3\times2\times1}}{(9+4\sqrt{5})(8+4\sqrt{5})}\approx 1.47343,
\end{align*}
\begin{align*}
\frac{16m}{r(r-1)}=\frac{16\times4}{(9+4\sqrt{5})(8+4\sqrt{5})}\approx 0.21049,
\end{align*}
\begin{align*}
&\frac{2(2m)(2m-1)(2m-2)(2m-3)}{3r(r-1)^2}\bigg(\frac{r+1}{r-1}\bigg)^{2m-4}\\
=&\frac{2(2\times4)(2\times4-1)(2\times4-2)(2\times4-3)}{3(9+4\sqrt{5})(8+4\sqrt{5})^2}\bigg(\frac{10+4\sqrt{5}}{8+4\sqrt{5}}\bigg)^{4}
\approx 0.339677,
\end{align*}
\begin{align*}
&\frac{8(2m)(2m-1)}{(r-1)^2}\bigg(\frac{r+1}{r-1}\bigg)^{2m-2}\\
=&\frac{8(2\times4)(2\times4-1)}{(8+4\sqrt{5})^2}\bigg(\frac{10+4\sqrt{5}}{8+4\sqrt{5}}\bigg)^{6}
\approx3.04763.
\end{align*}
So, one has
\begin{align*}
&\be_{max}(9+4\sqrt{5})\\
\leq&2\times0.84+8.97222/2+1.47343+0.21049+ 0.339677+3.04763+0.214541\\
\approx&11.4519<17.82=(4\times4+2)-0.18<|\al(9+4\sqrt{5},\tht)|.
\end{align*}

The left case is $m=3$. For $m=3$ and $r=(2m+1)+2\sqrt{m+m^2}=(2\times3+1)+2\sqrt{3+3^2}=7+4\sqrt{3}$, direct computation gives
\beqq
\frac{8m(m+1)+1}{r}=\frac{8\times3\times4+1}{7+4\sqrt{3}}\approx6.96429,
\eeqq
\begin{align*}
\frac{8C^3_{2m}}{r(r-1)}
=\frac{8\times\tfrac{6\times5\times4}{3\times2\times1}}{(7+4\sqrt{3})(6+4\sqrt{3})}\approx0.88856,
\end{align*}
\begin{align*}
\frac{16m}{r(r-1)}=\frac{16\times3}{(7+4\sqrt{3})(6+4\sqrt{3})}\approx0.266568,
\end{align*}
\begin{align*}
&\frac{2(2m)(2m-1)(2m-2)(2m-3)}{3r(r-1)^2}\bigg(\frac{r+1}{r-1}\bigg)^{2m-4}\\
=&\frac{2(2\times3)(2\times3-1)(2\times3-2)(2\times3-3)}{3(7+4\sqrt{3})(6+4\sqrt{3})^2}\bigg(\frac{8+4\sqrt{3}}{6+4\sqrt{3}}\bigg)^{2}
\approx0.137461,
\end{align*}
\begin{align*}
&\frac{8(2m)(2m-1)}{(r-1)^2}\bigg(\frac{r+1}{r-1}\bigg)^{2m-2}\\
=&\frac{8(2\times3)(2\times3-1)}{(6+4\sqrt{3})^2}\bigg(\frac{8+4\sqrt{3}}{6+4\sqrt{3}}\bigg)^{4}
\approx2.55277.
\end{align*}
So, one has
\begin{align*}
&\be_{max}(7+4\sqrt{3})\\
\leq&2\times0.84+6.96429/2+0.88856+0.266568+0.137461+2.55277+0.214541\\
\approx&9.22205<13.82=(3\times4+2)-0.18<|\al(7+4\sqrt{3},\tht)|.
\end{align*}

(b) It follows from \eqref{paraequ-57}, \eqref{paraequ-58}, and Lemma \ref{paraequ-13} (a) that
\begin{align*}
&|\arg\,(F(z)-z)-\arg\,\al(r,\tht)|\leq\bigg|\arg\,\bigg(1+\frac{\be(\zeta,r,\tht)}{\al(r,\tht)}\bigg)\bigg|\\
\leq&\arcsin\bigg|\frac{\be(\zeta,r,\tht)}{\al(r,\tht)}\bigg|\leq\arcsin\bigg(\frac{\be_{max}(r)}{|\al(r,\tht)|}\bigg).
\end{align*}
This, together with $\arg\,\al(r,\tht)=\arctan(\tfrac{\text{Im}\,\al(r,\tht)}{\text{Re}\,\al(r,\tht)})$, yields the conclusion of (b).

(c) By \eqref{paraequ-57}, one has
\begin{align*}
|\al(r,\tht)|-\be_{max}(r)\leq|F(z)-z|\leq|\al(r,\tht)|+\be_{max}(r).
\end{align*}

(d) Since $z=\varphi(\zeta)$, $F(z)=F\circ\varphi(\zeta)=Q(\zeta)$,  and $F'(z)\cdot\varphi'(\zeta)=Q'(\zeta)$, this is derived by Lemma \ref{paraequ-56} and Lemma \ref{paraequ-11} (f).
\end{proof}

\begin{remark}
In Lemma \ref{paraequ-55}, $\tht\in[-\tfrac{\pi}{2},\tfrac{\pi}{2}]$ can be replaced by
$\tht\in\mathbb{R}$,  the corresponding statements also hold but one should assume that $m$ is sufficiently large because of \eqref{paraequ-54}.
\end{remark}

\subsection{Repelling Fatou coordinate $\widetilde{\Phi}_{rep}$ on $X$}\label{repell-fatou-2021-12-29-2}

\begin{proof} (Proof of Proposition \ref{near-equat-2})
By the definitions of $g$, $\widetilde{\varphi}$ in Proposition \ref{paraequ-14}, and $\widetilde{Q}$ in \eqref{paraequ-18}, one has
\beqq
F\circ\pi_X\circ g=Q\circ \varphi^{-1}\circ\pi_X\circ\widetilde{\varphi}\circ\widetilde{Q}^{-1}=Q\circ\widetilde{Q}^{-1}=\pi_X.
\eeqq

So, $F$ has an inverse branch near $\infty$, denoted by $\ol{g}(z)=z-(4m+2-c_0)+o(1)$ as $z\to\infty$. It follows from Lemmas \ref{paraequ-13} and \ref{paraequ-11} that
\begin{align*}
\color{red}&|\arg(4m+2-c_0)|=\bigg|(4m+2)\arg\bigg(1+\frac{-c_0}{4m+2}\bigg)\bigg|
=\bigg|\arg\bigg(1+\frac{-c_0}{4m+2}\bigg)\bigg|\\
\leq&\arcsin\bigg(\frac{|e_0|+2|e_1|}{4m+2}\bigg)\leq\frac{\pi}{3}\cdot\frac{1.86}{4m+2}<\frac{\pi}{10}\ \forall m\geq2.
\end{align*}

For sufficiently large $L>0$, $\ol{g}$ exists and is injective in
\beqq
W=\cc\setminus\overline{\vv}(-L,\tfrac{\pi}{10}),
\eeqq
and $\bar{g}$ satisfies $|\arg(\bar{g}(z)-z)-\pi|<\tfrac{\pi}{10}$, so $\bar{g}(W)\subset W$, and
\beqq
\text{Re}\,\bar{g}(z)<\text{Re}\,z-(4m+2)+|e_0|+2|e_1|+1=\text{Re}\,z-(4m+2)+2.86=\text{Re}\,z-4m+0.86.
\eeqq
By Proposition \ref{paraequ-14} (a), the asymptotic formula for $\widetilde{Q}^{-1}(z)$, one has that $\pi_X\circ g\circ \pi^{-1}_X(z)\to\infty$ as $z\in\pi_X(X)$ and $z\to\infty$. Hence, $\pi_X\circ g\circ \pi^{-1}_X(z)$ and $\bar{g}(z)$ should coincide as the only inverse of $F$ near $\infty$. So, $\pi_X\circ g(z)=\bar{g}\circ \pi_X(z)$ for $\pi_X(z)$ sufficiently large.

By the discussions of the theory of Fatou coordinates, there is a Fatou coordinate $\Phi_{rep}(z)$, which is holomorphic and injective in $\{z:\ \text{Re}\,z<-L'\}$ for sufficiently large $L'>L$ and satisfies $\Phi_{rep}\circ\bar{g}(z)=\Phi_{rep}(z)-1$. The Fatou coordinate $\Phi_{rep}(z)$ can be extended to $W$, since $\Phi_{rep}(z)$ and $\bar{g}|W$ are injective.

Let $W'=\pi^{-1}_X(W\cap\pi_X(X))$, then $\pi_X|W'$ is injective if $L$ is large. Define
$\widetilde{\Phi}_{rep}\circ\pi_X$ on $W'$. It is evident that $\widetilde{\Phi}_{rep}\circ g(z)=\widetilde{\Phi}_{rep}(z)-1$ in $W'$.  By the following Lemma \ref{paraequ-19},  the Fatou coordinates $\widetilde{\Phi}_{rep}$ can be extended to $X$ as follows:
\beqq
\widetilde{\Phi}_{rep}(z)=\widetilde{\Phi}_{rep}(g^n(z))+n,
\eeqq
where $n$ is chosen such that $g^n(z)\in W'$. This is well-defined by the functional equation. By the injectivity of $\Phi_{rep}$ and $g$, $\widetilde{\Phi}_{rep}$ is injective. It is evident that for any $z\in X$, $\text{Re}\,\pi_X(g^n(z))\to-\infty$ as $n\to\infty$.
\end{proof}

\begin{lemma}\label{paraequ-19}
For any point $z\in X$, there exists a positive integer $n$ such that $g^n(z)\in W'$.
\end{lemma}

\begin{proof}
Pick a point $z_0\in W'$. Let $\partial W'$ be the boundary of $W'$ restricted to $X$. Then $\pi_X(\partial W')$ is a union of two finite segments. Since $X$ is isomorphic to $Y$ as a proper subdomain,  one has that $X$ is a hyperbolic Riemann surface. Since $\partial W'$ is relatively compact in $\cc\setminus\overline{\cd}(0,\rho)$, in which $g^n(z_0)$ tends to the boundary, the Poincar\'{e} distance $d_{\cc\setminus\overline{\cd}(0,\rho)}(g^n(z_0),\partial W')\to\infty$ as $n\to\infty$. The same estimate holds, where the the distance is the restriction of the Poincar\'{e} metric on $X$, denoted by $d_X$, where the Schwarz-Pick  theorem implies that the projection $\pi_X:X\to\cc\setminus\overline{\cd}(0,\rho)$ does not expand the Poincar\'{e} distance. So, for sufficiently large $n$, by applying the Schwarz-Pick theorem to $g^n$, one has
\beqq
d_X(g^n(z),g^n(z_0))\leq d_X(z,z_0)<d_X(g^n(z_0),\partial W').
\eeqq
So, $g^n(z)\in W'$ for these sufficiently large $n$.
\end{proof}

\subsection{Attracting Fatou coordinate $\Phi_{att}$ and the domains $D_1$ and $D^{\sharp}_1$}\label{attractingfatoucoor-2021-10-25-2}

\begin{definition}
For any $z\in\cc$, set
\beqq
pr_{+}(z):=\text{Re}(z e^{-i\pi/5})\ \text{and}\ pr_{-}(z):=\text{Re}(z e^{i\pi/5}),
\eeqq
which are the orthogonal projection to the line with angle $\pm\tfrac{\pi}{5}$ to the real axis.
\end{definition}

Set
\begin{align*}
H^{\pm}_1:&=\{z:\ pr_{\pm}(z)>u_1=(\sqrt{m+1}+\sqrt{m})^2+2.1\},\\
 H^{\pm}_2:&=\{z:\ pr_{\pm}(z)>u_2=cp_{Q1}=(\sqrt{m+1}+\sqrt{m})^2\},\\
H^{\pm}_3:&=\bigg\{z:\ pr_{\pm}(z)>u_3=pr_{+}(cv_Q)=\frac{4(m+1)^{m+1}}{m^{m}}\cos(\tfrac{\pi}{5})
=\frac{(1+\sqrt{5})(m+1)^{m+1}}{m^{m}}\bigg\}, \\
H^{\pm}_4:&=\bigg\{z:\ pr_{\pm}(z)>u_4=\frac{4(m+1)^{m+1}}{m^{m}}\cos(\tfrac{\pi}{5})-2.1=
\frac{(1+\sqrt{5})(m+1)^{m+1}}{m^{m}}-2.1\bigg\}.
\end{align*}

\begin{remark}\label{near-equat-1}
For any positive number $u$ and $\tht\in(0,\tfrac{\pi}{2})$,  the Euclidean coordinate $(x,y)$ is used to represent the points on the complex plane $\cc$, the boundary of $\{\zeta:\ \text{Re}(z e^{-i\tht})>u\}$ is given by the line equation
\beqq
y=\tan\left(\frac{\pi}{2}+\tht\right)(x-u\cos(\tht))+u\sin(\tht);
\eeqq
the boundary of $\{\zeta:\ \text{Re}(z e^{i\tht})>u\}$ is given by the line equation
\beqq
y=\tan\left(\frac{\pi}{2}-\tht\right)(x-u\cos(\tht))-u\sin(\tht).
\eeqq
\end{remark}

\begin{lemma} \cite[Page 11]{BorweinErdelyi1995}\label{paraequ-62}
Given any real-coefficient polynomial $P(x)$ of degree $n, n \geq 1,$ in the form of
\[
P(x)=a_{n}\left(x^{n}+a_{n-1} x^{n-1}+\cdots+a_{0}\right)
\]
the signs of both its value and the value of its derivative do not change in $(-\infty,-r) \cup(r, \infty),$ where
\[
r:=1+\max _{0 \leq k \leq n-1}\left|a_{k}\right|.
\]
\end{lemma}

\begin{lemma} (Attracting Fatou coordinate)\label{paraequ-63}
 Take the constants $e_1 = 0.84$, $e_{0}=-0.18$, and $e_{-1} = 0.6$.
 Suppose that $\varphi$ is specified in Lemma \ref{paraequ-11}. The following statements hold:
\begin{itemize}
 \item[(a)]  $H^{\pm}_1\subset\varphi(H^{\pm}_2)$, $H^{\pm}_3\subset\varphi(H^{\pm}_4)$. As a consequence, $F$ is defined on $H^{+}_1\cup H^{-}_1$.
 \item[(b)] $Q$ is injective in $H^{\pm}_2$, implying that $F$ is injective in $H^{\pm}_1$.
 \item[(c)] Assume $m\geq6$. If $z\in H^{\pm}_1$, then $|\arg\,(F(z)-z)|<\tfrac{3\pi}{10}$, hence $F(H^{\pm}_1)\subset H^{\pm}_1$. Therefore, the sector $H^{+}_1\cup H^{-}_1=\vv(u_0,\tfrac{7\pi}{10})$ is forward invariant under $F$ and contained in $Basin(\infty)$, where $u_0=\tfrac{u_1}{\cos(\tfrac{\pi}{5})}$ (see Remark \ref{near-equat-1}).
\item[(d)] An attracting Fatou coordinate $\Phi_{att}$ for $F$ exists in $\vv(u_0,\tfrac{7\pi}{10})$ and is injective in each of $H^{\pm}_1$.
\end{itemize}
\end{lemma}

\begin{remark}
In the following discussions, the Fatou coordinate $\Phi_{att}$ is normalized such that $\Phi_{att}(cv_Q)=1$.
\end{remark}

\begin{proof}

 Take the following constants $e_1 = 0.84$, $e_{0}=-0.18$, $e_{-1} = 0.6$.

By Lemma \ref{paraequ-11} (b) and $u_1>4m+1+2.1>4e_1=3.36$, $H^{\pm}_{1}$  is contained in the image of $\varphi$.

If $\zeta\in\partial H^{\pm}_{2}$, by $u_2=cp_{Q1}>|a_E(1)|+|e_0|=|e_1|+|e_{-1}|+|e_0|=0.84+0.18+0.6=1.08$, Lemma \ref{paraequ-11} (e), \eqref{paraequ-30-1-21-1}
\begin{align*}
&pr_{\pm}(\varphi(\zeta))=pr_{\pm}(\zeta)+ pr_{\pm}(-e_0)+ pr_{\pm}(c_0+e_{0})+ pr_{\pm}(\varphi_1(\zeta))\\
\leq&cp+|e_0|\cdot|\cos(\tfrac{\pi}{5})|+2|e_1|+\varphi_{1,max}(cp_{Q1})\\
\leq&cp+0.18\cdot\bigg(\frac{1}{4}(1+\sqrt{5})\bigg)+1.68+0.214541\\
\approx& cp+2.0401640589874903<(\sqrt{m+1}+\sqrt{m})^2+2.1,
\end{align*}
where $\zeta$ and $c_0+e_{0}$ might be complex numbers, and $-e_0$ is a real number,
implying that $\varphi(\zeta)\not\in H^{\pm}_{1}$. So, $\varphi^{-1}(H^{\pm}_1)$ is contained in one side of $\partial H^{\pm}_2$. Furthermore, for any point $\zeta\in H^{\pm}_{2}$ far from the boundary of $H^{\pm}_{2}$, it is evident that $\varphi(\zeta)\in H^{\pm}_1$. So, $\varphi^{-1}(H^{\pm}_1)$ is contained in $H^{\pm}_2$. Hence, $H^{\pm}_1\subset\varphi(H^{\pm}_2)$.

If $\zeta\in\partial H^{\pm}_{4}$ with $|u|>cp>|a_E(1)|+|x_E|$, then
\begin{align*}
&pr_{\pm}(\varphi(\zeta))=pr_{\pm}(\zeta)+ pr_{\pm}(-e_0)+ pr_{\pm}(c_0+e_0)+ pr_{\pm}(\varphi_1(\zeta))\\
\leq&u_4+|e_0|\cdot|\cos(\tht)|+|c_0+e_0|+\varphi_{1,max}(u_4)\\
\leq& u_4+0.18\cdot\bigg(\frac{1}{4}(1+\sqrt{5})\bigg)+1.68+0.214541\approx u_4+2.0401640589874903\\
<&2.1+\frac{(1+\sqrt{5})(m+1)^{m+1}}{m^{m}}-2.1
=\frac{(1+\sqrt{5})(m+1)^{m+1}}{m^{m}}.
\end{align*}

(c) For $z\in H^{\pm}_1$, then $\varphi^{-1}(z)\in H^{\pm}_2$ by (a). By Lemma \ref{paraequ-55} (b), we will show that
\begin{align*}
|\arg(F(z)-z)|\leq\max\{\text{Arg}\De F_{max}(cp_{Q},\pm\tfrac{\pi}{5}),\ -\text{Arg}\De F_{min}(cp_{Q},\pm\tfrac{\pi}{5})\}
\end{align*}

For convenience, set
\begin{align*}
Part(i):=\frac{\tfrac{(8m(m+1)+1)\sin(\tht)}{2r}}{(4m+2)+e_0+
\tfrac{(8m(m+1)+1)\cos(\tht)}{2r}}
\end{align*}
and
\begin{align*}
Part(ii):=\frac{\be_{max}(r)}{\sqrt{((4m+2)+e_0)^2+
\big(\tfrac{8m(m+1)+1}{2r}\big)^2+2((4m+2)+e_0)\tfrac{(8m(m+1)+1)\cos\tht}{2r}}}.\\
\end{align*}

\begin{align*}
&\arctan\bigg(\frac{\tfrac{(8m(m+1)+1)\sin(\tht)}{2r}}{(4m+2)+e_0+
\tfrac{(8m(m+1)+1)\cos(\tht)}{2r}}\bigg)\\
&-\arcsin\bigg(\frac{\be_{max}(r)}{\sqrt{((4m+2)+e_0)^2+
\big(\tfrac{8m(m+1)+1}{2r}\big)^2+2((4m+2)+e_0)\tfrac{(8m(m+1)+1)\cos\tht}{2r}}}\bigg)\\
\end{align*}

Suppose that $r\geq am+b$, where $a\geq4$, $b\geq1$, $c>0$, and $\tht\in[0,\tfrac{\pi}{2}]$ are constants, which are given later.

Next, some estimates on Parts (i) and (ii) are provided.

{\bf Part (i):}
\begin{align*}
\frac{\tfrac{(8m(m+1)+1)\sin(\tht)}{2r}}{(4m+2)+e_0+
\tfrac{(8m(m+1)+1)\cos(\tht)}{2r}}\leq c
\end{align*}
is equivalent to
\begin{align*}
&(8m(m+1)+1)\sin(\tht)-c(2r(4m+2+e_0)+(8m(m+1)+1)\cos(\tht))\\
=&8\sin(\tht)m^2+8\sin(\tht)m+\sin(\tht)\\
&-c(8m(am+b)+2(2+e_0)(am+b)+8\cos(\tht)m^2+8\cos(\tht)m+\cos(\tht))\\
=&(8\sin(\tht)-8ca-8c\cos(\tht))m^2+(8\sin(\tht)-8cb-2ca(2+e_0)-8c\cos(\tht))m\\
&+(\sin(\tht)-2c(2+e_0)b-c\cos(\tht))\leq0.\\
\end{align*}
This inequality can be derived by the following inequality:
\begin{align}\label{paraequ-59}
c\geq\max\bigg\{\frac{\sin(\tht)}{a+\cos(\tht)},\ \frac{4\sin(\tht)}{4b+a(2+e_0)+4\cos(\tht)},\
\frac{\sin(\tht)}{2(2+e_0)b+\cos(\tht)}\bigg\}.
\end{align}

{\bf Part (ii):}

Take the following constants $e_1 = 0.84$, $e_{0}=-0.18$, and $e_{-1} = 0.6$.

By \eqref{paraequ-30-1-21-1} and \eqref{paraequ-28}, one has
\begin{align*}
&\frac{\be_{max}(r)}{\sqrt{((4m+2)+e_0)^2+
\big(\tfrac{8m(m+1)+1}{2r}\big)^2+2((4m+2)+e_0)\tfrac{(8m(m+1)+1)\cos\tht}{2r}}}\\
=&\frac{2|e_1|+\tfrac{8m(m+1)+1}{2r}+Q_{2,max}(r)+\varphi_{1,max}(r)}{\sqrt{((4m+2)+e_0)^2+
\big(\tfrac{8m(m+1)+1}{2r}\big)^2+2((4m+2)+e_0)\tfrac{(8m(m+1)+1)\cos\tht}{2r}}}\\
<&\frac{1.68+\tfrac{8m(m+1)+1}{2r}+\tfrac{7}{6}m+\tfrac{1}{m}+6+0.22}{\sqrt{((4m+2)+e_0)^2+
\big(\tfrac{8m(m+1)+1}{2r}\big)^2+2((4m+2)+e_0)\tfrac{(8m(m+1)+1)\cos\tht}{2r}}}\\
<&\frac{\tfrac{8m(m+1)+1}{2r}+\tfrac{7}{6}m+\tfrac{1}{m}+8}{\sqrt{((4m+2)+e_0)^2+
\big(\tfrac{8m(m+1)+1}{2r}\big)^2+2((4m+2)+e_0)\tfrac{(8m(m+1)+1)\cos\tht}{2r}}}\\
\leq&\frac{\tfrac{8m(m+1)+1}{2r}+\tfrac{7}{6}m+\tfrac{1}{3}+8}{\sqrt{((4m+2)+e_0)^2+
\big(\tfrac{8m(m+1)+1}{2r}\big)^2+2((4m+2)+e_0)\tfrac{(8m(m+1)+1)\cos\tht}{2r}}}<c.
\end{align*}
Plugging $r=am+b$ into this inequality, this is transformed into
\begin{align}\label{paraequ-60}
&m^4 \left(-64 a^2 c^2+\frac{49 a^2}{9}-128 a c^2
   \cos \left(\tht\right)+\frac{112 a}{3}-64 c^2+64\right)\nonumber\\
+&m^3 \left(-32 a^2 c^2 e_0-64 a^2 c^2+\frac{700 a^2}{9}-128 a b c^2+\frac{98
   a b}{9}-32 a c^2 e_0 \cos \left(\tht\right)-192 a c^2 \cos \left(\tht\right)\right)\nonumber\\
&+m^3\left(304 a-128 b c^2 \cos \left(\tht\right)+\frac{112 b}{3}-128 c^2+128\right)\nonumber\\
+&m^2 \left(-4 a^2 c^2 e_0^2-16 a^2 c^2 e_0-16 a^2 c^2+\frac{2500 a^2}{9}-64 a b c^2 e_0-128 a b c^2+\frac{1400 a b}{9}\right)\nonumber\\
&+m^2\left(-32 a c^2 e_0 \cos \left(\tht\right)-80 a c^2 \cos \left(\tht\right)+\frac{814
   a}{3}-64 b^2 c^2+\frac{49 b^2}{9}\right)\nonumber\\
&+m^2\left(-32 b c^2 e_0 \cos \left(\tht\right)-192 b c^2 \cos \left(\tht\right)+304 b-80 c^2+80\right)\nonumber\\
+&m \left(-8 a b c^2 e_0^2-32 a b c^2 e_0-32 a b c^2+\frac{5000 a b}{9}-4 a c^2 e_0 \cos \left(\tht\right)-8 a c^2 \cos \left(\tht\right)\right)\nonumber\\
&+m\left(\frac{100
   a}{3}-32 b^2 c^2 e_0-64 b^2 c^2+\frac{700 b^2}{9}-32 b c^2 e_0 \cos \left(\tht\right)-80 b c^2 \cos \left(\tht\right)+\frac{814 b}{3}-16 c^2+16\right)\nonumber\\
+&\left(-4 b^2 c^2 e_0^2-16 b^2 c^2 e_0-16 b^2
   c^2+\frac{2500 b^2}{9}-4 b c^2 e_0 \cos \left(\tht\right)-8 b c^2 \cos \left(\tht\right)+\frac{100 b}{3}-c^2+1\right)<0.
\end{align}
Note that if the coefficient in front of $m^4$ is negative, then the inequality holds for sufficiently large $m$. The the coefficient in front of $m^4$ is negative is equivalent to
\begin{align}\label{paraequ-61}
c^2>\frac{\frac{49 a^2}{9}+\frac{112 a}{3}+64}{64 a^2
   +128 a \cos \left(\tht\right)+64}.
\end{align}

Next, we give the numerical estimates.

Since $m\geq3$, so $r\geq cp_{Q1}$, by Lemma \ref{paraequ-24}, taking $a=4$, $b=2\sqrt{6}-3$, and $\tht=\tfrac{\pi}{5}$, plugging these numbers into \eqref{paraequ-59} and \eqref{paraequ-61}, one has
\begin{align*}
&\max\bigg\{\frac{\sin(\tht)}{a+\cos(\tht)},\ \frac{4\sin(\tht)}{4b+a(2+e_0)+4\cos(\tht)},\
\frac{\sin(\tht)}{2(2+e_0)b+\cos(\tht)}\bigg\}\\
\approx&\max\{0.122226,\ 0.129811,\ 0.0761251\}= 0.129811
\end{align*}
and
\beqq
\frac{\frac{49 a^2}{9}+\frac{112 a}{3}+64}{64 a^2
   +128 a \cos \left(\tht\right)+64}\approx0.125444\approx(0.354181)^{2};
\eeqq
plugging this into the fourth polynomial \eqref{paraequ-60}, one has
\begin{align*}
&\left(-128 \sqrt{5} c^2-1216 c^2+\frac{2704}{9}\right) m^4+\left(-2831.45 c^2+\frac{1456 \sqrt{\frac{2}{3}}}{3}+\frac{21112}{9}\right) m^3\\
&+\left(-1933.93 c^2+\frac{16084
   \sqrt{\frac{2}{3}}}{3}+\frac{27097}{9}\right) m^2+\left(-564.92 c^2+\frac{36484 \sqrt{\frac{2}{3}}}{3}-\frac{14294}{3}\right) m\\
&-59.964 c^2-9800 \sqrt{\frac{2}{3}}+\frac{27203}{3}\\
\approx& m^4 \left(300.444\, -1502.22 c^2\right)+m^3 \left(2742.05\, -2831.45 c^2\right)+m^2 \left(7388.29\, -1933.93 c^2\right)\\
&+m \left(5165.02\, -564.92 c^2\right)-59.964 c^2+1066.
\end{align*}
For this fourth degree polynomial, choosing $c=0.72$, the above polynomial becomes
\begin{align*}
1034.91 + 4872.17 m + 6385.74 m^2 + 1274.23 m^3 - 478.305 m^4.
\end{align*}
By direct computation, one has
\begin{align*}
1034.91 + 4872.17\times 6 + 6385.74\times 6^2 + 1274.23\times 6^3 - 478.305\times 6^4\approx-84495<0.
\end{align*}
So, it follows from numerical computation that
\begin{align*}
1034.91 + 4872.17 m + 6385.74 m^2 + 1274.23 m^3 - 478.305 m^4<0\ \forall m\geq6.
\end{align*}
Note that it follows from Lemma \ref{paraequ-62} that this polynomial is always negative for
\beqq
m\geq15>1+\max\left\{\frac{1034.91}{ 478.305},\ \frac{4872.17}{ 478.305},\ \frac{6385.74}{ 478.305},\ \frac{1274.23}{ 478.305}\right\}\approx14.3508.
\eeqq

Hence, if $z\in H^{\pm}_1$, then $\zeta=\varphi^{-1}(z)\in H^{\pm}_2$. From Lemma \ref{paraequ-55} (b), it follows that
\begin{align*}
&|\arg\,(F(z)-z)|\leq\max\{\text{Arg}\,\De F_{max}(cp_{Q1},\pm\tfrac{\pi}{5}),\ -\text{Arg}\,\De F_{min}(cp_{Q1},\pm\tfrac{\pi}{5})\}\\
\leq&\arctan(0.13)+\arcsin(0.72)\approx0.129275+0.803802=0.933077<\frac{3\pi}{10}\approx0.942478.
\end{align*}

(d) This can be verified by the arguments in the proof of Proposition \ref{near-equat-2}.
\end{proof}

\begin{lemma}\label{near-equat-5}
For $m\geq22$, one has the following estimates:
\begin{itemize}
\item[(a)] the attracting Fatou coordinate $\Phi_{att}$ satisfies the following estimates:
\beq\label{paraequ-79}
-\frac{\pi}{5}<\arg \Phi'_{att}(z)<\frac{\pi}{4}\ \text{for}\ z\in H^{+}_3,
\eeq
\beq\label{paraequ-80}
-\frac{\pi}{4}<\arg \Phi'_{att}(z)<\frac{\pi}{5}\ \text{for}\ z\in H^{-}_3,
\eeq
and
\beq\label{near-equat-4}
\frac{0.800739}{4.35 + 7.25 m}<|\Phi'_{att}(z)|< 0.0613686\ \text{for}\ z\in H^{+}_3\cup H^{-}_3=\overline{\vv}\left(cv_Q,\frac{7\pi}{10}\right);
\eeq
\item[(b)] $\Phi_{att}$ is injective in $H^{+}_3\cup H^{-}_3=\overline{\vv}\left(cv_Q,\frac{7\pi}{10}\right)$. There is a domain $\mathcal{H}_1$ such that $\Phi_{att}$ is a homeomorphism from $\overline{\mathcal{H}}_1$ onto $\{z:\ \text{Re}\, z\geq1\}$, and $\mathcal{H}_1$ satisfies
$\overline{\vv}(cv_Q,\tfrac{3\pi}{10})\subset\mathcal{H}_1\cup\{cv_Q\}\subset\overline{\mathcal{H}}_1\subset\overline{\vv}(cv_Q,\tfrac{7\pi}{10})\cup\{cv_Q\}$
and $cv_Q\in\partial \mathcal{H}_1$.
\end{itemize}
\end{lemma}

\begin{proof}
First, we estimate the following third term:
\beqq
Part(iii):=\frac{1+8m+8m^2 }{2r^2}.
\eeqq
{\bf Part (iii):} For $r=am+b$ with $a>0$ and $b>0$,
\begin{align*}
\frac{1+8m+8m^2 }{2r^2}<c
\end{align*}
is equivalent to
\begin{align*}
(8-2ca^2)m^2+(8-4abc)m+1-2b^2c<0,
\end{align*}
which could be derived by
\beq\label{paraequ-39}
c>\max\bigg\{\frac{4}{a^2},\ \frac{2}{ab},\ \frac{1}{2b^2}\bigg\}.
\eeq
Note that if $2b>a$, then $\max\{\tfrac{4}{a^2},\ \tfrac{2}{ab},\ \tfrac{1}{2b^2}\}=\tfrac{4}{a^2}$.

In the following discussions, assume that $m\geq9$.

By Lemma \ref{paraequ-63}, suppose $z\in H^{+}_{3}$, $\zeta=\varphi^{-1}(z)\in H^{+}_{4}$, i.e., $\text{Re}\,(\zeta\,e^{-i\pi/5})\geq u_4=\frac{(1+\sqrt{5})(m+1)^{m+1}}{m^{m}}-2.1$.

Taking $r_4=0.48$, we show that
\beqq
F(z)\in\cd_{H^{+}_{1}}(z,s(r_4))
\eeqq
where $s(r)=\log\frac{1+r}{1-r}$ is introduced in Lemma \ref{paraequ-10} (b).
By Lemma \ref{paraequ-10} (b) with $H=H^{+}_1$, $t=u_1=(\sqrt{m+1}+\sqrt{m})^2+2.1$, $u=pr_{+}(z)-u_1$, $r=r_4$, $\tht=\tfrac{\pi}{5}$, this is equivalent to
\beqq
F(z)-z\in\cd\bigg(\frac{2ur^2_4e^{i\pi/5}}{1-r^2_4},\frac{2ur_4}{1-r^2_4}\bigg).
\eeqq

By Lemma \ref{paraequ-55}, $F(z)-z\in\cd(\al(u_4,\tfrac{\pi}{5}),\be_{max}(u_4))$, where
\beqq
\al(u_4,\tfrac{\pi}{5})=(4m+2)+
\frac{(8m(m+1)+1)e^{-\tfrac{i\pi}{5}}}{2u_4}+e_0
\eeqq
and
\beqq
\be_{max}(u_4)=2|e_1|+\tfrac{8m(m+1)+1}{2u_4}+Q_{2,max}(u_4)+\varphi_{1,max}(u_4).
\eeqq

Next, we provide an estimate on this item:
\begin{align}\label{paraequ-74}
&\bigg|\al(z)-\frac{2u_5r^2_4e^{i\pi/5}}{1-r^2_4}\bigg|+\be_{max}(u_4)-\frac{2u_5r_4}{1-r^2_4}
\end{align}
and show that if $r_4=0.48$ and $m\geq14$, then this term is negative.

By simple computation, one has
\begin{align}\label{paraequ-42}
4m=(2\sqrt{m})^2<(\sqrt{m+1}+\sqrt{m})^2<(2\sqrt{m+1})^2=4(m+1),\ \forall m\geq2.
\end{align}

Next, assume
\beqq
u_5=u_3-u_1=pr_{+}(cv_Q)-(\sqrt{m+1}+\sqrt{m})^2-2.1
=\frac{(1+\sqrt{5})(m+1)^{m+1}}{m^{m}}-(\sqrt{m+1}+\sqrt{m})^2-2.1,
\eeqq
it is easy to obtain that for $m\geq9$, by \eqref{paraequ-36} and \eqref{paraequ-42},
\begin{align*}
u_5&=\frac{(1+\sqrt{5})(m+1)^{m+1}}{m^{m}}-(\sqrt{m+1}+\sqrt{m})^2-2.1\\
\geq&(1+\sqrt{5})(m+1) \frac{(m+1)^{m}}{m^{m}}-4(m+1)-2.1
\geq(1+\sqrt{5})(m+1) \frac{10^{9}}{9^{9}}-4(m+1)-2.1
\end{align*}
and
\begin{align*}
u_5=(1+\sqrt{5})(m+1)\frac{(m+1)^{m}}{m^{m}}-4m-2.1<(1+\sqrt{5})(m+1)e-4m-2.1.
\end{align*}
Similarly, one has
\begin{align}\label{parqequ-74}
u_4&=\frac{(1+\sqrt{5})(m+1)^{m+1}}{m^{m}}-2.1\nonumber\\
&=(1+\sqrt{5})(m+1)\frac{(m+1)^{m}}{m^{m}}-2.1
\geq(1+\sqrt{5})(m+1)\frac{10^{9}}{9^{9}}-2.1
\end{align}
and
\begin{align}\label{paraequ-48}
u_4\leq(1+\sqrt{5})(m+1)e-2.1.
\end{align}

So, the numerical computation with $e<2.72$ gives us that
\begin{align}\label{paraequ-77}
4.35286m+2.25286\leq u_5\leq4.8021m+6.7021
\end{align}
and
\begin{align}\label{paraequ-65}
8.35286m+6.25286\leq u_4\leq 8.8021m+10.7021.
\end{align}

Choose $r_4=0.48$, one has
\begin{align}\label{paraequ-78}
\frac{2\cos(\tfrac{\pi}{5})r^2_4}{1-r^2_4}\approx0.484401,\ \frac{2\sin(\tfrac{\pi}{5})r^2_4}{1-r^2_4}\approx0.351938,\ \frac{2r_4}{1-r^2_4}\approx1.2474.
\end{align}

So, by \eqref{paraequ-77}, \eqref{paraequ-65}, and the estimates in Part (iii) with $a=8.35286$ and $b=6.25286$,
\begin{align*}
&(4m+2)+
\frac{(8m(m+1)+1)\cos(\tfrac{\pi}{5})}{2u_4}+e_0-\frac{2\cos(\tfrac{\pi}{5})u_5r^2_4}{1-r^2_4}\\
\leq&(4m+2)+\frac{4\cos(\tfrac{\pi}{5})}{8.35286^2} (8.8021m+10.7021)-0.18-0.484401\times(4.35286m+2.25286)\\
\approx&1.22509 + 2.29973 m
\end{align*}
and
\begin{align*}
&\frac{(8m(m+1)+1)\sin(\tfrac{\pi}{5})}{2u_4}+\frac{2\sin(\tfrac{\pi}{5})u_5r^2_4}{1-r^2_4}\\
\leq&\frac{4\sin(\tfrac{\pi}{5})}{8.35286^2} (8.8021m+10.7021)+0.351938\times(4.8021m+6.7021)\\
\approx&2.71937 + 1.98666 m.
\end{align*}

By \eqref{paraequ-65} and $m\geq9$,
$u_4\geq8.35286m+6.25286\geq 4m+2\sqrt{6}$. So,
for $r\geq u_4=\frac{2\sqrt{3}(m+1)^{m+1}}{m^{m}}-2.1 \geq4m+2\sqrt{6}$ and $m\geq9$, by \eqref{paraequ-28}, \eqref{paraequ-30-1-21-1}, and the estimates in Part (iii) with $a=8.35286$ and $b=6.25286$, one has
\begin{align*}
&\be_{max}(r)=2|e_1|+\tfrac{8m(m+1)+1}{2r}+Q_{2,max}(r)+\varphi_{1,max}(r)\\
<&1.68+\tfrac{8m(m+1)+1}{2r}+\frac{7}{6}m+\frac{1}{m}+6+0.22\\
<&\frac{4}{8.35286^2} (8.8021m+10.7021)+\frac{7}{6}m+8+\frac{1}{9}.
\end{align*}
So, by \eqref{paraequ-77} and \eqref{paraequ-78}, one has
\begin{align*}
&\be_{max}(u_4)-\frac{2u_5r_4}{1-r^2_4}\\
<&\frac{4}{8.35286^2} (8.8021m+10.7021)+\frac{7}{6}m+8+\frac{1}{9}-1.2474\times(4.35286m+2.25286)\\
<&5.91446 - 3.75846 m.
\end{align*}
Hence, one has
\begin{align}\label{near-equat-3}
&\bigg|\al(z)-\frac{2u_5r^2_4e^{i\pi/5}}{1-r^2_4}\bigg|+\be_{max}(u_4)-\frac{2u_5r_4}{1-r^2_4}\nonumber\\
<&\sqrt{(1.22509 + 2.29973 m)^2 + (2.71937 + 1.98666 m)^2}+5.91446 - 3.75846 m\nonumber\\
=&\sqrt{8.89582 + 16.4397 m + 9.23558 m^2}+5.91446 - 3.75846 m\nonumber\\
<&\sqrt{(3.03901 m + 2.70479)+1.3^2}+5.91446 - 3.75846 m\nonumber\\
<&(3.03901 m + 2.70479)+1.3+5.91446 - 3.75846 m=9.91925 - 0.71945 m.
\end{align}
Since $9.91925/0.71945 m\approx13.7873$, if $m\geq14$, then this number is negative.
\medskip

An estimate in Theorem \ref{paraequ-64} gives that
\begin{align*}
&\arg \Phi'_{att}(z)\leq-\arg(F(z)-z)+\frac{1}{2}|\log F'(z)|+\frac{1}{2}\log\frac{1}{1-r^2_4}\\
\leq& -\text{Arg}\De F_{min}(u_4,\tfrac{\pi}{5})+\frac{1}{2}\text{Log} DF_{max}(u_4)-\frac{1}{2}\log(1-r^2_4)\\
\leq&\arctan\bigg(\frac{\tfrac{(8m(m+1)+1)\sin(\tfrac{\pi}{5})}{2u_4}}{(4m+2)+e_0+
\tfrac{(8m(m+1)+1)\cos(\tfrac{\pi}{5})}{2u_4}}\bigg)(:=Part (i))\\
&+\arcsin\bigg(\frac{\be_{max}(u_4)}{\sqrt{((4m+2)+e_0)^2+
\big(\tfrac{8m(m+1)+1}{2u_4}\big)^2+2((4m+2)+e_0)\tfrac{(8m(m+1)+1)\cos(\tfrac{\pi}{5})}{2u_4}}}\bigg)\\
&(:=Part (ii))+\frac{1+8m+8m^2 }{2u_4^2}(:=Part (iii))+\frac{ 32 m + 96 m^2 + 64 m^3}{6u_4^3}(:=Part (iv))\\
&+\frac{1}{8}\bigg(\frac{\big(\tfrac{(\sqrt{m+1}+\sqrt{m})^2 }{u_4}\big)^4}{1-\tfrac{(\sqrt{m+1}+\sqrt{m})^2 }{u_4}}+\frac{\big(\tfrac{(\sqrt{m+1}-\sqrt{m})^2 }{u_4}\big)^4}{1-\tfrac{(\sqrt{m+1}-\sqrt{m})^2 }{u_4}}\bigg)(:=Part (v))+
\frac{\frac{(2m+1)}{5u_4^{5}}}{1-\tfrac{1}{u_4^2}}(:=Part (vi))\\
&-\frac{1}{2}\log\bigg(1-\bigg(\frac{|a_E|}{u_4-|x_E|}\bigg)^2\bigg)(:=Part (vii))-\frac{1}{2}\log(1-r^2_4)(:=Part (viii)).
\end{align*}
On the other hand,
\begin{align*}
&\arg \Phi'_{att}(z)\geq-\arg(F(z)-z)+\frac{1}{2}|\log F'(z)|-\frac{1}{2}\log\frac{1}{1-r^2_4}\\
\geq& -\text{Arg}\De F_{max}(u_4,\tfrac{\pi}{5})-\frac{1}{2}\text{Log} DF_{max}(u_4)+\frac{1}{2}\log(1-r^2_4)\\
\geq&\arctan\bigg(\frac{\tfrac{(8m(m+1)+1)\sin(\tfrac{\pi}{5})}{2u_4}}{(4m+2)+e_0+
\tfrac{(8m(m+1)+1)\cos(\tfrac{\pi}{5})}{2u_4}}\bigg)\\
&-\arcsin\bigg(\frac{\be_{max}(u_4)}{\sqrt{((4m+2)+e_0)^2+
\big(\tfrac{8m(m+1)+1}{2u_4}\big)^2+2((4m+2)+e_0)\tfrac{(8m(m+1)+1)\cos(\tfrac{\pi}{5})}{2u_4}}}\bigg)\\
&-\frac{1+8m+8m^2 }{2u_4^2}-
\frac{ 32 m + 96 m^2 + 64 m^3}{6u_4^3}\\
&-\frac{1}{8}\bigg(\frac{\big(\tfrac{(\sqrt{m+1}+\sqrt{m})^2 }{u_4}\big)^4}{1-\tfrac{(\sqrt{m+1}-\sqrt{m})^2 }{u_4}}-\frac{\big(\tfrac{(\sqrt{m+1}-\sqrt{m})^2 }{u_4}\big)^4}{1-\tfrac{(\sqrt{m+1}-\sqrt{m})^2 }{u_4}}\bigg)-
\frac{\frac{(2m+1)}{5u_4^{5}}}{1-\tfrac{1}{u_4^2}}\\
&+\frac{1}{2}\log\bigg(1-\bigg(\frac{|a_E|}{u_4-|x_E|}\bigg)^2\bigg)+\frac{1}{2}\log(1-r^2_4).
\end{align*}

For the estimate of these constants, we divide these estimates into the following parts:

Suppose that $r\geq am+b$, where $a$, $b$, $c$ are positive constants, $\tht\in[0,\tfrac{\pi}{2}]$ determined later.

Consider the $m\geq9$ in the following computation.

{\bf Part (i):}
This is already done in the proof of Lemma \ref{paraequ-63}. By \eqref{paraequ-59}, the minimal bound for Part (i) is
\beqq
\max\bigg\{\frac{\sin(\tfrac{\pi}{5})}{a+\cos(\tfrac{\pi}{5})},\ \frac{4\sin(\tfrac{\pi}{5})}{4b+a(2+e_0)+4\cos(\tfrac{\pi}{5})},\
\frac{\sin(\tfrac{\pi}{5})}{2(2+e_0)b+\cos(\tfrac{\pi}{5})}\bigg\}.
\eeqq

{\bf Part (ii):}
Take the following constants $e_1 = 0.84$, $e_{0}=-0.18$, $e_{-1} = 0.6$.

Please refer to the proof of Lemma \ref{paraequ-63}, where $m\geq3$ in Lemma \ref{paraequ-63}. Here, we assume $m\geq9$. So, we modify the arguments there.

For $r\geq4m+1$, by \eqref{paraequ-28}, \eqref{paraequ-30-1-21-1},
\begin{align*}
&\frac{\be_{max}(r)}{\sqrt{((4m+2)+e_0)^2+
\big(\tfrac{8m(m+1)+1}{2r}\big)^2+2((4m+2)+e_0)\tfrac{(8m(m+1)+1)\cos\tht}{2r}}}\\
=&\frac{2|e_1|+\tfrac{8m(m+1)+1}{2r}+Q_{2,max}(r)+\varphi_{1,max}(r)}{\sqrt{((4m+2)+e_0)^2+
\big(\tfrac{8m(m+1)+1}{2r}\big)^2+2((4m+2)+e_0)\tfrac{(8m(m+1)+1)\cos\tht}{2r}}}\\
<&\frac{1.68+\tfrac{8m(m+1)+1}{2r}+\tfrac{7}{6}m+\tfrac{1}{m}+6+0.22}{\sqrt{((4m+2)+e_0)^2+
\big(\tfrac{8m(m+1)+1}{2r}\big)^2+2((4m+2)+e_0)\tfrac{(8m(m+1)+1)\cos\tht}{2r}}}\\
<&\frac{\tfrac{8m(m+1)+1}{2r}+\tfrac{7}{6}m+8+\tfrac{1}{m}}{\sqrt{((4m+2)+e_0)^2+
\big(\tfrac{8m(m+1)+1}{2r}\big)^2+2((4m+2)+e_0)\tfrac{(8m(m+1)+1)\cos\tht}{2r}}}<c.
\end{align*}
This is equivalent to the following inequality with $r=am+b$,
\begin{align}\label{paraequ-38}
&m^4 \left(-64 a^2 c^2+\frac{49 a^2}{9}-128 a c^2 \cos \left(\tht\right)+\frac{112 a}{3}-64 c^2+64\right)\nonumber\\
+&m^3 \left(-32 a^2 c^2 e_0-64 a^2 c^2+\frac{224 a^2}{3}-128 a b c^2+\frac{98 a b}{9}-32 a c^2 e_0 \cos \left(\tht\right)-192 a c^2 \cos \left(\tht\right)\right)\nonumber\\
&+m^3\left(\frac{880 a}{3}-128 bc^2 \cos \left(\tht\right)+\frac{112 b}{3}-128 c^2+128\right)\nonumber\\
+&m^2 \left(-4 a^2 c^2 e_0^2-16 a^2 c^2 e_0-16 a^2 c^2+\frac{796 a^2}{3}-64 a b c^2 e_0-128 a b c^2+\frac{448 a b}{3}\right)\nonumber\\
&+m^2\left(-32 a c^2 e_0 \cos \left(\tht\right)-80 a c^2 \cos \left(\tht\right)+\frac{878 a}{3}-64
   b^2 c^2+\frac{49 b^2}{9}-32 b c^2 e_0 \cos \left(\tht\right)\right)\nonumber\\
&+m^2\left(-192 b c^2 \cos \left(\tht\right)+\frac{880 b}{3}-80 c^2+80\right)\nonumber\\
+&m \left(64 a^2-8 a b c^2 e_0^2-32 a b c^2 e_0-32 a b c^2+\frac{1592 a
   b}{3}-4 a c^2 e_0 \cos \left(\tht\right)-8 a c^2 \cos \left(\tht\right)\right)\nonumber\\
&+m\left(64 a-32 b^2 c^2 e_0-64 b^2 c^2+\frac{224 b^2}{3}-32 b c^2 e_0 \cos \left(\tht\right)-80 b c^2 \cos \left(\tht\right)+\frac{878
   b}{3}-16 c^2+16\right)\nonumber\\
+&\frac{8 a b+64 b^2+4
   b}{m}+\frac{4 b^2}{m^2}\nonumber\\
+&\left(4 a^2+128 a b+4 a-4 b^2 c^2 e_0^2-16 b^2 c^2 e_0-16 b^2 c^2\right)\nonumber\\
&+\left(\frac{796 b^2}{3}-4 b c^2 e_0 \cos \left(\tht\right)-8 b c^2 \cos \left(\tht\right)+64 b-c^2+1\right)<0.
\end{align}

{\bf Part (iii):} For $r=am+b$,
\begin{align*}
\frac{1+8m+8m^2 }{2r^2}<c
\end{align*}
has been studied above. By \eqref{paraequ-39}, the minimal bound for Part (iii) is
\beqq
\max\bigg\{\frac{4}{a^2},\ \frac{2}{ab},\ \frac{1}{2b^2}\bigg\}.
\eeqq

{\bf Part (iv):} For $r=am+b$,
\begin{align*}
\frac{ 32 m + 96 m^2 + 64 m^3}{6r^3}=\frac{16 m + 48 m^2 + 32 m^3}{3a^3 m^3+9 a^2 b m^2+9 a b^2 m+3b^3}<c
\end{align*}
is equivalent to
\begin{align*}
(3a^3c-32) m^3+(9 a^2 bc-48) m^2+(9 a b^2c-16) m+3b^3c>0,
\end{align*}
which could be derived by
\beq\label{paraequ-40}
c>\max\bigg\{\frac{32}{3a^3},\ \frac{48}{9 a^2 b},\ \frac{16}{9ab^2}\bigg\}.
\eeq

{\bf Part (v):} It is evident that the following function is decreasing with respect to $r$:
\begin{align*}
\frac{1}{8}\bigg(\frac{\big(\tfrac{(\sqrt{m+1}+\sqrt{m})^2 }{r}\big)^4}{1-\tfrac{(\sqrt{m+1}+\sqrt{m})^2 }{r}}+\frac{\big(\tfrac{(\sqrt{m+1}-\sqrt{m})^2 }{r}\big)^4}{1-\tfrac{(\sqrt{m+1}-\sqrt{m})^2 }{r}}\bigg).
\end{align*}
It is evident that
\begin{align*}
4m=(2\sqrt{m})^2<(\sqrt{m+1}+\sqrt{m})^2<(2\sqrt{m+1})^2=4(m+1),\ \forall m\geq2.
\end{align*}
and
\begin{align*}
\frac{4}{m}=\bigg(\frac{2}{\sqrt{m}}\bigg)^2>\bigg(\frac{1}{\sqrt{m+1}+\sqrt{m}}\bigg)^2=
(\sqrt{m+1}-\sqrt{m})^2>\bigg(\frac{2}{\sqrt{m+1}}\bigg)^2=\frac{4}{m+1}.
\end{align*}

For $r>4(m+1)$, one has
\begin{align*}
&\frac{1}{8}\bigg(\frac{\big(\tfrac{(\sqrt{m+1}+\sqrt{m})^2 }{r}\big)^4}{1-\tfrac{(\sqrt{m+1}+\sqrt{m})^2 }{r}}+\frac{\big(\tfrac{(\sqrt{m+1}-\sqrt{m})^2 }{r}\big)^4}{1-\tfrac{(\sqrt{m+1}-\sqrt{m})^2 }{r}}\bigg)\\
=&\frac{1}{8}\bigg(\frac{(\sqrt{m+1}+\sqrt{m})^8}{r^3(r-(\sqrt{m+1}+\sqrt{m})^2)}
+\frac{(\sqrt{m+1}-\sqrt{m})^8}{r^3(r-(\sqrt{m+1}-\sqrt{m})^2)}\bigg)\\
\leq&\frac{1}{8}\bigg(\frac{(4(m+1))^4}{r^3(r-4(m+1))}
+\frac{(\tfrac{4}{m})^4}{r^3(r-\tfrac{4}{m})}\bigg).
\end{align*}
Further, assume that $r=am+b$ with $a>4$, $b>4$, $2b-4>a$, where $2b-4>a$ implies that
$\tfrac{b}{a}>\tfrac{1}{2}$ and $\tfrac{b-4}{a-4}>\tfrac{1}{2}$,
\begin{align}\label{paraequ-67}
&\frac{1}{8}\bigg(\frac{(4(m+1))^4}{r^3(r-4(m+1))}
+\frac{(\tfrac{4}{m})^4}{r^3(r-\tfrac{4}{m})}\bigg)\nonumber\\
=&\frac{1}{8}\bigg(\frac{(4(m+1))^4}{(am+b)^3(am+b-4(m+1))}
+\frac{(\tfrac{4}{m})^4}{(am+b)^3(am+b-\tfrac{4}{m})}\bigg)\nonumber\\
=&\frac{1}{8}\bigg(\frac{(4(m+1))^4}{(am+b)^3((a-4)m+(b-4)))}
+\frac{(\tfrac{4}{m})^4}{(am+b)^3(am+b-\tfrac{4}{m})}\bigg)\nonumber\\
<&\frac{1}{8}\bigg(\frac{(4(m+1))^4}{a^3\times (m+1/2)^3\times (a-4)\times(m+1/2))}
+\frac{(\tfrac{4}{m})^4}{(am)^3(am)}\bigg)\nonumber\\
<&\frac{1}{8}\bigg(\frac{4^4}{a^3\times (a-4)}\times\bigg(\frac{m+1}{m+\tfrac{1}{2}}\bigg)^4
+\bigg(\frac{4}{am^2}\bigg)^4\bigg).
\end{align}

{\bf Part (vi):}
This function is decreasing with respect to $r$:
\begin{align*}
\frac{\frac{(2m+1)}{5r^{5}}}{1-\tfrac{1}{r^2}}=\frac{(2m+1)}{5r^3(r^2-1)}.
\end{align*}
Moreover, assume that $r=am+b$ with $a>4$, $b>4$, and $2b-2>a$, where $2b-2>a$ implies that $\tfrac{2}{a}(b-1)>1$, one has
\begin{align}\label{paraequ-68}
&\frac{(2m+1)}{5r^3(r^2-1)}=\frac{(2m+1)}{5(am+b)^3((am+b)^2-1)}\nonumber\\
=&\frac{(2m+1)}{5(am+b)^3(am+b-1)(am+b+1)}\nonumber\\
=&\frac{(2m+1)}{5(am+b)^3\cdot(a/2(2m+(b-1)\times(2/a)))\cdot(am+b)}\nonumber\\
\leq&\frac{(2m+1)}{5(am+b)^3\cdot(a/2(2m+1))\cdot(am+b+1)}
=\frac{2}{5a(am+b)^3(am+b+1)}.
\end{align}

Next, we give the numerical estimates for the above items.

For $m\geq9$, by \eqref{parqequ-74} and \eqref{paraequ-65},
\begin{align*}
u_4\geq8.35286m+6.25286.
\end{align*}
So, we choose $a=8.35286$ and $b=6.25286$.

For Part (i), by \eqref{paraequ-59},
$a=8.35286$, $b=6.25286$, $\tht=\tfrac{\pi}{5}$,
\begin{align}\label{paraequ-69}
&\max\bigg\{\frac{\sin(\tht)}{a+\cos(\tht)},\ \frac{4\sin(\tht)}{4b+a(2+e_0)+4\cos(\tht)},\
\frac{\sin(\tht)}{2(2+e_0)b+\cos(\tht)}\bigg\}\nonumber\\
\approx&\{0.0641555,\ 0.0541118,\ 0.0249385\}=0.0641555.
\end{align}

For Part (ii),
for  \eqref{paraequ-38}, if $m$ is sufficiently large and
\beqq
\frac{\frac{49 a^2}{9}+\frac{112 a}{3}+64}{64 a^2+128 a\cos \left(\tht\right)+64}< c^2,
\eeqq
then \eqref{paraequ-38} holds.

Let $e_0=-0.18$, $\tht=\tfrac{\pi}{5}$, and $m\geq9$, \eqref{paraequ-38} is simplified as follows:
\begin{align*}
&m^4 \left(755.7\, -5394.27 c^2\right)+m^3 \left(8589.84\, -12782.8 c^2\right)+m^2 \left(30883.6\, -11034.2 c^2\right)\\
&+m \left(37481.6\, -4101.87 c^2\right)+17773.1\, -555.863
   c^2+\frac{156.393}{m^2}+\frac{2945.13}{m}<0.
\end{align*}
Since $m\geq14$ in \eqref{near-equat-3}, plugging $m=14$ into $\tfrac{156.393}{m^2}+\tfrac{2945.13}{m}$, the above inequality is simplified as follows:
\begin{align*}
&m^4 \left(755.7\, -5394.27 c^2\right)+m^3 \left(8589.84\, -12782.8 c^2\right)+m^2 \left(30883.6\, -11034.2 c^2\right)\\
&+m \left(37481.6\, -4101.87 c^2\right)+17984.3\, -555.863
   c^2<0.
\end{align*}
Take $c = 0.45$, the above polynomial is written as
\beqq
17871.7 + 36651 m + 28649.2 m^2 + 6001.32 m^3 - 336.64 m^4<0
\eeqq
By numerical computation, one has that this inequality holds for $m\geq22$.

For Part (iii), by \eqref{paraequ-39},
$a=8.35286$, $b=6.25286$, $\tht=\tfrac{\pi}{5}$,
\begin{align}\label{paraequ-70}
c>&\max\bigg\{\frac{4}{a^2},\ \frac{2}{ab},\ \frac{1}{2b^2}\bigg\}=
\max\bigg\{\frac{4}{8.35286^2},\ \frac{2}{8.35286\times6.25286},\ \frac{1}{2\times 6.25286^2}\bigg\}\nonumber\\
\approx&\max\{0.057331,\ 0.0382927,\ 0.0127883\}=0.057331.
\end{align}

For Part (iv), by \eqref{paraequ-40}, $a=8.35286$, $b=6.25286$
\begin{align}\label{paraequ-71}
c>&\max\bigg\{\frac{32}{3a^3},\ \frac{48}{9 a^2 b},\ \frac{16}{9ab^2}\bigg\}\nonumber\\
=&\max\bigg\{\frac{32}{3\times 8.35286^3},\ \frac{48}{9\times 8.35286^2\times 6.25286},\ \frac{16}{9\times 8.35286\times 6.25286^2}\bigg\}\nonumber\\
\approx&\max\{0.018303,\ 0.012225,\ 0.00544358\}=0.018303.
\end{align}

For Part (v), $a=8.35286$, $b=6.25286$,  by \eqref{paraequ-67},
\begin{align}\label{paraequ-72}
&\frac{1}{8}\bigg(\frac{4^4}{a^3\times (a-4)}\times\bigg(\frac{m+1}{m+\tfrac{1}{2}}\bigg)^4
+\bigg(\frac{4}{am^2}\bigg)^4\bigg)\nonumber\\
<&\frac{1}{8}\bigg(\frac{4^4}{8.35286^3\times (8.35286-4)}\times\bigg(\frac{9+1}{9+\tfrac{1}{2}}\bigg)^4
+\bigg(\frac{4}{8.35286\times 9^2}\bigg)^4\bigg)\approx0.0154873.
\end{align}
where the function $(\tfrac{x+1}{x+1/2})^4$ is decreasing for $x>1$.

For Part (vi), $a=8.35286$, $b=6.25286$, by \eqref{paraequ-68}, one has
\begin{align*}
&\frac{2}{5a(am+b)^3(am+b+1)}\\
=&\frac{2}{5\times 8.35286(8.35286\times9+6.25286)^3(8.35286\times9+6.25286+1)}
\approx1.07601\times10^{-9}.
\end{align*}

For Part (vii),  by Remark \ref{paraequ-29}, $a=8.35286$, $b=6.25286$, $m\geq9$,
recall the constants $e_1 = 0.84$, $e_{0}=-0.18$, $e_{-1} = 0.6$,
one has
\begin{align}\label{paraequ-73}
&-\frac{1}{2}\log\bigg(1-\bigg(\frac{|a_E|}{u_4-|x_E|}\bigg)^2\bigg)\nonumber\\
\leq&-\frac{1}{2}\log\bigg(1-\bigg(\frac{0.84+0.6}{8.35286\times9+6.25286-0.18}\bigg)^2\bigg)
\approx0.000157084.
\end{align}
For Part (viii),
$r_4=0.48$,
\beq\label{paraequ-76}
-\frac{1}{2}\log(1-r^2_4)=0.130942.
\eeq

Hence, by the above estimates for Parts (i)--(vii) and \eqref{paraequ-31}, one has
\begin{align*}
&\arg \Phi'_{att}(z)\\
\leq&\arctan(0.0641555) + \arcsin(0.45) + 0.057331 +0.018303\\
&+0.0154873 +10^{-8}+ 0.000157084 + 0.130942\\
=&0.753053<\frac{\pi}{4}(\approx0.785398)
\end{align*}
and
\begin{align*}
&\arg \Phi'_{att}(z)\\
\geq&\arctan(0.0641555) - \arcsin(0.45) - 0.057331 - 0.018303\\
 &- 0.0154873 -10^{-8}-0.000157084 -0.130942\\
=&-0.624918>-\frac{\pi}{5}(\approx-0.628319).
\end{align*}

Next, we estimate $|\Phi'_{att}(z)|$ on $H^+_3$ or $H^-_3$.
It follows from Theorem \ref{paraequ-64} and Lemma \ref{paraequ-55} and the estimates in Part (iii)--(vi) that
\begin{align}\label{paraequ-49}
&|\Phi'_{att}(z)|\leq\exp\bigg(-\log|F(z)-z|+\frac{1}{2}|\log F'(z)|+\frac{1}{2}\log\frac{1}{1-r^2_4}\bigg)\nonumber\\
\leq&\frac{\exp\bigg(\frac{1}{2}\text{Log} DF_{max}(u_4)\bigg)}{\text{Abs}\De F_{min}(u_4,\tfrac{\pi}{5})\sqrt{1-r^2_4}}\nonumber\\
\leq&\frac{\exp(0.057331+0.018303+0.0154873 +10^{-8}+ 0.000157084)}{(0.55 + 2.2m)\sqrt{1-0.48^2}}\nonumber\\
\leq&
\frac{\exp( 0.0912784)}{(0.55 + 2.2\times9)\sqrt{1-0.48^2}}\nonumber\\
\approx&\frac{1.24885}{0.55 + 2.2\times9}\approx0.0613686
\end{align}
and
\begin{align*}
&|\Phi'_{att}(z)|\geq\exp\bigg(-\log|F(z)-z|-\frac{1}{2}|\log F'(z)|-\frac{1}{2}\log\frac{1}{1-r^2_4}\bigg)\\
\geq&\frac{\sqrt{1-r^2_4}}{\text{Abs}\De F_{max}(u_4,\tfrac{\pi}{5})\exp\bigg(\frac{1}{2}\text{Log} DF_{max}(u_4)\bigg)}\\
\geq&\frac{\sqrt{1-0.48^2}}{(4.35 + 7.25 m)\exp( 0.057331+0.018303+0.0154873 +10^{-8}+ 0.000157084)}\\
\approx&\frac{\sqrt{1-0.48^2}}{(4.35 + 7.25 m)\exp( 0.0912784)}
\approx\frac{0.800739}{4.35 + 7.25 m}.
\end{align*}

By estimates in Part (iii),  \eqref{paraequ-48}, one has

$$u_4=\frac{2\sqrt{3}(m+1)^{m+1}}{m^{m}}-2.1=2\sqrt{3}(m+1)\frac{(m+1)^{m}}{m^{m}}-2.1
\leq2\sqrt{3}(m+1)e-2.1,$$
and
\begin{align*}
&((4m+2)+e_0)^2+
\big(\tfrac{8m(m+1)+1}{2u_4}\big)^2+2((4m+2)+e_0)\tfrac{(8m(m+1)+1)\cos(\tfrac{\pi}{5})}{2u_4}\\
\leq&((4m+2)+e_0)^2+
\big(0.051 u_4\big)^2+2((4m+2)+e_0)\times\cos(\tfrac{\pi}{5})\times(0.051u_4)\\
<&((4m+2)+e_0)^2+
\big(0.051( 2\sqrt{3}(m+1)e-2.1)\big)^2\\
&+0.102((4m+2)+e_0)\cos(\tfrac{\pi}{5})(2\sqrt{3}(m+1)e-2.1)\\
<&((4m+2)-0.18)^2+
\big(0.051( 2\sqrt{3}(m+1)\times2.72-2.1)\big)^2\\
&+0.102((4m+2)-0.18)\cos(\tfrac{\pi}{5})(2\sqrt{3}(m+1)\times2.72-2.1)\\
\approx&4.55157 + 18.751 m + 19.341 m^2<9+30m+25m^2=(5m+3)^2.
\end{align*}
So, by estimates in Part (ii), one has
\begin{align*}
\text{Abs}\De F_{max}(u_4,\tfrac{\pi}{5})\leq(5m+3)(1+0.45)=4.35 + 7.25 m
\end{align*}
and
\begin{align*}
\text{Abs}\De F_{min}(u_4,\tfrac{\pi}{5})\geq(4m+1)(1-0.45)=0.55 + 2.2 m.
\end{align*}

(b)  First, we show that $\Phi_{att}$ is injective in $\overline{\vv}(cv_{Q},\tfrac{7\pi}{10})$.

Take two points $z_1$ and $z_2\in $ from $\overline{\vv}(cv_{Q},\tfrac{7\pi}{10})$, suppose these two points can be joined by a non-trivial line segment within $\overline{\vv}(cv_{Q},\tfrac{7\pi}{10})$, and this line segment is denoted by $[z_1,z_2]$ with $[z_1,z_2]\subset \overline{\vv}(cv_{Q},\tfrac{7\pi}{10})$. Applying a similar argument in \eqref{paraequ-49}, one has
\begin{align}\label{parqequ-49}
\text{if}\, \tht<\arg\,\Phi'_{att}(z)<\tht'\leq\tht+\pi\ \text{on}\ [z_1,z_2],\ \text{then}\ \tht<\arg\,\frac{\Phi_{att}(z_2)-\Phi_{att}(z_1)}{z_2-z_1}<\tht'.
\end{align}

This, together with \eqref{paraequ-49}, $-\tfrac{\pi}{5}<\arg\,\Phi'_{att}(z)<\tfrac{\pi}{4}$, implies that $$\text{Re}\,\Phi'_{att}(z)>0,\ \text{Re}\,\frac{\Phi_{att}(z_2)-\Phi_{att}(z_1)}{z_2-z_1}>0,\ \text{and}\ \Phi_{att}(z_2)\neq\Phi_{att}(z_1).$$

If two points $z_1,z_2\in\overline{\vv}(cv_{Q},\tfrac{7\pi}{10})$ cannot be joined by a segment within $\overline{\vv}(cv_{Q},\tfrac{7\pi}{10})$, there is $z_3$ so that $[z_1,z_3]$ and $[z_3,z_2]$ are two non-trivial line segments contained in  $\overline{\vv}(cv_{Q},\tfrac{7\pi}{10})$ satisfying that $\tfrac{3\pi}{10}\leq\arg\,(z_3-z_1)\leq\tfrac{7\pi}{10}$ and $\tfrac{3\pi}{10}\leq\arg\,(z_2-z_3)\leq\tfrac{7\pi}{10}$. By \eqref{parqequ-49}, $0<\tfrac{3\pi}{10}-\tfrac{\pi}{5}<\arg\,(\Phi_{att}(z_3)-\Phi_{att}(z_1))<\tfrac{7\pi}{10}+\tfrac{\pi}{4}<\pi$, thus $\text{Im}\,(\Phi_{att}(z_3)-\Phi_{att}(z_1))>0$. Similarly $0<\tfrac{3\pi}{10}-\tfrac{\pi}{5}<\arg\,(\Phi_{att}(z_2)-\Phi_{att}(z_3))<\tfrac{7\pi}{10}+\tfrac{\pi}{4}<\pi$ and $\text{Im}\,(\Phi_{att}(z_2)-\Phi_{att}(z_3))>0$. So, $\text{Im}\,(\Phi_{att}(z_2)-\Phi_{att}(z_1))=\text{Im}\,(\Phi_{att}(z_2)-\Phi_{att}(z_3))+\text{Im}\,(\Phi_{att}(z_3)-\Phi_{att}(z_1))>0$. Hence, $\Phi_{att}$ is injective in $\overline{\vv}(cv_{Q},\tfrac{7\pi}{10})$.

Take two points $z_1,z_2\in H^+_3$ with $z_1\neq z_2$ and
$-\tfrac{3\pi}{10}\leq\arg\,(z_2-z_1)\leq\tfrac{7\pi}{10}$.
Since $H^+_3$ is convex, $[z_1,z_2]\subset H^+_3$. By \eqref{paraequ-49},
\begin{align}\label{paraequ-50}
\arg\,(z_2-z_1)-\frac{\pi}{5}< \arg\,(\Phi_{att}(z_2)-\Phi_{att}(z_1))
< \arg\,(z_2-z_1)+\frac{\pi}{4}.
\end{align}
For points on the upper boundary of $\overline{\vv}(cv_{Q},\tfrac{7\pi}{10})$,
one has $\arg\,(z-cv_{Q})=\tfrac{7\pi}{10}$. This, together
with $\Phi_{att}(cv_{Q})=1$ and \eqref{paraequ-50}, yields that
$\tfrac{\pi}{2}=\tfrac{7\pi}{10}-\tfrac{\pi}{5}
<\arg\,(\Phi_{att}(z)-1)<\tfrac{7\pi}{10}+\tfrac{\pi}{4}<\pi$.
Thus $\text{Re}\,(\Phi_{att}(z)-1)<0$ and $\text{Im}\,(\Phi_{att}(z)-1)>0$.

Similarly, for points $z_1,z_2\in H^-_3$ with $z_1\neq z_2$ and
$-\tfrac{7\pi}{10}\leq\arg\,(z_2-z_1)\leq\tfrac{3\pi}{10}$, one has $[z_1,z_2]\subset H^-_3$ and
\begin{align}\label{paraequ-52}
\arg\,(z_2-z_1)-\frac{\pi}{4}< \arg\,(\Phi_{att}(z_2)-\Phi_{att}(z_1))
< \arg\,(z_2-z_1)+\frac{\pi}{5}.
\end{align}
For points on the lower boundary of $\overline{\vv}(cv_{Q},\tfrac{7\pi}{10})$,
one has $\arg\,(z-cv_{Q})=-\tfrac{7\pi}{10}$. So,
$-\pi<-\tfrac{7\pi}{10}-\tfrac{\pi}{4}
<\arg\,(\Phi_{att}(z)-1)<-\tfrac{7\pi}{10}+\tfrac{\pi}{5}=-\tfrac{\pi}{2}$.
Thus $\text{Re}\,(\Phi_{att}(z)-1)<0$ and $\text{Im}\,(\Phi_{att}(z)-1)<0$.

It follows from \eqref{paraequ-75}, \eqref{paraequ-79},\eqref{paraequ-80}, and \eqref{near-equat-4} that
\begin{align*}
\bigg|\frac{\Phi_{att}(z)-1}{z-cv_{Q}}\bigg|\geq\int^{1}_0\text{Re}\,\Phi'_{att}(cv_{Q}+t(z-cv_{Q}))dt\geq
\frac{0.800739}{4.35 + 7.25 m}\cos(\tfrac{\pi}{4})>0.
\end{align*}
So, $\Phi_{att}(z)\to\infty$ as $z\to\infty$ in $\overline{\vv}(cv_{Q},\tfrac{7\pi}{10})$.

For any given positive constant $R'>0$, choose
$R^{\prime\prime}>\frac{0.800739}{4.35 + 7.25 m}\cos(\tfrac{\pi}{4})R'$
and denote $G=\vv(cv_{Q},\tfrac{7\pi}{10})\cap\cd(cv_{Q},R^{\prime\prime}).$
By the above discussions, one has that
$\Phi_{att}(\partial G)\cap(\{z:\ \text{Re}\,z\geq0\}\cap
\overline{\cd}(1,R'))=\{cv_Q\}$.
By \eqref{paraequ-52}, $\{z:\ \text{Re}\,z\geq0\}\cap
\overline{\cd}(1,R')$
contains at least one point of $\Phi_{att}(G)$, for example
$cv_Q+t$ with $t>0$ small enough, such that Jordan curve
$\Phi_{att}(\partial G)$ has winding number $1$ around this point. Hence, this is true for any point
in $\{z:\ \text{Re}\,z\geq1\}\cap
\overline{\cd}(1,R')$ except $cv_{Q}$. Hence, it follows from the Argument Principle
that $\{z:\ \text{Re}\,z\geq1\}\cap
\overline{\cd}(1,R')\subset\Phi_{att}(G)\cup\{cv_Q\}$. Since $R'>0$ was arbitrary,
$\{z:\ \text{Re}\,z\geq1\}$
is contained in the image of  $\vv(cv_{Q},\tfrac{7\pi}{10})\cup\{cv_Q\}$.
Further, $R'>0$ was arbitrarily chosen, $\{z:\ \text{Re}\,z\geq1\}$ is contained in the image of $\vv(cv_{Q},\tfrac{7\pi}{10})\cup\{cv_Q\}$ by $\Phi_{att}$.

Set
\beq\label{h1region8-20-1}
\mathcal{H}_1:=\Phi^{-1}_{att}
(\{z:\ \text{Re}\,z>1\}).
\eeq
For $z\in\overline{\vv}(cv_Q,\tfrac{3\pi}{10})=H^+_3\cap H^-_{3}$,  the two inequalities \eqref{paraequ-79} and \eqref{paraequ-80} should hold at the same time, so for any  $z_1,z_2\in\overline{\vv}(cv_Q,\tfrac{3\pi}{10})$,
\begin{align}\label{paraequ-52}
\arg\,(z_2-z_1)-\frac{\pi}{5}< \arg\,(\Phi_{att}(z_2)-\Phi_{att}(z_1))
< \arg\,(z_2-z_1)+\frac{\pi}{5},
\end{align}
implying that
\beqq
|\arg\,(\Phi_{att}(z)-1)|<\frac{\pi}{5}+\frac{3\pi}{10}=\frac{\pi}{2}\ \forall z\in\overline{\vv}(cv_Q,\tfrac{3\pi}{10}).
\eeqq
Thus, $\Phi_{att}(\overline{\vv}(cv_Q,\tfrac{3\pi}{10}))\subset\{z:\ \text{Re}\,z>1\}\cup\{cv_Q\}$. Therefore,
\beqq
\overline{\vv}(cv_Q,\tfrac{3\pi}{10})\subset\mathcal{H}_1\cup\{cv_Q\}\subset\overline{\mathcal{H}}_1\subset
\overline{\vv}(cv_Q,\tfrac{7\pi}{10})\cup\{cv_Q\}.
\eeqq
\end{proof}

\begin{lemma}\label{connect-2021-12-29-3}
For $m\geq22$ and $F=Q\circ \varphi^{-1}$, one has
\begin{align*}
\text{Re}\,F(cv_Q)>cv_Q,
\end{align*}
implying that $W_1$ is connected, where $W_1$ is introduced in \eqref{paraequ-81}.
\end{lemma}

\begin{proof}
Recall that $\varphi$ satisfies \eqref{paraequ-9} and \eqref{near-equat-7}. By Lemma \ref{paraequ-25}, for $m\geq3$,
\beqq
4(m+1)\cdot e\geq cv_Q=4(m+1)\frac{(m+1)^{m}}{m^{m}}\geq4(m+1) \frac{(3+1)^{3}}{3^{3}}=\frac{64}{27}\cdot4(m+1)>9(m+1).
\eeqq

By \eqref{paraequ-66} of Lemma \ref{paraequ-12}, there is an estimate on $Q(\zeta)$
\beqq
Q(z)=z+(4m+2)+
\frac{8m(m+1)+1}{z}+Q_2(z).
\eeqq

Hence, by \eqref{near-equat-20} and $9(m+1)>9m+4=(9m+1)+3$,
\beqq
|Q_2\circ \varphi^{-1}(cv_Q)|\leq \frac{2^7}{3^6}m+\frac{2^4}{3^4}\cdot\frac{1}{m}+\frac{2^5}{3^3}.
\eeqq

Next, we give an estimate on the real part of $\tfrac{1}{a+b\cdot e^{i\tht}}$, where $a>2b>0$ are positive numbers and $\tht\in[0,2\pi).$ By direct computation, one has
\begin{align*}
\frac{1}{a+b\cdot e^{i\tht}}=\frac{1}{a+b(\cos(\tht)+i\sin(\tht))}=\frac{1}{(a+b\cos(\tht))+ib\sin(\tht)}
=\frac{(a+b\cos(\tht))-ib\sin(\tht)}{(a+b\cos(\tht))^2+(b\sin(\tht))^2}.
\end{align*}
The real part is
\begin{align*}
\frac{(a+b\cos(\tht))}{(a+b\cos(\tht))^2+(b\sin(\tht))^2}\geq
\frac{(a+b\cos(\tht))}{(a+b\cos(\tht))^2+b^2}\geq\frac{a+b}{(a+b)^2+b^2},
\end{align*}
where the last inequality can be derived by considering the function $\tfrac{x}{x^2+b^2}$ for $x\in[a-b,a+b]$, and $(\tfrac{x}{x^2+b^2})^{\prime}=\tfrac{b^2-x^2}{(x^2+b^2)^2}<0$ for $x\in[a-b,a+b]$.
Note that the function $\frac{a+b}{(a+b)^2+b^2}$ is decreasing with respect to the variable $a$.

Hence, the estimate of the lower bound of the real part of $\tfrac{8m(m+1)+1}{ \varphi^{-1}(cv_Q)}$ for $m\geq22$ is
\begin{align*}
&\frac{(8m(m+1)+1)\cdot (4e(m+1)+3)}{ (4e(m+1)+3)^2+3^2}
=\frac{(8m(m+1)+1)}{ (4e(m+1)+3)+\frac{3^2}{4e(m+1)+3}}\\
>&\frac{(8m(m+1)+1)}{ (4e(m+1)+3)+1}
>\frac{8m(m+1)}{ 4e(m+1)+4}=\frac{2m(m+1)}{ e(m+1)+1}>\frac{1.96m}{e},
\end{align*}
where $0.04\times e\times (m+1)>0.04\times 2.71\times (22+1)\approx2.4932>1.96$ is used in the last inequality.

So, by \eqref{paraequ-66}, for $m\geq22$, one has
\begin{align*}
&\text{Re}\,(Q\circ \varphi^{-1}(cv_Q))\\
\geq& \frac{23^{22}}{22^{22}}\cdot4(m+1)-3+(4m+2)+
\frac{1.96m}{e}-\left(\frac{2^7}{3^6}m+\frac{2^4}{3^4}\cdot\frac{1}{m}+\frac{2^5}{3^3}\right)\\
\geq&\frac{23^{22}}{22^{22}}\cdot4(m+1)-3+4(m+1)-2+
\frac{0.49}{e}\cdot 4(m+1)\\
&-\frac{1.96}{e}-\frac{2^5}{3^6}\cdot4(m+1)+\frac{2^7}{3^6}-\frac{2^4}{3^4}\cdot\frac{1}{22}-\frac{2^5}{3^3}\\
=&\left(\frac{23^{22}}{22^{22}}+\frac{0.49}{e}-\frac{2^5}{3^6}\right)\cdot4(m+1)+\left(-5-\frac{1.96}{e}+\frac{2^7}{3^6}-\frac{2^4}{22\times3^4}-\frac{2^5}{3^3}\right)\\
>&\left(\frac{23^{22}}{22^{22}}+\frac{0.49}{2.72}-\frac{2^5}{3^6}\right)\cdot4(m+1)+\left(-5-\frac{1.96}{2.71}+\frac{2^7}{3^6}-\frac{2^4}{22\times3^4}-\frac{2^5}{3^3}\right)\\
\approx&2.79522\cdot 4(m+1)-6.74183>2.72\cdot 4(m+1)>cv_Q,
\end{align*}
where $(2.79522 - 2.72)\times4\times(22+1)=6.92024>6.74183$ is used.
\end{proof}

Now, the proof of Proposition \ref{attracting2020912-1} is provided.

\begin{proof} (Proof of Proposition \ref{attracting2020912-1})

By the above discussions on $\Phi_{att}$, define the following sets:
\beqq
D_1=\Phi_{att}^{-1}(\{z:\ 1<\text{Re}\, z<2,\ -\eta<\text{Im}\,z<\eta\})\bigcap \overline{\vv}(cv_Q,\tfrac{7\pi}{10}),
\eeqq
\beqq
D^{\sharp}_1=\Phi_{att}^{-1}(\{z:\ 1<\text{Re}\, z<2,\ \text{Im}\,z>\eta\})\bigcap \overline{\vv}(cv_Q,\tfrac{7\pi}{10}),
\eeqq
\beqq
D^{\flat}_1=\Phi_{att}^{-1}(\{z:\ 1<\text{Re}\, z<2,\ \text{Im}\,z<-\eta\})\bigcap \overline{\vv}(cv_Q,\tfrac{7\pi}{10}).
\eeqq

Suppose that $|\arg\,(z-F(cv_Q))|\leq\tfrac{3\pi}{10}$, this is on the right of $W_1$, where $W_1$ is introduced in \eqref{paraequ-81}.

It follows from \eqref{paraequ-65} that $\text{Re}\,(\zeta\,e^{-i\pi/5})>cp_{Q1}$
for $\zeta\in H^{+}_4$, $H^{\pm}_3\subset\varphi(H^{\pm}_4)$ by Lemma \ref{paraequ-63}, so there is $\zeta\in H^{+}_4$ such that $\varphi(\zeta)=cv_Q$, thus
$|\arg\,(F(cv_Q)-cv_Q)|<\tfrac{3\pi}{10}$ by the estimates in Parts (i) and (ii) (of $\Phi'_{att}$) in the proof of Lemma \ref{near-equat-5} . Hence, by $\arg\,(z-cv_Q)=\arg\,(z-F(cv_Q)+F(cv_Q)-cv_Q)$, one has
$|\arg\,(z-cv_Q)|\leq\tfrac{3\pi}{10}$. By \eqref{paraequ-52}, one has
\begin{align*}
\arg\,(z-F(cv_Q))-\frac{\pi}{5}< \arg\,(\Phi_{att}(z)-\Phi_{att}(F(cv_Q)))
< \arg\,(z-F(cv_Q))+\frac{\pi}{5}.
\end{align*}
Thus, $|\arg\,(\Phi_{att}(z)-\Phi_{att}(F(cv_Q)))|<\tfrac{3\pi}{10}+\tfrac{\pi}{5}=\tfrac{\pi}{2}$,
implying that
$$\text{Re}\,\Phi_{att}(z)>\text{Re}\,\Phi_{att}(F(cv_Q))=2.$$ Therefore, $D_1$, $D^{\sharp}_1$,  and $D^{\flat}_1$ are contained in $W_1$.

Now, we show that for $z\in\{z:\ |\arg\,(z-cv_Q)|\leq\tfrac{\pi}{5}\}$, $\Phi_{att}(z)$ cannot be in $\{z:\ 1<\text{Re}\,z<2,\ |\text{Im}\,z|>\eta\}$. By direct computation, one has $|\arg\,(\Phi_{att}(z)-1)|<\tfrac{\pi}{5}+\tfrac{\pi}{5}=\tfrac{2\pi}{5}$, $\tan(\tfrac{2\pi}{5})=\sqrt{5 + 2\sqrt{5}}\approx3.07768<\eta$,
where
\begin{align*}
&\eta=\frac{1}{2\pi}\left(\log(12(m+1)\cdot 30^{2+2m})+1\right)
>\frac{1}{2\times 4}\left(\log(12)+\log(m+1)+ (2+2m)\log(30)+1\right)\\
>&\frac{1}{2\times 4}\left(\log(12)+\log(22+1)+ (2+2\times22)\log(30)+1\right)\approx
20.38\ \forall m\geq22.
\end{align*}

So, $D^{\sharp}_1$  and $D^{\flat}_1$ are in $\{z:\ \tfrac{\pi}{5}<|\arg\,(z-cv_Q)|<\tfrac{7\pi}{10}\}$.

Finally, we verify that $D_1\subset\cd(cv_Q,R_1)$. By \eqref{near-equat-4}, the derivative of $\Phi_{att}^{-1}$ is bounded by $\tfrac{4.35 + 7.25 m}{0.800739}$, and
$\{z:\ 1<\text{Re}\,z<2,\ -\eta<\text{Im}\,z<\eta\}\subset \cd(1,\sqrt{1+\eta^2})$, one has
$D_1\subset\cd(cv_Q,\sqrt{1+\eta^2}\times\tfrac{4.35 + 7.25 m}{0.800739})$.  So, it is sufficient to show that $\sqrt{1+\eta^2}\times\tfrac{4.35 + 7.25 m}{0.800739}<R_1=2.39\times 10^m$.

Since
\begin{align*}
& (1.5m+3.5)(4.35 + 7.25 m)=15.225 + 31.9 m + 10.875 m^2<16+32m+11m^2,\\
&16+32\times 3+11\times 3^2=211<1000=10^3,\ 22< (\log(10))^2\times10,\\ &32+22\times2=76<100<\log(10)\times10^2,
\end{align*}
one has
\beq\label{near-equat-8}
(1.5m+3.5)(4.35 + 7.25 m)<16+32m+11m^2<10^m\ \forall m\geq3.
\eeq
So, for $m\geq3$, one has
\begin{align*}
&\sqrt{1+\eta^2}=\sqrt{1+\left(\frac{1}{2\pi}\left(\log(12(m+1)\cdot 30^{2+2m})+1\right)\right)^2}\\
<&1+\frac{1}{2\pi}\left(\log(12(m+1)\cdot 30^{2+2m})+1\right)\\
<&\frac{1}{2\pi}\left(\log(12(m+1)\cdot 30^{2+2m})+9\right)\\
=&\frac{1}{2\pi}\left(\log(12)+\log(m+1)+(2+2m)\log(30)+9\right)\\
<&\frac{1}{2\times3}\left(3+(m+1)+4(2+2m)+9\right)=2+1.5(m+1)=1.5m+3.5\\
<&\frac{10^m}{4.35 + 7.25 m}<1.91377\times \frac{10^m}{4.35 + 7.25 m}=2.39\times 10^m\times \frac{0.800739}{4.35 + 7.25 m}.
\end{align*}

\end{proof}

\subsection{Locating domains $D_0$, $D'_0$, $D_{-1}$, and $D^{\prime\prime}_{-1}$}\label{location2021-10-25-1}

Note that $Q$ maps both $\mathcal{U}^{Q+}_1$ and $\mathcal{U}^{Q+}_2$ homeomorphically onto $\cc\setminus(-\infty,cv_Q]$. Define
\beqq
\widetilde{\mathcal{H}}_0=(Q|_{\mathcal{U}^{Q+}_1})^{-1}(\mathcal{H}_1),\
\widetilde{D}_0=(Q|_{\mathcal{U}^{Q+}_1})^{-1}(D_1),\
\widetilde{D}^{\sharp}_0=(Q|_{\mathcal{U}^{Q+}_1})^{-1}(D^{\sharp}_1),\
\widetilde{D}^{\prime}_0=(Q|_{\mathcal{U}^{Q+}_2})^{-1}(D_1),
\eeqq
where $\mathcal{H}_1$ is specified in \eqref{h1region8-20-1}.
By (b) of Lemma \ref{near-equat-5}, $\mathcal{H}_1\cup D^{\sharp}_1\subset\ol{\vv}(cv,\tfrac{7\pi}{10})\subset\cc\setminus\ol{\cd}(0,\rho)$; by the proof of Proposition \ref{attracting2020912-1}, $D_1\subset W_1$, so
$D_1\subset\ol{\cd}(0,R)\setminus\ol{\cd}(0,\rho)$.
This, together with (d) of Lemma \ref{paraequ-15}, implies that these domains, $\widetilde{\mathcal{H}}_0$,
$\widetilde{D}_0$,
$\widetilde{D}^{\sharp}_0$, and
$\widetilde{D}^{\prime}_0$, are contained in
$$\cc\setminus \bigg(E_{r_1}\cup\overline{\cd}\big(\cot(\tfrac{\pi}{m+1})i,\tfrac{1}{\sin(\tfrac{\pi}{m+1})}\big)\cup \overline{\cd}\big(-\cot(\tfrac{\pi}{m+1})i,\tfrac{1}{\sin(\tfrac{\pi}{m+1})}\big)\bigg).$$ So,
it is reasonable to define
\beqq
\mathcal{H}_0=\varphi(\widetilde{H}_0),\ D_0=\varphi(\widetilde{D}_0),\
D^{\sharp}_0=\varphi(\widetilde{D}^{\sharp}_0),\ D^{\prime}_0=\varphi(\widetilde{D}^{\prime}_0).
\eeqq

Since $Q$ maps $\mathcal{U}_{1+}\cup\mathcal{U}_{2-}\cup\ga_{b1+}$ and $\mathcal{U}_{1-}\cup\mathcal{U}_{2+}\cup\ga_{b2+}$ homeomorphically onto $\cc\setminus(-\infty,0]\cup(cv_Q,+\infty)$, we can define
\beqq
\widetilde{D}_{-1}=(Q|_{\mathcal{U}_{1+}\cup\mathcal{U}_{2-}\cup\ga_{b1+}})^{-1}(D_0)\ \text{and}\ \widetilde{D}^{\prime\prime}_{-1}=(Q|_{\mathcal{U}_{1-}\cup\mathcal{U}_{2+}\cup\ga_{b2+}})^{-1}(D_0).
\eeqq
These domains are contained in $\cc\setminus E_{r_1}$ by Lemma \ref{near-equata-1} below. It is reasonable to define
\beqq
D_{-1}=\varphi(\widetilde{D}_{-1})\ \text{and}\ D^{\prime\prime}_{-1}=\varphi(\widetilde{D}^{\prime\prime}_{-1}).
\eeqq

\begin{lemma}\label{near-equata-1}
Let $u_{\sharp}=cv_Q-(4m+2)$.
 \begin{itemize}
\item[(a)] $\widetilde{D}_0\cup\widetilde{D}^{\prime}_0\subset \cd(u_{\sharp},R_1+1)$, $D_0\cup D^{\prime}_0\subset\cd(u_{\sharp},R_1+4)$;
\item[(b)]
 $\widetilde{D}_{-1}\cup\widetilde{D}_{-1}^{\prime\prime}\subset\cd(4m,R_1+12m+9-cv_Q)$, and
$D_{-1}\cup D_{-1}^{\prime\prime}\subset \cd(4m,R_1+12m+12-cv_Q)$;
further,
$\widetilde{D}_0\cup\widetilde{D}^{\prime}_0\cup\widetilde{D}_{-1}\cup\widetilde{D}^{\prime\prime}_{-1}\subset\cd(0,R_1+8m+11)\cap\mathcal{U}_{12}\cap(\cc\setminus E_{r_1})$;
\item[(c)] $D_0\cup D^{\prime}_0\cup D_{-1}\cup D^{\prime\prime}_{-1}\subset\cd(0,R_1+8m+14)$.
\end{itemize}
\end{lemma}

\begin{lemma}\label{region2021-11-25-1}
Take $\tht_0\in(0,\pi/2)$ with $\cos(\tht_0)=\tfrac{1}{3}$, where $\arccos(\tfrac{1}{3})\approx0.391827\pi$.
Consider the two lines
$$l^+_1=\{\zeta=(4m+1)e^{i\tht_0}+s e^{i(\tht_0+\pi/2)}\ \mbox{with}\ s\geq0\ \mbox{is the parameter}\}$$
and
$$l^-_1=\{\zeta=(4m+1)e^{-i\tht_0}+s e^{i(-\tht_0-\pi/2)}\ \mbox{with}\ s\geq0\ \mbox{is the parameter}\},$$
and the arc $\ga_1=\ga_1(\tht)=cp_{Q1}\cdot e^{i\tht}$ with $\tht\in[-(\tht_0+\tht_1),(\tht_0+\tht_1)]$, where $\tht_1=\arccos(\tfrac{4m+1}{cp_{Q1}})$.
It is evident that the union $l^+_1\cup l^-_1\cup\ga_1$ separates the complex plane into two parts. Let
$\widetilde{W}_0$ be the region on the right side of the $l^+_1\cup l^-_1\cup\ga_1$.

Consider the two lines
$$l^+_2=\{\zeta=(4m-2)e^{i\tht_0}+s e^{i(\tht_0+\pi/2)}\ \mbox{with}\ s\geq0\ \mbox{is the parameter}\}$$
and
$$l^-_2=\{\zeta=(4m-2)e^{-i\tht_0}+s e^{i(-\tht_0-\pi/2)}\ \mbox{with}\ s\geq0\ \mbox{is the parameter}\},$$
and the arc $\ga_2=\ga(\tht)=(cp_{Q1}-3)\cdot e^{i\tht}$ with $\tht\in[-(\tht_0+\tht_2),(\tht_0+\tht_2)]$, where $\tht_2=\arccos(\tfrac{4m-2}{cp_{Q1}-3})$.
It is evident that the union $l^+_2\cup l^-_2\cup\ga_2$ separates the complex plane into two parts. Let
$W_0$ be the region on the right side of the $l^+_2\cup l^-_2\cup\ga_2$.

Then,
\begin{itemize}
\item[(a)] $\vv(cv_{Q},\tfrac{7\pi}{10})\subset Q(\widetilde{W}_0)\subset\cc\setminus(-\infty,cv_{Q}]$ and $\widetilde{W}_0\subset\mathcal{U}_1$;
\item[(b)] $\varphi(\widetilde{W}_0)\subset W_0$;
\item[(c)] $Q^{-1}(W_0)\setminus\overline{\cd}\subset\widetilde{W}_{-1}:
    =\vv(0,\tfrac{9\pi}{10})\setminus(\overline{\cd}\cup\{\zeta:\ \text{Re}\,\zeta\leq0\ \text{and}\ |\zeta|\leq4m-3\})$.
\end{itemize}
\end{lemma}

\begin{lemma}\label{equ-22-2-12-1} For $R=2.66\times 10^m$ and $R_1=2.39\times 10^m$,  and $W_0$, $\widetilde{W}_0$, and  $\widetilde{W}_{-1}$ introduced in Lemma \ref{region2021-11-25-1}, one has
\begin{itemize}
\item[(a)] $\widetilde{D}_0\subset\widetilde{W}_0\cap\cd(cv_Q-(4m+2),R_1+1)$; $D_0\subset W_0\cap\cd(cv_Q-(4m+2),R_1+4) $;
\item[(b)] $\widetilde{D}_0\cup \widetilde{D}^{\prime}_0\cup\widetilde{D}_{-1}\cup \widetilde{D}^{\prime\prime}_{-1}\subset\widetilde{W}_{-1}\cap\cd(0,R_1+cv_Q-(4m+1))\cap\mathcal{U}^{Q+}_{12}\cap(\cc\setminus E_{r_1})$;
\item[(c)] $D_0\cup D^{\prime}_0\cup D_{-1}\cup D^{\prime\prime}_{-1}\subset\cd(0,R_1+cv_Q-(4m+1))$.
\end{itemize}
\end{lemma}

\begin{proof}
This is a direct corollary by Lemmas \ref{near-equata-1} and \ref{region2021-11-25-1}.
\end{proof}

Now, the proof of Proposition \ref{attracting2020912-2} is given.

\begin{proof} (Proof of Proposition \ref{attracting2020912-2})
By Lemma \ref{equ-22-2-12-1} and the construction above, one can show that (a)--(d) of Proposition \ref{attracting2020912-2} hold. It is enough to check that
$ \ol{D}_0\cup\ol{D}^{\prime}_{0}\cup\ol{D}_{-1}\cup\ol{D}^{\prime\prime}_{-1}\setminus\{cv_Q\}\subset\cd(0,R)\setminus(\ol{\cd}(0,\rho)\cup\rr_{-}\cup\ol{\vv}(cv_Q,\tfrac{3\pi}{10}))=\pi_X (X_{2+})\cup\pi_{X}(X_{2-})$. It follows from (c) of Lemma \ref{equ-22-2-12-1} that $\ol{D_0\cup D^{\prime}_0\cup D_{-1}\cup D^{\prime\prime}_{-1}}\subset\ol{\cd}(0,R_1+cv_Q-(4m+1))$.

By direct computation, one has
\begin{align*}
R-R_1=(2.66-2.39)\times10^m=0.27\times 10^m>8m+14\ \forall m\geq3,
\end{align*}
since $0.27\times10^3=270>38=8\times3+14$ and $0.27\times\log(10)\times 10^3>0.27\times10^3>8$.

For $\zeta\in\ol{\widetilde{D}_0\cup\widetilde{D}^{\prime}_0\cup\widetilde{D}_{-1}\cup\widetilde{D}^{\prime\prime}_{-1}}$,
Lemma \ref{near-equata-1} (b) implies that $\zeta\not\in\mbox{int}E_{r_1}$. Hence, by Lemma \ref{nearequ-24}, $|\varphi(\zeta)|>\rho$. Furthermore, by Lemma \ref{equ2021-11-25-5}, if
$\text{Re}\,\zeta\geq0$, $\varphi(\zeta)\not\in\rr_{-}$.
If $\text{Re}\,\zeta\leq0$, then $\zeta\in\text{closure}(\widetilde{W}_{-1})$, hence $|\text{Im}\,\zeta|\geq(4m-3)\sin(\tfrac{9\pi}{10})>(4\times 21-3)\times \frac{(-1 + \sqrt{5})}{4} \approx25.0304>7$. So, since $|\zeta|\geq7$, one has $|\varphi(\zeta)-\zeta|<3$ by \eqref{near-equat-7}. Therefore, $\varphi(\zeta)\not\in\rr_{-}$.

For $z\in\ol{D}_0\cup\ol{D}^{\prime}_{0}\cup\ol{D}_{-1}\cup\ol{D}^{\prime\prime}_{-1}$, one has $F(z)\in\ol{\mathcal{H}}_0\cup\ol{\mathcal{H}}_1$ and $0\leq \text{Re}\,\Phi_{att}(F(z))\leq2$. On the other hand, by (b) of Lemma \ref{near-equat-5}, for $z'\in \overline{\vv}(cv_Q,\tfrac{3\pi}{10})\setminus\{cv_Q\}$, one has $\text{Re}\,\Phi_{att}(z')>1$ implying that $\text{Re}\,\Phi_{att}(F(z'))>2$. Hence, it is impossible that $z\in\overline{\vv}(cv_Q,\tfrac{3\pi}{10})\setminus\{cv_Q\}$. Hence, (e) of Proposition \ref{attracting2020912-2} is proved.

\end{proof}

\begin{proof} ( Proof of Lemma \ref{near-equata-1})
(a) Now, we give an estimate on $Q_2(z)$ in \eqref{paraequ-66}.
\begin{itemize}
\item By \eqref{paraequ-20}, if $a=\tfrac{1}{2^3}$ and $r>\tfrac{2^4}{\sqrt{3}}m^{3/2}+1$, then
\begin{align*}
\frac{8C^3_{2m}}{r(r-1)}<\frac{1}{2^3};
\end{align*}
\item by \eqref{paraequ-21}, if $a=\tfrac{1}{2^3}$ and $r>2^{7/2}m^{1/2}+1$, then
\begin{align*}
\frac{16m}{r(r-1)}<\frac{1}{2^3};
\end{align*}
\item by \eqref{paraequ-22}, if $a=\tfrac{1}{2^3}$ and $r>\max\{2^{8/3}m^{4/3}+1,4m\}$, then
\begin{align*}
\frac{2(2m)(2m-1)(2m-2)(2m-3)}{3r(r-1)^2}\bigg(\frac{r+1}{r-1}\bigg)^{2m-4}
<\frac{1}{2^3};
\end{align*}
\item by \eqref{paraequ-23}, if $a=\tfrac{1}{2^3}$ and $r>\max\{2\sqrt{3}m+1,4m\}$, then
\begin{align*}
\frac{8(2m)(2m-1)}{(r-1)^2}\bigg(\frac{r+1}{r-1}\bigg)^{2m-2}
<\frac{1}{2^3}.
\end{align*}
\end{itemize}

By direct computation, one has
\beqq
16m(m+1)+2\geq\max\bigg\{\tfrac{2^4}{\sqrt{3}}m^{3/2}+1,\ 2^{7/2}m^{1/2}+1,\ 2^{8/3}m^{4/3}+1,\ 2\sqrt{3}m+1,\ 4m\bigg\}.
\eeqq

If $|\zeta|\geq16m(m+1)+2$, by Lemma \ref{paraequ-12}, one has
\begin{align*}
|Q(\zeta)-(\zeta+4m+2)|\leq \frac{1}{2}+Q_{2,max}(16m(m+1)+2)<1.
\end{align*}

By direct computation,
\begin{align*}
u_{\sharp}\leq 4e(m+1)-(4m+2)=4m(e-1)+4e-2<4m(3-1)+4\times3-2=8m+10.
\end{align*}
One has $16m(m+1)+2+8m+10=16m^2+24m+12<R_1$ for $m\geq2$, since it follows from direct calculation that
\begin{align*}
&16\times 2^2+24\times2+12=124<2.39\times10^2=239,\\
& 32\times 2+24=88<2.39\times(\log10)\times10^2\approx550.318,\\
&32<2.39\times (\log10)^2\times 10^2\approx1267.15.
\end{align*}
As a consequence, if $|\zeta-u_{\sharp}|\geq R_1+1$, then $|\zeta|>R_1-u_{\sharp}>16m(m+1)+2$, implying that
\begin{align*}
|Q(\zeta)-cv_Q|=|(Q(\zeta)-(\zeta+4m+2))+(\zeta-u_{\sharp})|\geq|\zeta-u_{\sharp}|-|Q(\zeta)-(\zeta+4m+2)|>R_1.
\end{align*}
So,
\beqq
\widetilde{D}_0\cup\widetilde{D}^{\prime}_0\subset Q^{-1}(\cd(cv_Q,R_1))\subset\cd(u_{\sharp},R_1+1).
\eeqq
On the other hand, if $\zeta\in\cc\setminus E_{r_1}$ and $|\zeta-u_{\sharp}|<R_1+1$,  its image $\varphi(\zeta)$ is surrounded by the Jordan curve $\varphi(\{\zeta^{\prime}:\ |\zeta^{\prime}-u_{\sharp}|=R_1+1\})$, which is contained in $\cd(u_{\sharp},R_1+4)$ by \eqref{near-equat-7}. Thus, $D_0\cup D^{\prime}_0\subset\cd(u_{\sharp},R_1+4)$.

(b)
Now, we show $1\leq 12m+4-cv_Q<3m$. By direct computation,
\beqq
12m+4-cv_Q\geq 12m+4-4e(m+1)\geq 12m+4-4\times 2.72(m+1)=-6.88 + 1.12 m>0\ \forall m\geq7,
\eeqq
\beqq
12\times 2+4-cv_Q= 28-4\cdot\tfrac{3^3}{2^2}=1,\ m=2,
\eeqq
\beqq
12\times 3+4-cv_Q= 40-4\cdot\tfrac{4^4}{3^3}=\frac{56}{27},\ m=3,
\eeqq
\beqq
12\times 4+4-cv_Q= 52-4\cdot\tfrac{5^5}{4^4}=\frac{203}{64},\ m=4,
\eeqq
\beqq
12\times 5+4-cv_Q= 64-4\cdot\tfrac{6^6}{5^5}=\frac{13376}{3125},\ m=5,
\eeqq
\beqq
12\times 6+4-cv_Q= 76-4\cdot\tfrac{7^7}{6^6}=\frac{62921}{11664},\ m=6.
\eeqq
On the other hand,
\beqq
12m+4-cv_Q\leq 12m+4-4(m+1)\frac{3^2}{2^2}= 3m-5<3m\ \forall m\geq2.
\eeqq

 If $|\zeta-4m|\geq R_1+12m+9-cv_Q$, then $|\zeta|\geq R_1-4m>R_1-(8m+10)>16m(m+1)+2$, and
\begin{align*}
&|Q(\zeta)-u_{\sharp}|
=|(Q(\zeta)-(\zeta+(4m+2)))+(\zeta-4m)+(4m+(4m+2)-cv_Q+4m+2)|\\
\geq&|(\zeta-4m)|-|Q(\zeta)-(\zeta+(4m+2))|-|12m+4-cv_Q|
>R_1+4.
\end{align*}
So, $\widetilde{D}_{-1}\cup\widetilde{D}_{-1}^{\prime\prime}\subset Q^{-1}(\cd(u_{\sharp},R_1+4))\subset\cd(4m,R_1+12m+9-cv_Q)$, and
$D_{-1}\cup D_{-1}^{\prime\prime}\subset \cd(4m,R_1+12m+8-cv_Q)$.

(c) By the discussions in (a) and (b), it is evident that $\cd(u_{\sharp},R_1+4)\cup\cd(4m,R_1+12m+12-cv_Q)\subset\cd(0,R_1+8m+14)$. This completes the proof.

\end{proof}

\begin{proof} (Proof of Lemma \ref{region2021-11-25-1})
{\bf (a)} The proof is divided into two parts, each part is separated into several steps.

{\bf Step 1(i)} For the $\tht_0\in(0,\pi/2)$ with $\cos(\tht_0)=\tfrac{1}{3}$, one has $\arccos(\tfrac{1}{3})\approx0.391827\pi$.
So, by Remark \ref{nearequ-23}, one has
\beqq
\sin(\tht_0)=\sqrt{1-(\cos(\tht_0))^2}=\frac{2\sqrt{2}}{3}\approx0.942809,
\eeqq
\begin{align*}
&\cos(\tht_0+\pi/5)=\cos(\tht_0)\cos(\pi/5)-\sin(\tht_0)\sin(\pi/5)\\
=&
\frac{1}{3}\cdot\frac{1}{4}(1+\sqrt{5})-\frac{2\sqrt{2}}{3}\cdot\frac{1}{4} \sqrt{10-2 \sqrt{5}}\\
=&\frac{1+\sqrt{5}}{12}-\frac{\sqrt{2}}{6}\cdot \sqrt{10-2 \sqrt{5}}\approx-0.284497,
\end{align*}
\begin{align*}
&\cos(\tht_0-\pi/5)=\cos(\tht_0)\cos(\pi/5)+\sin(\tht_0)\sin(\pi/5)\\
=&
\frac{1}{3}\cdot\frac{1}{4}(1+\sqrt{5})+\frac{2\sqrt{2}}{3}\cdot\frac{1}{4} \sqrt{10-2 \sqrt{5}}\\
=&\frac{1+\sqrt{5}}{12}+\frac{\sqrt{2}}{6}\cdot \sqrt{10-2 \sqrt{5}}\approx0.823842,
\end{align*}
\begin{align*}
\cos(7\pi/10)=\cos(\pi/2+\pi/5)=-\sin(\pi/5)=-\frac{1}{4} \sqrt{10-2 \sqrt{5}}\approx-0.587785,
\end{align*}
\begin{align*}
\sin(7\pi/10)=\sin(\pi/2+\pi/5)=\cos(\pi/5)=\frac{1}{4}(1+\sqrt{5})\approx0.809017,
\end{align*}
and
\begin{align*}
&\cos(\tht_0+7\pi/10)=\cos(\tht_0)\cos(7\pi/10)-\sin(\tht_0)\sin(7\pi/10)\\
=&
-\frac{1}{3}\cdot\frac{1}{4} \sqrt{10-2 \sqrt{5}}-\frac{2\sqrt{2}}{3}\cdot\frac{1}{4}(1+\sqrt{5})\\
=&-\frac{\sqrt{10-2 \sqrt{5}}}{12}-\frac{\sqrt{2}}{6}\cdot (1+\sqrt{5})\approx-0.958677.
\end{align*}

{\bf Step 1(ii)}  For $\zeta\in l^+_1$, show $pr_{+}(Q(\zeta))<pr_{+}(cv_Q)$.

By Lemma \ref{paraequ-12}, one has
\beqq
Q(\zeta)=\zeta+(4m+2)+
\frac{8m(m+1)+1}{\zeta}+Q_2(\zeta).
\eeqq
By \eqref{paraequ-28}, one has
\beqq
|Q_2(\zeta)|\leq Q_{2,max}(r)\leq \frac{7}{6}m+\frac{1}{m}+6,\ |\zeta|=r\geq 4m+1.
\eeqq
By direct computation, one has, for $\zeta\in\l^+_1$,
\beq\label{equ2021-10-16-1}
\zeta=(4m+1)e^{i\tht_0}+s e^{i(\tht_0+\pi/2)}=[(4m+1)\cos(\tht_0)-s\sin(\tht_0)]+i[(4m+1)\sin(\tht_0)+s\cos(\tht_0)],
\eeq
$$pr_{+}(\zeta)=(4m+1)\cos(\tht_0-\tfrac{\pi}{5})-s\sin(\tht_0-\tfrac{\pi}{5})\leq (4m+1)\cos(\tht_0-\tfrac{\pi}{5}),$$
$$\tht_0\leq\arg\zeta\leq\tht_0+\tfrac{\pi}{2},$$
$$pr_+(4m+2)=(4m+2)\cos(\tfrac{\pi}{5}),$$
\begin{align*}
&\frac{e^{-i\pi/5}}{\zeta}=\frac{e^{-i\pi/5}}{(4m+1)e^{i\tht_0}+s e^{i(\tht_0+\pi/2)}}\\
=&\frac{e^{-i\pi/5}}{[(4m+1)\cos(\tht_0)-s\sin(\tht_0)]+i[(4m+1)\sin(\tht_0)+s\cos(\tht_0)]}\\
=&e^{-i\pi/5}\frac{[(4m+1)\cos(\tht_0)-s\sin(\tht_0)]-i[(4m+1)\sin(\tht_0)+s\cos(\tht_0)]}{(4m+1)^2+s^2}\\
=&e^{-i\pi/5}\frac{(4m+1)e^{-i\tht_0}+se^{-i(\tht_0+\pi/2)}}{(4m+1)^2+s^2}\\
=&\frac{(4m+1)e^{-i(\tht_0+\pi/5)}+se^{-i(\tht_0+7\pi/10)}}{(4m+1)^2+s^2},
\end{align*}
and
\begin{align*}
  pr_{+}\bigg(\frac{e^{-i\pi/5}}{\zeta}\bigg)=
\frac{(4m+1)\cos(-(\tht_0+\pi/5))+s\cos(-(\tht_0+7\pi/10))}{(4m+1)^2+s^2}.
\end{align*}

So, for $m\geq22$, one has
\begin{align*}
&pr_{+}(Q(\zeta))\\
\leq&(4m+1)\cos(\tht_0-\tfrac{\pi}{5})+(4m+2)\cos(\tfrac{\pi}{5})\\
&+
\frac{(4m+1)\cos(-(\tht_0+\pi/5))+s\cos(-(\tht_0+7\pi/10))}{(4m+1)^2+s^2}\cdot (8m(m+1)+1)+\frac{7}{6}m+\frac{1}{m}+6
\\
\approx&(4m+1)\cdot0.823842+(4m+2)\cdot 0.809017\\
&+
\frac{(4m+1)\cdot(-0.284497)+s\cdot(-0.958677)}{(4m+1)^2+s^2}\cdot (8m(m+1)+1)+\frac{7}{6}m+\frac{1}{m}+6
\\
\leq& (4m+1)\cdot0.823842+(4m+1)\cdot 0.809017+0.809017+
\frac{7}{24}\cdot(4m+1)-\frac{7}{24}+\frac{1}{m}+6
\\
\leq&\bigg(1.63286+\frac{7}{24}\bigg)\cdot(4m+1)+0.809017-\frac{7}{24}+\frac{1}{22}+6\\
\approx&1.92453\cdot(4m+1)+6.5628<2(4m+1)+7<2.1(4m+1)\\
<&pr_{+}(cv_Q)=4\frac{(m+1)^{m+1}}{m^m}\cos(\tfrac{\pi}{5})\approx 2.15115\cdot4(m+1),
\end{align*}
where $0.1\times 4\times 22=8.8$,
$\frac{23^{22}}{22^{22}}=\frac{907846434775996175406740561329}{341427877364219557396646723584}
\approx2.65897$,
and
$2.65897\times \cos(\tfrac{\pi}{5})\approx2.15115$ are used.

{\bf Step 1(iii)}  Show $\Im Q(\zeta)>0$.

Note $l^+_1$ is a half line intersect orthogonally
$\{\zeta':\ \arg\zeta'=\tht_0\}$ at distance $4m+1$ from the origin. Its image by $\zeta:\to\tfrac{1}{\zeta}$ is on the circle that passes through $0$ and intersects orthogonally $\{\zeta':\ \arg\zeta'=-\tht_0\}$ at distance $\tfrac{1}{4m+1}$ from the origin. So, the center of the circle is $\tfrac{1}{2(4m+1)}e^{-i\tht_0}$ and the radius is  $\tfrac{1}{2(4m+1)}$. The imaginary part on this circle is $-\tfrac{1}{2(4m+1)}(1+\sin(\tht_0))$. Hence, the imaginary part is
\begin{align*}
&\mbox{Im}\, Q(\zeta)\\
\geq& (4m+1)\sin(\tht_0)+s\cos(\tht_0)-(8m(m+1)+1)\cdot \tfrac{1}{2(4m+1)}(1+\sin(\tht_0))-\bigg(\frac{7}{6}m+\frac{1}{m}+6\bigg)\\
=& (4m+1)\frac{2\sqrt{2}}{3}+\frac{s}{3}-(8m(m+1)+1)\cdot \tfrac{1}{2(4m+1)}\bigg(1+\frac{2\sqrt{2}}{3}\bigg)-\bigg(\frac{7}{6}m+\frac{1}{m}+6\bigg)\\
\approx& (4m+1)\cdot0.942809+\frac{s}{3}-(8m(m+1)+1)\cdot \tfrac{1}{2(4m+1)}(1+0.942809)-\bigg(\frac{7}{6}m+\frac{1}{m}+6\bigg)\\
>&(4m+1)\cdot0.942809-(m+1)\cdot1.942809-\frac{7}{6}m-7\\
\approx&0.66176m-8\geq0.66176\times 22-8=6.5587
>0.
\end{align*}

{\bf Step 1(iv)} Show that $l^+_1$ does not intersect $ \cd\big(   \cot(\tfrac{\pi}{2(m+1)})i,\tfrac{1}{\sin(\tfrac{\pi}{2(m+1)})}\big)$. Similarly, one can show that $l^{-}_1$ does not intersect
$\cd\big(-\cot(\tfrac{\pi}{2(m+1)})i,\tfrac{1}{\sin(\tfrac{\pi}{2(m+1)})}\big)$.

The intersection of $l^+_1$ and the imaginary axis is the point $i\frac{(4m+1)}{\sin(\tht_0)}$, since this solution is corresponding to the $s$ such that
the real part of \eqref{equ2021-10-16-1} is zero, that is,
$s=\frac{(4m+1)\cos(\tht_0)}{\sin(\tht_0)}$, plugging this $s$ into the imaginary part of \eqref{equ2021-10-16-1}, we have the intersection of $l^+_1$ and the imaginary axis.

Denote by $w*$ the intersection of $l^+_1$ and the orthogonal line of $l^+_1$ passing through the point $\cot(\tfrac{\pi}{2(m+1)})i$.
Consider two triangles, one with the vertices $0$, $i\frac{(4m+1)}{\sin(\tht_0)}$, and $(4m+1)e^{i\tht_0}$, another with the vertices $\cot(\tfrac{\pi}{2(m+1)})i$, $i\frac{(4m+1)}{\sin(\tht_0)}$, and $w^*$. These two triangles are similar. By using this fact, the distance between the point $\cot(\tfrac{\pi}{2(m+1)})i$ and the line $l^+_1$ can be obtained as follows:
$$(4m+1)\cdot\frac{\tfrac{4m+1}{\sin(\tht_0)}-\cot(\tfrac{\pi}{2(m+1)})}{\tfrac{4m+1}{\sin(\tht_0)}}
=(4m+1)-\cot(\tfrac{\pi}{2(m+1)})\sin(\tht_0).$$

By (b) of Lemma \ref{paraequ-13}, $\frac{3}{\pi}x\leq \sin(x)$ for $0\leq x\leq\tfrac{3}{2\pi}$. So, for $m\geq22$, one has $\tfrac{\pi}{2(m+1)}<\tfrac{3}{2\pi}$ and
\begin{align*}
4m+1>\frac{\pi}{3}\cdot(\frac{4(m+1)}{\pi})=
\frac{2}{\tfrac{3}{\pi}\cdot(\frac{\pi}{2(m+1)})}\geq\frac{2}{\sin(\frac{\pi}{2(m+1)})}
\geq\frac{\cos(\tfrac{\pi}{2(m+1)})\sin(\tht_0)+1}{\sin(\tfrac{\pi}{2(m+1)})}.
\end{align*}
Hence,
\beqq
(4m+1)-\cot(\tfrac{\pi}{2(m+1)})\sin(\tht_0)>\frac{1}{\sin(\tfrac{\pi}{2(m+1)})},
\eeqq
implying that $l^+_1$ does not intersect $ \cd\big(   \cot(\tfrac{\pi}{2(m+1)})i,\tfrac{1}{\sin(\tfrac{\pi}{2(m+1)})}\big)$.

{\bf Step 2(i)} Consider the equation $Q(\zeta)-cv_Q$ for $\zeta$ on the arc $\ga_1$, that is, $\zeta=\ga_1(\tht)=cp_{Q1}\cdot e^{i\tht}$.

Note that
$\ga_1(\tht)=cp_{Q1}\cdot e^{i\tht}$ with $\tht\in[-(\tht_0+\tht_1),(\tht_0+\tht_1)]$, where $\tht_1=\arccos(\tfrac{4m+1}{cp_{Q1}})$, by Lemmas
\ref{paraequ-25} and \ref{paraequ-24}, one has $\arccos(\tfrac{4m+1}{cp_{Q1}})=\arcsin\bigg(\tfrac{\sqrt{cp^2_{Q1}-(4m+1)^2}}{cp_{Q1}}\bigg)$, and
\begin{align*}
\frac{\sqrt{cp^2_{Q1}-(4m+1)^2}}{cp_{Q1}}
=\frac{\sqrt{(\sqrt{m+1}+\sqrt{m})^2-(4m+1)^2}}{(\sqrt{m+1}+\sqrt{m})^2}
=\frac{-8m^2+4(2m+1)\sqrt{m(m+1)}}{(\sqrt{m+1}+\sqrt{m})^2}
\end{align*}
So,
\begin{align*}
&\frac{\sqrt{m}}{2(m+1)}=\frac{-8m^2+4(2m+1)m}{4(m+1)}\\
\leq&\frac{-8m^2+4(2m+1)\sqrt{m(m+1)}}{(\sqrt{m+1}+\sqrt{m})^2}\\
\leq&\frac{-8m^2+4(2m+1)(m+1)}{4m}=\frac{\sqrt{3m+1}}{2m}.
\end{align*}
Hence, one has, for $m\geq22$,
\beq\label{equ2021-11-25-1}
\tht_0+\tht_1\leq\arccos\bigg(\frac{1}{3}\bigg)+\frac{\sqrt{3m+1}}{2m}
\leq\arccos\bigg(\frac{1}{3}\bigg)+\frac{\sqrt{3\times22+1}}{2\times22}\approx
1.41699=0.451042\pi.
\eeq

It is evident that $Q(cp_{Q1})-cv_Q=0$. By the derivative of $Q$ in \eqref{paraequ-26}, one has
\begin{align*}
&Q(\zeta)-cv_Q=
\int_{\ga_1}Q'(z)dz\\
=&\int_{\ga_1}\bigg(1-\frac{(\sqrt{m+1}+\sqrt{m})^2 }{z}\bigg)\cdot\bigg(1-\frac{(\sqrt{m+1}-\sqrt{m})^2}{z}\bigg)\cdot\bigg(\frac{ 1+\frac{1}{z}}{1-\frac{1}{z}}\bigg)^{2 m+1}dz\\
=&\int_{\ga_1}\bigg(1-\frac{(\sqrt{m+1}+\sqrt{m})^2 }{z}\bigg)\cdot\bigg(1-\frac{(\sqrt{m+1}-\sqrt{m})^2}{z}\bigg)\cdot\bigg(\frac{ z+1}{z-1}\bigg)^{2 m+1}dz.
\end{align*}

{\bf Step 2(ii)}
For $z\in\ga_1$, one has, for $\tht\geq0$,
\begin{align*}
&\bigg(1-\frac{(\sqrt{m+1}+\sqrt{m})^2 }{z}\bigg)
=1-e^{-i\tht}=1-(\cos\tht-i\sin\tht)=(1-\cos\tht)+i\sin\tht\\
=&2\sin^2\frac{\tht}{2}+i2\sin\frac{\tht}{2}\cos\frac{\tht}{2}=
2\sin\frac{\tht}{2}\bigg(\sin\frac{\tht}{2}+i\cos\frac{\tht}{2}\bigg)
=2\sin\frac{\tht}{2}\bigg(\cos\bigg(\frac{\pi}{2}-\frac{\tht}{2}\bigg)+i\sin\bigg(\frac{\pi}{2}-\frac{\tht}{2}\bigg)\bigg).
\end{align*}
Note that, for $\tht\leq0$, the above formula should be written as
\begin{align*}
\bigg(1-\frac{(\sqrt{m+1}+\sqrt{m})^2 }{z}\bigg)
=2\sin\bigg(\frac{-\tht}{2}\bigg)\bigg(\cos\bigg(\frac{3\pi}{2}-\frac{\tht}{2}\bigg)+i\sin\bigg(\frac{3\pi}{2}-\frac{\tht}{2}\bigg)\bigg).
\end{align*}

{\bf Step 2(iii)}
For $z\in\ga_1$, one has
\begin{align*}
\bigg(1-\frac{(\sqrt{m+1}-\sqrt{m})^2}{z}\bigg)
=1-\frac{(\sqrt{m+1}-\sqrt{m})^2}{(\sqrt{m+1}+\sqrt{m})^2}e^{-i\tht}
=1-\frac{e^{-i\tht}}{(\sqrt{m+1}+\sqrt{m})^4}.
\end{align*}

By Lemma \ref{paraequ-13}, one has
\begin{align*}
&\bigg|\arg\bigg(1-\frac{(\sqrt{m+1}-\sqrt{m})^2}{(\sqrt{m+1}+\sqrt{m})^2}e^{-i\tht}\bigg)\bigg|=
\bigg|\arg\bigg(1-\frac{(\sqrt{m+1}-\sqrt{m})^2}{(\sqrt{m+1}+\sqrt{m})^2}e^{-i\tht}\bigg)-\arg 1\bigg|\\
\leq& \arcsin\bigg(\frac{1}{(\sqrt{m+1}+\sqrt{m})^4}\bigg)\leq
\frac{\pi}{3}\cdot\frac{1}{(\sqrt{m+1}+\sqrt{m})^4}.
\end{align*}

{\bf Step 2(iv)} Consider the image of the circle $\{z:\ |z|=\sqrt{R}\}$ under the map $w=\psi_3(z)=\frac{z+1}{z-1}$, where $R>1$ is a constant. The expression for the circle is $z\ol{z}=R$. Substituting $z=\frac{w+1}{w-1}$ (the inverse of $w=\frac{z+1}{z-1}$) into $z\ol{z}=R$, one has
\beqq
\bigg(\frac{w+1}{w-1}\bigg)\cdot\bigg( \ol{\frac{w+1}{w-1}}\bigg)=R.
\eeqq
So,
\beqq
(R-1)w\ol{w}-(R+1)(w+\ol{w})+R-1=0.
\eeqq
Suppose $w=x+iy$, one has
\beqq
(x^2+y^2)-\frac{R+1}{R-1}(2x)+1=0,
\eeqq
which can be written as
\beqq
\bigg(x-\frac{R+1}{R-1}\bigg)^2+y^2=\frac{4R}{(R-1)^2}.
\eeqq
This is a circle with center $\frac{R+1}{R-1}$ and radius $\frac{2\sqrt{R}}{R-1}$, denoted by $S_0$. The intersection of the image of the circle $\{z:\ z\ol{z}=R\}$
and the $x$-axis is at the points $\frac{\sqrt{R}-1}{\sqrt{R}+1}$ and $\frac{\sqrt{R}+1}{\sqrt{R}-1}$.

By direct computation, one has $\psi_{3}(\sqrt{R})=\frac{\sqrt{R}+1}{\sqrt{R}-1}$,
$\psi_{3}(-\sqrt{R})=\frac{\sqrt{R}-1}{\sqrt{R}+1}$,
\begin{align*}
&\psi_{3}(\sqrt{R}e^{i\tfrac{\pi}{2}})=
\frac{\sqrt{R}i+1}{\sqrt{R}i-1}=\frac{(\sqrt{R}i+1)^2}{-(\sqrt{R})^2-1}
=\frac{-R+2i\sqrt{R}+1}{-R-1}\\
=&\frac{R-1}{R+1}-i\frac{2\sqrt{R}}{R+1},
\end{align*}
\begin{align*}
&\psi_{3}((\sqrt{m+1}+\sqrt{m})^2)=
\frac{(\sqrt{m+1}+\sqrt{m})^2+1}{(\sqrt{m+1}+\sqrt{m})^2-1}
=\frac{\sqrt{m+1}}{\sqrt{m}},
\end{align*}
and
\begin{align*}
&\psi_{3}(re^{i\tht})=
\frac{re^{i\tht}+1}{re^{i\tht}-1}=
\frac{r\cos\tht+ir\sin\tht+1}{r\cos\tht+ir\sin\tht-1}
=\frac{[(r\cos\tht+1)+ir\sin\tht][(r\cos\tht-1)-ir\sin\tht]}{(r\cos\tht-1)^2+r^2\sin^2\tht}\\
=&\frac{r^2-1-2ir\sin\tht}{r^2-2r\cos\tht+1}
=\frac{r^2-1}{r^2-2r\cos\tht+1}
-i\frac{2r\sin\tht}{r^2-2r\cos\tht+1}.
\end{align*}
The estimate of the argument of this complex number is given by this following item:
\begin{align*}
\arctan\bigg(\frac{-2r\sin\tht}{r^2-1}\bigg).
\end{align*}
For $r=cp_{Q1}\geq 4m+1$, one has
\begin{align}\label{equ2021-11-25-2}
\frac{2r(2m+1)}{r^2-1}=(2m+1)\bigg(\frac{1}{r+1}+\frac{1}{r-1}\bigg)
\leq(2m+1)\bigg(\frac{1}{4m+2}+\frac{1}{4m}\bigg)=1+\frac{1}{4m},
\end{align}
and for $r=cp_{Q1}\leq 4(m+1)$, one has
\begin{align}\label{equ2021-11-25-3}
&\frac{2r(2m+1)}{r^2-1}=(2m+1)\bigg(\frac{1}{r+1}+\frac{1}{r-1}\bigg)
\geq(2m+1)\bigg(\frac{1}{4(m+1)+1}+\frac{1}{4(m+1)-1}\bigg)\nonumber\\
=&\frac{4m+2}{4(m+1)+1}=1-\frac{2}{4(m+1)+1}.
\end{align}

{\bf Step 2(v)} By the formula of Taylor expansion with Lagranginan error term, one has
\beqq
(\arctan x)^{\prime}=\frac{1}{1+x^2},\ \bigg(\frac{1}{1+x^2}\bigg)^{\prime}=\frac{-2x}{(1+x^2)^2},
\eeqq
and
\beqq
\bigg(\frac{-2x}{(1+x^2)^2}\bigg)^{\prime}=
\frac{-2(1+x^2)^2+2x(4x(1+x^2))}{(1+x^2)^4}
=\frac{6x^2-2}{(1+x^2)^3}.
\eeqq
So,
\beqq
\bigg|\frac{6x^2-2}{(1+x^2)^3}\bigg|\leq
\frac{6x^2+6}{(1+x^2)^3}=\frac{6}{(1+x^2)^2}\leq 6.
\eeqq
So,
\beqq
x-x^3\leq\arctan x\leq x+x^3\ \forall x\geq0.
\eeqq
Similarly, one has
\beqq
(\sin x)^{\prime}=\cos x,\ (\cos x)^{\prime}=-\sin x,\ (-\sin x)^{\prime}=-\cos x,
\eeqq
\beqq
x-\frac{x^3}{6}\leq\sin x\leq x,\  \forall x\geq0.
\eeqq
Hence, by $m\geq22$, $r=cp_{Q1}\geq 4m+1$, \eqref{equ2021-11-25-1}, \eqref{equ2021-11-25-2}, and the above discussions, one has, for $\zeta=cp_{Q1}\cdot e^{i\tht}$ with $\tht\in[0,\tht_0+\tht_1]$,
\begin{align*}
 &\arg(Q(\zeta)-cv_Q)\\
\geq&  \frac{\pi}{2}-\frac{\tht}{2}-\frac{\pi}{3}\cdot\frac{1}{(\sqrt{m+1}+\sqrt{m})^4}-\frac{2r(2m+1)\sin\tht}{r^2-1}-(2m+1)\bigg(\frac{2r\sin\tht}{r^2-1}\bigg)^3+\frac{\pi}{2}+\tht\\  \geq&
\frac{\pi}{2}-\frac{\tht}{2}-\frac{\pi}{3}\cdot\frac{1}{(\sqrt{m+1}+\sqrt{m})^4}-\bigg(1+\frac{1}{4m}\bigg)\tht-\bigg(1+\frac{1}{4m}\bigg)\bigg(\frac{1}{2m}\bigg)^2+\frac{\pi}{2}+\tht\\ \geq&
\pi-\frac{\tht}{2}-\frac{\pi}{3}\cdot\frac{1}{(\sqrt{22+1}+\sqrt{22})^4}-\frac{\tht}{4\times22}-\bigg(1+\frac{1}{4\times22}\bigg)\bigg(\frac{1}{2\times22}\bigg)^2\\ \geq&\pi-\frac{1.41699}{2}-\frac{\pi}{3}\cdot\frac{1}{(\sqrt{22+1}+\sqrt{22})^4}-\frac{1.41699}{4\times22}-\bigg(1+\frac{1}{4\times22}\bigg)\bigg(\frac{1}{2\times22}\bigg)^2\\
\approx&\pi-0.225521\pi-0.0000411624\pi-0.00512548\pi-0.000166285\pi\\
=&0.769146\pi>\frac{7\pi}{10}.
\end{align*}

Similarly, one has
\begin{align*}
  &\arg(Q(\zeta)-cv_Q)\\
  \leq&\frac{\pi}{2}-\frac{\tht}{2}+\frac{\pi}{3}\cdot\frac{1}{(\sqrt{m+1}+\sqrt{m})^4}-\frac{2r(2m+1)\sin\tht}{r^2-1}+(2m+1)\bigg(\frac{2r\sin\tht}{r^2-1}\bigg)^3+\frac{\pi}{2}+\tht\\
\leq&
\pi+\frac{\tht}{2}+\frac{\pi}{3}\cdot\frac{1}{(\sqrt{m+1}+\sqrt{m})^4}+\bigg(1-\frac{2}{4(m+1)+1}\bigg)\bigg(-\tht+\frac{\tht^3}{6}\bigg)+\bigg(1+\frac{1}{4m}\bigg)\bigg(\frac{1}{2m}\bigg)^2\\ \leq&
\pi+\frac{\tht}{2}+\frac{\pi}{3}\cdot\frac{1}{(\sqrt{22+1}+\sqrt{22})^4}+\bigg(1+\frac{1}{4\times22}\bigg)\bigg(\frac{1}{2\times22}\bigg)^2\\ \leq&\pi+\frac{1.41699}{2}+\frac{\pi}{3}\cdot\frac{1}{(\sqrt{22+1}+\sqrt{22})^4}+\bigg(1+\frac{1}{4\times22}\bigg)\bigg(\frac{1}{2\times22}\bigg)^2\\
\approx&\pi+0.225521\pi+0.0000411624\pi+0.000166285\pi\\
=&1.225728\pi<\frac{13\pi}{10},
\end{align*}
where $-\tht+\frac{\tht^3}{6}\leq0$ for $\tht\in[0,\tht_0+\tht_1]$ because of $\sqrt{6}\approx2.44949>1.41699$ by \eqref{equ2021-11-25-1}.

By applying similar discussions, one has, for $\zeta=cp_{Q1}\cdot e^{i\tht}$ with $\tht\in[-(\tht_0+\tht_1),0]$,
\beqq
\frac{7\pi}{10}<\arg(Q(\zeta)-cv_Q)<\frac{13\pi}{10}.
\eeqq

{\bf (b)} For $\zeta\in\widetilde{W}_0$, $|\zeta|\geq 4m+1>7$, by Lemma \ref{paraequ-24}, \eqref{paraequ-9}, \eqref{paraequ-30-1-21-2}, \eqref{near-equat-7}, and the definition of the region $W_0$,
one can show the conclusion.

{\bf (c)} This part is divided into four parts.

{\bf Step (i)} Suppose $\text{Re}\,\zeta\leq0$ and $1\leq|\zeta|\leq4m-3$, one has $|\zeta+1|\leq|\zeta-1|$ and
\begin{align*}
&|Q(\zeta)|=\bigg|\frac{(1+\zeta)^{2+2m}}{\zeta(1-\zeta)^{2m}}\bigg|=\bigg|\frac{(1+\zeta)(1-\zeta)}{\zeta}\cdot\frac{(1+\zeta)^{2m+1}}{(1-\zeta)^{2m+1}}\bigg|
\\
=&\bigg|\bigg(\frac{1}{\zeta}-\zeta\bigg)\bigg(\frac{1+\zeta}{1-\zeta}\bigg)^{2m+1}\bigg|\leq|\zeta|+\frac{1}{|\zeta|}\leq 4m-3+\frac{1}{4m-3}\\
<&4m-3+\frac{1}{4*22-3}\approx\ 4m-2.98824<4m-2,
\end{align*}
where the discussions in Remark \ref{near-equat-11} are used in the last but one inequality.

{\bf Step (ii)} The estimate of $|Q_2(\zeta)|$ for $|\zeta|\geq4m-3$ is provided, where $Q_2(\zeta)$ is introduced in \eqref{paraequ-66}.

Following the arguments in the proof of Lemma \ref{paraequ-12}, one has,
for $r\geq 4m-3$,
\begin{align}\label{near-equat-2021-11-22-1}
 &\bigg(\frac{r+1}{r-1}\bigg)^{2m-4}= \bigg(1+\frac{2}{r-1}\bigg)^{2m-4}\leq \bigg(1+\frac{2}{4m-4}\bigg)^{2m-4}\nonumber\\
=&\bigg(1+\frac{1}{2m-2}\bigg)^{2m-2}\cdot\bigg(1+\frac{1}{2m-2}\bigg)^{-2}
\leq e\cdot1=e
\end{align}
and
\begin{align}\label{near-equat-2021-11-22-2}
 &\bigg(\frac{r+1}{r-1}\bigg)^{2m-2}= \bigg(1+\frac{2}{r-1}\bigg)^{2m-2}\leq \bigg(1+\frac{2}{4m-4}\bigg)^{2m-2}\nonumber\\
=&\bigg(1+\frac{1}{2m-2}\bigg)^{2m-2}\leq e.
\end{align}

Hence, one has
\begin{itemize}
\item for $a_1=\tfrac{2m}{3}\big(\tfrac{m}{m-1}\big)^2$, if \beqq
r\geq(\tfrac{32}{3a_1}m^3)^{1/2}+1=
(16(m-1)^2)^{1/2}+1=4m-3,
\eeqq
then
\begin{align*}
\frac{8C^3_{2m}}{r(r-1)}<\frac{\tfrac{4}{3}(2m)^3}{(r-1)^2}<a_1
=\tfrac{2m}{3}\big(1+\tfrac{1}{m-1}\big)^2\leq
\frac{2m}{3}\bigg(\frac{22}{21}\bigg)^2\approx 0.73167m;
\end{align*}
\item for $a_1=\tfrac{m}{(m-1)^2}$, if $r>(\tfrac{16m}{a_1})^{1/2}+1=4m-3$, then
\begin{align*}
\frac{16m}{r(r-1)}<\frac{16m}{(r-1)^2}<a_1
=\tfrac{m}{(m-1)^2}<\frac{4m}{m^2}=\frac{4}{m};
\end{align*}
\item for $a_1=\tfrac{m^4}{2(m-1)^3}$, if $r>(\tfrac{2^5m^4}{a_1})^{1/3}+1=4m-3$, then
\begin{align*}
&\frac{2(2m)(2m-1)(2m-2)(2m-3)}{3r(r-1)^2}\bigg(\frac{r+1}{r-1}\bigg)^{2m-4}
<\frac{2(2m)^4}{3(r-1)^3}e<\frac{2(2m)^4}{(r-1)^3}\\
<&a_1
=\tfrac{m^4}{2(m-1)^3}
=\frac{m}{2}\bigg(\frac{m}{m-1}\bigg)^3
\leq\frac{m}{2}\bigg(\frac{22}{21}\bigg)^3\approx 0.574884m;
\end{align*}
\item for $a_1=\tfrac{6m^2}{(m-1)^2}$, if $r>\tfrac{4\sqrt{6}m}{\sqrt{a_1}}+1=4m-3$, then
\begin{align*}
&\frac{8(2m)(2m-1)}{(r-1)^2}\bigg(\frac{r+1}{r-1}\bigg)^{2m-2}
<\frac{8(2m)^2}{(r-1)^2}e<\frac{24(2m)^2}{(r-1)^2}\\
<&a_1=\frac{6m^2}{(m-1)^2}\leq6\bigg(\frac{22}{21}\bigg)^2
\approx6.58503.
\end{align*}
\end{itemize}
Hence, for $|\zeta|=r\geq4m-3$, one has
\beqq
|Q_2(\zeta)|\leq0.73167m+\frac{4}{m}+0.574884m+6.58503\approx
1.30655m+\frac{4}{m}+6.58503.
\eeqq

{\bf Step (iii)}
For the $\tht_0\in(0,\pi/2)$ with $\cos(\tht_0)=\tfrac{1}{3}$ and $|\zeta|=r\geq 4m-3$ and $\tfrac{9\pi}{10}\leq\arg\,\zeta\leq\tfrac{11\pi}{10}$, one would show
\begin{align*}
\text{Re}(Q(\zeta) e^{-i\tht_0})<4m-2.
\end{align*}

As for $\text{Re}(Q(\zeta) e^{-i\tht_0})$,
$|\zeta|=r\geq 4m-3$ and $\tfrac{9\pi}{10}\leq\arg\,\zeta\leq\tfrac{11\pi}{10}$,
one has $\text{Re}(\zeta\cdot e^{-i\tht_0})\leq0$ and $-\tfrac{11\pi}{10}-\tht_0\leq\arg\,(\tfrac{e^{-i\tht_0}}{\zeta})\leq-\tfrac{9\pi}{10}-\tht_0$, one has
\begin{align*}
&\text{Re}(Q(\zeta) e^{-i\tht_0})\\
\leq& 0+\text{Re}((4m+2) e^{-i\tht_0})+\frac{(8m(m+1)+1)\cos(-11\pi/10-\tht_0)}{r}+
\text{Re}\big( Q_2(\zeta)\cdot e^{-i\tht_0}\big)\\
\leq&(4m+2)\cos(\tht_0)+0+|Q_2(\zeta)|\\
\leq&(4m+2)\cos(\tht_0)+1.30655m+\frac{4}{m}+6.58503\\
\approx&7.2517 + \frac{4}{m} + 2.63988 m
<4m-2.
\end{align*}

{\bf Step (iv)}
As for $\text{Re}(Q(\zeta)\cdot e^{i\tht_0})$,
$|\zeta|=r\geq 4m-3$ and $\tfrac{9\pi}{10}\leq\arg\,\zeta\leq\tfrac{11\pi}{10}$,
one has $\text{Re}(\zeta\cdot e^{i\tht_0})\leq0$ and $-\tfrac{11\pi}{10}+\tht_0\leq\arg\,(\tfrac{e^{i\tht_0}}{\zeta})\leq-\tfrac{9\pi}{10}+\tht_0$, one has
\begin{align*}
&\text{Re}(Q(\zeta) e^{i\tht_0})\\
\leq&0+\text{Re}((4m+2) e^{i\tht_0})+\frac{(8m(m+1)+1)\cos(-9\pi/10+\tht_0)}{r}+
\text{Re}\big( Q_2(\zeta)\cdot e^{i\tht_0}\big)\\
\leq&(4m+2)\cos(\tht_0)+0+|Q_2(\zeta)|\\
\leq&(4m+2)\cos(\tht_0)+1.30655m+\frac{4}{m}+6.58503\\
\approx&7.2517 + \frac{4}{m} + 2.63988 m
<4m-2.
\end{align*}

\end{proof}

Finally, a simpler estimate is given.
\begin{lemma}
Let
\beqq
\widetilde{W}_0:=\bigg\{\zeta:\ \text{Re}\,\zeta>cp_Q\ \text{or}\ pr_{+}(\zeta)>\tfrac{cp_Q}{\cos(\tfrac{\pi}{5})}\ \text{or}\ pr_{-}(\zeta)>\tfrac{cp_Q}{\cos(\tfrac{\pi}{5})}\bigg\}.
\eeqq Then,
\begin{itemize}
\item[(a)]
$\varphi(\widetilde{W}_0)\subset W_0:=\{z:\  \text{Re}\,z>4m+0.18\ \text{or}\ pr_{+}(z)>4.9m+0.5\ \text{or}\ pr_{-}(z)>4.9m+0.5\}$;
\item[(b)] $Q^{-1}(W_0)\setminus\overline{\cd}\subset\widetilde{W}_{-1}:=\vv(0,\tfrac{7\pi}{10})\setminus(\overline{\cd}\cup\{\zeta:\ \text{Re}\,\zeta\leq0\ \text{and}\ |\zeta|\leq4m\})$.
\end{itemize}
\end{lemma}

\begin{proof}
{\bf Case (a)} Suppose $\zeta\in\widetilde{W}_0$,
if $\text{Re}\,\zeta>cp_Q$, then, by Lemma \ref{paraequ-24}, \eqref{paraequ-9} and \eqref{paraequ-30-1-21-1},
\begin{align*}
&\text{Re}\,\varphi(\zeta)\geq cp_Q-e_0-2|e_1|-\varphi_{1,max}(cp_Q)\\
\geq& 4m+1.89898+0.18-2\times 0.84-0.214541=4m+0.18444>4m+0.18;
\end{align*}
if $pr_{\pm}(\zeta)>\tfrac{cp_Q}{\cos(\tfrac{\pi}{5})}$, then,
\begin{align*}
&\text{Re}\,\varphi(\zeta)\geq \tfrac{cp_Q}{\cos(\tfrac{\pi}{5})}-e_0\cos(\tfrac{\pi}{5})-2|e_1|-\varphi_{1,max}(\tfrac{cp_Q}{\cos(\tfrac{\pi}{5})})\\
\geq& \frac{16m}{1+\sqrt{5}}+ \frac{4\times1.89898}{1+\sqrt{5}}+0.18\times0.809017-2\times 0.84-0.214541\\
=& \frac{16m}{1+\sqrt{5}}+2.34727-1.74892\approx4.94427m+0.59835>4.9m+0.5.
\end{align*}

{\bf Case (b)} By the assumption, one has $\cc\setminus(\widetilde{W}_{-1}\cup\ol{\cd})=
\{\zeta:\ \text{Re}\,\zeta\leq0\ \text{and}\ 1<|\zeta|\leq4m\}\cup\{\zeta:\ |\zeta|\geq 4m\ \text{and}\ \tfrac{3\pi}{10}\leq\zeta\leq\tfrac{13\pi}{10}\}$, implying that one needs to prove that if $\zeta$ is not in this set, then $Q(\zeta)\not\in W_0$.

Suppose $\text{Re}\,\zeta\leq0$ and $1\leq|\zeta|\leq4m$, one has $|\zeta+1|\leq|\zeta-1|$ and
\begin{align*}
&|Q(\zeta)|=\bigg|\frac{(1+\zeta)^{2+2m}}{\zeta(1-\zeta)^{2m}}\bigg|=\bigg|\frac{(1+\zeta)(1-\zeta)}{\zeta}\cdot\frac{(1+\zeta)^{2m+1}}{(1-\zeta)^{2m+1}}\bigg|
\\
=&\bigg|\bigg(\frac{1}{\zeta}-\zeta\bigg)\bigg(\frac{1+\zeta}{1-\zeta}\bigg)^{2m+1}\bigg|\leq|\zeta|+\frac{1}{|\zeta|}\leq 4m+\frac{1}{4m}<4m+0.18\ \forall m\geq2,
\end{align*}
where the discussions in Remark \ref{near-equat-11} are used in the last but one inequality.
Hence, $Q(\zeta)\in\cd(0,4m+0.18)\subset\cc\setminus W_0$.

By \eqref{near-equat-14} in Lemma \ref{paraequ-12}, the discussions in the proof of Lemma \ref{paraequ-12} (\eqref{near-equat-12} and \eqref{near-equat-13}),
for $|\zeta|\geq r\geq4m$,
\begin{align*}
&|\zeta Q_2(\zeta)|\leq
\frac{8C^3_{2m}+16m}{(r-1)}\\
&+\frac{2(2m)(2m-1)(2m-2)(2m-3)}{3(r-1)^2}\bigg(\frac{r+1}{r-1}\bigg)^{2m-4}+\frac{r8(2m)(2m-1)}{(r-1)^2}\bigg(\frac{r+1}{r-1}\bigg)^{2m-2}\\
\leq&\frac{8C^3_{2m}+16m}{4m-1}
+\frac{2(2m)(2m-1)(2m-2)(2m-3)e}{3(4m-1)^2}+\frac{8(4m)(2m)(2m-1)e}{(4m-1)^2}.
\end{align*}
For $m\geq3$, one has
\begin{align*}
\frac{4m}{4m-1}=\frac{1}{1-\frac{1}{4m}}\leq\frac{1}{1-\frac{1}{4\times3}}=\frac{12}{11}.
\end{align*}
So,
\begin{align*}
&\frac{8C^3_{2m}+16m}{4m-1}=\frac{4(2m)(2m-1)(2m-2)+12(4m)}{3(4m-1)}
\leq\frac{2(2m-1)(2m-2)}{3}\cdot\frac{12}{11}+\frac{48}{11}\\
=&\frac{2^4(2m-1)(m-1)}{11}+\frac{48}{11},
\end{align*}
\begin{align*}
&\frac{2(2m)(2m-1)(2m-2)(2m-3)e}{3(4m-1)^2}
=\frac{4m}{4m-1}\cdot\frac{4m-2}{4m-1}
\cdot\frac{(m-1)(2m-3)e}{3}\\
\leq&\frac{12}{11}\cdot\frac{(m-1)(2m-3)e}{3}
=\frac{4(m-1)(2m-3)e}{11},
\end{align*}
and
\begin{align*}
\frac{8(4m)(2m)(2m-1)e}{(4m-1)^2}
=\frac{4m}{4m-1}\cdot\frac{4m-2}{4m-1}\cdot
(2^3me)\leq\frac{12}{11}\cdot(2^3me)=\frac{2^5\cdot(3m)e}{11}.
\end{align*}
Hence, for $|\zeta|\geq4m$,
\begin{align*}
&|\zeta Q_2(\zeta)|\leq\frac{2^4(2m-1)(m-1)}{11}+\frac{48}{11}+\frac{4(m-1)(2m-3)e}{11}+
\frac{2^5\cdot(3m)e}{11}\\
\leq&\frac{2^4(2m-1)(m-1)}{11}+\frac{48}{11}+\frac{4(m-1)(2m-3)\times 2.72}{11}+
\frac{2^5\cdot(3m)\times 2.72}{11}\\
\approx&8.78545 + 14.4291 m + 4.88727 m^2.
\end{align*}

Now, we are ready to study the situation that $|\zeta|=r\geq 4m$ and $\tfrac{7\pi}{10}\leq\arg\,\zeta\leq\pi$. It is evident that $|\zeta-1|\geq|\zeta|=r\geq4m$. By the above discussions and \eqref{paraequ-66}, one has
\begin{align*}
&\text{Re}\,(Q(\zeta))\leq r\cos\left(\frac{7\pi}{10}\right)+(4m+2)+\frac{(8m(m+1)+1)\cos(\tfrac{7\pi}{10})}{r}+
\text{Re}\,\left(\frac{\zeta Q_2(\zeta)}{\zeta}\right)\\
\leq&4m\cos\left(\frac{7\pi}{10}\right)+(4m+2)+\frac{(8m(m+1)+1)\cos(\tfrac{7\pi}{10})}{r}+
\frac{8.78545 + 14.4291 m + 4.88727 m^2}{r}\\
\approx&-2.35114m+(4m+2)+\frac{8.19766 + 9.72682 m + 0.184988 m^2}{4m}\leq 4m.
\end{align*}

As for $pr_{+}(Q(\zeta))$, one has $pr_{+}(\zeta)\leq0$ and $-\tfrac{6\pi}{5}\leq\arg\,(\tfrac{e^{-i\pi/5}}{\zeta})\leq-\tfrac{9\pi}{10}$, one has
\begin{align*}
&pr_{+}(Q(\zeta))\leq 0+pr_{+}(4m+2)+\frac{(8m(m+1)+1)\cos(\tfrac{9\pi}{10})}{r}+
pr_{+}\left(\frac{\zeta Q_2(\zeta)}{\zeta}\right)\\
\leq&(4m+2)\cos\left(\frac{\pi}{5}\right)+\frac{1}{r}\left((8m(m+1)+1)\cos(\tfrac{9\pi}{10})+8.78545 + 14.4291 m + 4.88727 m^2\right)\\
\approx&(4m+2)\cos\left(\frac{\pi}{5}\right)+\frac{1}{4m}\left(7.83439 + 6.82065 m - 2.72118 m^2\right)<4.9m+0.5.
\end{align*}
As for  $pr_{-}(Q(\zeta))$, one has $pr_{-}(\zeta)\leq-4m\cos(\tfrac{\pi}{5})$ and
$-\tfrac{4\pi}{5}\leq\arg\,(\tfrac{e^{i\pi/5}}{\zeta})\leq\-\tfrac{\pi}{2}$, so $pr_{-}(\tfrac{1}{\zeta})\leq0$, one has
\begin{align*}
&pr_{-}(Q(\zeta))\leq-4m\cos\left(\frac{\pi}{5}\right)+pr_{-}(4m+2)+0+
pr_{-}\left(\frac{\zeta Q_2(\zeta)}{\zeta}\right)\\
\leq&-4m\cos\left(\frac{\pi}{5}\right)+(4m+2)\cos\left(\frac{\pi}{5}\right)+\frac{1}{r}\left(8.78545 + 14.4291 m + 4.88727 m^2\right)\\
\approx&2\cos\left(\frac{\pi}{5}\right)+\frac{1}{4m}\left(8.78545 + 14.4291 m + 4.88727 m^2\right)<4.9m+0.5.
\end{align*}

These three inequalities imply that $Q(\zeta)\not\in W_0$. The same conclusion holds for $\pi\leq\arg\,\zeta\leq\tfrac{13\pi}{10}$.
\end{proof}

\subsection{Construction of $\Psi_1$: Relating $D_n$'s to $P$}\label{proofpro2021-10-3-1}

In this subsection, the proof of Proposition \ref{attracting202091203} is provided.

The set $\pi_X(X_{1+}\cup X_{2-})$ contains the sets $D_0$, $D_0^{\prime}$, $D_{-1}$, $D^{\prime\prime}_{-1}$, $D_0^{\sharp}$, and $D^{\sharp}_1$, which could be thought of as the subsets of $X_{1+}\cup X_{2-}$.

Since
\beqq
F(D_0)=F(D_0')=D_1,\ F(D^{\sharp}_{0})=D^{\sharp}_1,\ F(D_{-1})=F(D_{-1}^{\prime\prime})=D_0,\
\eeqq
it is reasonable to define the following sets:
\beqq
D_{-n-1}=g^n(D_{-1}),\ D^{\sharp}_{-n}=g^n(D^{\sharp}_{0}),\ D^{\prime}_{-n}=g^{n}(D_0^{\prime}),\
D^{\prime\prime}_{-n-1}=g^{n}(D_{-1}^{\prime\prime}),\ \forall n\in\mathbb{N}.
\eeqq

Let
\beq\label{equ-2021-11-30-2}
\mathbf{D}=\{z:\ 0<\text{Re}\,z<1\ \text{and}\ |\text{Im}\,z|<\eta\}\ \text{and}\
\mathbf{D}^{\sharp}=\{z:\ 0<\text{Re}\,z<1\ \text{and}\ \eta<\text{Im}\,z\}.
\eeq
The boundary segments can be represented by the following notations:
\begin{align*}
&\partial^{l}_{+}\mathbf{D}=0+i[0,\eta];\ \partial^{l}_{-}\mathbf{D}=0+i[-\eta,0];\ \partial^{r}_{+}\mathbf{D}=1+i[0,\eta];\ \partial^{r}_{-}\mathbf{D}=1+i[-\eta,0];\\
&\partial^{h}_{+}\mathbf{D}=\partial^{h}\mathbf{D}^{\sharp}=[0,1]+i\eta;\ \partial^{h}_{-}\mathbf{D}=[0,1]-i\eta;\ \partial^{l}\mathbf{D}^{\sharp}=0+i[\eta,+\infty];\ \partial^r\mathbf{D}^{\sharp}=1+i[\eta,+\infty],
\end{align*}
where $l$,  $r$, and $h$ stand for left, right, and horizontal boundaries.  Figure \ref{fig2021-2-25-1} is an illustration diagram.

Since $\Phi_{att}(z)-1$ maps $D_1$ onto $\mathbf{D}$ homeomorphically including the boundaries, maps $D^{\sharp}_1$ onto $\mathbf{D}^{\sharp}$ homeomorphically including the boundaries. As a consequence, the boundary segments of $D_1$ and $D^{\sharp}_1$ are denoted by $\partial^l_{+}D_1$, $\partial^hD^{\sharp}_1$ etc according to their images by $\Phi_{att}(z)-1$.  The same notations are also defined for $D_n$, $D^{\prime}_n$, $D^{\prime\prime}_n$, $D^{\sharp}_n$, etc according to their images by iterates of $F$ for $n\leq0$. Figure \ref{fig2021-10-3-3}  is an illustration diagram.

\begin{figure}
\begin{center}
\begin{tikzpicture}[scale=1,line width=1pt]
\path (0,-2) coordinate (p00);
\path (5,-2) coordinate (p01);
\path (5,2) coordinate (p11);
\path (0,2) coordinate (p10);
\draw[color=blue,line width=3pt] (p00)--(p01)--(p11)--(p10)--cycle;
\draw [<->,line width=1.5pt] (0,6.5) node (yaxis) [above] {$y$}
    |- (6.5,0) node (xaxis) [right] {$x$};

\path (-0.25,-0.25) node(axis0){$O$};

\draw[->,color=blue,line width=3pt](0,2)--(0,5) --cycle;
\draw[->,color=blue,line width=3pt](5,2)--(5,5) --cycle;

\path (-0.5,1.5) node(text1)[below]{$\partial^l_{+}D$};
\path (-0.5,-1) node(text1)[below]{$\partial^l_{-}D$};

\path (5.5,1.5) node(text1)[below]{$\partial^r_{+}D$};
\path (5.5,-1) node(text1)[below]{$\partial^r_{-}D$};

\path (2.5,-2) node(text1)[below]{$\partial^h_{-}D$};

\path (2.5,2.8) node(text1)[below]{$\partial^h_{+}D=\partial^hD^{\sharp}$};

\path (-0.5,4.5) node(text1)[below]{$\partial^lD^{\sharp}$};
\path (5.5,4.5) node(text1)[below]{$\partial^rD^{\sharp}$};
\end{tikzpicture}
\end{center}
\caption{The illustration diagram of the boundary segments}\label{fig2021-2-25-1}
\end{figure}
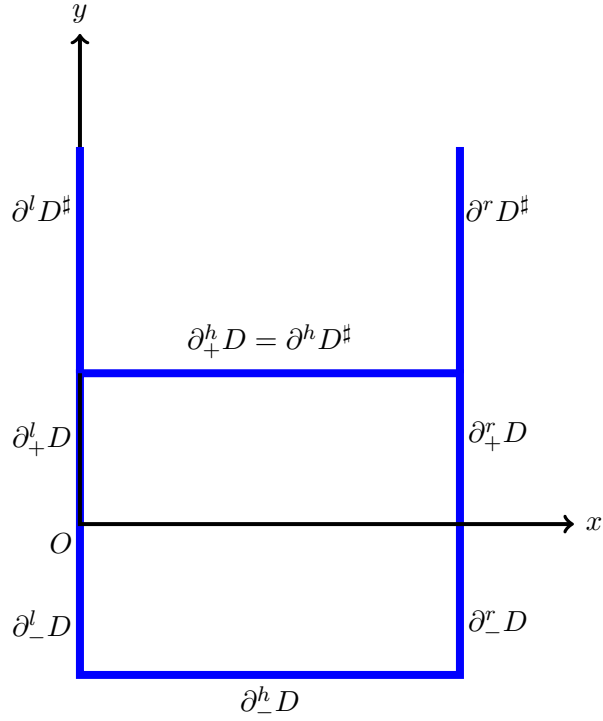

\begin{lemma}\label{gra-2021-10-3-2}
\begin{itemize}
\item[(a)] $g(D_0)=D_{-1}$ and $g(D^{\sharp}_1)=D^{\sharp}_0$.
\item[(b)] Among closed domains $\{\ol{D}_n,\ \ol{D}^{\prime}_n,\ \ol{D}^{\prime\prime}_{n-1},\ \ol{D}^{\sharp}_n\}$ for $n\in\mathbb{Z}_{-}=\{0,-1,-2,...\}$, the intersection of these subsets are given as follows:
\begin{equation}\label{equ-2021-11-30-1}
\left\{
\begin{array}{l}
(b.1)\ \ol{D}_{n}\cap\ol{D}_{n-1}=\partial^{l}_{+}D_{n}=\partial^{r}_{+}D_{n-1},\\
(b.2)\ \ol{D}_{n-1}\cap\ol{D}_{n}^{\prime}=\partial^{r}_{-}D_{n-1}=\partial^{l}_{-}D_{n}^{\prime},\\
(b.3)\ \ol{D}_{n}^{\prime}\cap\ol{D}_{n-1}^{\prime\prime}=\partial^{l}_{+}D_{n}^{\prime}=\partial^{r}_{+}D_{n-1}^{\prime\prime},\\
(b.4)\ \ol{D}_{n-1}^{\prime\prime}\cap\ol{D}_{n}=\partial^{r}_{-}D_{n-1}^{\prime\prime}=\partial^{l}_{-}D_{n},\\
(b.5)\ \ol{D}_{n}\cap\ol{D}_{n}^{\prime}=\ol{D}_{n-1}\cap\ol{D}_{n-1}^{\prime\prime}=\text{a singleton set},\\
(b.6)\ \ol{D}_{n}\cap\ol{D}_{n}^{\sharp}=\partial^{h}_{+}D_{n}=\partial^{h}D_{n}^{\sharp},\\
(b.7)\ \ol{D}_{n}^{\sharp}\cap\ol{D}_{n-1}^{\sharp}=\partial^{l}D_{n}^{\sharp}=\partial^{r}D_{n-1}^{\sharp},\\
(b.8)\ \ol{D}_{n}\cap\ol{D}_{n-1}^{\sharp}=\ol{D}_{n-1}\cap\ol{D}^{\sharp}_{n}=\text{a singleton set}.
\end{array}
\right.
\end{equation}

\end{itemize}
\end{lemma}
\begin{proof}
(a) By the definition of $g$ in, $g$ is the unique branch of $F^{-1}$ near $\infty$, and $g(\partial^rD^{\sharp}_1)=\partial^lD^{\sharp}_1$ near $\infty$. By the construction above, $\partial^lD^{\sharp}_1=\partial^rD^{\sharp}_0$, implying that $g$ maps the left side $\partial^rD^{\sharp}_1$ to the left side $\partial^rD^{\sharp}_0$. Hence, one has $g(D^{\sharp}_1)=D^{\sharp}_0$. So, $g(\partial^rD^{\sharp}_0)=g(\partial^lD^{\sharp}_1)=\partial^lD^{\sharp}_0$.
In the definition above, $D_0$ and $D_0^{\sharp}$ are defined such that $\partial^rD^{\sharp}_0\cup \partial^r_{+}D_0$ is a single arc joining the critical value $cv_{Q1}$ and $\infty$. Continuing the branch $g$ along this curve up to $cv_{Q1}$, one has $g(\partial^r_{+}D_0)=\partial^l_{+}D_0=\partial^r_{+}D_{-1}$. Consider the left side of these curves, one has $g(D_0)=D_{-1}$.

(b) First of all, consider the case $n=0$ in \eqref{equ-2021-11-30-1}. Recall that $D_0$, $D_0^{\prime}$,
$D_{-1}$, and $D_{-1}^{\prime\prime}$ are defined by the image of $\varphi$ on $\widetilde{D}_0$, $\widetilde{D}_0^{\prime}$,
$\widetilde{D}_{-1}$, and $\widetilde{D}_{-1}^{\prime\prime}$, which are inverse images of $D_a$ and $D_0$ by two-fold branched covering $Q:\mathcal{U}_1\cup\mathcal{U}_2\to\cc\setminus(-\infty,0]$ branched only over $cv_Q$. This, together with the fact that $D_1$ and $D_0$ meet at $cv_Q$ along $\partial^l_{+}D_1=\partial^r_{+}D_0$ and $\partial^l_{-}D_1=\partial^r_{-}D_0$, yields that (b.1)--(b.5) hold for $n=0$ and $\widetilde{D}_0$ etc., implying (b.1)--(b.5) hold for $D_0=\varphi(\widetilde{D}_0)$ etc.

For (b.6)--(b.8), the intersection of $\ol{D}^{\sharp}_0\cup \ol{D}^{\sharp}_1$ with $\ol{D}_0\cup\ol{D}_1$ can be lifted to $\partial^h_{+}D_{-1}\cup\partial^h_{+}D_0$, and the lift $\ol{D}^{\prime}_0\cup\ol{D}^{\prime\prime}_{-1}$ cannot intersect with $\ol{D}^{\sharp}_{-1}\cup \ol{D}^{\sharp}_0$. Hence, one can conclude that (b.6)--(b.8) hold for $n=0$.

For $n>0$, the intersection relations in \eqref{equ-2021-11-30-1} between two domains whose indices are the same or differ by one can be verified by $g$. Further, two domains whose indices differ by two or more cannot intersect, since there images are disjoint sets under $F$.
\end{proof}

Now, the proof of Proposition \ref{attracting202091203} is given.
\begin{proof} (Proof of Proposition \ref{attracting202091203})
Consider the following sets
\beqq
\mathcal{U}=\mathcal{U}^{P}_{1+}\cup\mathcal{U}^{P}_{1-}\cup \ga_{c1+},\ \mathcal{U}^{\prime}=\mathcal{U}^{P}_{2-}\cup\mathcal{U}^{P}_{3+}\cup \ga_{c2-},\
\mathcal{U}^{\prime\prime}=\mathcal{U}^{P}_{2+}\cup\mathcal{U}^{P}_{(m+1),-}\cup \ga_{c2+}.
\eeqq
Each region is mapped homeomorphically by $P$ onto $\cc\setminus(-\infty,0]$. The map $\Psi_0(z)=cv_P\cdot e^{2\pi iz}=-\frac{m^m}{(m+1)^{m+1}}\cdot e^{2\pi iz}$ maps $\mathbf{D}$ onto $(\cc\setminus(-\infty,0])\cap\{z:\ e^{-2\pi\eta}<|z|<e^{2\pi\eta}\}$ and
$\mathbf{D}^{\sharp}$ onto $(\cc\setminus(-\infty,0])\cap\{z:\ 0<|z|<e^{-2\pi\eta}\}$, where $\mathbf{D}$ and $\mathbf{D}^{\sharp}$ are introduced in \eqref{equ-2021-11-30-2}.

Define the following map
\beqq
\Psi_1(z)=\left\{
  \begin{array}{ll}
   (P|_{\mathcal{U}})^{-1}\circ\Psi_0\circ\widetilde{\Phi}_{att}\quad\quad &\text{for}\ z\in D_{n}\cup D_{n}^{\sharp}\\
 (P|_{\mathcal{U}^{\prime}})^{-1}\circ\Psi_0\circ\widetilde{\Phi}_{att}\quad\quad &\text{for}\ z\in D_{n}^{\prime}\\
 (P|_{\mathcal{U}^{\prime\prime}})^{-1}\circ\Psi_0\circ\widetilde{\Phi}_{att}\quad\quad &\text{for}\ z\in D_{n}^{\prime\prime}.\\
  \end{array}
\right.
\eeqq
This $\Psi_1$ is a homeomorphism on each domain, which can be extended continuously to the closure.

Now, one needs to show that two extensions are consistent on a common boundary on two domains. This consistence implies that $\Psi_1$ is holomorphic, where the points corresponding to the critical value of $P$ can be used by the removable singularity theorem.

Next, the matching conditions on the intersection of two domains will be checked by the intersection relation \eqref{equ-2021-11-30-1}.

For $z\in D_n$ tending to $\partial^l_{+}D_n$, one has that $\Psi_0\circ\widetilde{\Phi}_{att}(z)$ tends to $[cv_P,0)=\Ga^P_{a}$ from lower side, implying that $\Psi_1(z)\in\mathcal{U}$ tends to $[cp_P,0)=\ga^{P}_{a1}$ from lower side. For $z\in D_{n-1}$ tending to the same boundary curve $\partial^l_{+}D_n=\partial^r_{+}D_{n-1}$ from the other side, one has that $\Psi_0\circ\widetilde{\Phi}_{att}(z)$ tends to $\Ga^P_a$ from upper side, implying that $\Psi_1(z)\in\mathcal{U}$ tends to $\ga^{P}_{a1}$ from upper side.

Since $P$ is homeomorphic in a neighborhood of $\ga^P_{ai}$, $\Psi_1$ matches completely along $\ol{D}_n\cap\ol{D}_{n-1}=\partial^l_{+}D_n=\partial^r_{+}D_{n-1}$, and is holomorphic there.

Similarly, for $z\in D_{n-1}$ tending to $\partial^r_{-}D_{n-1}$, one has that $\Psi_1(z)\in\mathcal{U}$ tends to $\ga^{P}_{b1+}$, while for $z\in D^{\prime}_{n}$ tending to $\partial^l_{-}D^{\prime}_{n}$, one has that $\Psi_1(z)\in\mathcal{U}^{\prime}$ tends to $\ga^{P}_{b2-}=\ga^{P}_{b1+}$. Hence, $\Psi_1$ matches along $\ol{D}_{n-1}\cap\ol{D}^{\prime}_{n}=\partial^r_{-}D_{n-1}=\partial^l_{-}D^{\prime}_n$.

The rest intersection relation in \eqref{equ-2021-11-30-1} can be studied similarly.
By the definition,  $\partial^l_{+}D^{\prime}_n=\partial^r_{+}D^{\prime\prime}_{n-1}$ corresponds to $\ga^P_{a2}=\ga^P_{a2+}=\ga^P_{a2-}$, and $\partial^r_{-}D^{\prime\prime}_{n-1}=\partial^l_{-}D_n$ to $\ga^{P}_{b2+}=\ga^{P}_{b1-}$.

Hence, $\Psi_1$ can be defined on $U=\text{the interior of}\ \cup^{0}_{-\infty}(\ol{D}_n\cup\ol{D}^{\prime}_n\cup\ol{D}^{\prime\prime}_{n-1}\cup\ol{D}^{\sharp}_n)$.

By the above discussions, one has $P\circ\Psi_1=\Psi_0\circ\widetilde{\Phi}_{att}$ and it is surjective onto $V^{\prime}$, where $V^{\prime}$ is introduced in \eqref{paraequ-92}.

By the description of the images $\mathcal{U}$, $\mathcal{U}^{\prime}$, $\mathcal{U}^{\prime\prime}$ and matching relations, the case (b) of Proposition Proposition \ref{attracting202091203} can be verified. By (b), the map
\beqq
\psi=\Psi_0\circ\widetilde{\Phi}_{rep}\circ\Psi^{-1}_1:\ V^{\prime}\setminus\{0\}\to\cc^{*}
\eeqq
is well defined and injective.

The relation in (c) can be obtained directly by the construction above:
\begin{align*}
P\circ\psi^{-1}=P\circ\Psi_1\circ\widetilde{\Phi}_{rep}^{-1}\circ\Psi^{-1}_0
=\Psi_0\circ\widetilde{\Phi}_{att}\circ\widetilde{\Phi}^{-1}_{rep}\circ\Psi^{-1}_0=\Psi_0\circ E_F\circ\Psi^{-1}_0.
\end{align*}
So, $\widetilde{\Phi}_{att}$ and $\widetilde{\Phi}_{rep}$ are the lifted versions of $\Phi_{att}$ and $\Phi_{rep}$, implying that $\widetilde{\Phi}_{att}\circ\widetilde{\Phi}^{-1}_{rep}=E_F$. This, together with the normalization $E_F(z)=z+o(1)$ with $\text{Im}\,z\to+i\infty$, implies that $\psi$ can be extended holomorphically to $z=0$ satisfying that $\psi(0)=0$ and $\psi'(0)=1$.

Recall the formal expression \eqref{equ-2021-12-1} at the beginning of Subsection \ref{equ2021-12-2}, where $\Phi_{rep}\circ F^{-n}\circ\Phi^{-1}_{att}$ can be interpreted as follows: the inverse image of $\Phi_{att}$ in the right half plane $\{z:\ \text{Re}\,z>L\}$ can be chosen, since it is injective; next the inverse orbits
along an appropriate inverse branches of $F^{-1}$ can be studied; finally $\Phi_{rep}$ can be considered in the left plane $\{z:\ \text{Re}\,z<-L\}$, which is well defined.

The choice of the inverse branches should be defined according to the above construction. Further, The local branching should be included only when it is related to the critical orbit of $F$, which corresponds to $cp_{P}$ in the domain of definition of $\psi$. Given a holomorphic family $\varphi_{\ld}$, the definitions of the Fatou coordinates (on the right/left half planes) and local branches of $F^{-n}$ can be introduced such that they are dependent holomorphically on $\ld$ (the parameter should be adjusted if necessary), except along the critical orbit. Therefore, the function $\psi_{\ld}(z)$ is dependent holomorphically on $\ld$, except at $cp_{P}$. But the exception at the critical point can be removed by the removable singularity theorem and the holomorphic dependence for all $z\in V^{\prime}$ is thus proved.

\end{proof}

\subsection{Teichm\"{u}ller space}

In this subsection, the connection between the space $\mathcal{F}_1^P$ specified in Definition \ref{def-2021-9-26-1} and the Teichm\"{u}ller space of a punctured disk. The arguments are based on \cite[Section 6]{InouShishikura2016}. The outline of the arguments is provided, refer to \cite[Section 6]{InouShishikura2016} for more details.

\begin{definition} \cite[Definition of Teichm\"{u}ller space]{InouShishikura2016}
Let $W_1$ be a Jordan domain in $\widehat{\cc}=\cc\cup\{\infty\}$, a point $p\in W_1$ be fixed, and $W=W_1\setminus\{p\}$ (which is isomorphic to $\cd\setminus\{0\}$). The map $\varphi:\ol{W}\to\widehat{\cc}$ is called a quasi-conformal map if $\varphi:\ol{W}\to\varphi(\ol{W})\subset \widehat{\cc}$ is a homeomorphism and $\varphi:\ W\to\varphi(W)$ is quasi-conformal in the usual sense. The Teichm\"{u}ller space of $W$ is
\beqq
\mbox{Teich}(W)=\{\varphi:\overline{W}\to\widehat{\mathbb{C}}\ \mbox{quasi-conformal map}\}/\sim
\eeqq
where $\varphi\sim\psi$ if and only if there exists a conformal map $h:\varphi(W)\to\psi(W)$ (automatically extending homemorphically to the closure) which coincides with $\psi\circ \varphi^{-1}$ on the boundary.

The Teichm\"{u}ller distance between $[\varphi]$ and $[\phi]$ is defined to be
\begin{equation*}
d_{Teich}([\varphi],[\psi])=\inf\left\{\log K\ \bigg | \begin{array}{l}
\mbox{there is a K-quasi-conformal map}\  h:\varphi(W)\to\psi(W), \\
 \mbox{which coincides with}\ \psi\circ\varphi^{-1}\ \mbox{on the boundary}
\end{array}\right\}.
\end{equation*}
The Teichm\"{u}ller space is complete with this metric.
\end{definition}

\begin{remark}
The boundary of $W$ contains the point $p$. The Teichm\"{u}ller space can also be regarded as the quotient space of measurable Beltrami differentials $\bm{\mu}=\mu(z)\frac{d\bar{z}}{dz}$ with $\|\bm{\mu}\|_{\infty}<1$, where $\|\bm{\mu}\|_{\infty}=\mbox{esssup}|\mu(z)|$ is the $L^{\infty}$-norm. The connection of these two definitions is $\varphi\to\bm{\mu}_{\varphi}=(\tfrac{\partial\varphi}{\partial\ol{z}}/\tfrac{\partial\varphi}{\partial z})\tfrac{d\ol{z}}{dz}$.
\end{remark}

\begin{theorem} (Measurable Riemann mapping theorem)\cite[Proposition 1.5.2]{MorosawaNishimuraTaniguchiUeda2000}\label{equ202091-1}
Let $\bm{\mu}=\mu(z)\tfrac{d\ol{z}}{dz}$ be a measurable function on $\cc$ satisfying $\|\bm{\mu}\|_{\infty}<1$. Then there exists a quasi-conformal map $f:\cc\to\cc$ with $\bm{\mu}_f=\bm{\mu}$ almost everywhere on $\cc$. Such an $f$ is unique under the condition that $f(0)=0$ and $f(1)=1$.
\end{theorem}

\begin{lemma}\label{lemma-2021-9-26-5}\cite[Lemma 6.1]{InouShishikura2016}
Let $W$ be as above with the puncture at $p=\infty$ and assume that $V=\cc\setminus\ol{W}$ contains $0$, and the boundaries of $\partial W$ consist of two disjoint smooth and non-singular Jordan curves except for the point $p$. Define
\begin{equation*}
\mathcal{S}^{qc}(V):=\left\{\varphi:V\to\cc:\ \bigg | \begin{array}{l}
\varphi\ \text {is univalent with}\ \varphi(0)=0, \varphi^{\prime}(0)=1, \\
\text { and } \varphi \text { has a quasi-conformal extension to } \mathbb{C}
\end{array}\right\}.
\end{equation*}
Then there is a bijection $\rho:\mathcal{S}^{qc}(V)\to\mbox{Teich}(W)$ defined by $\rho(\varphi)=[\hat{\varphi}|W]$, where $\hat{\varphi}:\cc\to\cc$ is a quasi-conformal extension of $\varphi$. If $\varphi_n,\varphi\in\mathcal{S}^{qc}(V)$ and $d_{Teich}([\varphi_n],[\varphi])\to0$, then $\{\varphi_n\}$ converges to $\varphi$ uniformly on compact sets in $V$. A map $\tau(\ld)$ from a complex manifold $\Ld$ to $\mbox{Teich}(W)$ is holomorphic if and only if there exists a holomorphic function ${\mathbf\varphi}:\Ld\times V\to\cc$ such that $\varphi_{\ld}:={\bf{\varphi}}(\ld,\cdot)\in\mathcal{S}^{qc}(V)$ and $\rho(\varphi_{\ld})=\tau(\ld)$.
\end{lemma}

The arguments of the proof of Theorem \ref{ren2020912-2} is given.

Recall $V^{'}$ introduced in \eqref{paraequ-92} and $V$ defined in \eqref{para2021-9-1-2}. For the $V$ and $V^{\prime}$ introduced above, take a domain $V^{\prime\prime}$ such that $\ol{V}\subset V^{\prime\prime}\subset\ol{V^{\prime\prime}}\subset V^{\prime}$, and $\partial V^{\prime\prime}$ consists of a non-singular real-analytic Jordan curve. Denote by $W=\cc\setminus\ol{V}$ and $U=\cc\setminus\ol{V^{\prime\prime}}$, where they have a puncture at $p=\infty$.

For $f=P\circ\varphi^{-1}\in\mathcal{F}^P_1$, it follows from the definition $\varphi\in S^{qc}(V)$ and $\rho(\varphi)$ defines a point in $\mbox{Teich}(W)$. The one to one correspondence is verified in Lemma \ref{lemma-2021-9-26-5}. Let $\mathcal{R}_0^{Teich}$ be the induced map on $\mbox{Teich}(W)$ by the parabolic renormalization $\mathcal{R}_0$. So, $\mathcal{R}_0$ induces a map $\wt{\mathcal{R}}_0^{Teich}:\mbox{Teich}(W)\to \mbox{Teich}(U)$, defined by
$\rho(\varphi)\to\rho(\psi)$, where $\mathcal{R}_0(P\circ\varphi^{-1})=P\circ\psi^{-1}$. Hence,
$$d_{\mbox{Teich}(U)}(\wt{\mathcal{R}}_0^{Teich}(\tau_1),\wt{\mathcal{R}}_0^{Teich}(\tau_2))
\leq d_{\mbox{Teich}(W)}(\tau_1,\tau_2)\ \forall \tau_1,\tau_2\in\mbox{Teich}(W),$$
where the following Royden-Gardiner theorem is used.

\begin{theorem} (Royden-Gardiner) Any holomorphic map between Teichm\"{u}ller spaces does not expand the Teichm\"{u}ller distance.
\end{theorem}

\begin{theorem} (Extension map) \cite[Theorem 6.3]{InouShishikura2016}
Let $W_1$ and $U_1$ be  Jordan domains in $\widehat{\cc}=\cc\cup\{\infty\}$ such that $\ol{U_1}\subset W_1$. Fix a point $p\in U_1$ and set $W=W_1\setminus\{p\}$ and $U=U_1\setminus\{p\}$. Then, the inclusion $U\hookrightarrow W$ induces a canonical map
\beqq
\Xi:\ Teich(U)\to Teich(W)
\eeqq
such that $\Xi(\tau)=\tau^{\prime}$ if and only if there is a quasi-conformal map $\psi:\ W\to\widehat{\cc}$ satisfying $[\psi]=\tau^{\prime}$ in $Teich(W)$, $[\psi|_U]=\tau$ in $Teich(U)$ and $\tfrac{\partial \psi}{\partial\ol{z}}=0$ almost everywhere in $W\setminus U$. The image of $\Xi$ is relatively compact and bounded with respect to $d_{Teich(W)}$. Moreover, it is a uniform contraction with an explicit bound:
\beqq
d_{Teich(W)}(\Xi(\tau_1),\Xi(\tau_2))\leq\ld\, d_{Teich(W)}(\tau_1,\tau_2)\ \forall\tau_1,\tau_2\in Teich(U),
\eeqq
where $\ld=e^{-2\pi\,\text{mod}(W\setminus\ol{U})}<1$.

For the Teichm\"{u}ller spaces without removing the puncture $p$, the same conclusion holds for the map $Teich(U_1)\to Teich(W_1)$ with the factor $e^{-4\pi\,\text{mod}(W_1\setminus\ol{U_1})}$.
\end{theorem}

\subsection{Near-parabolic bifurcation}

In this subsection, some basic concepts and lemmas are introduced for the proof of Theorem \ref{ren2020912-1}. The proof of Theorem \ref{ren2020912-1} and Corollary \ref{cor-2021-9-26-6} can be obtained by the same arguments used in \cite[Section 7]{InouShishikura2016}.

Consider the following space of functions:
\begin{equation*}
\mathcal{F}_0=\left\{f: U\rightarrow \mathbb{C}\ \bigg | \begin{array}{l}
f \text { is defined and analytic in a neighborhood}\ U\ \text{of}\ 0\\
f(0)=0,\ f'(0)=1,\ \text{and}\ f^{\prime\prime}(0)\neq0.
\end{array}\right\}
\end{equation*}
For $f\in\mathcal{F}_0$, there is a linear conjugacy map $\phi(z)=\frac{z}{f^{\prime\prime}(0)}$ such that $$\phi^{-1}\circ f\circ \phi(z)=z+z^2+\cdots.$$

For any $f\in\mathcal{F}_0$, the parabolic fixed point $z = 0$ has multiplicity two as a fixed point,
i.e., as the solution of the fixed point equation $f_0(z)-z=0$. In general, it bifurcates into two fixed points after perturbation.
Suppose $f$ is small perturbation of an element in $\mathcal{F}_0$, whose second derivative is $1$ for convenience.
Suppose the derivative of $f$ at $0$ is
\beqq
f'(0)=\exp(2\pi i\cdot\al(f)),
\eeqq
where $\al(f)\in\cc$, $-\tfrac{1}{2}<\text{Re}\,\al(f)\leq\tfrac{1}{2}$, and $|\arg\,\al(f)|<\tfrac{\pi}{4}$. Assume $\si(f)$ is another fixed point near $0$, where
\beqq
\si(f)=-2\pi i\al(f)(1+o(1))\ \mbox{as}\ f\to f_0.
\eeqq
For details on the discussions, please refer to \cite{Shishikura2000}.

\begin{proof} (Proof of Theorem \ref{ren2020912-2})
This can be obtained by the same arguments used in the proof of Main Theorem 2 of \cite{InouShishikura2016}.
\end{proof}

\begin{proof} (Proof of Theorem \ref{ren2020912-1})
This can be obtained by the same arguments used in the proof of Main Theorem 3 of \cite{InouShishikura2016}.
\end{proof}


\begin{proof} (Proof of Corollary \ref{equ2022-2-8-1})
This can be obtained by the same arguments used in the proof of Corollary 4.1 of \cite{InouShishikura2016}.
\end{proof}

\begin{proof} (Proof of Corollary \ref{cor-2021-9-26-6})
This can be obtained by the same arguments used in the proof of Corollary 4.2 of \cite{InouShishikura2016}.
\end{proof}

\section{Julia set with positive area}\label{julia-set-equ-2022-2-12-2}

The whole structure of the arguments follows the idea used in \cite{BuffCheritat2012}.

\subsection{Strategy of the proof of Theorem \ref{positivearea-12-11-1}}

\begin{definition}
For $\al\in\cc$ and $m\in\mathbb{N}$, consider the following class of polynomial maps
\beq\label{paraequ-33}
P_{\al}:\ z\in\cc\to e^{2\pi i\al}z(1+z)^m.
\eeq
Let $K_{\al}=\{z\in\cc:\ (P^{\circ j}_{\al}(z))\ \text{is bounded}\}$ be the filled-in Julia set of $P_{\al}$, and $J_{\al}$ be the Julia set, where $J_{\al}=\partial K_{\al}$ and $P^{\circ j}_{\al}(z)$ is the $j$th iteration of $z$ under the map $P_{\al}$.
\end{definition}

The continued fraction expansion of the real number $\al$ is defined by
\beqq
\al=[a_0,a_1,a_2,a_3,...]=a_0+\cfrac{1}{a_1+\cfrac{1}{a_2+\cfrac{1}{a_3+\cdots}}},
\eeqq
where $(a_k)$ are integers, $a_k$ is the $k$th entry of the continued fraction,  $a_0$ may be any integer in $\zz$, $a_k$ ($k\geq1$) are all positive. For more information on continued fraction, see Subsection \ref{equ2021-2-3-3}.

\begin{definition}
For any positive integer $N\geq1$, set
\beq\label{paraequ-35}
\mathcal{S}_N:=\{\al=[a_0,a_1,a_2,...]\in\rr\setminus\mathbb{Q}:\ (a_k)\ \text{is bounded and}\ a_k\geq N\ \text{for all}\ k\geq1\}.
\eeq
\end{definition}
The set $\mathcal{S}_1$ is called the set of bounded type irrationals. And, for $\al\in\mathcal{S}_1$, $K_{\al}$ contains a Siegel disk \cite{Siegel1942} and has positive area by a classical result of Siegel, the Siegel disk is bounded by a
quasicircle containing the critical point, and the post-critical set of $P_{\al}$ is contained in the boundary of the Siegel disk \cite{Milnor2006}.

\begin{proposition}\label{equ-2021-12-11-3}
There is a nonempty set $\mathcal{S}$ of bounded type irrationals such that for any $\al\in\mathcal{S}$ and any $\vep>0$, there is $\al'\in\mathcal{S}$ satisfying that
\begin{itemize}
\item $|\al'-\al|<\vep$,
\item $P_{\al'}$ has a cycle in $D(0,\vep)\setminus\{0\}$,
\item $\text{area}\,(K_{\al'})\geq(1-\vep)\text{area}\,(K_{\al})$.
\end{itemize}
\end{proposition}

\begin{proof}
The proof follows from the arguments in the proof of \cite[Proposition 1]{BuffCheritat2012}.
\end{proof}

\begin{proposition}
The function $\al\in\cc\to\text{area}\,(K_{\al})\in[0,+\infty)$ is upper semi-continuous.
\end{proposition}
\begin{proof}
This follows the arguments in the proof of \cite[Proposition 2]{BuffCheritat2012}. For any sequence $(\al_{n})$ with $\al_n\to\al$, it follows from \cite{Douady1995} that for any polynomial $P$, the $P \to K(P)$ is upper semi-continuous (in the Hausdorff metric). So, for any open neighborhood $V$ of $K_{\al}$, if $n$ is sufficiently large, then $K_{\al_n}\subset V$. By the theory of Lebesgue measure, $\text{area}\,(K_{\al})$ is the infimum of the area of the open sets containing $K_{\al}$. Hence,
\beqq
\text{area}\,(K_{\al})\geq\limsup_{n\to+\infty}\text{area}\,(K_{\al_n}).
\eeqq
\end{proof}

\subsection{A stronger version of Proposition \ref{equ-2021-12-11-3}}

Proposition \ref{equ-2021-12-11-3} can be derived by the following proposition.
\begin{proposition}  \label{equ-2022-2-10-1}
For a sufficiently large $N$ and $\al\in\mathcal{S}_N$, take a sequence $(A_n)$ such that
\beqq
\lim_{n\to+\infty}\sqrt[q_n]{A_n}\to+\infty\ \text{and}\ \lim_{n\to+\infty}\sqrt[q_n]{\log\,A_n}=1.
\eeqq
Set
\beqq
\al_n:=[a_0,a_1,...,a_n,A_n,N,N,N,...].
\eeqq
For all $\vep>0$, if $n$ is sufficiently large, then
\begin{itemize}
\item[$\bullet$] $P_{\al_n}$ has a cycle in $D(0,\vep)\setminus\{0\}$;
\item[$\bullet$] $\text{area}\,(K_{\al_n})\geq(1-\vep)\text{area}(K_{\al})$.
\end{itemize}
\end{proposition}

\begin{proof}
This follows the arguments in the proof of \cite[Proposition 3]{BuffCheritat2012}.
\end{proof}

\subsection{The control of the cycle}

\begin{proposition}\label{exp-2021-12-11-1}
For each rational number $\tfrac{p}{q}$, where $p$ and $q$ are coprime, then there exists a holomorphic function
\beqq
\chi:\ \cd(0,1/q^{3/q})\to\cc
\eeqq
satisfying the following properties:
\begin{itemize}
  \item[(1)] $\chi(0)=0$;
 \item[(2)] $\chi'(0)\neq0$;
\item[(3)] for any $\de\in \cd(0,1/q^{3/q})\setminus\{0\}$, $\chi(\de)\neq0$;
\item[(4)] for any $\de\in \cd(0,1/q^{3/q})\setminus\{0\}$, if $\zeta=e^{2i\pi p/q}$ and $\tht=\tfrac{p}{q}+\de^{q}$, then
$$\langle\chi(\de),\chi(\zeta\de),...,\chi(\zeta^{q-1}\de)\rangle$$ forms a cycle of period $q$ of $P_{\tht}$. Further,
\beqq
\chi(\zeta\de)=P_{\al}(\chi(\de))\ \forall \de\in \cd(0,1/q^{3/q}).
\eeqq
\end{itemize}
\end{proposition}

\begin{proof}
This can be derived by Propositions \ref{explosionfun-1} and \ref{explosionfun-2}, and Lemma \ref{equ-2021-12-8-1} in Subsection \ref{equ-21-12-7-1}.
This proposition is corresponding to \cite[Proposition 4]{BuffCheritat2012}.
\end{proof}

\begin{definition}
For the function $\chi:\cd(0,1/q^{3/q})\to\cc$ in Proposition \ref{exp-2021-12-11-1}, it is called an explosion function at $\tfrac{p}{q}$.
\end{definition}

\begin{proposition}
For any irrational number $\al\in\rr\setminus \qq$ such that $P_{\al}$ has a Siegel disk $\De$, let $\tfrac{p_k}{q_k}$ be the $k$th convergent of $\al$ (see Subsection \ref{equ2021-2-3-3}). Let $r$ be the conformal radius of $\De$ at $0$ and let $\phi:\cd(0,r)\to\De$ be the isomorphism sending $0$ to $0$ with derivative $1$.
For $k\geq1$, let $\chi_k$ be an explosion function at $p_k/q_k$ and set $\ld_k:=\chi'_k(0)$. Then,
\begin{itemize}
\item[(1)] $|\ld_k|\to r$ as $k\to+\infty$;
\item[(2)] the sequence of maps $\psi_k:\de\to\chi_k(\de/\ld_k)$ converges uniformly on every compact subset of $\cd(0,r)$ to $\phi:\cd(0,r)\to\De$.
\end{itemize}
\end{proposition}

\begin{remark}
This proposition is corresponding to \cite[Proposition 5]{BuffCheritat2012}, and this explains how the explosion functions behave as $\tfrac{p}{q}$ ranges
in the set of $\tfrac{p_k}{q_k}$ converging to $\al$, where $P_{\al}$ has a Siegel disk.
\end{remark}

\begin{corollary}
Let $(\al_n)$ be the sequence defined in Proposition \ref{equ-2021-12-11-2}. For any $\ep>0$, if $n$ is sufficiently large, then $P_{\al_n}$ has a cycle in $\cd(0,\ep)\setminus\{0\}$.
\end{corollary}

\begin{proof}
This can be obtained by the same arguments used in the proof of \cite[Corollary 1]{BuffCheritat2012}.
\end{proof}


\subsection{Perturbed Siegel disks}

\begin{definition}\cite[Definition 3]{BuffCheritat2012}
Let $U$ and $X$ be measurable subsets of $\cc$ with $0<\mbox{area}(U)<+\infty$, set
\beqq
\mbox{dens}_U(X):=\frac{\mbox{area}(U\cap X)}{\mbox{area}(U)}.
\eeqq
\end{definition}

Let $\al$ be a Brjuno number, $\tfrac{p_n}{q_n}$ be the $n$th convergent of $\al$, and $\chi_n:\ D_n:=\cd(0,1/q^{3/q_n}_{n})\to\cc$ be the explosion function at $\tfrac{p_n}{q_n}$.

\begin{proposition}\label{equ-21-12-11-5}
Suppose $\al=[a_0,a_1,...]$ and $\tht=[0,t_1,t_2,...]$ are Brjuno numbers, and $\tfrac{p_n}{q_n}$ be the $n$th convergent of $\al$. Suppose
\beqq
\al_n:=[a_0,a_1,...,a_n,A_n,t_1,t_2,...]
\eeqq
with $(A_n)$ a sequence of positive integers such that
\beq
\limsup_{n\to+\infty}\sqrt[q_n]{\log A_n}\leq1.
\eeq
Let $\De$ be the Siegel disk of $P_{\al}$ and $\De'_{n}$ be the Siegel disk of the restriction of $P_{\al_n}$ to $\De$. For any nonempty open set $U\subset \De$,
\beqq
\liminf_{n\to+\infty}\mbox{dens}_U(\De'_n)\geq\frac{1}{2}.
\eeqq
\end{proposition}

Set
\begin{align}\label{equ-21-12-12-2}
\vep_n:=\al_n-\frac{p_n}{q_n}=\frac{(-1)^n}{q^2_n(A_n+\tht)+q_nq_{n-1}}\to\frac{(-1)^n}{q^2_nA_n}\ \text{as}\ n\to+\infty.
\end{align}

Note that $\sqrt[q_n]{\vep_n}\to\tfrac{1}{\sqrt[q_n]{A_n}}$ as $n\to+\infty$.

For $0<\rho<1$, set
\beqq
X_n(\rho):=\left\{z\in\cc:\ \frac{z^{q_n}}{z^{q_n}-\vep_n}\in\cd(0,s_n)\right\}\quad\text{with}\quad s_n=\frac{\rho^{q_n}}{\rho^{q_n}+|\vep_n|}.
\eeqq

Proposition \ref{equ-21-12-11-5} can be derived by Proposition \ref{equ-21-12-11-6}.

\begin{proposition}\label{equ-21-12-11-6}
Under the same assumptions as in Proposition \ref{equ-21-12-11-5}, for all $0<\rho<1$, if $n$ is large enough, the Siegel disk $\De'_n$ contains
$\chi_n(X_n(\rho))$.
\end{proposition}

\begin{proof}
This can be derived by the same arguments of the proof of \cite[Proposition 7]{BuffCheritat2012}.
\end{proof}

The proof of Proposition \ref{equ-21-12-11-6} follows by contradiction. Suppose there exist $0<\rho<1$ and an increasing sequence of integers $n_k$ such that
$\chi_{n_k}(X_{n_k}(\rho))$ is not contained in $\De'_{n_k}$. Without loss of generality, suppose this subsequence satisfies
\beqq
A^{1/q_{n_k}}_{n_k}\to A\in[1,+\infty].
\eeqq

The case $A>1$ can be studied by the same argument in the proof of \cite[Proposition 7]{BuffCheritat2012}.

Assume $A>1$. The limit values of the sequence $\chi_n:\ D_n=\cd(0,1/q^{3/q_n}_n)\to\cc$ are isomorphisms $\chi:\cd\to\De$. So, there is a sequence $r'_n$ tending to $1$ such that $\chi_n$ is univalent on $D'_n=\cd(0,r'_n)$ and the domain of the map
\beq\label{equ-21-12-12-3}
f_n=(\chi_n|_{D'_n})^{-1}\circ P_{\al_n}\circ (\chi_n|_{D'_n})
\eeq
eventually contains any compact subset of $\cd$. This case $A>1$ can be derived by the following Proposition \ref{equ-21-12-12-1}.

\begin{proposition}\label{equ-21-12-12-1}
Assume
\beqq
0\leq\frac{1}{A}<\rho<\rho'<1.
\eeqq
If $n$ is large enough, the orbit under iteration of $f_n$ of any point $z\in X_n(\rho)$ remains in $\cd(0,\rho')$.
\end{proposition}

The proof of Proposition \ref{equ-21-12-12-1} requires the introduction of a flow, where the time-1 map of the flow provides a good approximation of $f^{q_n}_n$, where the study of the flow needs the combination of several changes of coordinates and the analysis of the iteration in the new coordinates. So, this part is divided into two parts. For an illustration of the coordinate transformation, refer to  \cite[Figure 5]{BuffCheritat2012}.

\subsubsection{A vector field}

For $\vep_n$ in \eqref{equ-21-12-12-2} and $f_n$ in \eqref{equ-21-12-12-3}, the vector field is defined as follows:
\begin{align*}
\xi_n=\xi_n(z)\frac{d}{dz}=2\pi iq_nz(\vep_n-z^{q_n})\frac{d}{dz}.
\end{align*}
The vector field is transformed into the following form:
\begin{align*}
2\pi iq_nz(\vep_n-z^{q_n})\frac{d}{dz}
\xrightarrow[semi-conjugate]{z\to v=z^{q_n}}2\pi iq^2_nv(\vep_n-v)\frac{d}{dv}\xrightarrow{v\to w=\tfrac{v}{v-\vep_n}}2\pi iq^2_nw\frac{d}{dw}.
\end{align*}

The last vector field is tangent to Euclidean circles centered at the origin. The boundary of $X_n(\rho)$ is mapped to a Euclidean circle by the map $z\to w=\tfrac{z^{q_n}}{z^{q_n}-\vep_n}$. So, the vector field $\xi_n$ is tangent to the boundary of $X_n(\rho)$, which is invariant under the flow induced by the vector field $\xi_n$.

The unit disk is invariant by its real flow, and the open set
\beqq
\Om_n=\left\{z\in\cc:\ w=\frac{z^{q_n}}{z^{q_n}-\vep_n}\in\cd\right\}
\eeqq
is invariant by the real flow of the vector field $\xi_n$.

The map
\beqq
z\to w=\frac{z^{q_n}}{z^{q_n}-\vep_n}:\ \Om_n\to\cd
\eeqq
is ramified covering of degree $q_n$, ramified at $0$. Thus, there is an isomorphism $\psi_n:\Om_n\to\cd$ such that
\beqq
(\psi_n(z))^{q_n}=\frac{z^{q_n}}{z^{q_n}-\vep_n}.
\eeqq
The change of coordinates $z\to\tht=\psi_n(z):\Om_n\to\cd$ conjugates the vector field $\xi_n$ to
\beqq
2\pi iq_n\frac{d}{d\tht}.
\eeqq
Finally, let $\mathbb{H}$ be the upper half plane, $\pi_n\to\mathbb{H}\to\Om_n\setminus\{0\}$ be the universal covering given by
\beq\label{pin-2022-1-30-1}
\pi_n(Z)=\psi_n^{-1}(e^{2i\pi q_n\vep_n Z}),
\eeq
implying that
\beqq
\pi_n^*\xi_n=\frac{d}{dZ}.
\eeqq

\subsubsection{The coordinate straightening the vector field}

For convenience, consider the situation that $n$ is even such that $\vep_n>0$.
For $r\in[\rho,1)$, $X_n(\rho)\subset X_n(r)\subset\Om_n$,

the universal covering map $\pi_n:\mathbb{H}_n(r)\to X_n(r)\setminus\{0\}$ is defined on the region
\beqq
\mathbb{H}_n(r)=\{Z\in\cc:\ \text{Im}(Z)>\tau_n(r)\},
\eeqq
where
\beqq
\tau_n(r)=\frac{1}{2\pi q^2_n\vep_n}\log\left(1+\frac{\vep_n}{r^{q_n}}\right)\to\frac{1}{2\pi q^2_nr^{q_n}}\quad \text{as}\quad n\to+\infty.
\eeqq

Note that
\beqq
\sqrt[q_n]{\tau_n(r)}\to\frac{1}{r}\ \mbox{as}\ n\to+\infty.
\eeqq

\begin{definition}\cite[Definition 4]{BuffCheritat2012}
A sequence $(B_n)$ is called sub-exponential with respect to $q_n$ if
\beqq
\limsup_{n\to+\infty}\sqrt[q_n]{|B_n|}\leq1.
\eeqq
\end{definition}

\begin{proposition}
Recall the map $\pi_n$ is introduced in \eqref{pin-2022-1-30-1}. Assume $0<r<1$. If $n$ is large enough, there exist holomorphic maps $F_n:\mathbb{H}_n(r)\to\mathbb{H}$ and $G_n:\mathbb{H}_n(r)\to\mathbb{H}$ that
\begin{itemize}
\item $\pi_n$ semi-conjugates $F_n$ to $f^{q_n}_n$ and $G_n$ to $f^{q_{n-1}}_n$:
\beqq
\pi_n\circ F_n=f^{q_n}_n\circ\pi_n\ \mbox{and}\ \pi_n\circ G_n=f^{q_{n-1}}_n\circ\pi_n,
\eeqq
that is,
\begin{center}
\begin{tikzpicture}
    \node (E) at (0,0) {$\mathbb{H}_{n}(r)$ };
    \node[right=of E] (F) at (4,0){$\mathbb{H} $};
    \node[below=of F] (A) {$\Om_n(r)-\{0\}$};
    \node[below=of E] (Asubt) {$X_n(r)-\{0\}$};
   \draw[->] (E)--(F) node [midway,above] {$F_n$};
    \draw[->] (F)--(A) node [midway,right] {$\pi_n$}
                node [midway,left] {};
    \draw[->] (Asubt)--(A) node [midway,below] {$f_n^{\circ q_n}$}
                node [midway,above] {};
    \draw[->] (E)--(Asubt) node [midway,left] {$\pi_n$};
\end{tikzpicture}
\end{center}
and
\begin{center}
\begin{tikzpicture}
    \node (E) at (0,0) {$\mathbb{H}_{n}(r)$ };
    \node[right=of E] (F) at (4,0){$\mathbb{H} $};
    \node[below=of F] (A) {$\Om_n(r)-\{0\}$};
    \node[below=of E] (Asubt) {$X_n(r)-\{0\}$};
   \draw[->] (E)--(F) node [midway,above] {$G_n$};
    \draw[->] (F)--(A) node [midway,right] {$\pi_n$}
                node [midway,left] {};
    \draw[->] (Asubt)--(A) node [midway,below] {$f_n^{\circ q_{n-1}}$}
                node [midway,above] {};
    \draw[->] (E)--(Asubt) node [midway,left] {$\pi_n$};
\end{tikzpicture}.
\end{center}
\item $F_n-Id$ and $G_n-Id$ are periodic of period $1/(q_n\vep_n)$.
\item As $\text{Im}(Z)\to+\infty$, one has
\beqq
F_n(Z)=Z+1+o(1)\ \text{and}\ G_n(Z)=Z-(A_n+\tht)+o(1).
\eeqq
\end{itemize}
Furthermore, the two sequences:
\beqq
\sup_{z\in\mathbb{H}_n(r)}|F_n(Z)-Z-1|\ \text{and}\ \sup_{z\in\mathbb{H}_n(r)}|G_n(Z)-Z+A_n+\tht|
\eeqq
are sub-exponential with respect to $q_n$.
\end{proposition}

\begin{proof}
The proof of this proposition follows the arguments used in the proof of \cite[Proposition 8]{BuffCheritat2012}.
\end{proof}

\begin{remark}\label{paraequ-32}
By direct calculation, one has
\begin{align*}
P_{\al}(z)=&e^{2\pi i\al}z(1+z)^m=e^{2\pi i\al}z\bigg(\sum^m_{j=0}\tfrac{m!}{j!(m-j)!}z^j\bigg)=e^{2\pi i\al}\sum^m_{j=0}\tfrac{m!}{j!(m-j)!}z^{j+1}\\
=&e^{2\pi i\al}z+e^{2\pi i\al}mz^2+e^{2\pi i\al}\sum^m_{j=2}\tfrac{m!}{j!(m-j)!}z^{j+1},
\end{align*}
plugging into the the transformation $\tfrac{w}{e^{2\pi i\al}m}=z$, the conjugacy is \begin{align*}
&e^{2\pi i\al}m[(\tfrac{w}{e^{2\pi i\al}m})+e^{2\pi i\al}m(\tfrac{w}{e^{2\pi i\al}m})^2+e^{2\pi i\al}\sum^m_{j=2}\tfrac{m!}{j!(m-j)!}(\tfrac{w}{e^{2\pi i\al}m})^{j+1}]\\
=&e^{2\pi i\al}w+w^2+e^{2\pi i\al}\sum^m_{j=2}\tfrac{m!}{(e^{2\pi i\al}m)^{j}j!(m-j)!}w^{j+1}.
\end{align*}
So, we suppose the coefficient in front of $w^2$ is $1$.
\end{remark}
For convenience, consider the following map
\beq\label{paraequ-4}
P_{\al}(z)=e^{2\pi i\al}z+z^2+O(z^3).
\eeq

The following result will be used to show that the domains of $f^{q_n}_n$ and $f^{q_{n-1}}_n$ eventually contain any compact subset of $\cd$.

Consider a holomorphic function $f(z)=e^{2\pi i\al}z+O(z^2)$ with $\al\in\rr\setminus\mathbb{Q}$, which is linearizable near the point $0$.
The maximal Sigel disk, denoted by $\De$, and there exists an isomorphism $\varphi$ on $\De$ such that $\varphi:\De\to\cd$ with $\varphi(0)=0$ and $\varphi\circ f(z)=R_{\al}\circ\varphi(z)$ for $z\in\De$, where $R_{\al}(z)=e^{2\pi i\al}z$ is a rotation.

Let $\cd(0,r)=\{z\in\cc: |z|<r\}$, $\De(r)=\varphi^{-1}(\cd(0,r))$.

\begin{theorem} (Jellouli's theorem) \cite{Jellouli1994, Jellouli2000}
Assume that $P_{\al}$ has a Siegel disk $\De$ and let $\chi:\cd\to\De$ be a linearizing isomorphism. For $r<1$, set $\De(r):=\chi(\cd(0,r))$. Assume $\al_n\in\rr$ and $b_n\in\nn$ satisfy that $b_n\cdot|\al_n-\al|=o(1)$. For all $r^{\prime}_1<r^{\prime}_2<1$, if $n$ is sufficiently large, then
\beqq
\De(r^{\prime}_1)\subset\{z\in \De(r^{\prime}_2):\ \forall j\leq b_n,\ P^{\circ j}_{\al_n}(z)\in \De(r^{\prime}_2)\}.
\eeqq
\end{theorem}

\subsection{The control of the post-critical set}

\begin{definition}
Denote by $\partial$ the Hausdorff semi-distance:
$$
\partial(X, Y)=\sup _{x \in X} d(x, Y).
$$
\end{definition}

\begin{definition}
Denote by $\mathcal{PC}\left(P_{\alpha}\right)$ the post-critical set of $P_{\alpha}$:
$$\mathcal{P C}\left(P_{\alpha}\right):=\bigcup_{k \geq 1} P_{\alpha}^{ k}\left(\omega_{\alpha}\right),$$
where $ \omega_{\alpha}=cp_{\tiny{P}}=-\frac{1}{m+1}$ is specified in \eqref{critip2011-8-30-1}.
\end{definition}

\begin{proposition}
There exists $N$ such that as $\alpha^{\prime} \in \mathcal{S}_{N} \rightarrow \alpha \in \mathcal{S}_{N}$, we have
$$
\partial\left(\mathcal{P C}\left(P_{\alpha^{\prime}}\right), \ol{\Delta}_{\alpha}\right) \rightarrow 0
$$
with $\Delta_{\alpha}$ being the Siegel disk of $P_{\alpha}$.
\end{proposition}

\begin{proof}
This is corresponding to \cite[Proposition 11]{BuffCheritat2012}.
\end{proof}

\begin{corollary}
Let $\left(\alpha_{n}\right)$ be the sequence defined in Proposition \ref{equ-2022-2-10-1}. For all $\delta$, if $n$ is large enough, the post-critical set of $P_{\alpha_{n}}$ is contained in the $\delta$-neighborhood of the Siegel disk of $P_{\alpha}$.
\end{corollary}

\begin{proof}
This is corresponding to \cite[Corollary 4]{BuffCheritat2012}.
\end{proof}

\subsubsection{The class of maps in $\mathcal{F}^{P}_{2}$}

For any $m\geq22$, consider the polynomial
\beqq
P(z)=z(1+z)^m.
\eeqq
The critical point and the critical values are introduced in \eqref{critip2011-8-30-1}.
Set
$$v:=cv_{\tiny{P}},\ U:=V^{'},$$
where $V^{'}$ is specified in \eqref{paraequ-92}.

For the class of maps in $\mathcal{F}^{P}_{2}$ introduced in Definition \ref{renorm202091206}, for the convenience of the readers by using the the methods in \cite{BuffCheritat2012}, we adopt the notations in \cite{BuffCheritat2012}.

The $\mathcal{F}^{P}_{2}$ is represented by the following notation:
\beqq
\mathcal{IS}_{0}=\left\{f=P\circ\varphi^{-1}:\ U_f\to\cc\ \mbox{with}\ \bigg | \begin{array}{l}
\varphi:U\to U_f\
\mbox{univalent (isomorphism) such that} \\
\varphi(0)=0\ \mbox{and}\ \varphi^{\prime}(0)=1
\end{array}\right\}.
\eeqq

A map $f \in \mathcal{IS}_{0}$ fixes $0$ with multiplier $1$. The map $f$ has a critical point $\omega_{f}:=\varphi_{f}(\tfrac{-1}{m+1})$ which depends on $f$ and a critical value $v=cv_{\tiny{P}}$ which does not depend on $f$.

\subsubsection{Fatou coordinates}
Near $z=0$, elements $f \in \mathcal{IS}_{0}$ have an expansion of the form
$$
f(z)=z+c_{f} z^{2}+\mathcal{O}\left(z^{3}\right)
$$
The following result is an immediate consequence of the Koebe Distortion Theorem.

\begin{lemma}(Theorem \ref{paraequ-84}, part a) The set $\left\{c_{f} ; f \in\right.$ $\left.\mathcal{IS}_{0}\right\}$ is a compact subset of $\mathbb{C}^{*}$.
\end{lemma}

In particular, for all $f \in \mathcal{IS}_{0}, c_{f} \neq 0$ and $f$ has a multiple fixed point of multiplicity $2$ at $0$. The change of variables
$$
z=\tau_{f}(w):=-\frac{1}{c_{f} w}
$$
gives us $F(w)=w+1+o(1)$ near infinity. To lighten notation, we will write $f$ and $F$ for pairs of functions related as above; $\omega_{f}:=\varphi_{f}(\tfrac{-1}{m+1})$ and $\omega_{F}:=$ $\tau_{f}^{-1}\left(\omega_{f}\right)$ will denote their critical points.

\begin{lemma}
There exists $R_{0}$ such that for all $f \in \mathcal{IS}_{0}$,
\begin{itemize}
\item $F$ is defined and univalent in a neighborhood of $\mathbb{C} \backslash D\left(0, R_{0}\right)$;
 \item for all $w \in \mathbb{C} \backslash D\left(0, R_{0}\right)$,
$$
|F(w)-w-1|<\frac{1}{4} \quad \text { and } \quad\left|F^{\prime}(w)-1\right|<\frac{1}{4}.
$$
\end{itemize}
\end{lemma}

\begin{proof}
 This is derived by the compactness of $\mathcal{IS}_{0}$.
\end{proof}

If $R_{1}>\sqrt{2} R_{0}$, the regions
and
$$
\Omega^{\text {att }}:=\left\{w \in \mathbb{C}:\  \operatorname{Re}(w)>R_{1}-|\operatorname{Im}(w)|\right\}
$$
and
$$
\Omega^{\text {rep }}:=\left\{w \in \mathbb{C}:\  \operatorname{Re}(w)<-R_{1}+|\operatorname{Im}(w)|\right\}
$$
are contained in $\mathbb{C} \backslash D\left(0, R_{0}\right)$.
Then, for all $f \in \mathcal{IS}_{0}$,
$$
F\left(\Omega^{\text {att }}\right) \subset \Omega^{\text {att }} \quad \text { and } \quad F\left(\Omega^{\text {rep }}\right) \supset \Omega^{\text {rep }}
$$
In addition, there are univalent maps $\Phi_{F}^{\text {att }}: \Omega^{\text {att }} \rightarrow \mathbb{C}$ (attracting Fatou coordinate for $F$ ) and $\Phi_{F}^{\text {rep }}: \Omega^{\text {rep }} \rightarrow \mathbb{C}$ (repelling Fatou coordinate for $F$ ) such that
$$
\Phi_{F}^{\text {att }} \circ F(w)=\Phi_{F}^{\text {att }}(w)+1 \quad \text { and } \quad \Phi_{F}^{\text {rep }} \circ F(w)=\Phi_{F}^{\text {rep }}(w)+1
$$
when both sides of the equations are defined. The maps $\Phi_{F}^{\text {att }}$ and $\Phi_{F}^{\text {rep }}$ are unique up to an additive constant. Further, $\Phi_{F}^{\text {att }}-\Phi_{F}^{\text {rep }}$ converges to a constant as $w \in \Omega^{\text {att }} \cap \Omega^{\text {rep }}$ goes to infinity.

\begin{lemma}(Theorem \ref{paraequ-84}, part a) For all $f \in$ $\mathcal{IS}_{0}$, the critical point $\omega_{f}$ is attracted to $0$.
\end{lemma}

\begin{lemma}There exists $k$ such that for all $f \in \mathcal{IS}_{0}$, we have $F^{k}\left(\omega_{F}\right)$ $\in \Omega^{\mathrm{att}}$.
\end{lemma}

\begin{proof}
This is corresponding to \cite[Lemma 7]{BuffCheritat2012} using the compactness of the class $\mathcal{IS}_{0}$.
\end{proof}

Since the maps $\Phi_{F}^{\text {att }}$ and $\Phi_{F}^{\text {rep }}$ are only defined up to an additive constant, we can normalize $\Phi_{F}^{\text {att }}$ so that
$$
\Phi_{F}^{\mathrm{att}}\left(F^{ k}\left(\omega_{F}\right)\right)=k .
$$
Then, we can normalize $\Phi_{F}^{\text {rep }}$ so that
$\Phi_{F}^{\mathrm{att}}(w)-\Phi_{F}^{\mathrm{rep}}(w) \rightarrow 0$ when $\operatorname{Im}(w) \rightarrow+\infty \quad$ with $\quad w \in \Omega^{\mathrm{att}} \cap \Omega^{\mathrm{rep}} .$

Coming back to the $z$-coordinate, define
$$
\Omega_{\mathrm{att}, f}:=\tau_{f}\left(\Omega^{\mathrm{att}}\right) \quad \text { and } \quad \Omega_{\mathrm{rep}, f}:=\tau_{f}\left(\Omega^{\mathrm{rep}}\right).
$$
Set
$$
\Phi_{\mathrm{att}, f}:=\Phi_{F}^{\mathrm{att}} \circ \tau_{f}^{-1} \quad \text { and } \quad \Phi_{\mathrm{rep}, f}:=\Phi_{F}^{\mathrm{rep}} \circ \tau_{f}^{-1}.
$$
The univalent maps $\Phi_{\text {att }, f}: \Omega_{\text {att }, f} \rightarrow \mathbb{C}$ and $\Phi_{\text {rep }, f}: \Omega_{\text {rep }, f} \rightarrow \mathbb{C}$ are called attracting and repelling Fatou coordinates for $f$. Note that our normalization of the attracting coordinates is given by
$$
\Phi_{\text {att }, f}\left(f^{\circ k}\left(\omega_{f}\right)\right)=k.
$$

\begin{lemma} (Proposition \ref{attracting2020912-1}). For all $f \in \mathcal{I S}_{0}$, there exists an attracting petal $\mathcal{P}_{\text {att }, f}$ and an extension of the Fatou coordinate, which we will still denote $\Phi_{\mathrm{att}, f}: \mathcal{P}_{\mathrm{att}, f} \rightarrow \mathbb{C}$, such that
\begin{itemize}
\item $v \in \mathcal{P}_{\text {att}, f}$,
\item $\Phi_{\mathrm{att}, f}(v)=1$,
\item $\Phi_{\mathrm{att}, f}$ is univalent on $\mathcal{P}_{\mathrm{att}, f}$,
\item $\Phi_{\mathrm{att}, f}\left(\mathcal{P}_{\mathrm{att}, f}\right)=\{w:\  \operatorname{Re}(w)>0\}$.
\end{itemize}
\end{lemma}

\begin{definition} For $f \in \mathcal{IS}_{0}$, set
$$
V_{f}:=\left\{z \in \mathcal{P}_{\text {att }, f}:\  \operatorname{Im}\left(\Phi_{\text {att }, f}(z)\right)>0 \text { and } 0<\operatorname{Re}\left(\Phi_{\text {att }, f}(z)\right)<2\right\}
$$
and
$$
W_{f}:=\left\{z \in \mathcal{P}_{\text {att }, f}:\ -2<\operatorname{Im}\left(\Phi_{\text {att }, f}(z)\right)<2 \text { and } 0<\operatorname{Re}\left(\Phi_{\text {att }, f}(z)\right)<2\right\}.
$$
\end{definition}

The result below follows from Propositions \ref{attracting2020912-2} and \ref{attracting202091203}, and Subsection \ref{proofpro2021-10-3-1}. Our domain $V_{f}^{-k} \cup$ $W_{f}^{-k}$ below corresponds to the interior of
$$
\ol{D}_{-k} \cup \ol{D}_{-k}^{\sharp} \cup \ol{D}_{-k}^{\prime \prime} \cup \ol{D}_{-k+1} \cup \ol{D}_{-k+1}^{\sharp} \cup \ol{D}_{-k+1}^{\prime}.
$$
The set $W_{f}^{-k}$ itself corresponds to the interior of
$$
\ol{D}_{-k} \cup \ol{D}_{-k}^{\prime \prime} \cup \ol{D}_{-k+1} \cup \ol{D}_{-k+1}^{\prime}.
$$

\begin{lemma} For all $f \in \mathcal{IS}_{0}$ and all $k \geq 0$,
\begin{itemize}
\item the unique connected component $V_{f}^{-k}$ of $f^{-k}\left(V_{f}\right)$ that contains 0 in its closure is relatively compact in $U_{f}$ (the domain of $f$) and $f^{k}: V_{f}^{-k} \rightarrow$ $V_{f}$ is an isomorphism,
\item the unique connected component $W_{f}^{-k}$ of $f^{-k}\left(W_{f}\right)$ that intersects $V_{f}^{-k}$ is relatively compact in $U_{f}$ and $f^{ k}: W_{f}^{-k} \rightarrow W_{f}$ is a covering of degree $2$ ramified above $v$.
    \end{itemize}
In addition, if $k$ is large enough, then $V_{f}^{-k} \cup W_{f}^{-k} \subset \Omega_{\text {rep}, f}$.
\end{lemma}

The following lemma asserts that if $k$ is large enough, then for all map $f \in \mathcal{IS}_{0}$, the set $V_{f}^{-k} \cup W_{f}^{-k}$ is contained in a repelling petal of $f$, i.e., the preimage of a left half-plane by $\Phi_{\text{rep}, f}$.

\begin{lemma}\label{equ-2022-2-10-10} There is an $R_{2}>0$ such that for all $f \in \mathcal{I} S_{0}$, the set $\Phi_{\mathrm{rep}, f}\left(\Omega_{\mathrm{rep}, f}\right)$ contains the half-plane $\left\{w \in \mathbb{C} ;\right.$ Re $\left.w<-R_{2}\right\}$. There is an integer $k_{0}>0$ such that for all $k \geq k_{0}$, we have
$$
V_{f}^{-k} \cup W_{f}^{-k} \subset\left\{z \in \Omega_{\mathrm{rep}, f}:\  \operatorname{Re}\left(\Phi_{\mathrm{rep}, f}(z)\right)<-R_{2}\right\} .
$$
\end{lemma}

\begin{remark} $R_{2}$ can be replaced by any $R_{3} \geq R_{2}$, replacing if necessary $k_{0}$ by $k_{1}:=k_{0}+\left\lfloor R_{3}-R_{2}\right\rfloor+1$.
\end{remark}

\begin{proof}
This is corresponding to \cite[Lemma 8]{BuffCheritat2012}.
\end{proof}

\subsubsection{Perturbed Fatou coordinates}

For any $\al\in\rr$, denote by $\mathcal{IS}_{\al}$ the set of maps of the form $z\to f(e^{2i\pi\al}z)$ with $f\in\mathcal{IS}_0$. For any subset $A\subset\rr$, denote by \beqq
\mathcal{IS}_A=\bigcup_{\al\in A}\mathcal{IS}_{\al}.
\eeqq
So, one has
\beqq
\mathcal{IS}_{\al}=\left\{f=P\circ\varphi^{-1}:\ U_f\to\cc\ \mbox{with}\ \bigg | \begin{array}{l}
\varphi:U\to U_f\
\mbox{univalent (isomorphism) such that} \\
\varphi(0)=0\ \mbox{and}\ \varphi^{\prime}(0)=e^{-2i\pi\al}
\end{array}\right\}
\eeqq
and
\beqq
\mathcal{IS}_{\rr}=\left\{f=P\circ\varphi^{-1}:\ U_f\to\cc\ \mbox{with}\ \bigg | \begin{array}{l}
\varphi:U\to U_f\
\mbox{univalent (isomorphism) such that} \\
\varphi(0)=0\ \mbox{and}\ |\varphi^{\prime}(0)|=1
\end{array}\right\}.
\eeqq
The compact-open topology is added on this set of univalent maps.

\begin{proposition}
There are constants $K>0, \varepsilon_{1}>0$, and $R_{3} \geq R_{2}$ with $\frac{1}{\varepsilon_{1}}-R_{3}>1$, such that for all $f \in \mathcal{IS}_{(0, \varepsilon_{1})}$, the following holds:
\begin{itemize}
\item [(1)]
There is a Jordan domain $\mathcal{P}_{f} \subset U_{f}$ (a perturbed petal) containing the critical value $v$, bounded by two arcs joining $0$ to $\sigma_{f}$, where $f$ has only two fixed points $0$ and $\si(f)$ in a small neighborhood of the origin, and there is a branch of argument defined on $\mathcal{P}_{f}$ such that
$$
\sup _{z \in \mathcal{P}_{f}} \arg (z)-\inf _{z \in \mathcal{P}_{f}} \arg (z)<K .
$$
\item [(2)] There is a univalent map $\Phi_{f}: \mathcal{P}_{f} \rightarrow \mathbb{C}$ (a perturbed Fatou coordinate) such that
\begin{itemize}
\item[$\bullet$] $\Phi_{f}(v)=1 ;$
\item[$\bullet$] $\Phi_{f}\left(\mathcal{P}_{f}\right)=\left\{w \in \mathbb{C} ; 0<\operatorname{Re}(w)<1 / \alpha_{f}-R_{3}\right\}$;
\item[$\bullet$] $\operatorname{Im}\left(\Phi_{f}(z)\right) \rightarrow+\infty$ as $w \rightarrow 0$ and $\operatorname{Im}\left(\Phi_{f}(z)\right) \rightarrow-\infty$ as $w \rightarrow \sigma_{f} ;$
\item[$\bullet$] when $z \in \mathcal{P}_{f}$ and $\operatorname{Re}\left(\Phi_{f}(z)\right)<1 / \alpha_{f}-R_{3}-1$, then $f(z) \in \mathcal{P}_{f}$ and $\Phi_{f} \circ f(z)=\Phi_{f}(z)+1 .$
\end{itemize}
For $f \in \mathcal{I} S_{0}$, set
$$
\mathcal{P}_{\text{rep},f}:=\left\{z \in \Omega_{\text {rep }, f} ; \operatorname{Re}\left(\Phi_{\text {rep }, f}(z)\right)<-R_{3}\right\} .
$$
\item[(3)] If $\left(f_{n}\right)$ is a sequence of maps in $\mathcal{IS}_{(0, \varepsilon_{1})}\left[\right.$ converging to a map $f_{0} \in \mathcal{IS}_{0},$, then
\begin{itemize}
\item[$\bullet$] any compact $K \subset \mathcal{P}_{\text{att},f_{0}}$ is contained in $\mathcal{P}_{f_{n}}$ for $n$ large enough and the sequence $\left(\Phi_{f_{n}}\right)$ converges to $\Phi_{\text {att, } f_{0}}$ uniformly on $K$,
\item[$\bullet$] any compact $K \subset \mathcal{P}_{\text {rep}, f_{0}}$ is contained in $\mathcal{P}_{f_{n}}$ for $n$ large enough and the sequence $\left(\Phi_{f_{n}}-\frac{1}{\alpha_{f_{n}}}\right)$ converges to $\Phi_{\mathrm{rep}, f_{0}}$ uniformly on $K$.
\end{itemize}
\end{itemize}
\end{proposition}

\begin{proof}
This can be obtained by the same arguments used in the proof of \cite[Proposition 12]{BuffCheritat2012} and the tools developed in \cite{Shishikura2000}.
\end{proof}

\subsubsection{Renormalization}

By Lemma \ref{equ-2022-2-10-10}, for $k \geq 0$, there are components $V_{f}^{-k}$ and $W_{f}^{-k}$ properly mapped by $f^{\circ k}$ respectively to $V_{f}$ with degree $1$ and $W_{f}$ with degree $2$. In addition, there is an integer $k_{0}>0$ such that
$$
\forall f \in \mathcal{IS}_{0}, \quad V_{f}^{-k_{0}} \cup W_{f}^{-k_{0}} \subset \mathcal{P}_{\text {rep}, f}
$$

Take sufficiently small $\varepsilon_1>0$. If $f \in \mathcal{IS}_{(0, \varepsilon_{1})}$, set
$$
V_{f}:=\left\{z \in \mathcal{P}_{f}:\  \operatorname{Im}\left(\Phi_{f}(z)\right)>0 \text { and } 0<\operatorname{Re}\left(\Phi_{f}(z)\right)<2\right\}
$$
and
$$
W_{f}:=\left\{z \in \mathcal{P}_{f}:\ -2<\operatorname{Im}\left(\Phi_{f}(z)\right)<2 \text { and } 0<\operatorname{Re}\left(\Phi_{f}(z)\right)<2\right\}.
$$

\begin{proposition} There is a number $\varepsilon_{2}>0$ and an integer $k_{1} \geq 1$ such that for all $f \in \mathcal{IS}_{(0, \varepsilon_{2})}$ and for all integer $k \in\left[1, k_{1}\right]$, the following statements hold.
\begin{itemize}
\item[(1)] The unique connected component $V_{f}^{-k}$ of $f^{-k}\left(V_{f}\right)$ that contains 0 in its closure is relatively compact in $U_{f}$ (the domain of $f$ ) and $f^{ k}: V_{f}^{-k} \rightarrow V_{f}$ is an isomorphism.
\item[(2)] The unique connected component $W_{f}^{-k}$ of $f^{-k}\left(W_{f}\right)$ which intersects $V_{f}^{-k}$ is relatively compact in $U_{f}$ and $f^{ k}: W_{f}^{-k} \rightarrow W_{f}$ is a covering of degree $2$ ramified above $v$.
\item[(3)] $V_{f}^{-k_{1}} \cup W_{f}^{-k_{1}} \subset\left\{z \in \mathcal{P}_{f} ; 2<\operatorname{Re}\left(\Phi_{f}(z)\right)<\frac{1}{\alpha_{f}}-R_{3}-5\right\}$.
\end{itemize}
\end{proposition}

\begin{proof}
This can be obtained by the same arguments used in the proof of \cite[Proposition 13]{BuffCheritat2012}.
\end{proof}

\begin{lemma} (Theorem \ref{ren2020912-1} and Subsection \ref{proofpro2021-10-3-1})
If $f\in\mathcal{IS}_{(0,\vep_2)}$, the map
\beqq
\Phi_f\circ f^{k_1}\circ \Phi^{-1}_{f}:\ \Phi_f(V^{-k_1}_f\cup W^{-k_1}_f)\to\Phi_f(V_f\cup W_f)
\eeqq
projects via $w\to-\tfrac{4m^m}{(m+1)^{m+1}}\exp(2\pi i w)$ to a map $\mathcal{R}(f)\in\mathcal{IS}_{-1/\al_f}$.
\end{lemma}

The polynomial $P_{\al}$ dose not belong to the class $\mathcal{IS}_{\al}$, for $\al>0$ sufficiently close to $0$, there are perturbed Fatou coordinates, there is a renormalization $\mathcal{R}(P_{\al})$ that belongs to $\mathcal{IS}_{-1/\al}$.

\subsubsection{Renormalization tower}

For any positive integer $N$, denote
\beqq
\mbox{Irrat}_{\geq N}:=\{\al=[a_0,a_1,a_2,...]\in\rr\setminus\qq:\ a_k\geq N\ \forall k\geq1\},
\eeqq
where $\al=[a_0,a_1,a_2,...]$ is continued fraction expansion defined in \eqref{paraequ-34}. By \eqref{paraequ-34}, for any $\al=[a_0,a_1,a_2,...]$, one has
\beqq
\al_j=[0,a_{j+1},a_{j+2},a_{j+3},...]\ \forall j\geq0
\eeqq
and
\beqq
\al_{j+1}=\frac{1}{\al_j}-\bigg\lfloor\frac{1}{\al_j}\bigg\rfloor\ \forall j\geq1.
\eeqq

So, the assumption $\al=[a_0,a_1,a_2,...]\in\mbox{Irrat}_{\geq N}$ is equivalent to
\beqq
\al_j\in\bigg(0,\frac{1}{N}\bigg)\ \forall j\geq0.
\eeqq
Let $\frac{p_k}{q_k}$ be the $k$th convergent of $\al$ introduced in \eqref{equ-2022-2-6-1}.

For the given $\vep_2$, take $N$ sufficiently large such that $\frac{1}{N}<\vep_2$.

For any $f_0=P_{\al}$ or $f_0\in\mathcal{IS}_{\al}$, an infinite sequence of renormalizations can be defined, called a renormalization tower, denoted by
\beqq
f_{j+1}:=s\circ\mathcal{R}(f_j)\circ s^{-1},
\eeqq
where the conjugacy by $s:\ z\to\ol{z}$ is introduced so that
\beqq
f'_j(0)=e^{2i\pi\al_j}.
\eeqq

It is convenient to define
\beqq
\mbox{Exp}(\cdot)=-\tfrac{4m^m}{(m+1)^{m+1}}s\circ \exp(2\pi i\cdot):\ w\in\cc\to -\tfrac{4m^m}{(m+1)^{m+1}}s(\exp(2\pi i w))\in\cc^*.
\eeqq
Define
\beqq
\phi_j:=\mbox{Exp}\circ\Phi_{f_j}:\ \mathcal{P}_{f_j}\to\cc\ j\geq0.
\eeqq
The map $\phi_j$ goes from the $j$-th level of the renormalization tower to the next level.

\begin{lemma}\label{equ-2022-2-7-1}
There exists a positive constant $K$ such that for all $f\in\mathcal{IS}_{(0,\vep_2)}$, there is an inverse branch of $\mbox{Exp}$ which is defined on $\mathcal{P}_f$ and takes values in the strip $\{w\in\cc:\ 0<\mbox{Re}(w)<K\}$.
\end{lemma}

\begin{proof}
This can be obtained by the same arguments used in the proof of \cite[Lemma 11]{BuffCheritat2012}.
\end{proof}

Suppose $N$ is large enough such that
\beqq
\frac{1}{N}<\vep_2\ \mbox{and}\ N-R_3>K.
\eeqq
By Lemma \ref{equ-2022-2-7-1}, for any $j\geq1$, there exists an inverse branch $\psi_j$ of $\phi_{j-1}$ defined on the perturbed petal $\mathcal{P}_{f_j}$ with values in $\mathcal{P}_{f_{j-1}}$. Hence, one can introduce the univalent map defined on $\mathcal{P}_{f_j}$ with values in the dynamical plane of $f_0$ as follows:
\beqq
\Psi_j:=\psi_1\circ\psi_2\circ\cdots\circ\psi_j.
\eeqq

Recall that
\beqq
\Phi_{f_j}(\mathcal{P}_{f_j})=\{w\in\cc:\ 0<\Re(w)<1/\al_j-R_3\}.
\eeqq
Define $\mathcal{P}_{j}\subset\mathcal{P}_{f_j}$ and
$\mathcal{P}'_{j}\subset\mathcal{P}_{f_j}$ by
\beqq
\mathcal{P}_{j}:=\{z\in \mathcal{P}_{f_j}:\ 0<\Re(\Phi_{f_j}(z))<1/\al_j-R_3-1\}
\eeqq
and
\beqq
\mathcal{P}'_{j}:=\{z\in \mathcal{P}_{f_j}:\ 1<\Re(\Phi_{f_j}(z))<1/\al_j-R_3\}.
\eeqq
Note that $\mathcal{P}_{j}$ and $\mathcal{P}'_{j}$ is isomorphic by the map $f_j$. Set
\beqq
\mathcal{G}_j:=\Psi_j(\mathcal{P}_{j})\ \mbox{and}\ \mathcal{G}'_j:=\Psi_j(\mathcal{P}'_{j}).
\eeqq

\begin{proposition}
The map $f_j:\ \mathcal{P}_j\to \mathcal{P}'_j$ is conjugate with the map $f^{q_j}_0:\ \mathcal{G}_j\to \mathcal{G}'_j$ via the map $\Psi_j$. That is, the following commutative diagram holds:
 \begin{center}
\begin{tikzpicture}
    \node (E) at (0,0) {$\mathcal{G}_j\subset\Psi_j(\mathcal{P}_{f_j})$ };
    \node[right=of E] (F) at (4,0){$\mathcal{G}'_j\subset\Psi_j(\mathcal{P}_{f_j}) $};
    \node[below=of F] (A) {$\mathcal{P}'_j\subset\mathcal{P}_{f_j}$};
    \node[below=of E] (Asubt) {$\mathcal{P}_j\subset\mathcal{P}_{f_j}$};
   \draw[->] (E)--(F) node [midway,above] {$f^{q_j}_0$};
    \draw[<-] (F)--(A) node [midway,right] {$\Psi_j$}
                node [midway,left] {};
    \draw[->] (Asubt)--(A) node [midway,below] {$f_j$}
                node [midway,above] {};
    \draw[<-] (E)--(Asubt) node [midway,left] {$\Psi_j$};
\end{tikzpicture}.
\end{center}
\end{proposition}

\begin{proof}
This can be obtained by the same arguments used in the proof of \cite[Proposition 14]{BuffCheritat2012}.
\end{proof}

Set
\beqq
D_j:=V^{-k_1}_{f_j}\cup W^{-k_1}_{f_j},\ D'_j:=V_{f_j}\cup W_{f_j},
\eeqq
\beqq
C_j:=\Psi_j(D_j),\ \mbox{and}\ C'_j:=\Psi_j(D'_j).
\eeqq
From the discussions above, it follows that $f^{k_1}_j$ maps $D_j$ to $D'_j$.

\begin{proposition}
The map $f^{k_1}_j:\ D_j\to D'_j$ is conjugate with the map $f^{(k_1q_j+q_{j-1})}_0:\ C_j\to C'_j$ via the map $\Psi_j$. That is, the following commutative diagram holds:
 \begin{center}
\begin{tikzpicture}
    \node (E) at (0,0) {$C_j\subset\Psi_j(\mathcal{P}_{f_j})$ };
    \node[right=of E] (F) at (4,0){$C'_j\subset\Psi_j(\mathcal{P}_{f_j}) $};
    \node[below=of F] (A) {$D'_j\subset\mathcal{P}_{f_j}$};
    \node[below=of E] (Asubt) {$D_j\subset\mathcal{P}_{f_j}$};
   \draw[->] (E)--(F) node [midway,above] {$f^{(k_1q_j+q_{j-1})}_0$};
    \draw[<-] (F)--(A) node [midway,right] {$\Psi_j$}
                node [midway,left] {};
    \draw[->] (Asubt)--(A) node [midway,below] {$f^{k_1}_j$}
                node [midway,above] {};
    \draw[<-] (E)--(Asubt) node [midway,left] {$\Psi_j$};
\end{tikzpicture}.
\end{center}
\end{proposition}

\begin{proof}
This can be obtained by the same arguments used in the proof of \cite[Proposition 15]{BuffCheritat2012}.
\end{proof}

\subsubsection{Neighborhoods of the post-critical set}

\begin{proposition}
For all $\al\in\text{Irrat}_{\geq N}$ and all $f\in\mathcal{IS}_{\al}$, the post-critical set of $f$ is infinite.
\end{proposition}

\begin{proof}
This can be obtained by the same arguments used in the proof of \cite[Proposition 16]{BuffCheritat2012}.
\end{proof}

For every $\al\in\text{Irrat}_{\geq N}$, define $(U_j)$ of open sets containing the post-critical set of $P_{\al}$.

For $j\geq1$, the $j$-th renormalization of $f_0:=P_{\al}$ has a perturbed petal $\mathcal{P}_{f_j}$ and a perturbed Fatou coordinate
\beqq
\Phi_{f_j}:\mathcal{P}_{f_j}\to\{w\in\cc:\ 0<\text{Re}\, (w)<1/\al_j-R_3\}.
\eeqq

The set
\beqq
D_j:=V^{-k_1}_{f_j}\cup W^{-k_1}_{f_j}\subset\mathcal{P}_{f_j}
\eeqq
is mapped by $f^{k_1}_j$ to
\beqq
D'_j:=\{z\in\mathcal{P}_{f_j}:\ 0<\text{Re}\,(\Phi_{f_j}(z))<2\ \text{and}\ \text{Im}\,(\Phi_{f_j}(z))>-2\}.
\eeqq

There is a map $\Psi_j$, univalent on $\mathcal{P}_{f_j}$, with values in the dynamical plane of $P_{\al}$, conjugating $f^{k_1}_j: D_j\to D'_{j}$ to $P^{(k_1q_j+q_{j-1})}_{\al}:C_j\to C'_j$ with
\beqq
C_j=\Psi_{j}(D_j)\ \text{and}\ C'_j=\Psi_j(D'_j).
\eeqq

\begin{definition}
For $\al\in\text{Irrat}_{\geq N}$ and $j\geq1$, set
\beqq
U_j(\al):=\bigcup^{q_{j+1}+lq_j}_{k=0}P^{k}_{\al}(C_j),
\eeqq
where $l=k_1-R_3-4$.
\end{definition}

\begin{proposition}
For  all $j\geq1$, the post-critical set $\mathcal{PC}(P_{\al})$ is contained in $U_{j}(\al)$.
\end{proposition}

\begin{proof}
This can be obtained by the same arguments used in the proof of \cite[Proposition 17]{BuffCheritat2012}
\end{proof}

For the positive integer $N$, set
\beqq
\mathcal{S}_N:=\{\al\in\rr\setminus\qq:\ \al\in\text{Irrat}_{\geq N}\ \mbox{is a bounded type irrational number}\}.
\eeqq

\begin{proposition}
For any $\al\in\mathcal{S}_N$ and any $\vep>0$, if $j$ is sufficiently large, then the set $U_j(\al)$ is contained in the $\vep$-neighborhood of the Siegel disk $\De_{\al}$.
\end{proposition}

\begin{proof}
This can be obtained by the same arguments used in the proof of \cite[Proposition 18]{BuffCheritat2012}.
\end{proof}

\subsection{Lebesgue density near the boundary of a Siegel disk}

\begin{definition}\cite[Definition 10]{BuffCheritat2012}
If $\al$ is a Brjuno number and $\de>0$, denote by $\De$ the Siegel disk of $P_{\al}$
and by $K(\de)$ the set of points whose orbit under iteration of $P_{\al}$ remains at distance less than $\de$ from $\De$.
\end{definition}

\begin{theorem} \label{equ2021-2-3-1}(see Subsection \ref{lebesgue-density-12-12-1})
Let $\al$ be an irrational number of bounded type and $\de>0$ be any constant. Then, every point on
the boundary of the Siegel disk is a Lebesgue density point of $K(\de)$.
\end{theorem}

\begin{corollary}
Assume $\al$ is a bounded type irrational and $\de>0$. If
\beqq
d:=d(z,\partial\De)\to0\ \mbox{with}\ z\not\in\ol{\De},
\eeqq
then
\beqq
\mbox{dens}_{D(z,d)}(\cc\setminus K(\de))\to0.
\eeqq
\end{corollary}

\begin{proof}
This can be obtained by the same arguments as in the proof of \cite[Corollary 5]{BuffCheritat2012}.
\end{proof}

\subsection{The proof of Proposition \ref{equ-2022-2-10-1}}

By Proposition \ref{equ-21-12-11-5}, one has
\beqq
\liminf\text{area}(K_{\al_n})\geq\frac{1}{2}\text{area}(K_{\al}).
\eeqq
The aim of the following discussions is to prove that
\beqq
\liminf\text{area}(K_{\al_n})\geq\text{area}(K_{\al}).
\eeqq
This claim can be obtained by the same arguments in the discussions of \cite[Subsection 1.7]{BuffCheritat2012}.

\section{Appendix}\label{julia-set-equ-2022-2-12-3}

\subsection{Continued fraction}\label{equ2021-2-3-3}

Some materials on continued fraction are collected in this subsection \cite{HardyWright1938}.

For any real number $\al$, the number $\lfloor\al\rfloor\in\zz$ is the largest integer less than or equal to $\al$, which is called the integer part of $\al$, the number $\{\al\}=\al-\lfloor\al\rfloor$ is the fractional part of $\al$. Two sequences $(a_k)_{k\geq0}$ and $(\al_k)_{k\geq0}$ are defined recursively:
\beq\label{paraequ-34}
a_0=\lfloor\al\rfloor,\ \al_0=\{\al\},\ a_{k+1}=\bigg\lfloor\frac{1}{\al_{k}}\bigg\rfloor,\ \al_{k+1}=\bigg\{\frac{1}{\al_k}\bigg\},
\eeq
it is evident that
\beqq
\frac{1}{\al_k}=a_{k+1}+\al_{k+1}.
\eeqq
For the real number $\al$, the continued fraction expansion of $\al$ is denoted by
\beqq
\al=[a_0,a_1,a_2,a_3,...]=a_0+\cfrac{1}{a_1+\cfrac{1}{a_2+\cfrac{1}{a_3+\cdots}}},
\eeqq
where $(a_k)$ are integers, $a_k$ is the $k$th entry of the continued fraction,  $a_0$ may be any integer in $\zz$, $a_k$ ($k\geq1$) are all positive. If $\al$ is a rational number, then there is a positive integer $k$ such that
\beqq
\al=[a_0,a_1,a_2,a_3,...,a_k]=a_0+\cfrac{1}{a_1+\cfrac{1}{a_2+\cfrac{1}{a_3+\cdots+\cfrac{1}{a_k}}}};
\eeqq
if $\al$ is irrational, then $[a_0,a_1,a_2,a_3,...]$ is an infinitely long sequence.

If $\al$ is irrational and $\sup a_j<\infty$, then $\al$ is called an irrational number of bounded type.

There exist two sequences of integers $(p_k)_{k\geq-1}$ and $(q_k)_{k\geq1}$ defined as follows:
\beqq
p_{-1}=1,\ p_0=a_0,\ p_k=a_kp_{k-1}+p_{k-2},
\eeqq
\beqq
q_{-1}=0,\ q_0=1,\ q_k=a_kq_{k-1}+q_{k-2}.
\eeqq
The number $\tfrac{p_k}{q_k}$ is called the $k$th convergent of $\al$, and the numbers $p_k$ and $q_k$ satisfy:
\begin{itemize}
\item[$\bullet$] \beqq
q_kp_{k-1}-p_{k}q_{k-1}=(-1)^k;
\eeqq
\item[$\bullet$]
$p_k$ and $q_k$ are coprime;
\item[$\bullet$] if $a_1,a_2,...,$ are positive integers, then
\beq\label{equ-2022-2-6-1}
\frac{p_k}{q_k}=[a_0,a_1,...,a_k]\ \forall k\geq0.
\eeq
The number $p_k/q_k$ is said to be the $k$th convergent of $\al$.
\end{itemize}

Set
\beqq
\be_{-1}:=1\ \text{and}\ \be_k:=\al_0\al_1\cdots\al_{k},\ k\geq0.
\eeqq

\begin{proposition}\cite[Proposition 1]{AvilaBuffCheritat2004} \label{fractionprop}
For any irrational number $\al$, the sequences $(a_k)_{k\geq0}$, $(\al_k)_{k\geq0}$, $(\be_k)_{k\geq-1}$, $(p_k)_{k\geq-1}$, and $(q_k)_{k\geq-1}$ are defined as above, one has that, for all $k\geq0$,
\beqq
\al=\frac{p_k+p_{k-1}\al_k}{q_k+q_{k-1}\al_k},\ q_k\al-p_k=(-1)^k\be_k,
\eeqq
\beqq
q_{k+1}\be_k+q_k\be_{k+1}=1,\ \text{and}\ \frac{1}{q_{k+1}+q_k}<\be_k<\frac{1}{q_{k+1}}.
\eeqq
Further,
\beqq
\frac{1}{2q_kq_{k+1}}<\bigg|\al-\frac{p_k}{q_k}\bigg|<\frac{1}{q_kq_{k+1}}\ \forall k\geq0,
\eeqq
and
\beqq
\al_k=[0,a_{k+1},a_{k+2},...]\ \forall k\geq0.
\eeqq
\end{proposition}

\begin{definition} \cite{Yoccoz1995} (The Yoccoz function and Brjuno function) If $\al\in\rr$ is an irrational number, the function
\beqq
\Phi(\al)=\sum^{\infty}_{k=0}\be_{k-1}\log\frac{1}{\al_k}
\eeqq
is the called the Yoccoz function. If $\al$ is rational, then we set $\Phi(\al)=\infty$. If $\al\in\rr$ and $\Phi(\al)<\infty$, then $\al$ is said to be a Brjuno number.
\end{definition}

\subsection{Parabolic explosion for high degree polynomials}\label{equ-21-12-7-1}

In this subsection, Ch\'{e}ritat's techniques of parabolic explosion \cite{Cheritat2000} is generalized to the high degree polynomial. For a description of the following result for the quadratic polynomials, please refer to \cite{BuffCheritat}.

\begin{definition}\cite[Definition 10]{BuffCheritatSmooth}
For each rational number $p/q$, let $\mathcal{P}_q$ be the set of parameters $\al\in\cc$ such that $P^{\circ q}_{\al}$ has a parabolic fixed point with multiplier $1$. Let $R_{p/q}$ be the largest real number such that $P^{\circ q}_{\al}$ has no multiple fixed point for $\al\in B(\tfrac{p}{q},R_{p/q})\setminus\{\tfrac{p}{q}\}$, that is,
\beqq
R_{p/q}=d\bigg(\frac{p}{q},\mathcal{P}_q\setminus\bigg\{\frac{p}{q}\bigg\}\bigg).
\eeqq
\end{definition}

We give the main term for $P^{\circ q}_{\al}(z)$. This could be derived by direct calculation,

\begin{lemma}\label{equ20208101}
For the polynomial map \eqref{paraequ-33}, assume $\al=\tfrac{p}{q}$ with $p$ and $q$ coprime, one has
\beq\label{near-equat-22}
P^{\circ q}_{p/q}(z)=z+A(p/q)z^{q+1}+O(z^{q+2}).
\eeq
Further, for $\vep\in\cc$ with sufficiently small $|\vep|$, one has
\beq\label{paraequ-5}
P^{\circ q}_{p/q+\vep}(z)-z=z\cdot(2i\pi q\vep+A(p/q)z^{q}+O(z\vep)+O(z^{q+1})).
\eeq
\end{lemma}

\begin{proof}
The proof is divided into two steps.

{\bf Step 1.} An algebraic proof of the expression \eqref{near-equat-22} is provided.

Assume that $P^{\circ q}_{p/q}(z)=z+a_lz^l+a_{l+1}z^{l+1}+\cdots$ with $a_l\neq0$. First, we show that $l=ql_0+1$.

It follows from $P^{\circ q}_{p/q}\circ P_{p/q}=P_{p/q}\circ P^{\circ q}_{p/q}$ and \eqref{paraequ-4} that
\begin{align*}
&P^{\circ q}_{p/q}\circ P_{p/q}(z)\\
=&(e^{2\pi ip/q}z+z^2+O(z^3))+a_l(e^{2\pi ip/q}z+z^2+O(z^3))^l\\
&+a_{l+1}(e^{2\pi ip/q}z+z^2+O(z^3))^{l+1}+\cdots\\
=&P_{p/q}\circ P^{\circ q}_{p/q}(z)\\
=&e^{2\pi ip/q}(z+a_lz^l+a_{l+1}z^{l+1}+\cdots)+(z+a_lz^l+a_{l+1}z^{l+1}+\cdots)^2\\
&+O((z+a_lz^l+a_{l+1}z^{l+1}+\cdots)^3).
\end{align*}
There are four different cases:
 \begin{itemize}
 \item[(i)] $l>\max\{3,\text{degree}\,P_{\al}(z)\}$;
 \item[(ii)] $l\leq3$ and $\text{degree}\,P_{\al}(z)=3$;
 \item[(iii)] $l\leq3$ and $\text{degree}\,P_{\al}(z)>3$;
 \item[(iv)] $3<l\leq\text{degree}\,P_{\al}(z)$.
\end{itemize}

For Case (i) $l>\max\{3,\text{degree}\,P_{\al}(z)\}$, it is sufficient to compare the coefficients in front of $z^l$, that is $a_le^{2\pi ilp/q}z^l=a_le^{2\pi ip/q}z^l$, implying that $e^{2\pi i (l-1)p/q}=1$, that is, $l=ql_0+1$ with $l_0\in\mathbb{N}$,

For Case (ii) $l\leq3$ and $\text{degree}\,P_{\al}(z)=3$, that is, $P_{\al}(z)=e^{2\pi i\al}z+z^2+bz^3$, so,
\begin{align*}
&P^{\circ q}_{p/q}\circ P_{p/q}(z)\\
=&(e^{2\pi ip/q}z+z^2+bz^3)+a_l(e^{2\pi ip/q}z+z^2+bz^3)^l+a_{l+1}(e^{2\pi ip/q}z+z^2+bz^3)^{l+1}+\cdots\\
=&P_{p/q}\circ P^{\circ q}_{p/q}(z)\\
=&e^{2\pi ip/q}(z+a_lz^l+a_{l+1}z^{l+1}+\cdots)+(z+a_lz^l+a_{l+1}z^{l+1}+\cdots)^2\\
&+b(z+a_lz^l+a_{l+1}z^{l+1}+\cdots)^3.
\end{align*}
Hence, if $l=2$, then one has $z^2+a_le^{2\pi ilp/q}z^l=e^{2\pi ip/q}a_lz^l+z^2$, implying that $e^{2\pi ilp/q}=e^{2\pi ip/q}$, that is, $l=ql_0+1$ and thus $q=l_0=1$; if $l=3$, then one has $bz^3+a_le^{2\pi ilp/q}z^l=e^{2\pi ip/q}a_lz^l+bz^3$, one has $l=ql_0+1$ and thus $ql_0=2$.

Similarly, one could study Case (iii) $l\leq3$ and $\text{degree}\,P_{\al}(z)>3$, and obtain $l=ql_0+1$.

There is only one case left $3<l\leq\text{degree}\,P_{\al}(z)$. For this case, suppose the coefficient in front of $z^l$ is $b$ in $P_{\al}(z)$. By applying similar arguments, one has that $l=ql_0+1$.

This is a purely algebraic proof. For a dynamical proof, please refer to \cite[Lemma 1.4]{Milnor2006}.

{\bf Step 2.} We show that $l_0=1$.

By a local analysis of the fixed point $0$ in \cite{Milnor2006} and the arguments used in the Proof of Lemma 4.5.2 in \cite{Shishikura2000}, there exist $q$ disjoint simply connected open subsets $V^{(1)}$,...,$V^{(q)}$, which are belong to the domain of the polynomial $P_{p/q}$ and are called the attracting petals, such that $0\in\partial V^{(i)}$, $P_{p/q}(\ol{V}^{(i)})\subset V^{(i)}\cup\{0\}$, $P_{p/q}$ is injective on $V^{(i)}$; and the orbit spaces $V^{(i)}/\sim$ are isomorphic to $\cc$, where $z\sim P_{p/q}(z)$ if $z$ and $P_{p/q}(z)$ are both in $V^{(i)}$. The basin of attraction (parabolic basin) of $0$ is the set of points $\{z\in\cc:\ \exists n\geq0\ \text{with}\ P^n_{p/q}(z)\in \cup_i V^{(i)}\}$. The immediate basin of attraction is the connected component of the basin of attraction containing $\cup_i V^{(i)}$.

Let $B$ be an immediate basin of $0$, and $B^{\sharp}$ be the parabolic basin. So, $B\subset B^{\sharp}$, $B\cap(\cup_iV^{(i)})\neq\emptyset$, and $\cup_iV^{(i)}\subset B^{\sharp}$. So, $B$ contains one of $V^{(i)}$. Thus, there are at most $q$ immediate basins.

Now, we define a sequence of open subsets inductively. Let $V^{(i)}_0=V^{(i)}$; if $P_{p/q}(V^{(i)}_n)\subset P_{p/q}(V^{(i)}_n)$, then $V^{(i)}_{n+1}$ is defined as the connected component of $P^{-1}_{p/q}(V^{(i)}_n)$ containing $V^{(i)}_n$. It is evident that this definition is reasonable and $P_{p/q}(V^{(i)}_{n+1})=V^{(i)}_n\subset V^{(i)}_{n+1}$. By this definition, $\widetilde{B}=\cup_{n\geq0}V^{(i)}_n$ is an immediate parabolic basin and $\widetilde{B}$ is a component of $B^{\sharp}$. So, $P_{p/q}:\ \widetilde{B}\to\widetilde{B}$ is a branched covering. If there is no critical value or no critical point, then $P_{p/q}: V^{(i)}_{n+1}\to V^{(i)}_n$ is a covering map, and $P^n_{p/q}: V^{(i)}_{n}\to V^{(i)}$ is also a covering map, which inducing a covering map $\widetilde{B}\to V^{(i)}/\sim$. Further, by the assumption that $V^{(i)}$ is simply connected, one has $\widetilde{B}$ is simply connected, implying that $\widetilde{B}$ is isomorphic to $\cc$ and $P_{p/q}$ is a M\"{o}bius transformation. Therefore, $\widetilde{B}$ contains at least a critical point or a critical value.

By Proposition \ref{paraequ-2}, $P_{\al}$ has two critical points $-1$ and $-\tfrac{1}{m+1}$, and the image of $-1$ is $0$ for any parameter $\al$. It is evident that $-1$ is a superattracting fixed point for $P_{\al}$. By the construction above, $0\not\in V^{(i)}$,  $V^{(i)}_n$ does not contain the critical point $-1$. So, the basin only contains the critical point $-\tfrac{1}{m+1}$.

Recall that the Julia set contains all repelling fixed points and all neutral fixed points that do not correspond to Siegel disks; the Fatou set contains all attracting fixed points and all neutral fixed points corresponding to Siegel disks \cite[Theorem 1.1]{Carleson1993}. And, the Julia set is completely invariant, that is, the set and its complement are invariant \cite[Chapt III, Theorem 1.3]{Carleson1993}. So, for $\al=\tfrac{p}{q}$, the point $0$ is a neutral fixed point that does not correspond to a Siegel disk, thus $0$ is contained in the Julia set. This, together with the invariance of the Julia set, yields that the critical point $-1$ is also in the Julia set.

Recall that if $z_0$ is a rationally neutral periodic point, then each immediate basin of attraction associated with the cycle of $z_0$ contains a critical point \cite[Chapt III, Theorem 2.3]{Carleson1993}. So, the critical point $-\tfrac{1}{m+1}$ is contained in the basin of attraction associated with fixed point of $0$.  Each basin of a rationally neutral fixed point of a holomorphic map is contained in the Fatou set \cite[Lemma 10.5]{Milnor2006}.

By contradiction, suppose that $l_0>1$, the multiplicity of the fixed point $0$ for $P^{\circ q}_{p/q}(z)$ is $ql_0+1$. the number of attraction vector is $ql_0$, where the attraction vector $\mathbf{v}$ satisfies that $ql_0a'_l\mathbf{v}^{ql_0}=-1$. The attraction vectors are denoted by $\mathbf{v}_0,...,\mathbf{v}_{ql_0-1}$. And, the orbit of any point for the map $P^{\circ q}_{p/q}$ in the basin of attraction along the direction of attraction vector is convergent to $0$ \cite[Lemma 10.1]{Milnor2006}. Set $L_{j}:=\{z\in\cc:\ z\neq0,\ \lim_{k\to+\infty}P^{\circ kq}_{p/q}(z)\to0\ \text{from the direction}\ \mathbf{v}_j\}$, $j=0,1,...,ql_0-1$.  This, together with $P_{p/q}(L_{j})=L_{j'}$, $0\leq j,j'\leq ql_0-1$, implies that the original map $P_{p/q}$ should have $l_0$ different cycles of basins of attraction for the fixed point $0$. However, the above discussions tell us that each basin of attraction for the fixed point $0$ contains a critical point for the map $P_{p/q}$. By the above discussions, the number of critical point contained in the basin is only $1$.  Hence, $l_0=1$.
\end{proof}

\begin{proposition}\label{explosionfun-1}
For each rational number $\tfrac{p}{q}$,
\begin{itemize}
\item[(1)] there exists a function $\chi:B=B(0,R^{1/q}_{p/q})\subset\cc\to\cc$ such that
\begin{itemize}
  \item[$\bullet$] $\chi(0)=0$,
 \item[$\bullet$] $\chi'(0)\neq0$,
\item[$\bullet$] for any $\de\in B$, $\langle\chi(\de),\chi(\zeta\de),...,\chi(\zeta^{q-1}\de)\rangle$ forms a cycle of period $q$ of $P_{\al}$ with $\zeta=e^{2i\pi p/q}$ and
\beqq
\al=\frac{p}{q}+\de^q.
\eeqq
This understands the following relation: for any $\de\in B(0,r)$,
\beqq
\chi(\zeta\de)=P_{\al}(\chi(\de)).
\eeqq
\end{itemize}
\item[(2)] Any function satisfying these conditions is of the form $\de\to\chi(\zeta^k\de)$ for some $k\in\{0,1,...,q-1\}$.
\item[(3)] A computations gives us
\beqq
\chi'(0)^q=-\frac{2i\pi q}{A(p/q)}.
\eeqq
\end{itemize}
\end{proposition}

The normalization of $\chi$ is given by $\psi(\de)=\chi(\de/\chi'(\de))$, this is equivalent to replace the relation $\al=\tfrac{p}{q}+\de^q$ by
\beqq
\al=\frac{p}{q}-\frac{A(p/q)}{2i\pi q}\de^q.
\eeqq
Let
\beqq
r_{p/q}=\bigg|\frac{2\pi qR_{p/q}}{A(p/q)}\bigg|^{1/q}.
\eeqq

\begin{proposition}\label{explosionfun-2}
 For each rational number $p/q$, there exists a unique holomorphic function $\psi:B(0,r_{p/q})\subset\cc\to\cc$ satisfying
\begin{itemize}
\item[(1)] $\psi(0)=0$,
\item[(2)] $\psi'(0)=1$,
\item[(3)] for any $\de\in B(0,r_{p/q})$,
\beqq
\langle\psi(\de),\psi(\zeta\de),...,\psi(\zeta^{q-1}\de)\rangle
\eeqq
forms a cycle of period $q$ of $P_{\al}$ with $\zeta=e^{2i\pi p/q}$ and
\beq\label{paraequ-6}
\al=\frac{p}{q}-\frac{A(p/q)}{2i\pi q}\de^q.
\eeq
In particular, for any $\de\in B(0,r_{p/q})$, one has
\beqq
\psi(\zeta\de)=P_{\al}(\psi(\de)).
\eeqq
\end{itemize}
\end{proposition}

\begin{remark}
These two propositions are similar with Propositions 8 and 9 in \cite{BuffCheritatSmooth}, and the proof is a modification of the arguments there.
\end{remark}

\begin{proof} (Proof of Proposition \ref{explosionfun-2})
For convenience, let $A=A(p/q)$, $r=r_{p/q}$, and
\beqq
\al(\de)=\frac{p}{q}-\frac{A}{2i\pi q}\de^q.
\eeqq
Consider a function $F(\de,z)=P^{\circ q}_{\al(\de)}(z)-z:(\de,z)\in B(0,r)\times\cc\to\cc$. By direct calculation, $F(\de,0)=P^{\circ q}_{\al(\de)}(0)-0=0$, implying that $F(\de,z)=z\cdot h(\de,z)$, where $h:B(0,r)\times\cc\to\cc$ is analytic. By \eqref{paraequ-5}, $h(\de,0)\approx-A\de^q$. The set of fixed points different from $0$ is the set of $(\de,z)\subset B(0,r)\times\cc$ satisfying that $h(\de,z)=0$, this set is denoted by $\mathcal{M}$. For all $(\de,z)\in\mathcal{M}$, $z$ is a periodic point of $P_{\al(\de)}$ of period dividing $q$. Let $\pi_1:B(0,r)\times\cc\to B(0,r)$ and $\pi_2:B(0,r)\times\cc\to\cc$ be the projection to the first and second coordinates, respectively, i.e., $\pi_1(\de,z)=\de$ and $\pi_2(\de,z)=z$.

{\bf Claim} The restriction $\pi_1:\mathcal{M}\to B(0,r)$ is proper, where the inverse images of compact subsets are compact.

Take a sequence $(\de_n,z_n)\in\mathcal{M}$ satisfying that $\de_n\to\de\in B(0,r)$. Since $\de_n\in B(0,r)$, one has $\al(\de_n)\in B(p/q,R_{p/q})$. It follows from $\zz\subset\mathcal{P}_q$ for any $q$ that $R_{p/q}\leq1$. By a fact that the periodic orbits are contained in a bounded region for any polynomial maps \cite{MorosawaNishimuraTaniguchiUeda2000}, the sequence $z_n$ is in a bounded region. So, there is a subsequence of $z_n$ such that it is convergent to some point $z\in\cc$. Let $n\to\infty$ for the equation $F(\de_n,z_n)=P^{\circ q}_{\al(\de_n)}(z_n)-z_n=0$, one has $P^{\circ q}_{\al(\de)}(z)=z$. So, $(\de,z)\in\mathcal{M}$ and $\pi_1:\mathcal{M}\to B(0,r)$ is proper.

For the equation $P^{\circ q}_{\al(\de)}(z)-z=z\cdot h(\de,z)$, taking derivative with respect to $z$ on both sides of this equation, one has $(P^{\circ q}_{\al(\de)})'(z)-1= h(\de,z)+z\cdot h'_z(\de,z)$. This, together with $h(\de,z)=0$ for $(\de,z)\in\mathcal{M}$, yields that $(P^{\circ q}_{\al(\de)})'(z)-1= z\cdot h'_z(\de,z)$. So, for $(\de,z)\in\mathcal{M}\setminus\{(0,0)\}$, one has
\beqq
h'_z(\de,z)=\frac{\big(P^{\circ q}_{\al(\de)}\big)'(z)-1}{z}\neq0,
\eeqq
where $\big(P^{\circ q}_{\al(\de)}\big)'(z)$ is the multiplier of the periodic point $z$, and the definition of the $\mathcal{P}_q$ is used. So, it follows from the implicit function theorem that $\de$ is a local coordinate for $\mathcal{M}$.

By \eqref{paraequ-5} and \eqref{paraequ-6}, for $(\de,z)$ near the point $(0,0)\in\mathcal{M}$, one has
\beqq
h(\de,z)=A\cdot(z^q-\de^q)+O(\de^qz)+O(z^{q+1})=A\cdot(z^q-\de^q)+\text{higher order terms}.
\eeqq

Thus, $\mathcal{M}$ is locally the union of $q$ smooth manifolds through $(\de,z)=(0,0)$, for which $\de$ is a local parameter by introducing the new variable $\ld=z/\de$. A new coordinate chart with $(\de,\ld)$ is given, $(\de,\ld)\in B(0,r)\times\cc$, this new chart extends the complex manifold $(\de,z)\in B(0,r)\times \cc\setminus\{(0,0)\}$ into some complex manifold $\mathcal{C}$. It is natural to define an analytic projection from $\mathcal{C}$ back to $B(0,r)\times\mathcal{C}$, and set $\widetilde{\pi}=\pi_1\circ\eta:\mathcal{C}\to B(0,r)$. So, the manifold $\mathcal{M}\setminus\{(0,0)\}$ has a natural extension into a complex submanifold $\widetilde{\mathcal{M}}$ in $\mathcal{C}$ by adding $q$ points $(\de,\ld)=(0,\exp(i2\pi k/q))$ with $k\in\zz$. It is evident that $\eta:\widetilde{\mathcal{M}}\to \mathcal{M}$ is proper. This, together with the above Claim that $\pi_1:\mathcal{M}\to B(0,r)$ is proper, yields that the composition $\widetilde{\pi}=\pi_1\circ\eta:\widetilde{\mathcal{M}}\to B(0,r)$ is proper. Note that the coordinate transformation with $\ld=z/\de$ is invertible, thus $\widetilde{\pi}|_{\widetilde{\mathcal{M}}}$ is a covering. This, together with the fact $B(0,r)$ is simply connected, implies that $\widetilde{\pi}$ is a trivial covering. Hence, $\widetilde{\mathcal{M}}$ is a disjoint union of graphs of functions of $\de$ over $B(0,r)$. In other words, there exist $k_0\in\nn$ and for $j=1,...,k_0$, there exist holomorphic functions $\si_j:B(0,r)\to\mathcal{C}$ such that for any $\de\in B(0,r)$, $\widetilde{\pi}\circ\si_j(\de)=\de$, and $\widetilde{\mathcal{M}}$ is the disjoint union of the images of $\si_j$. Without loss of generality, assume that $\si_1$ is the one whose image goes through $(\de,\ld)=(0,1)$.

Now, we show the uniqueness. Let us consider another function $\psi$. It follows from the assumptions that $\psi(0)=0$, $\psi'(0)=1$, and for any $\de\in B(0,r)$, the point $(\de,z)=(\de,\psi(\de))\in\mathcal{M}$. Since $\lim_{\de\to0}\tfrac{\psi(\de)}{\de}=1$, the map $\de\to(\de,z)=(\de,\psi(\de))$ can be lifted to a continuous function $\widetilde{\psi}$ from $B(0,r)$ to $\widetilde{\mathcal{M}}$, which maps $0$ to $(\de,\ld)=(0,1)$. This, together with the union of the leaves composing $\widetilde{\mathcal{M}}$ is disjoint, and $\widetilde{\psi}(0)=\si_1(0)$, yields that $\widetilde{\psi}=\si_1$ by the continuity.

Now, we show the existence. Take $\psi=\pi_2\circ\eta\circ\si_1$. It is sufficient to show that for any $\de\in B(0,r)$, $\psi(e^{2i\pi p/q}z)=P_{\al(\de)}(\psi(\de))$. By checking the function $g(\de)=P_{\al(\de)}(\psi(e^{-2i\pi p/q}\de))$ satisfying the three conditions, it follows from the uniqueness of these functions that $g=\psi$.

This completes the whole proof.

\end{proof}
\begin{lemma}\label{inequ-57}\cite[Theorem 1.B]{Hubbard1993} (Pommerenke-Levin-Yoccoz inequality)
Let $P$ be a monic polynomial of degree $d$ with $K_P$ connected, and $z\in K_P$ a repelling fixed point with $P^{\prime}(z)=\omega$. The $q'$ external rays of $K_P$ which land at $z$ are transitively permuted by $P$, and this permutation must preserve their circular order since $P$ is a local homeomorphism at $z$; thus the permutation must be a circular permutation, which shifts each ray to the one which is $p'$ further counterclockwise for some $p'<q'$. Let $m=\text{gcd}(p',q')$ be the cycle number of $P$ at $z$, and $\tfrac{p'}{q'}=\tfrac{p}{q}$ with $p$ and $q$ coprime, $\tfrac{p}{q}$ is called the combinatorial rotation number of $P$ at $z$.

There exists a branch $\tau$ of $\log P^{\prime}(z)$ which satisfies
\beqq
\frac{\text{Re}\,\tau}{|\tau-2\pi ip/q|^2}\geq\frac{mq}{2\log d}.
\eeqq
\end{lemma}

\begin{definition}\cite[Definition 3.6]{GoldbergMilnor1993}
Let $\bar{\De}\subset\cc$ be a topological embedded closed disk with interior $\De$. A map $f:\bar{\De}\to\cc$ will be called weakly polynomial-like of degree $d$ if $f(\partial\De)\cap\De=\emptyset$, and if the induced map on integer homology
\beqq
f_*:\ H_2(\bar{\De},\partial\De)\cong\zz\to H_2(\cc,\cc\setminus\{z_0\})\cong\zz
\eeqq
is multiplication by $d$, where $z_0$ can be any base point in $\De$.
\end{definition}

\begin{lemma}\cite[Lemma 3.7]{GoldbergMilnor1993}\label{equ-20208103}
 If $f:\bar{\De}\to\cc$ is weakly polynomial-like of degree $d$, with isolated fixed points, then each
point $f(z_i)=z_i$ can be assigned a Lefschetz index $\iota(f,z_i)\in\zz$, which is a local invariant, so that the sum of these Lefschetz indices is equal to the degree $d$.
\end{lemma}

\begin{remark}\label{equ-2020892} \cite{GoldbergMilnor1993}
If $f$ is a polynomial of degree $d$ and $\De$ is a disk with center at the origin and sufficiently large radius, then $f|\bar{\De}$ is polynomial-like of degree $d$, and the Lefschetz index is $+1$ at a simple fixed point and $\mu$ at a $\mu$-fold fixed point.
\end{remark}

\begin{lemma}\label{equ-2021-12-8-1}
Suppose the Julia set is connected (or the postcritical set is bounded), then for every rational number $\tfrac{p}{q}$, one has
\beqq
R_{p/q}\geq\frac{C}{q^3},
\eeqq
where $C$ is some positive constant dependent on $m$.
\end{lemma}

\begin{proof}
Consider the simple situation that $q=1$.  The fixed points of $P_{\al}(z)$ are the solutions to $e^{2\pi i\al}z(1+z)^m=z$. By direct computation, the fixed points are $z=0$, and $z=-1+e^{(-2\pi i\al+2k \pi i)/m}$ for $k=0,1,...,m-1$. The derivative of  $P_{\al}(z)$ is $e^{2\pi i\al}(1+z)^m+me^{2\pi i\al}z(1+z)^{m-1}$. So, the multipliers at the points $-1+e^{(-2\pi i\al+2k \pi i)/m}$, $k=0,1,...,m-1$, are
\begin{align}\label{equ20208102}
&e^{2\pi i\al}(e^{(-2\pi i\al+2k \pi i)/m})^m+me^{2\pi i\al}(-1+e^{(-2\pi i\al+2k \pi i)/m})(e^{(-2\pi i\al+2k \pi i)/m})^{m-1}\nonumber\\
=&1+m(-e^{2\pi i\al+(-2\pi i\al+2k \pi i)(m-1)/m}+e^{2\pi i\al+(-2\pi i\al+2k \pi i)/m+(-2\pi i\al+2k \pi i)(m-1)/m})\nonumber\\
=&1+m(-e^{2\pi i(\al-k)/m}+1)=m+1-m(e^{2\pi i(\al-k)/m}),\ k=0,1,...,m-1.
\end{align}
So, $P_{\al}(z)$ has $m+1$ distinct simple fixed points. Suppose that $\al=\al_0+i \al_1$ with $\al_0$, $\al_1\in\rr$. Plugging this into $m+1-m(e^{2\pi i(\al-k)/m})$, if some fixed point has multiplier $1$, then $e^{2\pi i(\al-k)/m}=1$, thus one has that $\al_1=0$ and $\al_0=lm+k$, where $k=0,1,...,m-1$, and $l\in\zz$.  In particular,
if $\al$ is an integer, then there is only one fixed point with multiplier $1$, and the norm of the multipliers at the other fixed points are bigger than $1$.

Next, consider the situation $q\geq2$. For any rational number $\tfrac{p}{q}$ with $p$ and $q$ coprime, take $\al\in\mathcal{P}_{q}\setminus\{p/q\}$, $P^{\circ q}_{\al}$ has a multiple fixed point $z_0$. In other words, $P_{\al}$ has a parabolic cycle $\langle z_0,z_1,...,z_{q_1-1}\rangle$ of period $q_1$ with $q_1|q$, that is, $q$ is a multiple of $q_1$. Each
immediate basin of attraction associated with the cycle of contains
a critical point \cite[Chapter 3, Theorem 2.3]{Carleson1993}. This. together with the arguments in the proof of Lemma \ref{equ20208101} (Step 2), yields that there is only one parabolic cycle. Further, the forward orbits of the two critical points are contained in the filled-in Julia set, implying that the Julia set is connected.

Now, we show that the fixed point $0$ is either parabolic or repelling.
 Let the union of all the forward iterates of the critical points be the postcritical set. By \cite[Theorem 11.17]{Milnor2006}, every Cremer fixed point or periodic point
for a rational map is contained in the closure of the postcritical set;
similarly, the boundary of any Siegel disk or cycle of Siegel disks
is contained in the postcritical set. Since the critical point is contained in the basin of the parabolic cycle, the fixed point $0$ can not be either Cremer or Siegel fixed point.
By \cite[Theorem 2.2]{Carleson1993}, the immediate basin of an attracting cycle contains at least one critical point, it follows from the arguments of this statement (\cite[Theorem 2.2]{Carleson1993}) that the critical point is $-\tfrac{1}{m+1}$ (This is also similar with the arguments in the proof of Lemma \ref{equ20208101} (Step 2)). Since this critical point is already contained in the parabolic basin, it is impossible that $0$ is an attracting fixed point.

If $0$ is parabolic, then $\al=\tfrac{p'}{q}$, where $p'$ might not be prime to $q$. So, the distance between $\al$ and $\tfrac{p}{q}$ is bounded from below by $\tfrac{1}{q}$.

If $0$ is repelling, then by \eqref{paraequ-5}, $\text{Im}\,\al<0$. Since the Julia set is connected, it follows from Douady's landing theorem or Goldberg and Milnor's fixed point portrait \cite{GoldbergMilnor1993} that there are finitely many external rays landing at the repelling fixed point $z=0$, the number of external rays is denoted by $s$. Those $s$ rays are ordered cyclically with respect to the increasing of their arguments counterclockwise, and are permuted by $P_{\al}$. Each ray is mapped to some other ray with the number of index counterclockwise plus $r$. Suppose $r/s=r'/s'$ where $t=\text{g.c.d}\,(r,s)$. It follows from Lemma  \ref{inequ-57} that $\al$ is contained in a closed disk with radius $\tfrac{\log (m+1)}{2\pi s}$ and tangent to the real axis at $r/s$.
\beqq
\frac{\text{Re}\,2\pi i\al}{|2\pi i\al-2\pi ir/s|^2}=\frac{\text{Re}\,\tau}{|\tau-2\pi ip/q|^2}\geq\frac{mq}{2\log d}=\frac{s}{2\log(m+1)}.
\eeqq
that is,
\beqq
\frac{\text{Re}\, i\al}{|\al- r/s|^2}\geq\frac{\pi s}{\log(m+1)}.
\eeqq
Suppose the radius is $R=\tfrac{\log (m+1)}{2\pi s}$ and $\al=x+iy$, the above inequality can be transformed into the following inequality,
\beqq
(x-r/s)^2+(y+R)^2\leq R^2,
\eeqq
which is a closed disk.

By Pythagoras theorem, one has
\beqq
\bigg| \alpha-\frac{p}{q}\bigg|\geq
\sqrt{\bigg(\frac{r}{s}-\frac{p}{q}\bigg)^2+\bigg(\frac{\log(m+1)}{2\pi s}\bigg)^2}-\frac{\log(m+1)}{2\pi s}.
\eeqq

Since $q>s$ by Lemma \ref{equ2020891}, one has
\beqq
\bigg| \alpha-\frac{p}{q}\bigg|\geq \frac{2\pi}{q^3\log(m+1)\bigg(\sqrt{\bigg(\frac{2\pi}{q\log(m+1)}\bigg)^2+1}+1\bigg)}.
\eeqq
Let $q$ approach $\infty$, the right hand side has the asymptotic expression $\tfrac{\pi}{q^3\log(m+1)}$. This completes the proof.

\end{proof}

\begin{lemma}\label{equ2020891}
For the $s$ and $q$ in the proof of the above lemma, $0$ is repelling,
\beqq
s<q.
\eeqq
\end{lemma}

\begin{proof}

By assumption, one has $q>1$. Consider the complement in $\cc$ of the $s$ external rays landing at $0$ together with this point $0$. There are $s$ connected components around the fixed point $0$. Choose the connected component containing the critical point $-\tfrac{1}{m+1}$, denoted by $V$. The orbit of the critical point must first visit each connected component of this complement, before coming back somewhere in $V$. Since the critical point belongs to the immediate basin of the parabolic cycle, this implies that the period is bigger than or equal to $s$, yielding that $q\geq s$.

Now, we show that $q>s$. By contradiction. suppose that $q=s$. The fixed point $0$ has two distinct preimages $0$ and $-1$ (with multiplicity $m$). Consider the $q$ external rays landing at the fixed point $0$, the pre-images of these external rays are $mq$ rays landing at $1$. Let $U$ be the component of the complement of the union of $(m+1)q$ rays, where $U$ contains the critical point $-\tfrac{1}{m+1}$. By Proposition \ref{paraequ-2}, $P^{\circ q}_{\al}(U)=V$ and $P^{\circ q}_{\al}: U\to V$ is a proper ramified covering of degree $2$. Let $f$ be the restriction $P^{\circ q}_{\al}: \ol{U}\to\ol{V}$. Recall the definition of $V$, $\ol{U}\subset\ol{V}$.
The point $z=0$ is fixed, and the point of the parabolic cycle whose immediate basin contains the critical point is a multiple fixed point of $f$. Thus, the sum of Lefschetz indices is $\geq m+2$.
By the calculation in \eqref{equ20208102} and Remark \ref{equ-2020892}, the Lefschetz index at any of these fixed point is $+1$. By Lemma \ref{equ-20208103}, the sum of Lefschetz indices is $m+1$. This is a contradiction.
\end{proof}

In the following discussions, the polynomials $P_{\al}$ and $P_{\al_n}$ have Siegel disks $\De$ and $\De_n$, respectively; let $r$ and $r_n$ be the corresponding conformal radius of $\De$ and $\De_n$ at $0$, respectively; the conformal isomorphisms are $\phi:D(0,r)\to\De$ and $\phi_n:D(0,r_n)\to\De_n$ which maps $0$ to $0$ with derivative $1$, respectively.

\subsection{The points on the boundary of the Siegel disk}\label{lebesgue-density-12-12-1}

In this subsection, we show Theorem \ref{equ2021-2-3-1}. As a consequence, the Hausdorff dimension of the Julia set of $P_{\al}(z)=e^{2\pi i\al}z(1+z)^m$ is strictly less than two, where $\al$ is an irrational number of bounded type. The results in this section are the natural extension of the work of McMullen in \cite{McMullen1998}.

\subsubsection{Siegel renormalization}

Let $S^1=\rr/\zz$ and a map $F:S^1\to S^1$ defined by $F(x)=x+\al$, for $x\in S^1$, $\{x\}$ denote the unique real number representing $x$ in $(-1/2,1/2]$, $\|x\|=|\{x\}|$ denote the distance from $x$ to $0$. The points $q_k\al$, $k\in\nn$, are successive closet returns of the origin to itself under the rotation by $\al$ by Proposition \ref{fractionprop}, that is,
$\al>\|q_1\al\|>\|q_2\al\|>\|q_3\al\|>\cdots$, and $q_n's$ can be defined as the smallest integers such that this inequalities hold. Further, $(-1)^n\{q_n\al\}>0$, where the fractional part is given by $\{q_n\al\}=q_n\al-p_n$.

For $x,x^{\prime}\in S^1$ with $x\neq\tfrac{1}{2}+x^{\prime}$, let $[x,x^{\prime}]$ be the shorter interval bounded by these two points, for convenience, $[x,x^{\prime}]=[x^{\prime},x]$.

\begin{definition}\cite{McMullen1998}
Let $I_n=[q_n\al,q_{n+1}\al]$ be the interval bounded by two successive closet returns to the origin.
Since $(-1)^n\{q_n\al\}>0$, $0\in I_n$. For $x\in I_n$, the first return map on $I_n$ is for any $x\in I_n$, the minimal positive integer $i$ such that $F^i(x)\in I_n$.

The renormalized dynamical system is defined by the two transformations:
\beqq
F^{q_k}:\ [0,q_{k+1}\al]\to[q_k\al,x_k],
\eeqq
\beqq
F^{q_{k+1}}:\ [q_{k}\al,0]\to[x_k,q_{k+1}\al],
\eeqq
where $x_k=(q_k+q_{k+1})\al$.
\end{definition}

\begin{definition}\cite{McMullen1998}
For $k\geq3$, let $r_k=q_{k-1}-q_k$, $L_k=[r_{k+1}\al,r_k\al]$, the extended renormalization $\mathcal{R}_k(F)$ is defined by the pair of maps:
\beqq
F^{q_k}:\ [0,r_k\al]\to[q_k\al,q_{k-1}\al]\subset L_k
\eeqq
\beqq
F^{q_{k+1}}:\ [r_{k+1}\al,0]\to[q_k\al,q_{k+1}\al]\subset L_k.
\eeqq
Since $I_k\subset L_k$, the second pair of maps is called the extension.
\end{definition}

\subsubsection{Hyperbolic geometry}

\begin{definition}(\cite[Chap. 2]{McMullen3ManfildRenor}, \cite{BenedettiPetronio1992})
The hyperbolic space $\mathbb{H}^n$ is a complete simply connected $n$-manifold of constant sectional curvature $-1$. It can be represented by several isometric models. For example, consider the symmetric bilinear form of signature $(n,1)$ in $\rr^{n+1}$:
\beqq
\langle x|y\rangle_{(n,1)}=\sum^n_{i=1}x_i\cdot y_i-x_{n+1}\cdot y_{n+1},
\eeqq
where $x=(x_1,...,x_{n+1})$, $y=(y_1,...,y_{n+1})\in\rr^{n+1}$. The upper fold of the hyperboloid naturally associated to $\langle\cdot|\cdot\rangle_{(n,1)}$:
\beqq
I_n=\{x\in\rr^{n+1}:\ \langle x|x\rangle_{(n,1)}=-1,\ x_{n+1}>0\}.
\eeqq
It is evident that $I_n$ is the pre-image of a regular value of a differentiable function, it is a differentiable oriented hypersurface in $\rr^{n+1}$. The tangent space at $x\in I_n$ is
\beqq
T_xI_n=\{y\in\rr^{n+1}:\ \langle x|y\rangle_{(n,1)}=0\}:=\{x\}^{\perp}.
\eeqq
Since $\langle x|x\rangle_{(n,1)}=-1$, $\langle \cdot|\cdot\rangle_{(n,1)}$ restricted to $\{x\}^{\perp}$ is positive-definite, which is a scalar product on $\{x\}^{\perp}$. So, this scalar product naturally induces a differentiable metric, which can be thought of as a Riemann metric. The set $I_n$ with this metric is a Riemann manifold, denoted by $\mathbb{I}_n$.

A Poincar\'{e} ball model is defined by the stereographic projection from $\mathbb{I}_n$ to $\rr^{n}\times\{0\}\approx\rr^{n}$, defined by
\beqq
\pi(x)=\frac{(x_1,...,x_n)}{1+x_{n+1}}.
\eeqq
It is evident that $\pi$ is a diffeomorphism of $\mathbb{I}_n$ onto the open Euclidean unit ball of $\rr^n$. The metric defined on this set is the given by the pull-back metric with respect to $\pi^{-1}$, that is,
\beqq
ds^2=\frac{4dx^2}{(1-r^2)^2},
\eeqq
where $x=(x_1,...,x_n)\in\rr^n$ and $r=\sqrt{x^2_1+\cdots+x^2_n}$.

The boundary of the Poincar\'{e} ball models the sphere at infinity $S^{n-1}_{\infty}$
for hyperbolic space, and the isometries of $\mathbb{H}^n$ prolong to
conformal maps on the boundary.
\end{definition}

In three-dimensional $\mathbb{H}^3$, there is a natural identification between the sphere at infinity
 $S^{2}_{\infty}$ and the Riemann sphere $\wh{\cc}$, this gives an isomorphism between the orientation
preserving group $\text{Isom}^{+}(\mathbb{H}^3)$ and the group of fractional linear
transformations $\text{Aut}(\wh{\cc})\cong PSL_2(\cc)$.

Let $M$ be a Riemannian manifold. If $\si$ is a piecewise differentiable path in $M$, denote by $L(\si)$ its length. Each loop in $M$ is homotopic to a piecewise differentiable loop based at the same point, so that each element in the fundamental group, denoted by $\Pi_1(M)$, can be thought of as a piecewise differentiable loop up to homotopy. Given any point $x\in M$, the fundamental group with the base point $x$ is denoted by $\Pi_1(M,x)$. For any $\ep>0$, set
\beqq
M_{(0,\ep]}:=\{x\in M:\ \exists\langle\si\rangle\in\Pi_1(M,x)\setminus\{1\}\ s.t.\ L(\si)\leq\ep\}
\eeqq
and
\beqq
M_{[\ep,\infty)}:=\{x\in M:\ \forall\langle\si\rangle\in\Pi_1(M,x)\setminus\{1\}\ s.t.\ L(\si)\geq\ep\}.
\eeqq
The subset $M_{(0,\ep]}$ is called the $\ep$-thin part of $M$, $M_{[\ep,\infty)}$ is said to be the $\ep$-thick part of $M$. If this constant $\ep$ is fixed, these two parts are called the thin and thick parts of $M$.

Let $T$ be a locally compact Hausdorff topological space; a group $\Ga$ of homeomomorphisms of $T$ is said to operate:
\begin{itemize}
\item freely if $\ga\in\Ga$, $x\in T$ and $\ga(x)=x$ implies $\ga=id$;
\item properly discontinuously if for any pair $H,K$ of compact subsets of $T$ the set $\Ga(K,H)=\{\ga\in\Ga:\ \ga(K)\cap H\neq\emptyset\}$ is finite.
\end{itemize}

Let $G$ be a fixed group,
\begin{itemize}
\item if $H$ and $K$ are subgroups of $G$, denote by $[H,K]$ the subgroup of $G$
generated by all the elements $[h,k]=hkh^{-1}k^{-1}$, for $h\in H$ and $k\in K$;
\item the $m$-th commutator $G^m$ of $G$ is recursively defined by
\beqq
\left\{
  \begin{array}{ll}
    G^1=[G,G], & \\
    G^{m+1}=[G,G^m], & m\geq1;
  \end{array}
\right.
\eeqq
\item $G$ is said to be nilpotent if for some integer $m$, $G^m=\{1\}$;
\item a subgroup $H$ of $G$ is said to have finite index if the quotient set $G/H$ is finite;
\item if $G$ admits a subgroup of finite index enjoying a property $\mathcal{P}$, $G$
is said to enjoy the property almost-$\mathcal{P}$.
\end{itemize}

\begin{theorem}(Margulis' Lemma) \cite[Theorem D.1.1]{BenedettiPetronio1992}
For any $n\in\nn$, there is $\ep_n\geq0$ such that for any properly discontinuous
subgroup $\Ga$ of $\text{Isom}(\mathbb{H}^n)$ and for any $x\in\mathbb{H}^n$, the group
$\Ga_{\ep_n}(x)$ generated by the set
\beqq
F_{\ep_n}(x)=\{\ga\in\Ga:\ d(x,\ga(x))\leq\ep_n\}
\eeqq
is almost-nilpotent.
\end{theorem}
Geometrically, the Margulis' Lemma tells us that the thick part of $M$ is not very complicated.

A Kleinian group $\Ga$ is a discrete subgroup of $\text{Isom}(\mathbb{H}^n)$. A Kleinian
group is elementary if it contains an Abelian subgroup of finite index.

A hyperbolic $n$-manifold $M$ is a complete Riemannian manifold of constant curvature
$-1$. Any such manifold can be presented as a quotient $M=\mathbb{H}^n/\Ga$ of hyperbolic
space by a Kleinian group $\Ga$.

For $E\subset S^{n-1}_{\infty}$, the convex hull of $E$, denoted by $\text{hull}(E)$,
is the smallest convex subset of $\mathbb{H}^n$ containing all geodesics
with both endpoints in $E$. The convex core $K$ of a hyperbolic manifold
$M=\mathbb{H}^n/\Ga$ is given by $K=\text{hull}(E)/\Ga$.

\begin{definition}\cite{McMullen3ManfildRenor}
Let $\Ld\subset S^2_{\infty}$ be a compact set with convex hull $K\subset\mathbb{H}^3$. A point $x\in\Ld\subset S^2_{\infty}$ is a deep point of $\Ld$
if there is a geodesic ray $\ga:[0,\infty)\to K$, converging to $x$ and parameterized by hyperbolic arclength, such that
\beqq
d(\ga(s),\partial K)\geq \de s>0\ \forall s>0,
\eeqq
where $\de$ is a positive constant. Otherwise, $x$ is called a shallow point.

Let $\Ld\subset\cc$ be a compact set, for all $z\in\Ld$, $z$ is a deep point of $\Ld$ if and only if there exists $\de>0$ such that for all $0<r<1$, one has
\beqq
B(y,s)\subset B(z,r)-\Ld\to s=O(r^{1+\de});
\eeqq
we say that $z$ is a measurable deep point of $\Ld$, if there exists $\de>0$ such that for all $r>0$, one has
\beqq
\text{area}(B(z,r)-\Ld)=O(r^{2+\de}).
\eeqq
\end{definition}

For more information on hyperbolic geometry, please refer to \cite{BenedettiPetronio1992}.

\begin{definition}\cite{Hatcher2002}
A covering space or cover of a space $X$ is a space $\wt{X}$ together with a map $p:\wt{X}\to X$ satisfying the following condition:
every point $x\in X$ has an open neighborhood $U_x\subset X$ such that $p^{-1}(U_x)$ is a disjoint union of open subsets, each of which
is mapped by $p$ homeomorphically onto $U_x$.
\end{definition}

\subsubsection{Quasi-conformal surgery}

By Proposition \ref{paraequ-2}, the polynomial $P_{\al}$ has two finite critical points $-1$ and $-\tfrac{1}{m+1}$ with the corresponding critical values $0$ and $-\tfrac{m^m}{(m+1)^{m+1}}e^{2\pi i\al}$.
By \cite{Zhang2011}, if $\al$ is irrational of bounded type, then the  Siegel disk at $0$ is a quasi-disk, and the critical point $-\tfrac{1}{m+1}$ belongs to the boundary of the Siegel disk. The main tool is the combination of the Blaschke model and the quasi-conformal surgery. We sketch the main arguments in \cite{Shen2006} for the completeness of the discussions. Based on these results, we show that the boundary of Siegel disk is Lebesgue density point of $K(\de)$ by applying the method used in \cite{McMullen1998}.

By direct verification, the polynomial map $P_{\al}(z)=e^{2\pi i\al}z(1+z)^m$ is topologically conjugate with the map $\wt{f}(z)=e^{2\pi i\al}z(1+z/m)^m$ by the conjugacy map $\psi(z)=z/m$, that is, the following commutative diagram holds:
\begin{center}
\begin{tikzpicture}
    \node (E) at (0,0) {$\mathbb{C}$ };
    \node[right=of E] (F) at (4,0){$\mathbb{C} $};
    \node[below=of F] (A) {$\mathbb{C}$};
    \node[below=of E] (Asubt) {$\mathbb{C}$};
   \draw[->] (E)--(F) node [midway,above] {$\wt{f}(z)=e^{2\pi i\al}z(1+z/m)^m$};
    \draw[->] (F)--(A) node [midway,right] {$z/m$}
                node [midway,left] {};
    \draw[->] (Asubt)--(A) node [midway,below] {$P_{\al}(z)=e^{2\pi i\al}z(1+z)^m$}
                node [midway,above] {};
    \draw[->] (E)--(Asubt) node [midway,left] {$z/m$};
\end{tikzpicture}.
\end{center}
By direct computation, $\wt{f}(z)$ has two critical points $c_0=-\frac{m}{m+1}$ and $c_1=-m$, and a fixed point $0$. In this section, we will study the dynamics of $g(z)$.

\begin{definition}
Let $K$ be a subset of $\cc$ and $\be$ be a positive number.
\begin{itemize}
\item the $\be$-dimensional Hausdorff content is
\beqq
\mathcal{H}^{\be}_{\infty}(K)=\inf\bigg\{\sum_i \mbox{diam}(U_i)^{\be}:\ K\subset\bigcup_{i} U_i\bigg\},
\eeqq
where $(U_i)$ is any countable cover of $K$.
\item The Hausdorff dimension of $K$ is
\beqq
\mbox{dim}(K)=\inf\{\be:\ \mathcal{H}^{\be}_{\infty}(K)=0\}.
\eeqq
\end{itemize}
\end{definition}

\begin{theorem}\cite[Theorem 1.1]{Shen2006}
For any $\al\in\rr\setminus\qq$ of bounded type, the Hausdorff dimension of the Julia set of $\wt{f}(z)$ is strictly less than $2$.
\end{theorem}

\begin{lemma}\cite[Lemma 2.1]{Shen2006}
Let $\mathbb{T}$ be the unit circle. For any $\al>0$, there is $t(\al)\in[0,1)$ such that the Blaschke product
\beq\label{equ2021-2-3-2}
F(z)=e^{2\pi i t(\al)}z^{m+1}\big(\tfrac{z-(2m+1)}{1-(2m+1)z}\big)^m
\eeq
has the following properties:
\begin{itemize}
\item[(1)] $F(z)$ has a unique critical point at $z=1$ in $\mathbb{T}$ with $F^{\prime}(1)=F^{\prime\prime}(1)=0$ and $F^{\prime\prime\prime}(1)\neq0$;
\item[(2)] the restriction of $F$ to the unit circle $\mathbb{T}$ is a real-analytic homeomorphism with the rotation number $\al$.
\end{itemize}
\end{lemma}

\begin{lemma}\cite[Beurling-Ahlfors]{Ahlfors1966}\label{equ2021-2-16-1}
An orientation-preserving circle homeomorphism $h$ has a quasi-conformal extension to the disk if and only if
it is quasi-symmetric in the sense that its lift $\wt{h}$ to $\rr$ satisfies that
\beqq
\sup_{x\in\rr}\sup_{t>0}\frac{\wt{h}(x+t)-\wt{h}(x)}{\wt{h}(x)-\wt{h}(x-t)}<+\infty.
\eeqq
\end{lemma}

\begin{remark}
The circle can be modeled either by $\mathbb{T}=\rr/\zz$ or $S^1=\partial\cd$,
the coordinates are different on these two models, where the rotation on $\mathbb{T}$ is given by $R_{\al}(x)=x+\al\ \text{mod}\ 1$, and the rotation on $S^1$ is represented by $R_{\al}(z)=e^{2\pi i\al}z$. For convenience of the expressions,
we adopt these two kinds of expressions without mentioning of this again.
\end{remark}

\begin{lemma} \cite{Herman1987,Swiatek1998,Yoccoz1994,Peterson1996,Peterson2004} \label{equ2021-2-16-2}
Let $\mathbb{T}=\partial\cd$. Suppose $\phi:\mathbb{T}\to\mathbb{T}$ is a real analytic homeomorphism with critical point $z=1$. If the rotation number $\al$ of $\phi$ is irrational of bounded type, then there is a unique quasi-symmetric homeomorphism $h:\mathbb{T}\to\mathbb{T}$ with $h(1)=1$ such that $h\circ \phi=R_{\al}\circ h$, where $R_{\al}(z)=e^{2\pi i\al}z$ is the rigid rotation.
\end{lemma}

Since $h$ is quasi-symmetric, $h$ can be extended to a quasi-conformal homeomorphism $H:\cd\to\cd$ with $H(0)=0$. Set
\beqq
\widetilde{F}(z):=\left\{
                \begin{array}{ll}
                  F(z) & \hbox{for}\ |z|\geq1 \\
                  H^{-1}\circ R_{\al}\circ H(z) & \hbox{for}\ |z|<1.
                \end{array}
              \right.
\eeqq
By the construction, $\wt{F}$ is topologically conjugate with the irrational rotation on the unit disk $\cd$, and the Julia set of $\wt{F}$ can be thought of as the set $\ol{\cup_{j\geq0}\wt{F}^{-j}(\partial\cd)}$. Let $\mu_H$ be the Beltrami coefficient of $H(z)$, define an $\wt{F}$ invariant Beltrami coefficient:
\beqq
\mu(z)=\left\{
         \begin{array}{ll}
           \mu_H, & \hbox{for}\ z\in\cd \\
           (\wt{F}^n)^*\mu_H, & \hbox{for}\ z\in\wt{F}^{-n}(\cd) \\
           0, & \hbox{otherwise.}
         \end{array}
       \right.
\eeqq
It is evident that $\|\mu(z)\|_{\infty}=\|\mu_H\|_{\infty}<1$. By the measurable Riemann mapping theorem, there is a quasi-conformal homeomorphism $\psi:\wh{\cc}\to\wh{\cc}$ with the Beltrami coefficient $\mu_{\psi}(z)=\mu(z)$ almost everywhere for $z\in\cc$ satisfying that $\psi(0)=0$, $\psi(2m+1)=-m$, and $\psi(\infty)=\infty$. So, the map $\psi\circ\wt{F}\circ \psi^{-1}(z)$ is a degree $m+1$ holomorphic branched cover of the sphere $\wh{\cc}$. This, together with the fact that $F(z)$ has a critical point $\infty$ with multiplicity $m$ and a critical point $2m+1$ with multiplicity $m-1$, implies that
$\psi\circ\wt{F}\circ \psi^{-1}(z)$ has a critical point $\infty$ with multiplicity $m$ and a critical point $-m$ with multiplicity $m-1$. So, $\psi\circ\wt{F}\circ \psi^{-1}(z)$ has a critical point $-m$ with multiplicity $m-1$. It is evident that $\psi\circ\wt{F}\circ \psi^{-1}(z)$ has a Siegel disk at the origin
with multiplier $e^{2\pi i\al}$. Hence, $\psi\circ\wt{F}\circ \psi^{-1}(z)=e^{2\pi i\al}z(1+z/m)^m=\wt{f}(z)$ is a polynomial.

So, the dynamics of $\wt{f}(z)$ can be described by a quasi-conformal surgery from $\wt{F}(z)$, since the
Julia set of $\wt{f}(z)$ can be derived by that of $\wt{F}(z)$. (For an illustration diagram for an quadratic polynomial, please refer to the Peterson's work
on quadratic polynomials \cite{Peterson1996}.)

\begin{proposition}\cite[Proposition 2.4]{Shen2006}
Let $\wt{f}(z)=e^{2\pi i\al}z(1+z/m)^m$ and $\al$ be an irrational number of bounded type, then the boundary of the Siegel disk is a quasi-circle passing through a critical point.
\end{proposition}

Let $D_0$ be the Siegel disk of $\wt{f}$. For $\wt{f}=e^{2\pi i\al}z(1+z/m)^m$ with $\al$ bounded type, the Julia set is
$J(\wt{f})=\ol{\cup_{j\geq0}\wt{f}^{-j}(\partial D_0)}$ and the closure of the forward orbits of the critical points of $\wt{f}$ is $P(\wt{f})=\partial D_0\cup\{0,\infty\}$.

So, there is a quasi-conformal map between
$(D_0,0)$ and $(\cd,0)$, denoted by $R$, such that $R\circ \wt{f}\circ R^{-1}=R_{\al}$. This,
together with the fact that $\partial D_0$ is quasi-circle and Lemmas \ref{equ2021-2-16-1} and \ref{equ2021-2-16-2}, implies that $R(z)$ can be
extended to a $K$-quasi-conformal map $\varphi:\wh{\cc}_z\to\wh{\cc}_w$, where $\wh{\cc}_z$
is the domain of $\varphi$ and $\wh{\cc}_w$ is the range of $\varphi$.

\begin{definition}
Let $P(\wt{f})$ be the closure of the forward orbits of the critical points, where the number of elements
of $P(\wt{f})$ is bigger than $2$. Set
\beqq
\|\wt{f}^{\prime}(z_0)\|=\lim_{z\to z_0}\frac{d_{\wh{\cc}-P(\wt{f})}(\wt{f}(z),\wt{f}(z_0))}{d_{\wh{\cc}-P(\wt{f})}(z,z_0)},
\eeqq
where $d_{\wh{\cc}-P(\wt{f})}$ is the hyperbolic metric on $\wh{\cc}-P(\wt{f})$.
\end{definition}

\begin{lemma}\cite{McMullen1994Complex2021}
For $z,\wt{f}(z)\in \wh{\cc}-P(\wt{f})$, one has $\|\wt{f}^{\prime}(z)\|\geq1$.
\end{lemma}

Let $\wh{\rho}(w)=e^{2\pi i\al}w:\wh{\cc}_w\to \wh{\cc}_w$ be the rotation map. We can define the
following commutative diagram:
\begin{center}
\begin{tikzpicture}

    \node (E) at (0,0) {$\wh{\mathbb{C}}_z$ };
    \node[right=of E] (F) at (4,0){$\wh{\mathbb{C}}_z $};
    \node[below=of F] (A) {$\wh{\mathbb{C}}_w$};
    \node[below=of E] (Asubt) {$\wh{\mathbb{C}}_w$};

   \draw[->] (E)--(F) node [midway,above] {$\wt{f}(z)=e^{2\pi i\al}z(1+z/m)^m$};
    \draw[->] (F)--(A) node [midway,right] {$\varphi$}
                node [midway,left] {};
    \draw[->] (Asubt)--(A) node [midway,below] {$\wh{f}$}
                node [midway,above] {};
    \draw[->] (E)--(Asubt) node [midway,left] {$\varphi$};

\end{tikzpicture}
\begin{tikzpicture}

    \node (E) at (0,0) {$D_0$ };
    \node[right=of E] (F) at (4,0){$D_0 $};
    \node[below=of F] (A) {$\cd$};
    \node[below=of E] (Asubt) {$\cd$};

   \draw[->] (E)--(F) node [midway,above] {$\rho$};
    \draw[->] (F)--(A) node [midway,right] {$\varphi$}
                node [midway,left] {};
    \draw[->] (Asubt)--(A) node [midway,below] {$\wh{\rho}(w)=e^{2\pi i\al}w$}
                node [midway,above] {};
    \draw[->] (E)--(Asubt) node [midway,left] {$\varphi$};

\end{tikzpicture},
\end{center}
where $\wh{f}(w)=\varphi\circ \wt{f}\circ \varphi^{-1}(w)$ is defined on the left and $\rho(z)=\varphi^{-1}\circ\wh{\rho}\circ\varphi(z)$ is defined on the right, and
$\wh{f}(w)=\wh{\rho}(w)$ on $\cd$.

\subsubsection{Lebesgue density points on the boundary of a Siegel disk}

\begin{proposition}\cite[Proposition 4.9]{McMullen3ManfildRenor}\label{equ2021-2-16-11}
Let $\iota:X\hookrightarrow Y$ be an inclusion of one hyperbolic Riemann
surface into another, and let $s=d(x,Y-X)$. Then with respect to the hyperbolic
metrics on $X$ and $Y$,
\beqq
\|\iota'(x)\|<C(s)<1,
\eeqq
where $C(s)$ decreases to zero as $s\to0$.
\end{proposition}

For convenience, let $A=O(B)$ denote $A<C B$ for some implicit positive constant $C$,
and $A \asymp B$ denote $B / C<A<C B$ for some implicit positive constant $C$.
Let
$$d(z_1, z_2)=|z_1-z_2|$$
denote the Euclid distance between $z_1$ and $z_2$, and
$$d_{\widehat{\cc}}(z_1,z_2)=\frac{2|z_1-z_2|}{\sqrt{1+z^2_1}\sqrt{1+z^2_2}}$$
denote the hyperbolic metric on $\widehat{\cc}=\cc\cup\{\infty\}$.

We define $\Om=\wh{\cc}_z\setminus\ol{D_0}$ and $\wt{\Om}=\varphi(\Om)=\wh{\cc}_w\setminus\ol{\cd}$.
For any pointed disk $(U, u) \subset \mathbb{C},$ set
$$r_{in}(U, u):=\sup \{r: \cd(u, r) \subset U\}\ \text{and}\ r_{out}(U, u):=\inf \{r: \cd(u, r) \supset U\},$$
where $\cd(u, r)$ is the Euclidean disk centered at $u$ with radius $r$.

We define the hyperbolic structure on $\Om$. The thick part of $\Om$ is the region where the injectivity radius under the hyperbolic
metric exceeds some small constant $\ep_0$, otherwise it is in the thin part. It is evident that $\Om$ is a punctured disk, implying that the thin part is simply standard horoball
neighborhood of the cusp at $\infty$. The constant $\ep_0$ is chosen such that the Julia set of $\wt{f}$ is contained in the thick part.

The hyperbolic metric $\rho_{\Om}$ in the thick part is comparable to the $(1/d)-$metric (\cite[Theorem 2.3]{McMullen3ManfildRenor}):
\beqq
\rho_{\Om}(z)|dz|\asymp\frac{|dz|}{d(z,\partial\Om)}.
\eeqq

Let $\al$ be a bounded type irrational number and $\de>0$ be a given constant. Set
\beqq
K_{\de}(\wt{f}):=\{z:\ \text{Orb}(z,\wt{f})\subset\cd(K(\wt{f}),\de)\},
\eeqq
where $K(\wt{f})$ is filled-in Julia set of $\wt{f}$.

\begin{lemma}\cite[Lemma 3.4]{Shen2006}
For sufficiently small $\wh{r}>0$ and $\wh{z}\in\partial\cd$, there are neighborhoods
$\wh{U}$ of $\wh{z}$, $\wh{V}$ of $1$ in $\wh{\cc}_w$, and a point $\wh{y}\in\wh{U}$ such that for some
$j\geq0$,
\beqq
\wh{f}^{j}: (\wh{U},\wh{y})\to (\wh{V},1)
\eeqq
is bijective and satisfies
\begin{itemize}
\item[(1)] $d_{\wh{\Om}}(\wh{f}^j(\cdot),\wh{\rho}^j(\cdot))=O(1)$ for any point in
$\wh{\Om}\cap\wh{U}$, where $d_{\wh{\Om}}$ is the hyperbolic metric on $\wh{\Om}$;
\item[(2)] $r_{in}(\wh{U},\wh{y})\asymp r_{out}(\wh{U},\wh{y})\asymp \wh{r}$ in the Euclidean metric;
\item[(3)] $d_{\wh{U}}(\wh{y},\wh{z})=O(1)$ in the hyperbolic metric on $\wh{U}$;
\item[(4)] the points $0$, $\tfrac{1}{2m+1}$, $2m+1$ are not in $\wh{U}$.
\end{itemize}
\end{lemma}

\begin{lemma}\cite[Lemma 3.5]{Shen2006}
For small enough $r>0$ and $z\in\partial D_0$, there exist neighborhoods $U$ of $z$, $V$ of $c_0=-\tfrac{m}{m+1}$ in $\wh{\cc}_z$, and a point
$y\in U$ such that for some $j\geq0$,
\beqq
\wt{f}^j: (U,y)\to (V,c_0)
\eeqq
is a univalent map and satisfies:
\begin{itemize}
\item[(1)] $d_{\Om}(\wt{f}^j(\cdot),\rho^j(\cdot))=O(1)$ for any point in $U\cap\Om$;
\item[(2)] $r_{in}(U,y)\asymp r_{out}(U,y)\asymp r$;
\item[(3)] $d_{U}(y,z)=O(1)$;
\item[(4)] the points $0$ and $-m$ are not in $U$.
\end{itemize}
\end{lemma}

\begin{lemma}\cite[Lemma 3.6]{Shen2006}
For any $z\in J(\wt{f})\setminus P(\wt{f})$, there exists a point $y\in\Om$, a hyperbolic disk $B\subset\Om$ centered at $y$ and a
neighborhood $B'$ of $c_0$, such that for some $j\geq0$,
\beqq
\wt{f}^j: (B,y)\to (B',c_0)
\eeqq
is a univalent map, where $\text{diam}_{\Om}B\asymp 1$ and $d_{\Om}(z,y)=O(1)$.
\end{lemma}

\begin{lemma}\cite[Lemma 3.7]{Shen2006}
For any $z\in J(\wt{f})$ and $r>0$, there exist an Euclidean disk $U$ centered at $y$ and a univalent map
\beqq
\wt{f}^j:(U,y)\to (V,c_0)\ \text{for some}\ j\geq0
\eeqq
such that the Euclidean radius $r(U)\asymp r$ and $|y-z|=O(r)$.
\end{lemma}

\begin{theorem}
The critical point $c_0=-\tfrac{m}{m+1}$ is deep point of $K(\wt{f})$. Furthermore, for any given $\de>0$ and any $z\in J(\wt{f})$, $z$ is
a deep point of $K_{\de}(\wt{f})$.
\end{theorem}

\begin{proof}
The idea of the proof follows the arguments used in the proof of Theorem 4.2 in \cite{McMullen1998}.

Consider a point $z$ near the critical point $c_0$ such that the forward iterates of $z$ escapes to infinity.
Each time the orbit comes back near the critical point, it will become closer to the critical point, if it comes back to the critical point,
it will shadow the partial orbit of the critical value for a while, where the orbit of the critical value is a rotation on the boundary of
the Siegel disk, it moves near the critical point.

The points, not far from the Siegel disk and escaping to infinity, must visit the critical point $c_0$.
Recall $D_0$ is the Siegel disk of $\wt{f}$. It is evident that $D_0\subset \wt{f}^{-1}(D_0)$. Let $D'$ be its pre-image, which is not $D_0$.
On each visit, there is a definite chance of landing in the
pre-image $D'$ of $D_0$. Thus a random point close to $\partial D_0$ has a high probability
of eventually landing in $D'$, implying that the density of the filled-in Julia set $K(\wt{f})$ tends to one. The study of $\wt{f}$
near $\partial D_0$ is facilitated by working in the linear coordinate system, a quasi-conformal
chart in which $D_0$ is the unit disk, $\wt{f}|D_0$ is a rigid rotation, and the iterates of $\wt{f}$ are
uniformly quasi-regular.

It is obvious that
\beqq
\wt{f}:\ (\cc-\ol{D_0}\cap\ol{D'})\to(\cc-\ol{D_0})=\Om
\eeqq
is a covering map, and hence an isometry between the respective hyperbolic metrics. This,
together with Proposition \ref{equ2021-2-16-11}, implies the expansion in the hyperbolic metrics.

Other arguments are similar with those used in the proof of Theorem 4.2 in \cite{McMullen1998}.

\end{proof}

\section{Open problem} \label{julia-set-equ-2022-2-12-4}

Some interesting problems are collected in this section.

\begin{question}
For $m\in\{4,5,6,...,19,20,21\}$, the same parabolic and near-parabolic bifurcation should be obtained similarly, but the corresponding constants should be chosen carefully. We leave this for further work. The case $m=3$ was already studied in \cite{QiaoQu2020}.
\end{question}

\begin{question}
The choices of some constants play an important role in the whole arguments, is it possible to find some optimal constants for the whole discussions?
\end{question}

\begin{question}
The parabolic and near-parabolic bifurcation for general high degree polynomials is incomplete. There are many interesting problems, for example, is it possible to find other classes of functions as the study of the existence of Julia sets with positive area?
\end{question}

\begin{question}
By combining the Blaschke product in \eqref{equ2021-2-3-2}, quasi-conformal surgery, and the approach and the method used in Peterson's work \cite{Peterson1996}, one might be able to show that the Lebesgue measure of the Julia set of $P_{\al}(z)$ is zero, where $\al$ is an irrational number of bounded type.
\end{question}

\begin{question}
Inspired by the well-known result of Yoccoz, if an irrational number $\al$ is not Brjuno, then the quadratic map $g(z)=e^{2i\pi\al}z+z^2$ dose not have a Siegel disk at the origin, we guess that if the irrational number $\al$ is not Brujuno, then the map $f(z)=e^{2i\pi\al}z(1+z)^m$ does not have a Siegel disk at the origin either.
\end{question}

\end{document}